\documentclass{article}

\usepackage{amsmath, amsthm, amssymb, dsfont, mathtools, mathrsfs, nccmath}
\usepackage{graphicx}
\usepackage{natbib}
\usepackage[inline]{enumitem}
\usepackage{hyperref}
\usepackage[margin=1.4in]{geometry}
\hypersetup{colorlinks,citecolor=blue,urlcolor=blue,linkcolor=blue}
\usepackage{nikos_tex}
\usepackage{array}  
\usepackage{booktabs, tabu}
\usepackage[table]{xcolor}

\usepackage{multirow}
\usepackage{makecell}

\usepackage[utf8]{inputenc}
\usepackage[english]{babel}

\usepackage{tcolorbox}

\graphicspath{{./figures/}}

\theoremstyle{definition}
\newtheorem{prop}{Proposition}

\newtheorem{coro}[prop]{Corollary}
\newtheorem{lemm}[prop]{Lemma}
\newtheorem{theo}[prop]{Theorem}
\newtheorem{exam}[prop]{Example}
\newtheorem{rema}[prop]{Remark}

\date{Draft Manuscript, October 2026}

\title{Stein's unbiased risk estimate\\ and Hyv{\"a}rinen's score matching}

\author{
\begin{tabular}{ll}
Sulagna Ghosh & \hspace{1cm} Nikolaos Ignatiadis\\
\texttt{sulagnag@uchicago.edu} & \hspace{1cm} \texttt{ignat@uchicago.edu}\\
\vspace{0.05cm}\\
Frederic Koehler & \hspace{1cm} Amber Lee\\
\texttt{fkoehler@uchicago.edu} & \hspace{1cm} \texttt{amberlee0516@uchicago.edu}\\
\vspace{0.1cm}\\
\end{tabular}
}

\begin{document}

\maketitle

\begin{abstract}
Given a collection of observed signals corrupted with Gaussian noise, how can we learn to optimally denoise them? This fundamental problem arises in both empirical Bayes and generative modeling. In empirical Bayes, the predominant approach is via nonparametric maximum likelihood estimation (NPMLE), while in generative modeling, score matching (SM) methods have proven very successful. In our setting, Hyvärinen's implicit SM is equivalent to another classical idea from statistics---Stein's Unbiased Risk Estimate (SURE).  Revisiting SURE minimization, we establish, for the first time, that SURE achieves nearly parametric rates of convergence of the regret in the classical empirical Bayes setting with homoscedastic noise. We also prove that SURE-training can achieve fast rates of convergence to the oracle denoiser in a commonly studied misspecified model. We show how to practically implement our method in settings involving heteroscedasticity and side-information, such as in an application to the estimation of economic mobility in the Opportunity Atlas. Our empirical results demonstrate the superior performance of SURE-training over NPMLE under misspecification. Collectively, our findings advance SURE/SM as a strong alternative to the NPMLE for empirical Bayes problems in both theory and practice.\\

\end{abstract}

\section{Introduction} 
\label{sec:introduction}

Minimization of Stein's Unbiased Risk Estimate \citep[SURE]{stein1981estimation} is one of the most successful statistical strategies for model selection and parameter tuning in the Gaussian sequence model and beyond. SURE often plays a role similar to information criteria and cross-validation for tuning hyperparameters \citep{li1985stein, donoho1995adapting}. Here, instead, we use SURE as our primary loss function and minimize it over a nonparametric class of denoisers. 
This loss is equivalent to Hyv{\"a}rinen's \citeyearpar{hyvarinen2005estimation} score matching (SM) objective \citep{raphan2006learning,vincent2011connection}. Our key message is that SURE-training (equivalently, SM) should often replace nonparametric maximum likelihood estimation (NPMLE) in empirical Bayes denoising applications: while NPMLE optimizes the wrong objective under misspecification of the prior, SURE directly targets mean squared error, yielding superior performance both in theory and in practice for settings with side-information and heteroscedasticity.

To preview our methodology, consider (temporarily) the simplest empirical Bayes setting with
\begin{equation}
\label{eq:gaussian_EB}
\mu_i \simiid G_{\star},\;\; Z_i \mid \mu_i \simindep \mathrm{N}(\mu_i, 1),\;\;i=1,\dotsc,n,
\end{equation}
where the prior distribution $G_{\star}$ is unknown---our proposed methodology accommodates a generalization of~\eqref{eq:gaussian_EB} with side-information and heteroscedasticity (Section~\ref{sec:general_setting}) { and also extends to multivariate denoising (Section~\ref{subsec:multivariate_sure_pm})}.
While SURE applies to a broad class of denoisers, we focus on absolutely continuous denoisers of the form $t(z) = z+ h(z)$, whose mean squared error $\MSE(t(\cdot)) := \sum_{i=1}^n \EE{( t(Z_i) - \mu_i)^2}/n$ may be unbiasedly estimated via,
\begin{equation}
\label{eq:SURE_homosc_no_covariates}
\SURE(t(\cdot)) := 1+ \frac{1}{n}\sum_{i=1}^n \cb{ h(Z_i)^2 +2 \frac{\partial}{\partial z}h(Z_i)}\,\, \text{with}\,\, t(z) = z + h(z),
\end{equation}
so that $\EE{\SURE(t(\cdot))} =  \MSE(t(\cdot))$. If we knew $G_{\star}$ in~\eqref{eq:gaussian_EB}, then the optimal MSE denoiser would be equal to the posterior mean, $t_{\star}(z) = \EE[G_{\star}]{\mu_i \mid Z_i=z}$.
In an empirical Bayesian analysis~\citep{robbins1956empirical, efron2010largescale}, we do not know $G_{\star}$, but seek to learn about $G_{\star}$ via $Z_1,\dotsc,Z_n$ and to subsequently imitate the oracle Bayes estimator. Our empirical Bayes proposal is as follows: each distribution $G$ implies a denoiser via $t(z) = \EE[G]{\mu_i \mid Z_i=z}$ and so we can estimate $G_{\star}$ by minimizing $\SURE$ in~\eqref{eq:SURE_homosc_no_covariates} over the implied Bayes denoiser of any candidate $G$, that is,
$
 \widehat{G} \in \argmin_{G \in \mathcal{G}} \cb{\SURE(G)}$, $\SURE(G) := \SURE( \EE[G]{\mu_i \mid Z_i=\cdot}),
$
where $\mathcal{G}$ is a class of distributions that may (or may not) contain $G_{\star}$. Finally, we estimate $\mu_i$ via $\hat{\mu}_i = \EE[\widehat{G}]{\mu_i \mid Z_i}$.

The connection to Hyv{\"a}rinen's SM is as follows. Via the Eddington/Tweedie formula~\citep{dyson1926method, efron2011tweedie}, the Bayes denoiser may be expressed in terms of the score $s_G(z)$, that is, the derivative of the log-marginal density $f_G(z)$ of the observations $Z_i$ (when $\mu_i \sim G)$:
\begin{equation}
\label{eq:score_and_marginal}
\EE[G]{\mu_i \mid Z_i=z} = z + \score_G(z), \text{ with } \score_G(z) := \frac{\partial}{\partial z} \log f_G(z),\, f_G(z) := \!\int\!\!  \varphi(z-\mu)G(d\mu),
\end{equation}
and where $\varphi$ is the standard Gaussian density function. Thus, SURE in~\eqref{eq:SURE_homosc_no_covariates} takes the form:
\begin{equation}
\label{eq:simple_SURE}
\SURE(G) = 1+\frac{1}{n}\sum_{i=1}^n \cb{ s_G(Z_i)^2 +2 \frac{\partial}{\partial z}s_G(Z_i)}.
\end{equation}
In turn, this objective is (up to an additive constant) equal to the score matching (SM) objective of~\citet{hyvarinen2005estimation}, { $\SM(G) := n^{-1}\sum_{i=1}^n \{ \score_G(Z_i)^2 +2 \tfrac{\partial}{\partial z}\score_G(Z_i)\}$}. Score matching was introduced by~\citet{hyvarinen2005estimation} as an alternative to maximum likelihood estimation of parametric models that does not require computation of the normalizing constant (since the score does not depend on it). SM may also be used in nonparametric settings and is closely tied to the Fisher divergence,
\begin{equation}
\label{eq:fisher_divergence}
\Dfisher{f_{G_{\star}}}{f_G} := \int \cb{s_{G_\star}(z) - s_{G}(z)}^2 f_{G_\star}(z)dz.
\end{equation}
In particular, $\EE{\SM(G)} =\Dfisher{f_{G_{\star}}}{f_G}+c'$ (for a constant $c'$), and so $\SM$ can be used as a loss function for estimating the score $s_{G_{\star}}$ of $f_{G_{\star}}$. Thus, our empirical Bayes approach via SURE simultaneously provides a method for nonparametric score estimation.

Estimating priors by SURE in empirical Bayes settings is not a new idea. \citet{tibshirani2019excess} explain that minimizing SURE recovers a minor modification of the classical estimator of~\citet{james1961estimation}. { \citet*{xie2012sure} advocate for SURE-training in heteroscedastic Gaussian mean estimation, establishing its asymptotic optimality within classes of linear shrinkage estimators and demonstrating lower mean squared error than maximum likelihood in simulations under prior misspecification.} As we describe further in the related work (Section~\ref{subsec:related_work}) such methods have become quite popular in both statistics and economics. However, the current state of affairs, as noted by~\citet{chen2026empirical}, is that SURE-training methods only apply to restricted classes of priors. In this paper, we seek to advance SURE-training as a compelling alternative to the NPMLE for empirical Bayes through the following contributions:
\begin{itemize}[noitemsep, leftmargin=*]
\item We develop SURE-PM (Particle Modeling, Section~\ref{sec:fast_rates}), which uses as its working model $\mu_i \sim G$ with $G$ fully nonparametric but not incorporating side-information.
In the homoscedastic Gaussian sequence model with fully nonparametric and compactly supported prior, we prove in Theorem~\ref{theo:rate_homoscedastic}  that the regret of SURE-PM decays at a rate $O(1/n)$ (up to logarithmic factors), thus matching the celebrated rate shown for the NPMLE by~\citet{jiang2009general}. { We also provide regret guarantees for SURE-PM with light-tailed priors and for the multivariate Gaussian sequence model with $\mu_i, Z_i \in \mathbb R^d$ for $d \in \mathbb N$.}
\item We introduce SURE-LS (Least Squares, Section~\ref{sec:regression}) with working priors that are Gaussian with mean and variance as nonparametric functions of side-information, deriving sharp oracle inequalities with fast rates that hold under misspecification of both prior and noise models (Theorem~\ref{theo:reg}). This encompasses and improves upon existing results including~\citet{xie2012sure, weinstein2018grouplinear, ignatiadis2019covariatepowered}. { By contrast, Example~\ref{exam:mle-suboptimal} gives a misspecified Gaussian working model in which the denoiser based on maximum likelihood has limiting mean squared error strictly above that of the best denoiser in the same working class.}
\item Our proofs develop new connections between the functional inequality/score-based diffusion model literature and M-estimation. 
As a consequence of the equivalence of SM and SURE-training, our Theorem~\ref{theo:rate_homoscedastic} provides the first instance of minimax rate optimality (up to logarithmic factors) of Hyv{\"a}rinen's SM in a nonparametric setting in which the $\SM$ solution does not have a closed-form representation. {We summarize the key proof elements in Section~\ref{sec:proof_elements}.}
\item We provide practical PyTorch implementations (Section~\ref{sec:computation}) of SURE-PM, SURE-LS, and a third method called SURE-THING (This Helps In Neural-G modeling). SURE-THING uses neural networks to learn fully nonparametric priors whose shape, location, and scale is modulated by side-information in a flexible way. 
\item { We demonstrate strong empirical performance for our proposals in simulations (Section~\ref{sec:simulations}), where SURE-trained methods often achieve lower mean squared error than the NPMLE under prior misspecification.} We also apply our methods to measuring economic mobility for Black children from low-income families in the Opportunity Atlas \citep{chetty2026opportunity} (Section~\ref{sec:opportunity}). In this application, Census tracts with larger Black populations have both lower mobility and smaller standard errors~\citep{chen2026empirical}, violating the mean-variance independence assumption of the standard NPMLE. { SURE-PM embeds the same assumption in its working model, yet achieves substantially lower estimated mean squared error than the NPMLE in our empirical comparison.} SURE-LS and SURE-THING directly model the heteroscedasticity and side-information for further gains.
\end{itemize}
Overall, our paper establishes SURE-based training via end-to-end differentiable programming as a promising and theoretically grounded strategy for empirical Bayes estimation.

\subsection{Related Work}
\label{subsec:related_work}
\paragraph{Connection to empirical Bayes literature.}
\citet{efron2014two} classifies empirical Bayes estimation strategies as G- or F-modeling based. In G-modeling, one seeks to directly estimate $G_{\star}$ in~\eqref{eq:gaussian_EB}, by \smash{$\hG$}, say, and the empirical Bayes analysis proceeds by plugging-in \smash{$\hG$} in place of the true $G$~\citep{jiang2009general, efron2016empirical}. By contrast, in F-modeling, one starts by estimating properties of the marginal distribution $F_{G_{\star}}$ of $G_{\star}$, for instance, by estimating the marginal density $f_{G_{\star}}$ defined in~\eqref{eq:score_and_marginal} by a kernel density estimator \smash{$\hat{f}$}. Then one can approximate the Bayes denoiser via \smash{$\hEE{\mu_i \mid Z_i=z} = z + \{\tfrac{\partial}{\partial z} \hat{f}(z)\}/\hat{f}(z)$}~\citep{zhang1997empirical, li2005convergence,  raphan2007empirical, brown2009nonparametric}. We view our approach as a G-modeling strategy that uses SURE as the objective. 

Perhaps closest to our approach is the paper by~\citet{zhao2021simultaneous}, which in its Section 5.1 proposes an approach for the homoscedastic Gaussian problem in~\eqref{eq:gaussian_EB} that is effectively the same as our proposal (see Supplement~\ref{subsec:homosc_simulations} for more detail on the connection); however, without accompanying theory.~\citet{zhao2021simultaneous} also considers a setting with side-information, that is a special case of the model that we introduce in equation~\eqref{eq:EB_heterosc_side_info} of Section~\ref{sec:general_setting}. The focus therein is on arguing consistency (that is, that the empirical Bayes regret converges to zero), however, the proof technique used can only establish a rate of $O(n^{-1/2})$ (up to log factors), rather than the rate $O(n^{-1})$ we establish in  Theorem~\ref{theo:rate_homoscedastic}. Further important references that may be interpreted as G-modeling empirical Bayes facilitated by SURE include~\citet{jiang2011best, xie2012sure, tan2016steinized, kou2017optimal, zhang2017empirical, ignatiadis2019covariatepowered, rosenman2023combining, li2025predictionpowered,
kwon2023optimal}, however, as argued by~\citet{chen2026empirical} these methods apply only to a restricted class of models. \citet{cohen2013empirical} deal with side-information by minimizing an asymptotically consistent estimate of the risk (which is different from SURE).

Going beyond G- and F-modeling,~\citet[Chapter 4]{biscarri2019statistical} and~\citet{barbehenn2023nonparametric} introduce the term E-modeling for approaches that directly model the posterior expectation $\EE{\mu_i \mid Z_i=z}$.~\citet{zhao2022regression} and~\citet{barbehenn2023nonparametric} use SURE as their objective function and directly model the posterior mean (rather than the prior). From a learning theory perspective, E-modeling implements improper learning~\citep{daniely2014average}, as it models the posterior mean without enforcing that it arises from a proper prior. By contrast, G-modeling implements proper learning by explicitly modeling a valid prior distribution.
The rates derived by~\citet{barbehenn2023nonparametric} for the empirical Bayes regret are of order $O(n^{-1/2})$, and so suboptimal. See Supplement~\ref{subsec:emodeling} for further bibliographic comments on E-modeling.

\paragraph{Connection to score matching literature.} \citet{hyvarinen2005estimation} first introduced score matching for estimation in parametric models with an intractable (or expensive to compute) normalizing constant. The SM objective is not limited to parametric settings---SM can be traced back to~\citet{cox1985penalty} in nonparametric contexts, and it has since been applied in various nonparametric settings~\citep{kingma2010regularized, sasaki2014clustering, strathmann2015gradientfree, sriperumbudur2017density, feng2024optimal}. Nonetheless, \citet{vincent2011connection} laments that ``what happens in the transition from $J_{\text{ISM}_q}$ [$\Dfisher{\cdot}{\cdot}$ in our notation] to finite-sample version $J_{\text{ISM}_{q0}}$ [$\SM(G)$ in our notation] is, however, not entirely clear,'' and arguably we are still far from understanding finite-sample properties of Hyv{\"a}rinen's SM. For instance, as far as we know, the only result on minimax rate optimality of Hyv{\"a}rinen's SM  in a nonparametric setting is available for certain infinite dimensional kernel exponential families studied by~\citet{sriperumbudur2017density}. In the latter case, (regularized) SM permits an explicit representation as the Tikhonov-regularized solution of an explicit linear system and this representation facilitates analysis. Our~Theorem~\ref{theo:rate_homoscedastic} provides the first instance of minimax rate optimality (up to logarithmic factors) of Hyv{\"a}rinen's SM in a nonparametric setting in which the $\SM$ solution does not have a closed-form representation. 

Several recent works have studied the statistical properties of Hyv{\"a}rinen's SM and its computational and statistical tradeoffs versus maximum likelihood estimation for density estimation, especially in parametric settings. See, e.g., \citet{forbes2015linear,koehler2023statistical,pabbaraju2024provable,koehler2023sampling,qin2024fit,koehler2024efficiently,chewi2025ddpm}. In contrast to those works, for us the main goal  is learning to denoise (and learning the score function) instead of density estimation in Kullback-Leibler divergence; however, we build upon techniques in these works, especially the connections between score matching and isoperimetric/functional inequalities.

The importance of estimating the score has gained prominence through the success of score-based generative modeling~\citep{song2019generative} and diffusion models~\citep{ho2020denoising}, spurring new theoretical investigations. A central question in this area is how well one can estimate the score $s_{G_\star}$ of the convolution of $\mathrm{N}(0,\sigma^2)$ and $G_{\star}$ from 
samples of $G_{\star}$~\citep{dou2024optimal,zhang2024minimax}. Related results are derived by~\citet{wibisono2024optimal}, who focus on minimax estimation of the score of $G_{\star}$ itself when $G_{\star}$ satisfies certain smoothness properties (e.g., that its score exists and is H\"{o}lder continuous). A further strand of the literature has considered estimating the score through denoising score matching~\citep{vincent2011connection, saremi2018deep, block2022generative, oko2023diffusion}, which we discuss further in Remark~\ref{rema:denoising_score_matching} below. Our work differs from these studies: while they assume direct observations of $\mu_i$ in~\eqref{eq:gaussian_EB}, in our setting we only require access to the noisy observations $Z_i$.

Finally, we highlight the work of~\citet*{feng2024optimal} who apply nonparametric SM to estimate the antitonic (decreasing) score projection. The score is then used as a data-driven loss for M-estimation of coefficients in linear regression. Remarkably, the procedure asymptotically attains minimal variance among all convex M-estimators. A key message of both~\citet{feng2024optimal} and our paper is that the choice of objective for nonparametric estimation under misspecification depends crucially on the downstream task. For~\citet{feng2024optimal}, M-estimation of regression coefficients requires Hyv{\"a}rinen's SM criterion rather than maximum likelihood. Similarly, for our task of denoising $\mu_i$, the right objective under misspecification is Hyv{\"a}rinen's SM (equivalently, SURE), not maximum likelihood. This message (in the case of denoising) is also emphasized by~\citet{hyvarinen2008optimal} and~\citet{xie2012sure}.

\section{The general setting}
\label{sec:general_setting}
We start by extending the scope of~\eqref{eq:gaussian_EB}. We allow for side-information encoded via covariates $X_i$ taking values in a generic space $\mathcal{X}$, and heteroscedasticity (where each observation has its own noise variance $\sigma_i^2$). Our model is as follows,
\begin{equation}
\label{eq:EB_heterosc_side_info}
\mu_i \mid X_i \simindep G_{\star}(\cdot \mid X_i),\;\;\;\;\quad Z_i \mid \mu_i, X_i \simindep \mathrm{N}(\mu_i, \sigma_i^2),\;\;\;\;\quad i=1,\dotsc,n,
\end{equation}
and we observe $W_i := (Z_i, X_i)$ but not $\mu_i$. Above, $G_{\star}(\cdot \mid X_i)$ represents the unknown conditional prior distribution of the parameter $\mu_i$ given covariates $X_i$. As is common in the literature, we assume that the noise variance $\sigma_i^2$ for each $i$ is known exactly~\citep{xie2012sure, weinstein2018grouplinear,  soloff2024multivariate}. 
In our notation, we suppress $\sigma_i^2$ by absorbing it into $X_i$ (that is, by concatenating $\sigma_i^2$ with the original covariates).
Throughout this paper, we treat \smash{$\boldX:=(X_1,\dotsc,X_n)$} as fixed (similar to fixed-$X$ regression~\citep{rosset2020fixedx}).\footnote{Our results can also be stated in a random-X setting with $X_i \simiid \mathbb P^X$ for a covariate distribution $\mathbb P^X$.} 
{ The framework also extends to multivariate $Z_i$. We focus on univariate $Z_i$ here, but develop a special case of the multivariate framework in Section~\ref{subsec:multivariate_sure_pm}; \citet{kwon2023optimal} studies another.}

To streamline exposition, we treat $\mu_i, Z_i$ as random, generated via~\eqref{eq:EB_heterosc_side_info}. Our framework also accommodates treating \smash{$\boldmu := (\mu_1, \dotsc, \mu_n)$} as fixed, following compound decision theory~\citep{robbins1951asymptotically, zhang2003compound}. Indeed, below we  state Theorem~\ref{theo:reg} (Section~\ref{sec:regression}) for fixed $\boldmu$.
When $\boldmu$ is fixed, the only source of randomness is \smash{$Z_i \mid \boldmu, \boldX \sim \mathrm{N}(\mu_i, \sigma_i^2)$}. For clarity, in contexts where $\boldmu$ is treated as fixed, we will denote expectations as \smash{$\EE[\boldmu]{\cdot}$}, or as \smash{$\EE[\mu_i]{\cdot}$} when the dependence on \smash{$\boldmu$} is only through $\mu_i$.

In the homoscedastic setting without covariates, say, with $X_i = \sigma_i^2=1$ for all $i$, the model in~\eqref{eq:EB_heterosc_side_info} collapses to~\eqref{eq:gaussian_EB}.
Inclusion of covariates means that our framework also encapsulates the more traditional regression setting: take the conditional distribution $G_{\star}(\cdot \mid X_i)$ to be equal to a Dirac point mass at $m_{\star}(X_i)$, where $m_{\star}: \mathcal{X} \to \RR$ is a function. Then $\mu_i = m_{\star}(X_i)$ almost surely and we can rewrite~\eqref{eq:EB_heterosc_side_info} as the regression problem
$
Z_i \mid X_i \sim \mathrm{N}(m_{\star}(X_i), \sigma_i^2).
$
We refer to e.g.,~\citet{fayiii1979estimates, cohen2013empirical, ignatiadis2019covariatepowered} and Section~\ref{sec:regression} below for more discussion on empirical Bayes estimation with covariates.

In this more general setting, when writing ``$G$'' (with some abuse of notation) we refer to the conditional distributions $\{G(\cdot \mid x): x \in \mathcal{X}\}$.\footnote{Given the fixed-X setting, these need to be specified only for $x \in \cb{X_1,\dotsc,X_n}.$} Analogously we also write $\mathcal{G}$ for a class of such conditional distributions. 
Below we slowly unpack a suitable generalization of definitions and the Eddington/Tweedie formula in~\eqref{eq:score_and_marginal}. First, given any $G$, we define the marginal density of $Z_i$ given $X_i=x$,
\begin{equation}
f_G(z \mid x) := \int \varphi(z-\mu; \sigma^2) G(d\mu \mid x),
\label{eq:conditional_marginal}
\end{equation}
where $\varphi(z; \sigma^2)$ is the density function of a centered Gaussian with variance $\sigma^2$. Recalling our notational convention that $X_i$ includes $\sigma_i^2$, we see that $\sigma_i^2$ influences the marginal density in two ways: first, it determines the noise level (variance, $\sigma_i^2)$ of the Gaussian kernel used in the convolution, and second, it may influence the distribution of $\mu_i$ via the conditional distribution $G(\cdot \mid X_i)$. Other covariates in $X_i$ only influence the marginal density of $Z_i$ via the latter mechanism (that is, via $G(\cdot \mid X_i)$). We also define the conditional score given $x$ and use shorthand notation for the conditional score of the data-generating prior $G_{\star}$,
\begin{equation}
\score_G(w) \equiv \score_G(z,x) := \frac{\partial}{\partial z} \log f_G(z \mid x),\quad s_{\star}(w) := s_{G_\star}(w),
\label{eq:conditional_score}
\end{equation}
where $w = (z,x)$.
With these definitions in place, we can verify the following Eddington/Tweedie type formula that generalizes~\eqref{eq:score_and_marginal},\footnote{
Expectations of the form $\EE[G]{\cdot}$ with a prior $G$ as a subscript indicate that we are integrating over both $\mu_i \sim G(\cdot \mid X_i)$ and $Z_i \mid \mu_i, X_i \sim \mathrm{N}(\mu_i, \sigma_i^2)$.
}
\begin{equation}
\EE[G]{\mu_i \mid W_i=w} = \EE[G]{\mu_i \mid Z_i=z, X_i=x}  = z + \sigma^2 \score_G(w),
\label{eq:general_tweedie}
\end{equation}
where $\sigma^2$ is implicitly a function of $x$.

\subsection{Population risks, SURE, and Score Matching}
We now provide more details on some results highlighted in the introduction (Section~\ref{sec:introduction}). Our results herein are known, see e.g.,~\citet{raphan2006learning, raphan2011least, vincent2011connection}. However,
we provide a self-contained exposition and also describe all results in the more general setting laid out after~\eqref{eq:EB_heterosc_side_info} with side-information and heteroscedasticity.

We start by defining risks of interest. We use the following notational convention: we write \smash{$(X, \mu, Z)$} for a fresh draw that is independent of everything else and generated as follows: first, \smash{$X \sim \mathrm{Unif}(\{X_1,\dotsc,X_n\})$}, and then \smash{$(\mu, Z)$} is generated as in~\eqref{eq:EB_heterosc_side_info} conditioning on $X=X_i$. Writing $W=(Z,X)$, the empirical Bayes regret is defined as
\begin{equation}
\Regret(G_{\star}, G) := \EE[G_{\star}]{ (\mu - \EE[G]{\mu \mid W})^2} - \EE[G_{\star}]{ (\mu - \EE[G_\star]{\mu \mid W})^2}.
\label{eq:ebregret_defi_sideinfo}
\end{equation}
In words, when $\mu$ in~\eqref{eq:EB_heterosc_side_info} is generated according to $G_{\star}$, but we instead use the working prior model $G$, what is the difference in incurred risk? By an orthogonality argument, $\Regret(G_{\star}, G)$ can also be expressed as the mean squared error in estimating the posterior mean $\mathbb E_{G_{\star}}[\mu \mid W]$, that is,
\begin{equation}
\label{eq:regret_as_posterior_mean_mse}
\Regret(G_{\star}, G) = \EE[G_{\star}]{ \p{\EE[G]{\mu \mid W} - \EE[G_{\star}]{\mu \mid W} }^2}.
\end{equation}
In the case without side-information and with homoscedasticity ($\sigma_i^2=1$ for all $i$), the above formula alongside the Eddington/Tweedie formula in~\eqref{eq:general_tweedie} implies that $\Regret(G_{\star}, G)$ is equal to the Fisher divergence $\Dfisher{f_{G_{\star}}}{f_G}$ defined in~\eqref{eq:fisher_divergence}. 
This connection between empirical Bayes regret and Fisher divergence at the population level also holds in the general setting of this section (Supplement~\ref{subsec:fisher_div_ebregret}) and manifests itself in the two (effectively equivalent) estimation strategies we pursue: minimizing SURE and minimizing the score matching (SM) objective. To make this connection as clear as possible, we next present the arguments underlying SURE and SM in a unified way. In each case we start with the risk of interest, e.g., the MSE, $\mathbb E[\{\mu_i - (Z_i + \sigma_i^2 s(W_i))\}^2]$, for SURE and Fisher divergence, $\mathbb E[\{s_\star(W_i) - s(W_i)\}^2]$, for SM. Moreover, in each case we subtract a constant term that does not depend on $s(\cdot)$, namely $\mathbb E[(\mu_i - Z_i)^2]$ for SURE and $\mathbb E[s_\star(W_i)^2]$ for SM.

\[
\begin{array}{@{}c@{\qquad\vrule\qquad}c@{}}
\textsc{SURE} & \textsc{SM} \\[1em]
\begin{aligned}
& \EE{\cb{\mu_i - (Z_i + \sigma_i^2 s(W_i))}^2} - \EE{(\mu_i - Z_i)^2} \\[0.5em]
& = \sigma_i^4 \EE{ s(W_i)^2} + 2\colorbox{blue!15}{$\sigma_i^2\EE{(Z_i-\mu_i)s(W_i)}$} \\[0.5em]
& = \sigma_i^4\EE{ s(W_i)^2} + 2\colorbox{blue!15}{$\sigma_i^4\EE{ \tfrac{\partial}{\partial z}s(W_i)}$}.
\end{aligned}
&
\begin{aligned}
& \EE{\cb{s_\star(W_i) - s(W_i)}^2} - \EE{s_\star(W_i)^2} \\[0.5em]
& = \EE{s(W_i)^2} - 2\colorbox{blue!15}{$\EE{s_\star(W_i)s(W_i)}$} \\[0.5em]
& = \EE{s(W_i)^2} + 2\colorbox{blue!15}{$\EE{\tfrac{\partial}{\partial z}s(W_i)}$}.
\end{aligned}
\end{array}
\]
In each case the crucial argument relies on partial integration. For SURE, suppose that, fixing $X_i$, $z \mapsto s(z,X_i)$ is absolutely continuous and that 
$\mathbb E[\abs{\partial s(Z_i, X_i)/\partial z}] < \infty$. Then the equality $\EE{(Z_i-\mu_i)s(W_i)} =  \sigma_i^2 \mathbb E[\partial s(Z_i, X_i)/\partial z]$ follows by Stein's celebrated Lemma~\citep{stein1981estimation} since $Z_i \mid \mu_i, X_i \sim \mathrm{N}(\mu_i, \sigma_i^2)$. See Supplement~\ref{subsec:partial_integration} for the partial integration argument underlying SM, and its connection to the argument for SURE.

We get the following default versions of SURE, respectively SM:
\begin{align}
\label{eq:SURE}
\SURE(G) &:=\frac{1}{n}\sum_{i=1}^n \sqb{\sigma_i^2 + \sigma_i^4 \cb{ s_G(W_i)^2 +2 \frac{\partial}{\partial z}s_G(W_i)}}, \\ 
\SM(G) &:=  \frac{1}{n}\sum_{i=1}^n \cb{ s_G(W_i)^2 +2 \frac{\partial}{\partial z}s_G(W_i)}.
\end{align}
In the case wherein all $\sigma_i^2$ are identical, the above objectives are equivalent. Otherwise, they correspond to  different objectives. Below, for most of our analyses we focus on~\eqref{eq:SURE}, due to the central role of mean squared error in the empirical Bayes literature~\citep{weinstein2018grouplinear}. However, our results extend to objectives of the form
$
\Loss(G) := n^{-1}\sum_{i=1}^n w(X_i) \{ s_G(W_i)^2 +2 \tfrac{\partial}{\partial z}s_G(W_i)\},
$
with non-negative weights $w(\cdot)$. The SURE objective corresponds to $w(X_i) = \sigma_i^4$, and the vanilla score matching objective to $w(X_i)=1$. For both tasks (empirical Bayes denoising and score matching), we may want to use different weights: for instance, for denoising we may seek optimality with respect to the inverse variance weighted mean squared error as in e.g.,~\citet{banerjee2023nonparametric}, in which case we would choose $w(X_i) = \sigma_i^2$.

\begin{rema}[Denoising score matching]
\label{rema:denoising_score_matching}
The partial integration argument in the derivations above enables us to learn  optimal denoisers for $\mu_i$ without ever observing $\mu_i$ by first learning the score. In denoising score matching~\citep{vincent2011connection, song2019generative} the same partial integration argument is used in reverse: we can learn the score, by learning to denoise $\mu_i$ via supervised regression. This is only applicable in settings wherein we do indeed observe the clean $\mu_i$. However, supervision is actually unnecessary---for this reason, the approach in our paper is also called regression  ``without supervision''~\citep{raphan2006learning} or ``implicit'' score matching \citep{vincent2011connection}.
\end{rema}

\begin{rema}[An alternative expression for SURE]
\label{rema:variance_sure}
A second order generalization of the Eddington/Tweedie formula yields \smash{$\Var[G]{\mu_i \mid W_i} = \sigma_i^2 + \sigma_i^4 \tfrac{\partial}{\partial z}s(W_i)$}, see e.g.,~\citet[Equation (2.8)]{efron2011tweedie}. By this formula, we get that
$$\SURE(G) =\frac{1}{n}\sum_{i=1}^n \cb{-\sigma_i^2 + \p{Z_i-\EE[G]{\mu_i \mid W_i}}^2 +2 \Var[G]{\mu_i \mid W_i}}.$$
The interpretation is as follows: we seek to choose a prior $G$ such that both $\EE[G]{\mu_i \mid W_i} \approx Z_i$ and posterior uncertainty $\Var[G]{\mu_i \mid W_i}$ remains small. The term $\Var[G]{\mu_i \mid W_i}$ acts as a regularizer that encourages more concentrated priors.
\end{rema}

\subsection{SURE-training and uniform convergence}
\label{subsec:uniform_convergence}

Framing SURE-training as a general M-estimation problem~\Citep{geer2000empirical, wainwright2019highdimensional} with a specific loss, we can immediately get guarantees on the risk via standard uniform convergence arguments.
In particular, let \smash{$\hG \in \argmin\{\SURE(G): G \in \mathcal{G}\}$} and suppose that the uniform convergence error
\begin{equation}
\label{eq:sure_uniform_error}
\Delta_n(\mathcal{G}) := \EE{\sup_{G \in \mathcal{G}}\abs{ \SURE(G) - \frac{1}{n}\sum_{i=1}^n \p{\mu_i - \EE[G]{\mu_i \mid W_i}}^2}}
\end{equation}
converges to zero as $n\to\infty$. Then, since $\SURE(\hG) \leq \SURE(G)$ for all $G \in \mathcal{G}$, it follows that,
\begin{equation}
\label{eq:sure_uniform_oracle}
\EE{ \frac{1}{n}\sum_{i=1}^n \p{\mu_i - \EE[\widehat{G}]{\mu_i \mid W_i}}^2 - \inf_{G \in \mathcal{G}} \frac{1}{n}\sum_{i=1}^n \p{\mu_i - \EE[G]{\mu_i \mid W_i}}^2} \leq 2\Delta_n(\mathcal{G}) \to 0.
\end{equation}
The above results hold in both the frequentist setting, where $\boldmu$ is fixed, as well as when integrating over the data-generating prior $G_{\star}$. A corresponding guarantee for $\Regret$ in~\eqref{eq:ebregret_defi_sideinfo}, in the well-specified setting, is given in Proposition~\ref{prop:uniform_control_rademacher} of Supplement~\ref{sec:uniform_convergence_rademacher}.

Uniform convergence without localization has been the predominant approach for analyzing SURE-trained estimators~\citep{li1985stein, li1986asymptotic, li1987asymptotic, xie2012sure,xie2016optimal, zhang2017empirical, kou2017optimal, brown2018empirical, abadie2019choosing, banerjee2020adaptive, zhao2021simultaneous, rosenman2023combining, kwon2023optimal, barbehenn2023nonparametric, li2025predictionpowered}.\footnote{Notable alternative approaches include SURE for SURE~\citep{bellec2021secondorder} and excess optimism~\citep{tibshirani2019excess,cauchois2021comment}.}  This approach has an inherent limitation: it typically yields suboptimal rates rather than the sharp rates we establish in the next section.  Nevertheless, uniform convergence provides insight into SURE-training's fundamental properties. In particular, it guarantees that even under misspecification (when $G_{\star} \notin \mathcal{G}$), asymptotically we will perform at least as well as the denoiser associated to the best possible notional prior $G_{\oracle} \in \mathcal{G}$. 
{ By contrast, Example~\ref{exam:mle-suboptimal} gives a misspecified Gaussian working model in which the denoiser based on maximum likelihood has limiting mean squared error strictly above that of the best denoiser in the same working class.}

In Supplement~\ref{sec:uniform_convergence_rademacher} we provide quantitative versions of the uniform convergence guarantees above stated in terms of Rademacher complexity of suitable function classes. These results allow for nonparametric priors, side-information,
and heteroscedasticity.

\section{SURE-trained methods with sharp statistical guarantees}
\label{sec:stat_guarantees}

In this section, we study two instantiations of our framework for estimating $G_\star$: SURE-PM (Particle Modeling) and SURE-LS (Least Squares). In our framework, an empirical Bayes method is determined by the class of (conditional) distributions $\mathcal{G}$ (as well as the computational implementation; see Section~\ref{sec:computation}). Our goal is to demonstrate that SURE-training enjoys fast regret rates in important settings. { We postpone a discussion of the main ideas behind the proofs to Section~\ref{sec:proof_elements}, with complete arguments appearing in the supplement.}

\subsection{SURE-PM (Particle Modeling)}
\label{sec:fast_rates}

Our first estimator, SURE-PM, posits the working model that ``$\mu_i \indep (\sigma_i^2, X_i)$,''\footnote{{To be precise, since $\sigma_i^2$ and $X_i$ are fixed in our development, by ``$\mu_i \indep (\sigma_i^2, X_i)$'' we mean that the working prior distribution of $\mu_i$ is the same for all $i$ and does not depend on $(\sigma_i^2, X_i)$.}} by specifying $\mathcal{G}^{\mathrm{PM}} :=
\{ G(\,\cdot \mid \sigma_i^2, X_i) = H(\cdot)
\text{ for all } i \,:\, H \in \mathcal{P}(M)\}$, where $\mathcal{P}(M)$ is the class of all univariate distributions that are supported on $[-M,M]$. While this working model embeds a mean-variance independence assumption, a key advantage of SURE-training is that it optimizes the right objective even under misspecification. Indeed, SURE-PM enjoys the guarantees of Section~\ref{subsec:uniform_convergence} under heteroscedasticity without requiring mean-variance independence, unlike the NPMLE whose existing theoretical guarantees~\citep{jiang2020general, soloff2024multivariate} require that ``$\mu_i \indep \sigma_i^2$''---a strong assumption in practice~\citep{chen2026empirical}.

Our next theorem shows that these properties under misspecification
are compatible with fast regret rates in the well-specified setting. Specifically, we consider SURE-PM in the homoscedastic Gaussian means problem of~\eqref{eq:gaussian_EB}, wherein $\SURE$ takes the form~\eqref{eq:simple_SURE}.

\begin{theo}[SURE/SM rate in the homoscedastic Gaussian means problem]
\label{theo:rate_homoscedastic}
Suppose that \smash{$\mu_i \simiid G_{\star}$} with $G_{\star} \in  \mathcal{P}(M)$ and that $\sigma_i^2=1$ for all $i$. Let \smash{$\widehat{G}$} be the minimizer of $\SURE(G)$ (equivalently, $\SM(G)$) over all $G \in \mathcal{P}(M)$. Then, with probability at least $1-n^{-2}$, it holds that,
$$
\Regret(G_{\star}, \hG) \equiv \Dfisher{f_{G_{\star}}}{f_{\hG}}  \leq C  \frac{(1+M)^6 \log^4 n}{n},
$$
where $C >0$ is a universal constant.
\end{theo}
{
For fixed $M$, the logarithmic factor in our bound is faster than the celebrated $O(n^{-1}\log^5 n)$ rate of \citet{jiang2009general}. Recently, \citet{chen2026sharp} showed that the optimal regret bound is $O(n^{-1}(\log n/\log\log n)^2)$, matching the minimax lower bound of \citet{polyanskiy2021sharp}. See Supplement~\ref{subsec:homosc_comparison} for further comparisons with existing results.}

Theorem~\ref{theo:rate_homoscedastic} also advances the theory of Hyv{\"a}rinen's SM. While prior work established rates for parametric models~\citep{forbes2015linear, barp2019minimum} or cases where the SM estimator has an explicit form~\citep{sriperumbudur2017density}, we obtain fast rates for the Fisher divergence without requiring an explicit representation. As a consequence, Theorem~\ref{theo:rate_homoscedastic} provides the first instance of minimax rate optimality (up to logarithmic factors) of Hyv{\"a}rinen's SM in a nonparametric setting where the $\SM$ solution does not have a closed-form representation.

\subsubsection{SURE-PM with light-tailed priors}

The result of Theorem~\ref{theo:rate_homoscedastic} is stated for priors $G_{\star}$ supported on a compact interval $[-M,M]$. Here we show that similar fast rates hold for light-tailed priors, such as sub-Gaussian priors. Specifically, for $\tau>0$, define the class of sub-Gaussian priors with parameter $\tau$,
\begin{equation}
\label{eq:prior_exponential_moment}
\mathcal{SG}(\tau):=\cb{G\;:\;\int\exp(u^2/\tau^2)\,G(du)\leq2}.
\end{equation}
We have the following regret guarantee for SURE-PM.

\begin{coro}
\label{coro:rate_subgaussian}
Suppose that~\eqref{eq:gaussian_EB} holds with $G_\star\in\mathcal{SG}(\tau)$ and $n\geq2$. Let \smash{$\widehat{G}$} be the minimizer of $\SURE(G)$ in~\eqref{eq:simple_SURE} over all $G\in\mathcal{P}(M_n)$, where $M_n:=2\tau\sqrt{\log n}$. Then, with probability at least $1-2n^{-2}$, it holds that,
\begin{equation}
\label{eq:rate_exponential_tails}
\Regret(G_\star,\hG)\equiv\Dfisher{f_{G_\star}}{f_{\hG}}
\leq C\frac{(1+\tau)^6\log^7 n}{n},
\end{equation}
where $C>0$ is a universal constant.
\end{coro}
For a subexponential prior $G_{\star}$ satisfying $\int\exp(|u|/\tau)\,G_\star(du)\leq2$,an analogous argument with $M_n:=4\tau\log n$ gives the bound $\{C'(1+\tau)^6\log^{10}\!n\}/n$ for another $C'$. 
For fixed $\tau$, these rates are worse in terms of logarithmic factors than those of \citet{jiang2009general}, who obtain $n^{-1}\log^5\! n$ and $n^{-1}\log^{11/2}\!n$ for sub-Gaussian and subexponential priors, respectively.

The results of~\citet{jiang2009general} also cover heavier-tailed priors with only finitely many moments. Our current analysis is more brittle to such assumptions and would give slower rates. We do not know whether this is a limitation of our analysis or of SURE-training itself. We conjecture that using the regularized Bayes rules studied by \citet{jiang2009general} may be important for both the practical and theoretical performance of SURE-training with heavier-tailed priors.

\subsubsection{SURE-PM for multivariate Gaussian denoising}
\label{subsec:multivariate_sure_pm}

We now turn to a multivariate extension of our methodology and of Theorem~\ref{theo:rate_homoscedastic}. To set things up, we proceed in parallel to Section~\ref{sec:introduction}. In lieu of the model in~\eqref{eq:gaussian_EB}, we posit that
\begin{equation}
\label{eq:multivariate_gaussian_EB}
\mu_i \simiid G_{\star},\;\; Z_i \mid \mu_i \simindep \mathrm{N}(\mu_i, \Sigma),\;\;i=1,\dotsc,n,
\end{equation}
where $\mu_i \in \RR^d$ are unknown, while $\Sigma \succ 0$ is a known $d\times d$ covariance matrix. For any prior $G$ on $\RR^d$, the Eddington/Tweedie formula in~\eqref{eq:score_and_marginal} becomes
\begin{equation}
\label{eq:multivariate_score_and_marginal}
\EE[G]{\mu_i\mid Z_i=z} = z+\Sigma s_G(z),\;\; s_G(z) := \nabla\log f_G(z),\;\; f_G(z) := \int \varphi_d(z-\mu;\Sigma)G(d\mu),
\end{equation}
where $\varphi_d(\cdot;\Sigma)$ is the density of $\mathrm{N}(0,\Sigma)$ and $\nabla$ denotes the gradient with respect to $z$. Writing $\Norm{\cdot}$ for the Euclidean norm, SURE takes the following form (generalizing~\eqref{eq:simple_SURE}):
\begin{equation}
\label{eq:multivariate_SURE}
\SURE(G) := \tr(\Sigma)+\frac{1}{n}\sum_{i=1}^n \cb{\|\Sigma s_G(Z_i)\|^2+2\tr\bigl(\Sigma\nabla s_G(Z_i)\Sigma\bigr)},
\end{equation}
where $\nabla s_G(z)=\nabla^2\log f_G(z)$ is the Hessian of $\log f_G(z)$.
In turn, the objective in~\eqref{eq:multivariate_SURE} is (up to an additive constant) equal to the weighted score matching objective~\citep{barp2019minimum, huang2021variational}, $\SM_\Sigma(G) := n^{-1}\sum_{i=1}^n \{ \|\Sigma s_G(Z_i)\|^2 +2\tr(\Sigma^2\nabla s_G(Z_i))\}$, which is unbiased (up to a constant that does not depend on $G$) for the weighted Fisher divergence,
\begin{equation}
\label{eq:multivariate_regret_fisher}
\Dfishersigma{f_{G_\star}}{f_G}:=\int\|\Sigma\cb{s_{G_\star}(z)-s_G(z)}\|^2 f_{G_\star}(z)\,dz.
\end{equation}
For a fresh draw $(\mu,Z)$ independent of the training data, the empirical Bayes regret is, in analogy to~\eqref{eq:ebregret_defi_sideinfo},
\begin{equation}
\label{eq:multivariate_regret}
\Regret(G_\star,G)
:=\EE[G_\star]{\|\mu-\EE[G]{\mu\mid Z}\|^2}-\EE[G_\star]{\|\mu-\EE[G_\star]{\mu\mid Z}\|^2}.
\end{equation}
By~\eqref{eq:multivariate_score_and_marginal}, we have that $\Regret(G_\star,G)=\Dfishersigma{f_{G_\star}}{f_G}$.

Writing $\mathcal{P}([-M,M]^d)$ for the class of all distributions supported on $[-M,M]^d$, we obtain the following extension of Theorem~\ref{theo:rate_homoscedastic}.
\begin{theo}[SURE rate in the multivariate Gaussian problem]
\label{theo:rate_homoscedastic_multivariate}
Suppose that~\eqref{eq:multivariate_gaussian_EB} holds with $G_\star\in\mathcal{P}([-M,M]^d)$ and $\underline{k}I_d\preceq\Sigma\preceq\overline{k}I_d$ for $0<\underline{k}\leq\overline{k}$. Let \smash{$\widehat{G}$} be the minimizer of $\SURE(G)$ in~\eqref{eq:multivariate_SURE} over all $G\in\mathcal{P}([-M,M]^d)$. Then, with probability at least $1-n^{-2}$, it holds that,
$$
\Regret(G_\star,\hG)\equiv\Dfishersigma{f_{G_\star}}{f_{\hG}}\leq C_{d,M,\underline{k},\overline{k}}\frac{\log^{d+3}n}{n},
$$
where $C_{d,M,\underline{k},\overline{k}}>0$ is a constant that depends only on $d$, $M$, $\underline{k}$, and $\overline{k}$.
\end{theo}
This rate is faster than the $n^{-1}\log^{\max\{d+4, 3d/2\}}n$ rates obtained for the multivariate NPMLE by~\citet{saha2020nonparametric, soloff2024multivariate}.

\subsection{SURE-LS (Least Squares)}
\label{sec:regression}

Our second estimator, SURE-LS, uses the working assumption that  $\mu_i$ follows a Gaussian distribution conditional on $X_i$, i.e., $\mathcal{G}^{\mathrm{LS}} := \{ G(\,\cdot \mid X_i=x) = \mathrm{N}\p{m(x), A(x)}\}$, where $m(\cdot)$, resp. $A(\cdot)$, represent the prior mean and variance as a function of covariates. Variants of this model (with different restrictions on $m(\cdot)$ and $A(\cdot)$) appear throughout the literature; see below for several examples. 

Parameterizing any candidate $G$ by $m(\cdot)$ and $A(\cdot)$ so that $G(\cdot \mid X_i) = \mathrm{N}\p{m(X_i), A(X_i)}$, we have the following expressions for the posterior mean of $\mu_i$ in~\eqref{eq:EB_heterosc_side_info}, the conditional score in~\eqref{eq:conditional_score} and its derivative:
\begin{equation*}
\EE[G]{\mu \mid W=w} =  \frac{\sigma^2}{\sigma^2 + A(x)}m(x)  +   \frac{ A(x)}{\sigma^2 + A(x)}z,\,\,s_{G}(w) = - \frac{z - m(x)}{\sigma^2+A(x)},\,\,\frac{\partial}{\partial z}s_{G}(w) = -\frac{1}{\sigma^2+A(x)}.
\end{equation*}
The SURE objective is equal to 
\begin{equation}
\SURE(G) = \frac{1}{n}\sum_{i=1}^n \sigma_i^2 + \frac{1}{n} \sum_{i=1}^n  \sigma_i^4 \frac{(Z_i - m(X_i))^2}{(\sigma_i^2 +A(X_i))^2} - \frac{2}{n} \sum_{i = 1}^n \frac{\sigma_i^4}{\sigma_i^2+A(X_i)}.
\label{eq:sure_regression}
\end{equation}
In what follows, it will be convenient to parameterize $G$ instead via the functions $\lambda(\cdot)$, $b(\cdot)$ which are related to $m(\cdot)$, $A(\cdot)$ as follows:
\begin{equation}
\lambda(x) = \frac{\sigma^2}{\sigma^2 + A(x)},\,\,\,b(x) = \lambda(x) m(x)\,\;\; \longleftrightarrow\,\;\;  A(x) = \sigma^2 \frac{1-\lambda(x)}{\lambda(x)},\,\, m(x) = \frac{b(x)}{ \lambda(x)}. 
\end{equation}
With this new parameterization, the   optimal shrinkage function and the score take the form
\begin{equation}
\label{eq:linear_shrinkage}
\EE[G]{\mu \mid W=w} = b(x) \,+\, (1-\lambda(x))z,\,\,\,\, s_G(w) = \cb{b(x)-\lambda(x)z}/\sigma^2,
\end{equation}
while the SURE objective takes the form:
\begin{equation}
\SURE(G) \equiv \SURE(\lambda, b)=  \frac{1}{n}\sum_{i=1}^n \sigma_i^2 + \frac{1}{n} \sum_{i=1}^n  \cb{\lambda(X_i) Z_i - b(X_i)}^2 - \frac{2}{n} \sum_{i = 1}^n \sigma_i^2 \lambda(X_i).
\label{eq:SURE_gauss}
\end{equation}
Below we make assumptions directly on $\lambda(\cdot)$ and $b(\cdot)$, positing that $\lambda(\cdot) \in \mathcal{L},\,\, b(\cdot) \in \mathcal{B}$ for two classes $\mathcal{L}$ and $\mathcal{B}$, and so we identify $\mathcal{G} \equiv \mathcal{L} \times \mathcal{B}$. We assume that $\lambda(x) \in [0,1]$ for any $\lambda(\cdot) \in \mathcal{L}$.\footnote{The correspondence with proper Gaussian priors requires $\lambda(x)>0$, but Theorem~\ref{theo:reg} also allows $\lambda(x)=0$.}
We then estimate $\lambda(\cdot)$ and $b(\cdot)$  by minimizing~\eqref{eq:SURE_gauss}:
\begin{equation}
(\hat{\lambda},\, \hat{b}) \in \argmin \cb{\SURE(\lambda, b)\,:\, \lambda(\cdot) \in \mathcal{L},\, b(\cdot) \in \mathcal{B}}.
\label{eq:sure_minim_semiparametric}
\end{equation}
The implied denoiser $\hat{\mu}_i = \hat{b}(X_i) + (1-\hat{\lambda}(X_i))Z_i$ with $\hat{\lambda}, \hat{b}$ in~\eqref{eq:sure_minim_semiparametric}  has strong guarantees well-beyond our working model in~\eqref{eq:EB_heterosc_side_info} with conditional prior $\mathrm{N}(m(X_i), A(x_i))$. To make this clear, we state our next results in the fixed-X and compound decision theoretic setting with both $X_i$ and $\mu_i$ fixed. Moreover, we relax Gaussianity of $Z_i$ and assume that for $K>0$,
\begin{equation}
\label{eq:subgaussian}
Z_i = \mu_i + \xi_i,\,\,\; \EE{\xi_i}=0,\,\,\;\Var{\xi_i} = \sigma_i^2,\,\,\; \EE{\exp(t \xi_i)} \leq \exp\p{K^2\sigma_i^2 t^2/2} \text{ for all }t, 
\end{equation}
that is $\xi := Z_i - \mu_i$ is $K\sigma_i$-sub-Gaussian. The reason we can relax Gaussianity is that as noted e.g., by~\citet{kou2017optimal, ignatiadis2019covariatepowered}, the conditionally linear nature of the shrinkage rules in~\eqref{eq:linear_shrinkage}  implies that $\SURE$ is unbiased for any conditional distribution that has the correct structure of the first two moments, i.e., $\EE{Z_i \mid \mu_i, X_i} = \mu_i$ and $\Var{Z_i \mid \mu_i, X_i}=\sigma_i^2$. In this setting, instead of data-generating choices of $\lambda_{\star}$, $b_{\star}$, we define oracle choices of $\lambda \in \mathcal{L}$  and $b \in \mathcal{B}$ that optimize the mean squared error, with expectation taken only over $\xi_i$ in~\eqref{eq:subgaussian},
\begin{equation}
(\lambda_{\oracle}, b_{\oracle}) \in \argmin_{\lambda \in \mathcal L,b \in \mathcal B}\cb{ \frac{1}{n}\sum_{i=1}^n \EE[\boldmu]{\cb{ \mu_i - (b(X_i) \,+\, (1-\lambda(X_i))Z_i) }^2 }}.
\label{eq:oracle_params}
\end{equation}
We assume that $\mathcal{L}$ is star-shaped about $\lambda_{\oracle}$. This means that for any $\lambda \in \mathcal{L}$ and any $\eta \in [0,1]$, it also holds that $\eta \lambda + (1-\eta) \lambda_{\oracle} \in \mathcal{L}$. Similarly, we assume that $\mathcal{B}$
is star-shaped about $b_{\oracle}$. We will state our main result below in terms of the complexity of the shifted classes $\mathcal{B}_{\oracle} := \mathcal{B}-b_{\oracle}$ and $\mathcal{L}_{\oracle} : = \mathcal{L}-\lambda_{\oracle}$. The notion of complexity we will use is that of local Gaussian complexity (see e.g.,~\citealt[equation (13.15)]{wainwright2019highdimensional}), which is defined as follows for a class of functions $\mathcal{H} \subset \cb{h: \mathcal{X} \to \RR}$ and $t>0$,
\begin{equation}
\mathscr{G}_n(t; \mathcal{H}) := \EE{\sup \cb{ \frac{1}{n} \sum_{i = 1}^n h(X_i)\zeta_i\,\,:\,\, h \in \mathcal{H},\,\,\, \frac{1}{n}\sum_{i=1}^n h^2(X_i) \le t^2 }},
\label{eq:local_gaussian_complexity}
\end{equation}
where the expectation is taken only over $\zeta_i \simiid \mathrm{N}(0,1)$. For the shifted classes used below, which contain zero, we set $\mathscr{G}_n(0;\mathcal H)=0$, the continuous extension at zero.

Our main result below upper bounds the in-sample regret of the SURE-trained procedure by the Gaussian local complexity of $\mathcal{B}_{\oracle}$ and $\mathcal{L}_{\oracle}$. 

\begin{theo}
\label{theo:reg}
Let $n\ge 2$. Suppose $\mathcal L$ is star-shaped about $\lambda_{\oracle}$,  $\mathcal B$ is star-shaped about $b_{\oracle}$, that the $\xi_i$ are independently $K\sigma_i$-sub-Gaussian and that all  $\sigma_i \in [\sqrt{2},\sigma_{\text{max}}]$.\footnote{This is without loss of generality by rescaling the problem.} Moreover suppose that $\abs{\mu_i}$, $\abs{2 \lambda_{\oracle}(X_i)\mu_i - b_{\oracle}(X_i)}$, $\abs{\lambda_{\oracle}(X_i)} \leq M$ for all $i$ (for $M>0$). Then, there exist constants $C,c$ that depend on $K, \sigma_{\text{max}}$, and $M$ such that with probability at least $1 - \delta$ (for $\delta \in (0,1/2])$, it holds that,
$$
\begin{aligned}
&\frac{1}{n}\sum_{i=1}^n \cb{( \hat \lambda(X_i) - \lambda_\oracle(X_i))Z_i - (\hat b(X_i) - b_\oracle(X_i))}^2 \le t_*^2,\,\, \text{ and }\\
&\sqrt{ \frac{1}{n} \sum_{i=1}^n \cb{ \mu_i - (\hat{b}(X_i) \,+\, (1-\hat{\lambda}(X_i))Z_i) }^2} \leq \sqrt{ \frac{1}{n} \sum_{i=1}^n \cb{ \mu_i - (b_{\oracle}(X_i) \,+\, (1-\lambda_{\oracle}(X_i))Z_i) }^2}  \;  +\; t_*,
\end{aligned}
$$
where $t_*$ is defined as the infimum
$$
\inf \cb{ t \ge 0 \, :\, \frac{t^2}{C} \ge \sqrt{\log\p{\frac{n}{\delta}}}\mathscr{G}_n(ct; \mathcal{B}_{\oracle})+ \log\p{\frac{n}{\delta}}\mathscr{G}_n(ct; \mathcal{L}_{\oracle})\ + t\abs{\log \delta} \sqrt{\frac{\log n}{n}}+ \abs{ \log \delta}^2 \frac{\log n}{n}}.
$$
\end{theo}
For low complexity classes with e.g., bounded VC dimension, we will have \smash{$t_*^2 = \tilde{O}(n^{-1})$}\footnote{
We write $\tilde{O}(\psi_n)$ as shorthand for $O(\psi_n\psi_n')$ where $\psi_n'$ is polylogarithmic in $n$.
} as a function of the number of samples, so the first inequality in the guarantee yields a fast \smash{$\tilde{O}(n^{-1})$} rate of convergence to the oracle in terms of in-sample squared error, and the second inequality is a corresponding oracle inequality. This form of oracle inequality commonly appears in the literature (see e.g., \cite{jiang2009general,bellec2018slope}). At a technical level, this bound is an example of what is called an ``optimistic rate'' guarantee (see \citet{panchenko2003symmetrization,vapnik2006estimation,srebro2010optimistic,zhou2021optimistic}), as well an ``asymptotically exact'' oracle inequality (see e.g., \cite{cavalier2002oracle}) since it guarantees asymptotic convergence to the oracle MSE.

We next provide a simple example to show that a guarantee such as that of Theorem~\ref{theo:reg} in general does not need to hold for maximum likelihood when the working prior class is misspecified.

\begin{exam}[Maximum likelihood under misspecification.]
\label{exam:mle-suboptimal}
Suppose $X_i=\sigma_i^2$, and consider the working model of this section with $A(\cdot)=A$ fixed (for some $A>0$) and $m(\cdot)=m$ constant. Thus, we only estimate the common prior mean $m$. Writing $\lambda(\sigma^2)=\sigma^2/(\sigma^2+A)$, this corresponds to
$\mathcal{L}:=\cb{\lambda(\cdot)},\,\,\mathcal{B}:=\cb{m\lambda(\cdot):m\in\RR}.$
Since $\mathcal{L}$ is a singleton and $\mathcal{B}$ is one-dimensional, under the conditions of Theorem~\ref{theo:reg} we get $t_*^2=\tilde{O}(n^{-1})$. Thus, SURE-training converges to the oracle denoiser even when the constant prior mean is misspecified.

By contrast, maximum likelihood over the same working model need not converge to the oracle denoiser. For instance, take $A=4$, with $(\mu_i,\sigma_i^2)=(3,2)$ for half the observations and $(\mu_i,\sigma_i^2)=(0,8)$ for the other half. We can check that the oracle denoiser in~\eqref{eq:oracle_params} is given by the choice $m_{\oracle}=3/5$. Meanwhile, as $n\to \infty$, it holds that $\hat m_{\mathrm{MLE}}=2+O_{\mathbb P}(n^{-1/2})$ and $\hat m_{\mathrm{SURE}}=m_{\oracle}+O_{\mathbb P}(n^{-1/2})$. Writing $\hat\mu_i(m)=\lambda(X_i)m+(1-\lambda(X_i))Z_i$, the oracle MSE is $n^{-1}\sum_{i=1}^n\EEInline[\boldmu]{\cb{\mu_i-\hat\mu_i(m_{\oracle})}^2}=58/45$, while, as $n\to\infty$,
$$
\frac{1}{n}\sum_{i=1}^n\EE[\boldmu]{\cb{\mu_i-\hat\mu_i(\hat m_{\mathrm{MLE}})}^2}\to\frac{11}{6},\qquad
\frac{1}{n}\sum_{i=1}^n\EE[\boldmu]{\cb{\mu_i-\hat\mu_i(\hat m_{\mathrm{SURE}})}^2}\to\frac{58}{45}.
$$
See Supplement~\ref{subsec:sure_vs_mle} for the calculation and further discussion on the relationship of maximum likelihood and SURE-training.
\end{exam}

We next illustrate Theorem~\ref{theo:reg} in two important settings. In Supplement~\ref{subsec:further_further_examples}, we consider two more settings: the pure regression setting, and the group-linear estimation strategy of~\citet{weinstein2018grouplinear}, also termed CLOSE-Gauss by~\citet{chen2026empirical}.

\begin{exam}[Semiparametric isotonic SURE shrinkage.]
\label{exam:isotonic_sure}
The semiparametric isotonic SURE shrinkage estimator of~\citet{xie2012sure} is defined as follows. Suppose $X_i=\sigma_i^2$ (i.e., we have no covariates beyond the variances). Then let
$\mathcal{L}^{\text{iso}} := \cb{\lambda(\cdot) \in [0,1] \text{ and isotonic in } \sigma^2},\,\, \mathcal{B} := \cb{0}.$
The above classes imply denoisers of the form $\EE[G]{\mu_i \mid Z_i=z, \sigma_i^2=\sigma^2} = (1-\lambda(\sigma^2))z$. The motivation of~\citet{xie2012sure} was as follows. Suppose $\mu_i \sim \mathrm{N}(0,\, A)$ and $\mu_i$ is independent of $\sigma_i^2$. Then, $\EE{\mu_i \mid Z_i, \sigma_i^2} = (1-\lambda(\sigma_i^2))Z_i$ with $\lambda(\sigma^2) = \sigma^2/(A+\sigma^2)$, and so $\lambda(\cdot) \in \mathcal{L}^{\text{iso}}$. 

We note that the local Gaussian complexity in~\eqref{eq:local_gaussian_complexity} of $\mathcal{L}^{\text{iso}}$ satisfies $\mathscr{G}_n(t; \mathcal{L}_{\oracle}) \lesssim (t/n)^{1/2}$.  The latter result follows, e.g., by entropy numbers for $\mathcal{L}_{\oracle}$ (\citealt{birman1967piecewise} and \Citealt[equation (2.5)]{geer2000empirical}) and chaining (e.g.,~\citet[Corollary 13.7]{wainwright2019highdimensional}). Thus, under the conditions of Theorem~\ref{theo:reg}, the same theorem (for fixed $\delta$) shows that $t_*^2 = \tilde{O}( n^{-2/3})$. By contrast, although~\citet{xie2012sure} do not provide rates for any of their results, carrying out their argument would yield a rate of $O_{\mathbb P}(n^{-1/2})$, which is also the same rate we would get by applying a uniform consistency argument as in Section~\ref{subsec:uniform_convergence}. By contrast, Theorem~\ref{theo:reg} shows that a faster rate is attained.

Our result also allows for a flexible structure on $\mathcal{B}$. For instance, suppose that instead of shrinking only toward $0$, we consider a broader class for $\mathcal{B}$ such as $\mathcal{B}:= \mathcal{B}^{\text{iso}} \equiv \mathcal{L}^{\text{iso}}$. Then if we optimize SURE over $\mathcal{L}^{\text{iso}} \times \mathcal{B}^{\text{iso}}$, Theorem~\ref{theo:reg} still yields $t_*^2 = \tilde{O} (n^{-2/3})$.
\end{exam}

\begin{exam}[Covariate-powered empirical Bayes estimation.]
\label{exam:EBCF}
\citet{ignatiadis2019covariatepowered}  consider the working model of this section and provide regret guarantees allowing for misspecification of both the prior model and the Gaussianity of $\xi_i$ (as we do in~\eqref{eq:subgaussian}). The proposed approach is called empirical Bayes with cross-fitting (EBCF) and proceeds as follows:
\begin{enumerate}[noitemsep]
    \item Partition $\cb{1,\dotsc,n}$ into $K$ (say, $K=5$) folds $I_1,\dotsc,I_K$.
    \item For the $\ell$-th fold $I_{\ell}$, fit a nonparametric regression (a supervised learning model) of $Z_i \sim X_i$ based on $i \in \cb{1,\dotsc,n}\setminus I_{\ell}$, and denote the learned model by $\hat{m}_{-I_{\ell}}(\cdot)$.
    \item For $i \in I_{\ell}$, estimate $\mu_i$ by $\hat{\mu}_i(\hat{A}_{I_{\ell}})$, where $\hat{\mu}_i(A) := \{A/(A+\sigma_i^2)\} Z_i + \{\sigma_i^2/(A+\sigma_i^2)\}\hat{m}_{-I_{\ell}}(X_i)$ and $A$ is estimated as $\hat{A}_{I_{\ell}}$ by minimizing $\SURE$ of $\hat{\mu}_i(A)$ over $i \in I_{\ell}$.
\end{enumerate}
Suppose the data generating mechanism is given by our working model with $X_i \in [0,1]$, $A(x)=c$ for all $x$ (for some $c > 0$), $\sigma_i^2=1$ for all $i$ (with $\sigma_i^2$ not included in $X_i$ here), and $m(\cdot) \in \mathcal{B}^{\text{Lip}} := \cb{b(\cdot) \in [-B,B] \mbox{ and $L$-Lipschitz in }  x},
$
for some $L,B>0$.
In this setting,~\citet{ignatiadis2019covariatepowered} show that using a suitable $\hat{m}$ the empirical Bayes regret is of order $O(n^{-2/3})$ and that this is minimax rate optimal for the regret among all possible estimators (with worst case taken over $m(\cdot) \in \mathcal{B}^{\text{Lip}}$).  Theorem~\ref{theo:reg} with $\mathcal{B}^{\text{Lip}}$ and \smash{$\mathcal{L}^{\text{const}}:=\cb{ \lambda(\cdot) \text { is constant}}$} yields \smash{$t_*^2 = \tilde{O}(n^{-2/3})$} without cross-fitting.\footnote{We note that the cross-fitting of EBCF has one important advantage. It ensures that if $|I_{\ell}| \geq 5$ for all folds, and if $Z_i \sim \mathrm{N}(\mu_i,1)$, then the EBCF estimator has a finite sample James-Stein~\citeyearpar{james1961estimation} property, that is, $n^{-1}\sum_{i=1}^n \EE[\boldmu]{ (\mu_i - \hat{\mu}_i)^2}\, < 1.$ 
}
\end{exam}

\section{Computational strategy for SURE-training}
\label{sec:computation}
In this section we describe our computational strategy for implementing SURE-PM and SURE-LS. Moreover, we present a third method, called SURE-THING (\underline{T}his \underline{H}elps \underline{I}n \underline{N}eural \underline{G}-modeling). SURE-THING incorporates a neural network that uses side-information and flexibly models the conditional distribution $G(\cdot \mid X_i)$.
In all cases, parameters of the model are optimized using standard gradient-based training with the SURE loss: we use full batch gradient descent using the Adam optimizer~\citep{kingma2015adam} with learning rate $0.01$; tolerance $10^{-6}$; and $2,000$ iterations. We discuss the implementation used in our numerical results in more detail in Supplement~\ref{subsec:imp_details}. 

\paragraph{SURE-PM.} SURE-PM from Section~\ref{sec:fast_rates} derives its name from our computational strategy. Namely, we parameterize any prior $G$ (which for SURE-PM has no dependence on the side-information) as $K$ discrete particles, \smash{$G = \sum_{j=1}^K \pi_j \delta_{u_j}$} with \smash{$u = (u_1,\ldots,u_K) \in \mathbb{R}^K$} and \smash{$\pi \in \Delta^{K-1}$}, where $\delta_{u_j}$ denotes the Dirac mass at $u_j$ and $\Delta^{K-1}$ denotes the probability simplex. 
In our implementation, we use a different parameterization for $u$ and $\pi$ in which all parameters are unconstrained (and so we can directly use gradient descent), and that avoids the label-switching problem for the grid locations. For the probabilities $\pi$, we introduce an unconstrained vector $\tilde{\pi} \in \mathbb{R}^K$ and let
$
\pi_j := \text{Softmax}(\tilde{\pi})_j := \exp(\tilde{\pi}_j)/\sum_{\ell} \exp(\tilde{\pi}_\ell)$ for $j=1,\dotsc,K.$
For the grid locations, we use a composition of transformations with parameters $\tilde{u} \in \mathbb{R}^{K-1}$, $s \in \mathbb{R}$, and $m \in \mathbb{R}$, namely
\smash{$
u_j := (\sum_{k=1}^{j-1} \text{Softmax}(\tilde{u})_k - \frac{1}{2})e^s + m$} for $j=1,\dotsc,K$.
Here, $\text{Softmax}(\tilde{u})_j$ represents the relative distance between $u_j$ and $u_{j+1}$. 
The partial sum \smash{$\sum_{k=1}^{j-1} \text{Softmax}(\tilde{u})_k$} represents the relative position of the $j$-th 
grid point. The subtraction of $1/2$ positions the grid points so that a point occurring at relative position $1/2$ will be located at $m$, while $e^s$ controls the spread of the points around $m$.

\paragraph{SURE-LS.} For SURE-LS, we need to model the prior mean $m(x)$ and the prior variance $A(x)$. We model these with a feedforward neural network that takes as input $x \in \mathcal{X}$ and outputs the parameters $(m(x), \tilde{A}(x))$, and we equate the variance $A(x)$ to $\exp(\tilde{A}(x))$.

\paragraph{SURE-THING.} We propose a third method to demonstrate the flexibility of the SURE-training framework. 
SURE-THING extends SURE-PM by allowing the atoms and weights to vary with covariates:
\smash{$
\mathcal{G}^{\text{THING}} := \{ G\,:\,  G(\cdot \mid x) = \sum_{j=1}^K \pi_j(x) \delta_{u_j(x)},\; \pi(x) \in \Delta^{K-1},\; u(x) \in \mathbb{R}^K\}.
$}
We parameterize our model by a feedforward neural network $\mathcal X \to \mathbb{R}^{2K}$ so that
for each $x \in \mathcal{X}$, the network takes as input the covariates and outputs parameters $(\tilde{\pi}(x), \tilde{u}(x), s(x))$ which are transformed to atoms and weights via the same mechanism as SURE-PM.

\subsection{Further optimization considerations}
\label{subsec:optimization}

Our statistical results of Section~\ref{sec:stat_guarantees} concern global minimizers of SURE, whereas the gradient-based procedures described above do not guarantee convergence to a global minimum.
We therefore discuss when SURE minimization is convex, how we approach
optimization in the non-convex cases, and how to make use of fits from
competing estimation methods.

\paragraph{Convexity.} For the NPMLE, fixing the particle locations $u_1,\dotsc,u_K$ yields a convex optimization problem in the mixture weights $\pi$, which is central to the computational strategy of~\citet{koenker2014convex}. Other methods for computing the NPMLE include the vertex direction method \citep{lindsay1983geometry} and mix-SQP \citep{kim2020fast}. For SURE-PM, however, the objective can be non-convex even when only the weights are optimized (see  Supplement~\ref{subsubsec:pm_nonconvexity} for an example).

For SURE-LS, the parameterization in~\eqref{eq:SURE_gauss} gives an objective that is jointly convex in $(\lambda,b)$. Thus, minimizing SURE over convex classes $\mathcal{L}$ and $\mathcal{B}$ is a convex optimization problem. In Example~\ref{exam:isotonic_sure}, we consider the semiparametric isotonic SURE shrinkage estimator of~\citet{xie2012sure}, with $\mathcal{L}=\mathcal{L}^{\text{iso}}$ and $\mathcal{B}=\{0\}$. As already explained by~\citet{xie2012sure},
this estimator can be efficiently computed by weighted isotonic regression. We give the explicit optimization problem and its isotonic specialization in Supplement~\ref{subsubsec:ls_convexity}. This convexity does not extend to our neural network parameterization of SURE-LS.

\paragraph{Gradient-based optimization.} In our experiments, the gradient-based implementations perform well across problems including heteroscedasticity, side-information,  misspecified working priors, and with varying sample size $n$ (Section~\ref{sec:simulations}). Therein we obtain strong statistical performance  using the same optimization settings throughout, without further manual tuning for individual problems. Better understanding the operating characteristics of  gradient-based procedures for SURE-training is an interesting direction for future work.

For any given application, we recommend the usual checks for gradient-based training~\citep[Chapter 8]{goodfellow2016deep}: monitor SURE over the optimization iterations, reduce the learning rate if the objective oscillates, and increase the number of iterations if it is still decreasing. We can also repeat the optimization from different initial values and retain the fit with the smallest SURE.

\paragraph{Using competing methods.} Often, we will have a competing way of fitting the prior available, such as the NPMLE $\widehat{G}^{\mathrm{NPMLE}}$. Recall from Section~\ref{subsec:uniform_convergence} and Supplement~\ref{sec:uniform_convergence_rademacher} that when $\Delta_n(\mathcal{G})$ in~\eqref{eq:sure_uniform_error} is small, SURE is close to squared error simultaneously over all priors in $\mathcal{G}$. This includes priors fitted to the same observations by different methods.
We can therefore compare $\SURE(\widehat{G}^{\mathrm{NPMLE}})$ with SURE at our current fit (e.g., $\widehat{G}$ using gradient-based optimization on the SURE objective) and select the fit with the smaller value. If both fits lie in $\mathcal{G}$, the expected excess squared error of this choice relative to the better fitted denoiser is at most $2\Delta_n(\mathcal{G})$ even if neither fit minimizes SURE over $\mathcal{G}$.
We can also use the competing model to initialize SURE-training, as we describe for the NPMLE in Supplement~\ref{subsubsec:fitted_comparison}.

\section{Numerical results}
\label{sec:simulations}

This section evaluates the SURE-trained methods empirically, comparing them against existing methods.
We evaluate all methods (in all simulation setups) through the in-sample mean-squared error, which we define as follows. Suppose that the sample size in an experiment is $n \in \mathbb N$ and that we run $B$ Monte Carlo replicates of the experiment. For the $b$-th Monte Carlo replicate, we generate parameters \smash{$\mu_i^{(b)}$}, $i=1,\dotsc,n$ and observations \smash{$(X_i^{(b)}, Z_i^{(b)})$}, $i=1,\dotsc,n$. Each method takes \smash{$(X_i^{(b)}, Z_i^{(b)})$}, $i=1,\dotsc,n$, as input and returns estimates \smash{$\hat{\mu}^{(b)}_i$}, $i=1,\dotsc,n$. Our reported metric is
the estimated in-sample MSE, \smash{$
\widehat{\mathrm{MSE}} := \frac{1}{B} \sum_{b=1}^B \frac{1}{n} \sum_{i=1}^n (\mu_i^{(b)} - \hat{\mu}^{(b)}_i)^2.
$}

In Supplement~\ref{subsec:homosc_simulations} we conduct a simulation study in the Gaussian means problem of~\eqref{eq:gaussian_EB} in which $\sigma_i^2=1$ for all $i$ and there is no side-information. These simulations corroborate the results of Theorem~\ref{theo:rate_homoscedastic} as SURE-PM nearly matches the risk of the oracle Bayes estimator.

\begin{figure}
\centering
\includegraphics[width=1\linewidth]{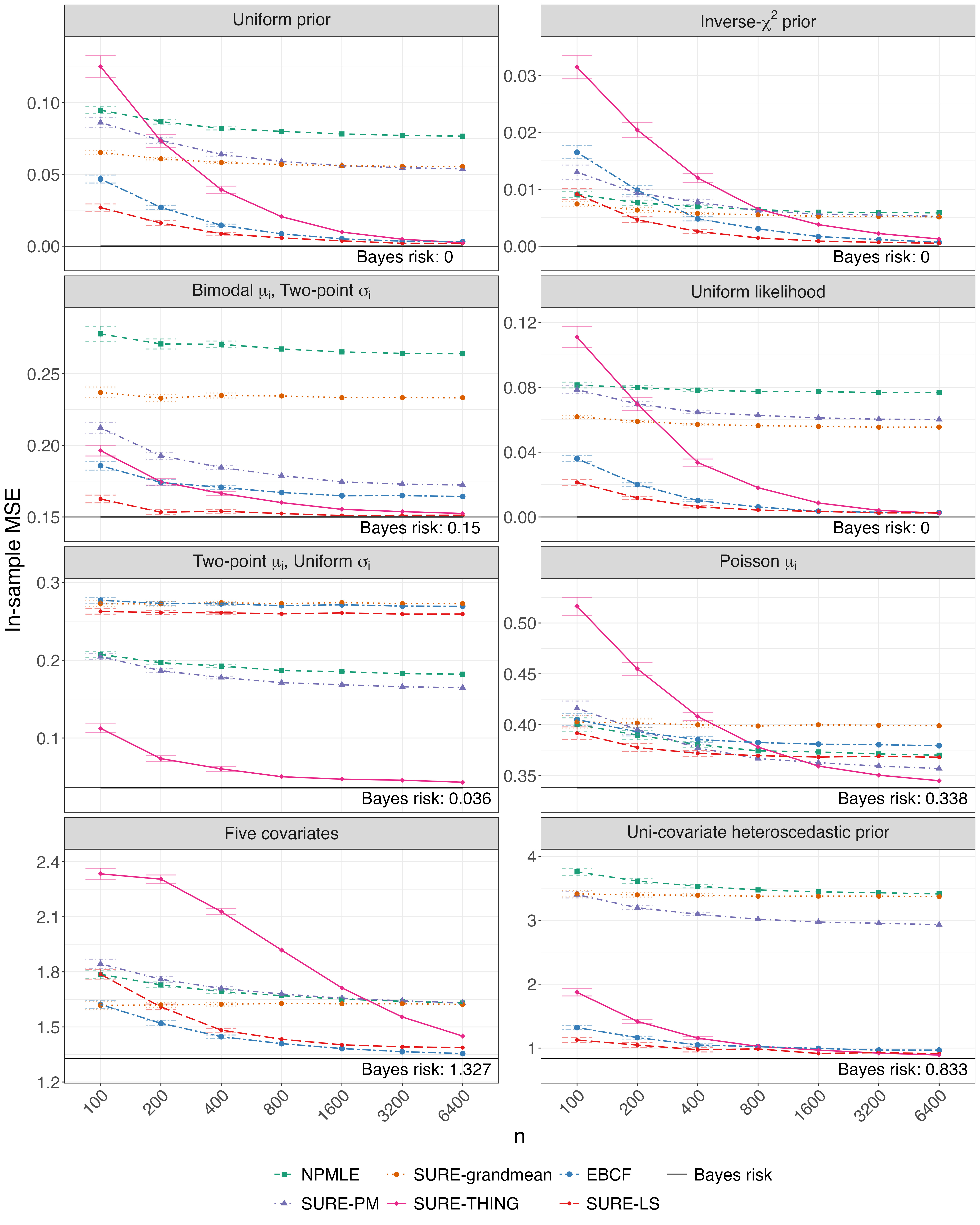}
\caption{Heteroscedastic simulations. Each panel corresponds to a different setting. We plot the in-sample MSE ($\pm$ 1 standard error) versus sample size $n$.}
\label{fig:xie_mse}
\end{figure}

We next conduct a simulation study with heteroscedasticity (varying $\sigma_i^2$) across problems and further side-information. We identify $\sigma_i^2$ with $X_i$ for the cases where we do not simulate further covariates. Part of our simulation study is inspired by the numerical study in~\citet[Simulation studies (c)-(f) in their Fig. 1]{xie2012sure}. We consider eight settings:
\begin{itemize}[leftmargin=*,noitemsep]
\item uniform prior: $\;\;\; \sigma_i^2 \simiid \text{Unif}(0.1, 1), \;\; \mu_i = \sigma_i^2,\;\;\; Z_i  \simindep \mathrm{N} \bigl( \mu_i, \sigma_i^2 \bigr);$
\item inverse $\chi^2$ prior: $\sigma_i^2 \simiid \mathrm{Inv-}\chi^2_{10}, \;\; \mu_i = \sigma_i^2,\;\;\; Z_i \simindep \mathrm{N} \bigl( \mu_i, \sigma_i^2 \bigr);$
\item bimodal prior with two-point variance: 

$\quad \sigma_i^2 \simiid \frac{1}{2}(\delta_{0.1} + \delta_{0.5}), \;\; \mu_i \simindep \begin{cases}
\mathrm{N}(2, 0.1), & \text{if } \sigma_i^2 = 0.1 \\
\mathrm{N}(0, 0.5), & \text{if } \sigma_i^2 = 0.5
\end{cases},\;\;\; Z_i \simindep \mathrm{N}(\mu_i, \sigma_i^2);$

\item uniform likelihood: $\;\sigma_i^2 \simiid \text{Unif}(0.1, 1), \;\; \mu_i=\sigma_i^2,\;\;\; Z_i  \simindep \text{Unif}\left(\mu_i - \sqrt{3}\sigma_i , \mu_i + \sqrt{3} \sigma_i \right);$ 
\item two-point prior with uniform variance: 

$\quad \sigma_i^2 \simiid \text{Unif}(0.1, 0.5), \;\;\; \mu_i \simindep \frac{1}{2}(\delta_{\sigma_i^2} + \delta_{10\sigma_i^2}), \;\;\; Z_i \simindep \mathrm{N} \bigl( \mu_i, \sigma_i^2 \bigr);$

\item Poisson prior: $\;\;\; \sigma_i^2 \simiid \text{Unif}(0.1, 1), \;\; \mu_i \simindep \text{Poisson}(2\sigma_i^2),\;\;\; Z_i \simindep \mathrm{N} \bigl( \mu_i, \sigma_i^2 \bigr);$ 
\item prior with multiple covariates: $\quad \sigma_i^2 \simiid \text{Unif}(1.5, 2.5), \;\; X_i \simiid \text{Unif}(0,1)^5,$

$ \quad \mu_i \simindep \mathrm{N}\bigl(\pi X_{i,1}X_{i,2} + 20(X_{i,3}-0.5)^2 + 5 X_{i,4}, 4\bigr), \;\; Z_i \simindep \mathrm{N} \bigl( \mu_i, \sigma_i^2 \bigr);$

\item heteroscedastic prior with one covariate: 

$\quad X_i \simiid \text{Unif}(0, 1), \;\; \sigma_i^2 = 2X_i^2 + 5X_i + 1, \;\; \mu_i \simindep \mathrm{N}\bigl(2\sigma_i^2+0.5, 0.25\sigma_i^2\bigr), \;\; Z_i \simindep \mathrm{N} \bigl( \mu_i, \sigma_i^2 \bigr).$

\end{itemize} 
The first six settings are designed to capture mean/variance ($\mu_i/\sigma_i^2$) relationships. The uniform prior, inverse $\chi^2$ prior and uniform likelihood settings provide models of the strongest possible mean–variance dependence: $\mu_i = \sigma_i^2$ deterministically. The fourth one of these settings (uniform likelihood) is misspecified with respect to the general modeling assumption in~\eqref{eq:EB_heterosc_side_info} since the distribution of the noise is not Gaussian. Meanwhile, the bimodal prior model can be written as conditionally Gaussian with $\sigma_i^2$ that takes on only two values. (We discuss the bimodal prior model in more detail below.) The two-point prior provides a model where the marginal distribution becomes a mixture of two well-separated Gaussian distributions and the sixth one considers a discrete prior which takes any non-negative integer value. In the last two settings, $\mu_i$ depends not only on $\sigma_i^2$, but also on other covariates. 

We vary the sample size $n \in \cb{2^k \cdot 100\,:\, k=0,\dotsc,6}$ in each setting ($n$ is a simulation parameter) and conduct $B=500$ Monte Carlo replicates of each setting. We compare seven estimators:
\begin{itemize}[leftmargin=*,noitemsep]
\item the Bayes estimator (which is equal to $\hat{\mu}_i = \sigma_i^2$ in the uniform prior, inverse $\chi^2$ prior and uniform likelihood settings, and so has MSE equal to $0$);
\item the NPMLE that models $\mu_i \indep \sigma_i^2$ (mean/variance independence) and $\mu_i \mid \sigma_i^2 \sim G$ (see~\citet{jiang2020general, soloff2024multivariate} for an analysis of this estimator when the assumption  $\mu_i \indep \sigma_i^2$ is correct) using the computational strategy in~\citet{koenker2014convex};
\item SURE-grandmean, proposed by~\citet{xie2012sure}, which proceeds as follows: let $\bar{Z} := n^{-1}\sum_{i=1}^n Z_i$ be the grand mean of all the observation $Z_i$ and consider the class of estimators $\hat{\mu}_i(A) := \{A/(A+\sigma_i^2)Z_i\} + \{\sigma_i^2/(A+\sigma_i^2)\}\bar{Z}$, then choose $\hat{A}$ by minimizing SURE over this class of estimators;
\item EBCF, the covariate-powered empirical Bayes estimation method proposed in~\citet{ignatiadis2019covariatepowered} (described in Example~\ref{exam:EBCF}), where we fit the nonparametric regression model $\hat{m}$ using a two-layer feedforward neural network with 8 neurons per hidden layer and \textrm{ReLU} activations for each fold, and squared error (i.e., the mean of $\{Z_i-\hat{m}(X_i)\}^2$) as the objective function; 
\item our proposed SURE-PM,  SURE-LS, and SURE-THING (as described in Section~\ref{sec:computation}).
\end{itemize}
The NPMLE, SURE-grandmean, and SURE-PM are misspecified in all of the settings: they all effectively operate on a class $\mathcal{G}$ of prior distributions that are not functions of $\sigma_i^2$, even though in the data generating process $\mu_i$ strongly depends on $\sigma_i^2$ and so $G_{\star} \notin \mathcal{G}$. By contrast, SURE-THING is well-specified in all settings except the uniform likelihood setting. SURE-LS is well-specified for the settings where the prior of $\mu_i$ (conditional on $X_i$) and the likelihood are Gaussian, while EBCF is well-specified when moreover $\Var{\mu_i \mid X_i}$ is constant.

The results of the simulation are shown in Figure~\ref{fig:xie_mse}. We summarize some key observations: SURE-THING outperforms the NPMLE, SURE-PM, and SURE-grandmean (as well as SURE-LS and EBCF for the fifth and sixth case) when the sample size is large enough (say, $n \geq 1600$), and for $n=6400$ almost matches the Bayes risk in all eight settings (including the setting with the misspecified uniform likelihood). However, its performance for small $n$ can be suboptimal, likely due to the challenge of fitting the neural network weights with few samples. SURE-THING outperforms all estimators (beyond, of course, the oracle Bayes rule) in the fifth case (two-point prior) for each value of $n$. SURE-PM is the second best estimator.

SURE-LS generally outperforms the other data-driven estimators, except in the two-point prior and five-covariate settings, and in the Poisson prior setting for larger sample sizes. Its strong performance in settings with Gaussian conditional priors and Gaussian noise is expected, since SURE-LS is well-specified while modeling the prior less flexibly than SURE-THING.
EBCF also performs similarly to SURE-LS in homoscedastic prior cases, outperforming the latter when there are several covariates. When the prior is heteroscedastic, EBCF does not perform as well as SURE-LS since it posits a homoscedastic Gaussian prior as a working model during the estimation. 

Among the other methods, we note that SURE-PM outperforms the NPMLE across the board (with one exception for the inverse $\chi^2$ prior for lower values of $n$). These simulations provide empirical evidence that SURE-PM can achieve smaller mean squared error than the NPMLE under misspecification. Furthermore, SURE-grandmean outperforms NPMLE in some of the cases, while not performing too well in other cases. 

SURE-grandmean has the flattest curves in the sense that the performance gains kick in already for small sample sizes (as it uses SURE to tune a single parameter instead of training a nonparametric prior). SURE-grandmean and SURE-PM have comparable performance for the uniform prior and the inverse $\chi^2$ prior. For the bimodal (with two-point $\sigma_i$), two-point (with uniform $\sigma_i$), Poisson and heteroscedastic (with one covariate) priors, SURE-PM substantially outperforms SURE-grandmean. We will provide intuition for SURE-PM's strong performance in the bimodal setting below. 
By contrast, for the uniform likelihood, SURE-grandmean outperforms SURE-PM. This aligns with our theoretical development in Section~\ref{sec:regression}, particularly equation~\eqref{eq:subgaussian} that allows for non-Gaussian noise. SURE-grandmean remains effective in this setting as it only requires the correct structure of the first two moments, while SURE-PM's purely nonparametric approach has no theoretical guarantees when the noise distribution is misspecified: while $\SURE$ provides an unbiased estimate of MSE for Gaussian noise, this property does not extend to non-Gaussian settings, except for specific formulations such as SURE-LS.

\paragraph{Comparing SURE-PM vs NPMLE for the bimodal prior.}
We now zoom into one of the simulations with a bimodal prior to explicitly demonstrate in what way SURE behaves differently than the NPMLE under misspecification. We show results for a single Monte Carlo replicate with sample size n=6400. We recall that in this setting the variances $\sigma_i^2$ can take on one of two values: $0.1$ or $0.5$. Then, according to the value of $\sigma_i^2$, $\mu_i$ is drawn from a different Gaussian distribution,
\begin{equation}
\text{(Low variance)}\; \mu_i \mid \sigma_i^2 = 0.1 \sim \mathrm{N}(2, 0.1),\;\; \text{(High variance)}\; \mu_i \mid \sigma_i^2 = 0.5 \sim \mathrm{N}(0, 0.5).
\label{eq:component_dbns}
\end{equation}
Although both SURE-PM and NPMLE have access to both $\sigma_i^2$ and $Z_i$, they are forced to always use the same distribution $G$ (because $\mathcal{G}$ is misspecified, as mentioned above), and are unable to use a different $G$ for low and high variance observations (as in the data generating process).

\begin{figure}
\centering
\includegraphics[width=0.8\linewidth]{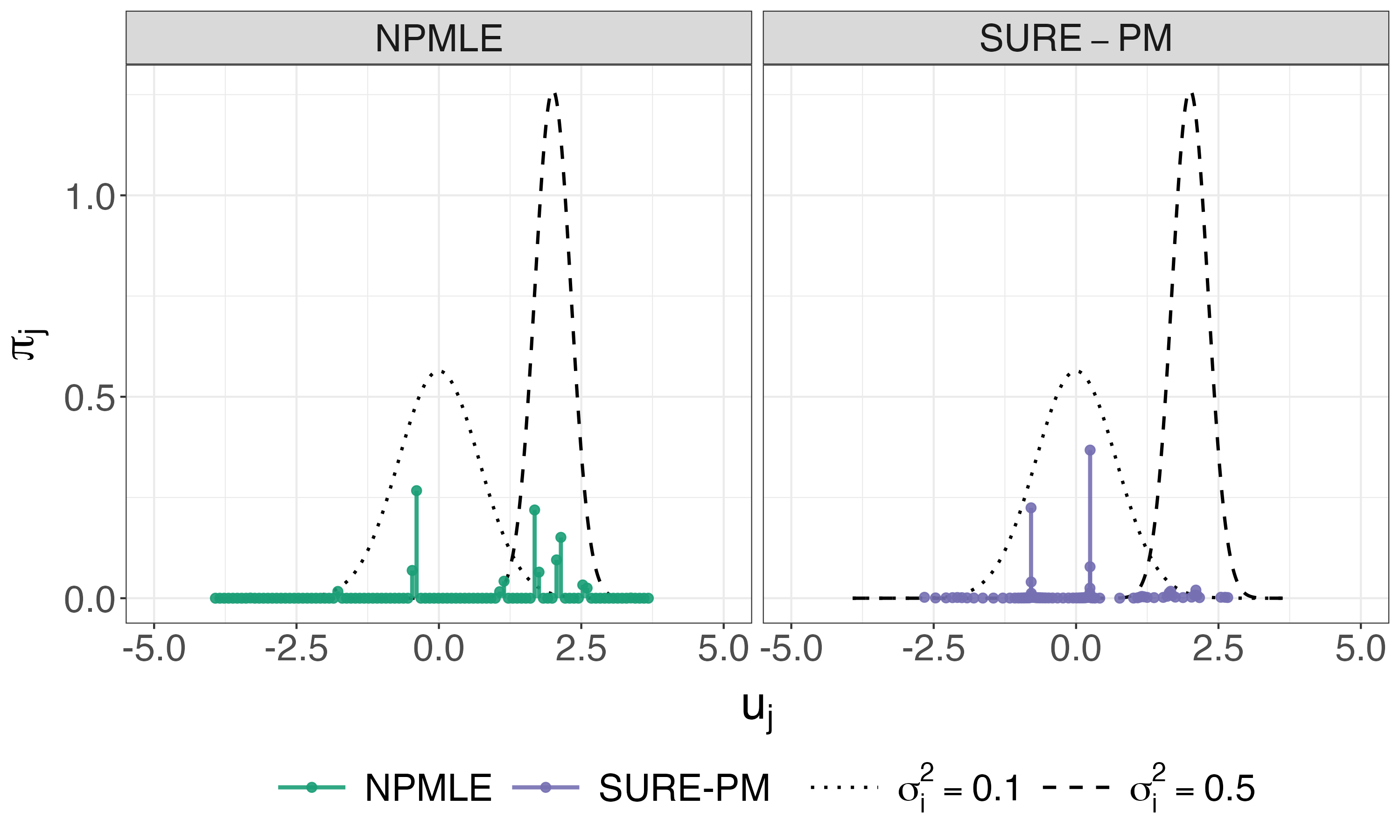}
\caption{Estimated priors \smash{$\sum_{j=1}^K \pi_j \delta_{u_j}$} from NPMLE and SURE-PM compared with the two true component distributions in~\eqref{eq:component_dbns} (low variance, high variance) for the bimodal prior setting.}
\label{fig:prior}
\end{figure}

\begin{figure}
\centering
\includegraphics[width=0.8\linewidth]{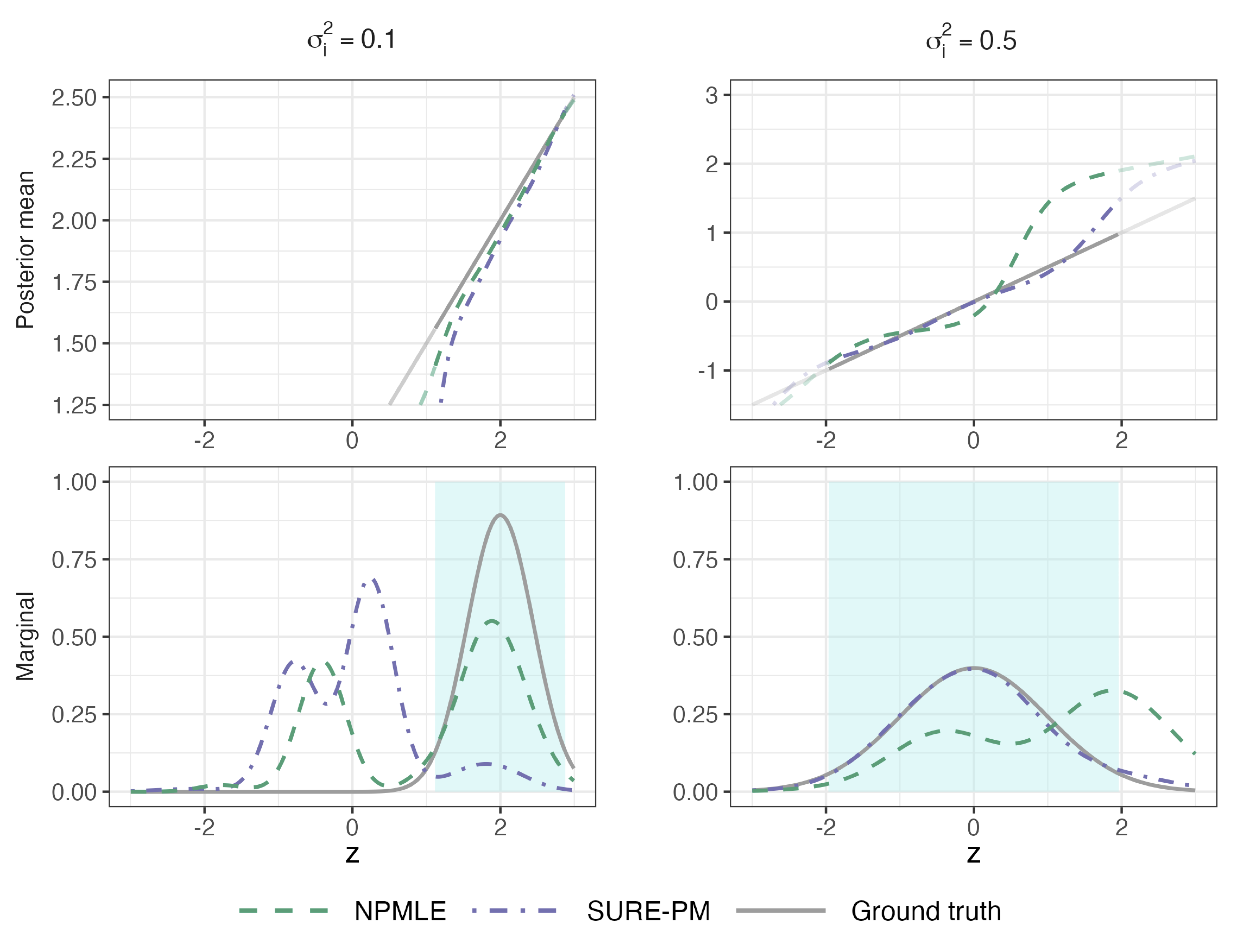}
\caption{Marginal densities (bottom row) and posterior mean functions (top row) for low variance observations ($\sigma_i^2=0.1$, left column) and high variance observations ($\sigma_i^2=0.5$, right column). The posterior means from SURE-PM better match the oracle for high variance observations while maintaining reasonable performance for low variance observations.}
\label{fig:marginal_and_posterior_mean}
\end{figure}

In Figure~\ref{fig:prior} we show the discrete priors estimated by the NPMLE, respectively SURE-PM, as well as the densities of the two Gaussian prior components (corresponding to high and low variance). Qualitatively we observe the following difference: SURE-PM assigns almost no mass to large $\mu_i$ that fall in the support of the low variance component $\mathrm{N}(2, 0.1)$, while the NPMLE places substantial mass therein. How is it possible that the in-sample MSE of SURE-PM is so much smaller than of the NPMLE?

One explanation is provided by
Figure~\ref{fig:marginal_and_posterior_mean}.  We first focus on its first column which pertains to the low variance component ($\sigma_i^2 = 0.1$). The bottom row plots the marginal density of the low variance component \smash{$f_G(z \mid \sigma_i^2 = 0.1)$} for \smash{$G= \widehat{G}^{\text{NPMLE}}$}, \smash{$G=\widehat{G}^{\text{SURE}}$} and for \smash{$G=\mathrm{N}(2,0.1)$} (the true distribution of $\mu_i$ given $\sigma_i^2=0.1$). The ground truth marginal density is the $\mathrm{N}(2, 0.2)$ density, which results from adding the prior variance ($0.1$) and the noise variance ($0.1$). The blue box encloses the 2.5-97.5\% quantiles of this ground truth marginal density. Observe that SURE does not fit this marginal density at all, the mass it places is way too small, while the NPMLE does a much better job (also recall the fitted priors in Figure~\ref{fig:prior}). Meanwhile, the top plot shows the induced posterior means including the oracle posterior mean \smash{$\EE[G_{\star}]{\mu_i \mid Z_i, \sigma_i^2=0.1} = 0.5(Z_i + 2)$} and the posterior means based on the NPMLE and SURE-PM: despite the mismatch of the marginal density with SURE-PM, the implied posterior mean tracks along quite well with the true posterior mean and is only slightly worse than the posterior mean of the NPMLE. The reason  is that the implied posterior mean, via the Eddington/Tweedie formula in~\eqref{eq:general_tweedie} is a function of the score and does not depend on modeling the height of the density precisely, but just its shape.

For the high variance component (right column of Figure~\ref{fig:marginal_and_posterior_mean}), SURE-PM is doing a better job of approximating the marginal density than the NPMLE (e.g., the NPMLE misplaces an additional mode around $2$). The consequence is that SURE-PM matches the posterior mean (top panel) well, while the posterior mean of the NPMLE is biased upward and shrinks observations coming from the $\mathrm{N}(0, 0.5)$ component toward the $\mathrm{N}(2, 0.1)$ component. SURE-PM is able to avoid this by putting very little mass to the $\mathrm{N}(2, 0.1)$ component so that it can match the posterior mean at the high variance component, yet it does not sacrifice performance too much for denoising low variance observations.

This example illustrates why SURE-based methods can outperform NPMLE under misspecification: they optimize for denoising performance rather than density estimation accuracy (in Kullback-Leibler divergence), making trade-offs that minimize overall mean squared error.

\section{Application to the Opportunity Atlas}
\label{sec:opportunity}
We turn to a real-world dataset to see how misspecification affects the performance of SURE-based estimators versus the NPMLE. Following the analysis in \cite{chen2026empirical}, we apply shrinkage methods to a measure of economic mobility from the Opportunity Atlas \citep{chetty2026opportunity}. For each Census tract $i$ (with $n=10,056$), the Opportunity Atlas contains a range of economic mobility estimates and standard errors, which we denote by  $Z_i$ and $\sigma_i$. Specifically, we analyze a measure of economic mobility that captures whether Black children from low-income families (25th income percentile) in tract $i$ will reach high incomes (top 20th percentile) as adults.

As noted by~\citet{chen2026empirical}, Census tracts with larger Black populations have both lower $\sigma_i^2$ (more observations for estimating the economic mobility measure) and lower $\mu_i$ (less economic mobility), which implies strong dependence between $\mu_i$ and $\sigma_i^2$. We apply NPMLE and SURE-PM, both of whose working priors incorrectly posit independence ``$\mu_i \indep \sigma_i^2$''. Figure~\ref{fig:shrinkage_and_prior} suggests that while NPMLE systematically underestimates $\mu_i$ for high-uncertainty tracts, SURE-PM performs better under this misspecification.
This can be understood by examining the different discrete priors learned by the methods, shown in the right panel of  Figure~\ref{fig:shrinkage_and_prior}. While nearly all the prior mass from the NPMLE is placed on $u_j$ less than 0.1, SURE-PM places most of its prior mass above 0.1. (This is similar to what occurred in the simulated example from Figure~\ref{fig:marginal_and_posterior_mean}.)

\begin{figure}
\includegraphics[width=\linewidth]{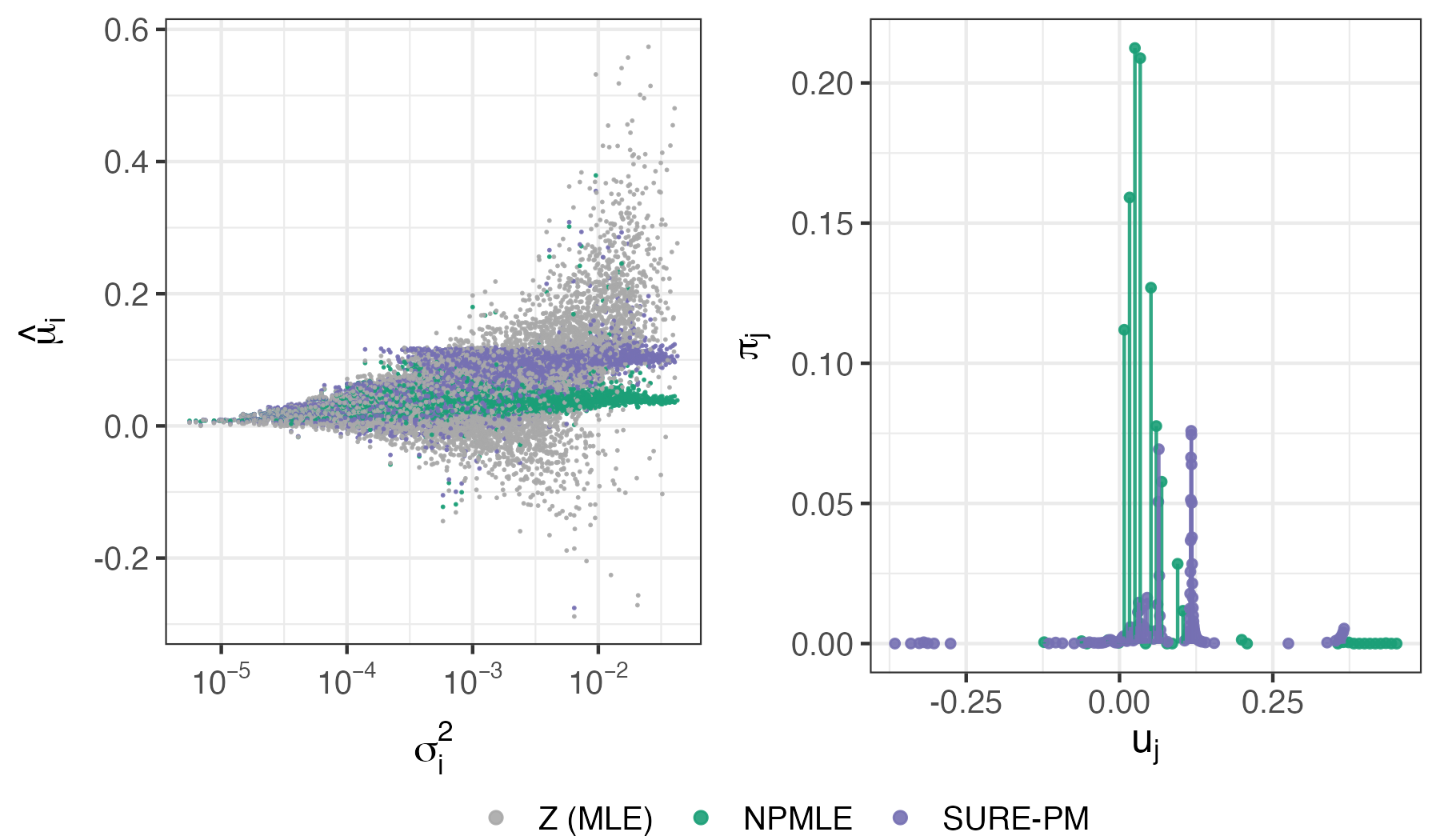}
\caption{Scatter plot (left)  of point estimates $\hat{\mu}_i$ (corresponding to the MLE ($Z_i$), the NPMLE, and SURE-PM) versus $\sigma_i^2$ and estimated priors (right) of NPMLE and SURE-PM applied to the Atlas dataset. Grid points $u_j$ associated with zero probability mass are excluded in the estimated prior plot.}
\label{fig:shrinkage_and_prior}
\end{figure}

We next turn to a quantitative comparison of SURE-PM, as well as SURE-LS and SURE-THING, versus the NPMLE.
Moreover, we consider two variants of SURE-LS/SURE-THING: first, we apply them using as only covariate the standard deviation $\sigma_i$. Second, we apply them using further tract-level side-information available from the Opportunity Atlas, e.g., poverty level percentage of college-educated individuals. We use the same nine covariates as in~\citet[Online Appendix, Table 5.2]{chen2026empirical}.

\begin{table}
\caption{Evaluation via data fission: Normalized performance gains of SURE-trained methods as compared to NPMLE in the Opportunity Atlas.}
    \renewcommand{\arraystretch}{1.2} 
\begin{tabular*}{\linewidth}{@{\extracolsep{\fill}}*{4}{c}}
\hline
Covariates $X_i$ &  Estimator & \makecell{Normalized performance gains} &  Standard error \\
\hline 
\multirow{ 2}{*}{None} & NPMLE & 100.0\% & --\\
 & SURE-PM & 127.7\% & 0.6\% 
\\ 
\hline 

\multirow{ 2}{*}{$\sigma_i$} & SURE-LS & 153.0\% & 1.7\%\\
 & SURE-THING & 155.2\% & 1.2\% 
 \\ 
\hline 

\multirow{ 2}{*}{$\bigl( X_{i1}, \ldots, X_{i9}, \sigma_i \bigr)$} & SURE-LS & 175.4\% & 1.4\%
\\
& SURE-THING & 180.0\% & 1.6\% \\

\hline
\end{tabular*}
\label{tab:data_fission}
\end{table}

For the evaluation we use data fission to estimate the MSE under doubled noise variance~\citep{leiner2023data, oliveira2024unbiased}. 
For $B=25$ replicates and $b=1,\ldots,B$, we generate \smash{$\varepsilon_i^{(b)} \simiid \mathrm{N}(0, \sigma_i^2)$} and then let \smash{$\Zone = Z_i +  \varepsilon_i^{(b)}, \,\Ztwo = Z_i - \varepsilon_i^{(b)}$}.
This yields the two iid data copies with doubled variance,
from which we can use one dataset fold for training and another fold for evaluation. Table~\ref{tab:data_fission} shows the relative performance of these estimators, defined as the squared error improvement over the MLE (\smash{$\hat{\mu}_i^{(b)}= Z_i^{(1,b)}$}), normalized as a multiple of the improvement of the NPMLE over MLE (see Supplement~\ref{sec:data_fission_details} for details). We observe that SURE-PM outperforms the NPMLE (even though both make the incorrect working assumption that ``$\mu_i \indep \sigma_i^2$''). Moreover, by using side-information with SURE-LS or SURE-THING, we almost double the relative performance gains in MSE, as compared to an empirical Bayes analysis using the NPMLE.

\section{Proof elements}
\label{sec:proof_elements}
In this section, we provide a sketch of the key technical ideas underlying our proofs. We hope that this overview will facilitate the analysis of risk for other SURE-based empirical Bayes estimators.

\subsection{Proof elements for SURE-PM (Theorem~\ref{theo:rate_homoscedastic})}
\label{subsec:basic_ineq}

As is well known from nonparametric regression problems, uniform convergence arguments as the ones sketched in Section~\ref{subsec:uniform_convergence} (and detailed in Supplement~\ref{sec:uniform_convergence_rademacher}), often yield suboptimal convergence rates.  
To show fast rates, we pursue an alternative approach based on basic inequalities and localization techniques, 
following corresponding proof strategies for general M-estimation problems~\Citep{geer2000empirical, wainwright2019highdimensional}.

We begin our development by establishing a basic inequality for SURE-PM in the well-specified setting where $G_{\star} \in \mathcal{G}$. In Proposition~\ref{prop:basic_inequality} and Lemma~\ref{lemm:deterministic_inequality} below, we assume that $\sigma_i^2 = 1$ for all $i$ and that we do not have side-information (as in the setting of Theorem~\ref{theo:rate_homoscedastic}). Analogous statements also hold in the heteroscedastic problem with side-information (see proofs in
Supplements~\ref{subsec:proof_basic_inequality} and~\ref{subsec:proof_deterministic_inequality}).

\begin{prop}[Basic inequality]
\label{prop:basic_inequality}
Suppose that $G_{\star} \in \mathcal{G}$ and that we estimate $\hG$ by minimizing $\SURE(G)$ in~\eqref{eq:simple_SURE} over $\mathcal{G}$.  Writing $s_{\star}=s_{G_\star}$, then:
\begin{equation*}
\begin{aligned}
\frac{1}{n}\sum_{i=1}^n  \cb{\score_{\hG}(Z_i)-\score_{\star}(Z_i)}^2 \leq \frac{2}{n}\sum_{i=1}^n  \!\sqb{ \score_{\star}(Z_i)\cb{\score_{\star}(Z_i)-\score_{\hG}(Z_i)}  + \frac{\partial}{\partial z}\!\cb{\score_{\star}(Z_i)-\score_{\hG}(Z_i)}\!}.
\end{aligned}
\end{equation*}
\end{prop}
It may be worthwhile to compare the above result to the basic inequality for least squares in regression. Consider the model $Z_i=m_\star(X_i)+\xi_i$, where $\EE{\xi_i\mid X_i}=0$, and let $\widehat m$ minimize least squares over a class $\mathcal M$ containing $m_\star$. The basic inequality is 
$$
\frac{1}{n}\sum_{i=1}^n \{\widehat m(X_i)-m_\star(X_i)\}^2 \leq \frac{2}{n}\sum_{i=1}^n \xi_i\cdot \{\widehat m(X_i)-m_\star(X_i)\}.
$$
With some abuse of notation, we can rewrite the summands on the right-hand side of Proposition~\ref{prop:basic_inequality} as $(s_\star(Z_i)+\tfrac{\partial}{\partial z})\cdot\{s_\star(Z_i)-s_{\widehat G}(Z_i)\}$. Thus, $s_\star(Z_i)+\tfrac{\partial}{\partial z}$ informally plays the role of the regression noise $\xi_i$.

The basic inequality leads to the following deterministic lemma.

\begin{lemm}[Deterministic inequality]
\label{lemm:deterministic_inequality}
In the setting of Proposition~\ref{prop:basic_inequality}, define the local complexity functional by
\begin{equation*}
\begin{aligned}
\widehat W(r) := &\sup\Bigg\{ \frac{2}{n}\sum_{i=1}^n \sqb{s_{\star}(Z_i)\cb{s_{\star}(Z_i) - s_G(Z_i)}  +\frac{\partial}{\partial z}\cb{\score_{\star}(Z_i)-\score_{G}(Z_i)}}: \\ 
&\;\;\;\;\;\;\;\;\;\;\;\;\;\;\;\;\;\;\;  G \in \mathcal{G}\,\, \text{ s.t. }\,\, \frac{1}{n}\sum_{i=1}^n \cb{s_{\star}(Z_i) - s_G(Z_i)}^2 \leq r^2 \Bigg\}.
\end{aligned}
\end{equation*}
and define the corresponding greatest (post)fixed point\footnote{From the fact that $\widehat{W}(r)$ is monotonically increasing, one can directly show that $(\hat r)^2 = \widehat W(\hat r)$. So $\hat r$ is the greatest fixed point of $r \mapsto \widehat{W}^{1/2}(r)$---this is a special case of the Knaster-Tarski theorem \citep{tarski1955lattice}.} by $\hat r := \sup \{ r \ge 0 \;:\; r^2 \le \widehat W(r) \}.$
Then it holds that
$\tfrac{1}{n}\sum_{i=1}^n \cb{s_{\star}(W_i) - s_{\hG}(W_i)}^2 \le \hat r^2. 
$
\end{lemm}
To turn Lemma~\ref{lemm:deterministic_inequality} into sharp convergence guarantees, we require an implication of the following form:
\begin{align}
&\,\frac{1}{n}\sum_{i=1}^n\cb{s_{\star}(W_i) - s_G(W_i)}^2 \text{ is ``small'' }    \label{eq:wrong_isometry} \\
 \Longrightarrow \;\;\;
& \abs{\frac{1}{n}\sum_{i=1}^n \sqb{s_{\star}(W_i)\cb{s_{\star}(W_i) - s_G(W_i)}  +\frac{\partial}{\partial z}\cb{\score_{\star}(W_i)-\score_{G}(W_i)}}} \text{ is ``small.''} \nonumber
\end{align}
The major technical challenge in our analysis is that an implication of the form in~\eqref{eq:wrong_isometry} does not hold in general. For example, two functions being close in $L^2$ does not always imply that their derivatives are close---in general, one or both of the functions may not even be differentiable! Yet, precisely such an implication holds in our setting, as demonstrated by the next theorem.
\begin{theo} 
\label{theo:apply-gaussian-strip-interpolation}
Let $G, G_\star \in \mathcal{P}(M)$. There exist constants $C, C_0 > 0$ such that  
\begin{equation*}
\int \left( \frac{\partial}{\partial z} \score_{G_\star}(z) - \frac{\partial}{\partial z}  \score_{G}(z)  \right)^2 f_{G_\star}(z)dz \leq C(1 + M^2) \left[ 1 +  \log^2_+\left( \frac{C_0 M^2}{\Dfisher{f_{G_\star}}{f_G}} \right) \right] \Dfisher{f_{G_\star}}{f_G},
\end{equation*}
where $\log_+(z) := \max\{\log z, 0\}$.
The right-hand side is interpreted as zero when the Fisher divergence is zero.
\end{theo}

Theorem \ref{theo:apply-gaussian-strip-interpolation} can be interpreted as a stability result for posterior moments. Indeed, by the Eddington/Tweedie formula, $\EE[G]{\mu_i \mid Z_i=z} = z+\score_G(z)$, while $\tfrac{\partial}{\partial z} \EE[G]{\mu_i \mid Z_i=z} = \Var[G]{\mu_i \mid Z_i=z}$  (recall Remark~\ref{rema:variance_sure}).
Thus, the theorem shows that closeness of posterior means in $L^2(f_{G_\star})$ implies closeness of posterior variances, up to logarithmic factors. 

As mentioned above, an implication as that of Theorem~\ref{theo:apply-gaussian-strip-interpolation} does not hold for general smooth functions. The key property making this possible here is the analytic structure of Gaussian mixtures. For $G\in\mathcal P(M)$, the posterior mean extends to a uniformly bounded holomorphic function on a complex strip around the real axis of width proportional to $1/M$. Our proof exploits this extension through a Gaussian-weighted Fourier interpolation argument. This is conceptually related to the root-exponential convergence of Hermite approximations for functions analytic in a strip, studied by \citet{wang2026convergence}. 

In addition to Lemma~\ref{lemm:deterministic_inequality} and Theorem~\ref{theo:apply-gaussian-strip-interpolation}, the proof of Theorem~\ref{theo:rate_homoscedastic} also builds on various results from empirical process theory. For instance, in Lemma~\ref{lemm:second-covering} of Supplement~\ref{sec:proofs_sure_homoscedastic}, we construct coverings of $\{\frac{\partial}{\partial z}s_G(z)\,:\, G \in \mathcal{P}(M)\}$; building on existing constructions for controlling the complexity of Gaussian mixture densities~\citep{ghosal2001entropies, zhang2009generalized}.
We present the full proof in
Supplement~\ref{sec:proofs_sure_homoscedastic}.

\begin{rema}[Log-Sobolev inequalities and Jiang--Zhang induction]
\label{rema:proof}
An earlier version of our proof obtained an analogue of
Theorem~\ref{theo:apply-gaussian-strip-interpolation} using the
powerful induction argument of~\citet[Lemma 1]{jiang2009general}
together with a log-Sobolev inequality. The argument of~\citet{jiang2009general} shows that
when two Gaussian mixture densities are close in Hellinger distance,
their first derivatives are also close in a weighted $L^2$ norm.
We extended their induction argument to second derivatives, which appear
in SURE through $\frac{\partial}{\partial z}s_G$.
To control the Hellinger distance by the Fisher divergence, we used
the fact that Gaussian mixtures $f_G$ with $G \in \mathcal{P}(M)$
have a finite log-Sobolev constant
$C_{\mathrm{LS}}(f_G)$~\citep{zimmermann2016elementary,
bardet2018functional, chen2021dimensionfree}, which gives
$\Dhel^2(f_{G_\star},f_G) \leq
\tfrac14 C_{\mathrm{LS}}(f_G)\Dfisher{f_{G_\star}}{f_G}$
\citep{koehler2023statistical}.
We retain these arguments in Supplement~\ref{subsec:homosc_proof_elements} because we find the connection to log-Sobolev inequalities conceptually appealing. However, this approach has two disadvantages: the bound $C_{\mathrm{LS}}(f_G) \lesssim (M+1)^2\exp(4M^2)$~\citep{chen2021dimensionfree} leads to an exponential dependence on $M^2$, and even for fixed $M$, the final regret bound is larger by a factor of $\log^2 n$ than that in Theorem~\ref{theo:rate_homoscedastic}. \end{rema}

\subsection{Proof elements for SURE-LS (Theorem~\ref{theo:reg})}

Our argument for SURE-LS builds on the following inequality, whose proof (in Supplement~\ref{subsec:deterministic_argument_misspecified}) is closely tied to the assumption that $\mathcal{L}_{\oracle}$ and $\mathcal{B}_{\oracle}$ are star-shaped.
\begin{prop}
\label{prop:regression-basic}
Suppose that $\mathcal L, \mathcal B$ are star shaped about the oracles $\lambda_{\oracle}, b_{\oracle}$ respectively.
Then, the following deterministic inequality holds for $\hat{\lambda}, \hat{b}$:
\begin{align*}
&\frac{1}{n} \sum_{i = 1}^n \cb{\p{\hat \lambda(X_i) - \lambda_{\oracle}(X_i)} Z_i - \p{\hat b(X_i) - b_{\oracle}(X_i)}}^2 \\ 
&\;\;\;\; \le \;\; \frac{1}{n} \sum_{i = 1}^n \cb{\p{2\lambda_{\oracle}(X_i) \mu_i - b_{\oracle}(X_i)}(\lambda_{\oracle}(X_i) - \hat{\lambda}(X_i))- \lambda_{\oracle}(X_i) (b_{\oracle}(X_i) - \hat b(X_i))}\xi_i \\ 
&\;\;\;\;\;\; + \; \frac{1}{n} \sum_{i=1}^n \cb{\lambda_{\oracle}(X_i)(\lambda_{\oracle}(X_i) - \hat \lambda(X_i))}(\xi_i^2 - \sigma_i^2).
\end{align*}
\end{prop}
\noindent The challenge associated with turning the inequality of Proposition~\ref{prop:regression-basic} into the fast rates of Theorem~\ref{theo:reg} is conceptually related to our discussion following equation~\eqref{eq:wrong_isometry}. Our argument would be relatively standard, if we could show an implication of the following form:
\begin{align}
&\,\frac{1}{n} \sum_{i = 1}^n \cb{\p{\hat \lambda(X_i) - \lambda_{\oracle}(X_i)} Z_i - \p{\hat b(X_i) - b_{\oracle}(X_i)}}^2  \text{ is ``small'' }    \label{eq:first_line_cancellation} \\
\Longrightarrow \;\;\;& \frac{1}{n} \sum_{i = 1}^n \cb{\hat \lambda(X_i) - \lambda_{\oracle}(X_i)}^2 \quad \text{and} \quad \frac{1}{n}\sum_{i = 1}^n \cb{\hat b(X_i) - b_{\oracle}(X_i)}^2 \;\text{ are ``small.''}   \nonumber
\end{align}
To establish such an inequality, we need to preclude the possibility of cancellation of errors in~\eqref{eq:first_line_cancellation}.  The following anti-concentration/lower isometry estimate is crucial to our argument---it allows us to argue that terms of the form $c_i + f_i \xi_i$ do not typically cancel out as long as $c_i$ and $f_i$ are from a low complexity class of functions, without explicit dependence on the size of the $c_i$. We accomplish this via the following key technical lemma, which we present in the special case of bounded noise variables $\xi_i$ (postponing the statement for sub-Gaussian noise to  Lemma~\ref{lemm:subgaussian-isometry} of Supplement~\ref{subsec:lower_isometry}).
\begin{lemm}
\label{lemm:isometry_bounded_noise}
Suppose that $\xi_1,\ldots,\xi_n$ are independent, \smash{$\abs{\EE{\xi_i}} \le \sqrt{\EE{\xi_i^2}}/2$}, each satisfy $\EE{\xi_i^2} \ge 1$, and are valued in $[-\upperbound, \upperbound]$ for $\upperbound \ge 1$. 
Consider a separable set $\mathcal H \subset \mathbb{R}^n \times [-1,1]^n$ and  denote a generic element of $\mathcal{H}$ by $(c,f) = (c_1,\ldots,c_n,f_1,\ldots,f_n) \in \mathcal{H}$. Let
$$r^2 := \sup_{(c,f) \in \mathcal H} \frac{1}{n}\sum_{i=1}^n f_i^2,\;\quad \rademacher(\mathcal{H}) := \EE{\sup_{(c,f) \in \mathcal H} \abs{ \frac{1}{n}\sum_{i =1}^n  \varepsilon_i c_i} + B\sup_{(c,f) \in \mathcal H} \abs{ \frac{1}{n}\sum_{i =1}^n  \varepsilon_i f_i}},$$
where the right-hand side expectation is taken over iid Rademacher random variables $\varepsilon_i$.
Then, for any $x \geq 0$, with probability at least $1 - e^{-x}$, uniformly over all $(c, f) \in \mathcal H$, it holds that,
$$\frac{1}{n}\sum_{i = 1}^n (c_i + f_i \xi_i)^2 \ge \frac{1}{8n}\sum_{i = 1}^n (c_i^2 + f_i^2) - 72 B\rademacher(\mathcal{H}) - 17 \upperbound r \sqrt{\frac{\max_{i=1}^n\EE{\xi_i^2} x}{n}} - 180\frac{\upperbound^2 x}{n}.$$
\end{lemm}
We use Lemma~\ref{lemm:isometry_bounded_noise} to prove an implication of the form in~\eqref{eq:first_line_cancellation} as follows. Recall that $Z_i = \mu_i + \xi_i$.  For $\lambda(\cdot) \in \mathcal{L}$ and $b(\cdot) \in \mathcal{B}$, let 
$
f_i := \lambda(X_i) - \lambda_{\oracle}(X_i)$,
$c_i := \mu_i \cb{\lambda(X_i) - \lambda_{\oracle}(X_i)} - \cb{b(X_i) - b_{\oracle}(X_i)}$.
Then, if, $n^{-1}\sum_{i = 1}^n \{(\lambda(X_i) - \lambda_{\oracle}(X_i)) Z_i - (\hat b(X_i) - b_{\oracle}(X_i))\}^2$ is small, Lemma~\ref{lemm:isometry_bounded_noise} implies that $n^{-1}\sum_{i=1}^n \{ \lambda(X_i) - \lambda_{\oracle}(X_i)\}^2$ is also small.
Using the deterministic inequality in Proposition~\ref{prop:regression-basic}, localization techniques, and empirical process theory, eventually yields the statement of Theorem~\ref{theo:reg}. The complete proof details, incorporating these techniques along with our approximate lower isometry bound, are provided in Supplement~\ref{sec:appendix_regression_proofs}.

\subsection*{Reproducibility}
All numerical results in this paper are fully third-party reproducible with code available on Github:
\url{https://github.com/sulagna-ghosh/score-matching-empirical-bayes}

\subsection*{Acknowledgments}
We would like to thank William Biscarri, Jiafeng Chen, Yan Chen, Claire Donnat, Bodhisattva Sen and Jake Soloff for helpful discussions, and Aaron Schein for suggesting the acronym ``SURE-THING.'' Theorem~\ref{theo:apply-gaussian-strip-interpolation} and its proof were suggested to us by GPT-5.4. This theorem improves the dependence on $M$ and logarithmic factors in $n$ in Theorem~\ref{theo:rate_homoscedastic} relative to the corresponding result in our earlier preprints, which was not AI-assisted; see also Remark~\ref{rema:proof}. Part of the computing for this project was conducted on UChicago's Data Science Institute cluster. 
N.I. gratefully acknowledges support from NSF (DMS 2443410). 

\bibliographystyle{plainnat}
\bibliography{ebscore}

\appendix 

\setcounter{prop}{0}

\renewcommand{\thetable}{S\arabic{table}}
\renewcommand{\theprop}{S\arabic{prop}}
\renewcommand{\thetheo}{S\arabic{theo}}
\renewcommand{\thelemm}{S\arabic{lemm}}

\section{Remarks and discussions omitted in the main text}
\label{sec:omitted_remarks}

\subsection{Further bibliographic remarks on E-modeling}
\label{subsec:emodeling}

Beyond the Gaussian empirical Bayes problem,~\citet{james2022irrational} and~\citet{jana2023empirical} develop E-modeling approaches for the Beta, respectively Poisson empirical Bayes problems by using an unbiased risk estimate 
based on the Stein operator for the Beta, respectively Poisson distribution. The rates in~\citet{jana2023empirical}  are sharp.
Closely related procedures to E-modeling via SURE have appeared under the names ``regression without supervision''~\citep{raphan2006learning}, ``SURE2PLS'' (SURE-optimized parametric least squares)
~\citep{raphan2011least}, and ``SDA-SURE'' (stacked denoising autoencoder SURE)~\citep{soltanayev2018training}.
Other E-modeling approaches include those of \citet{ignatiadis2023empiricala} and~\citet{ignatiadis2024empirical} based on sample splitting, respectively data fission~\citep{leiner2023data}, as well as image denoising methods such as Noise2Self~\citep{batson2019noise2self} and Noise2Score~\citep{kim2021noisescore}.

\subsection{Fisher divergence and empirical Bayes regret}
\label{subsec:fisher_div_ebregret}

The Fisher divergence in~\eqref{eq:fisher_divergence} may be extended to the general setting of Section~\ref{sec:general_setting} by considering the expectation of the conditional Fisher divergence~\citep{arbel2018kernel},
\begin{equation}
\label{eq:fisher_divergence_conditional}
\Dfisher{f_{G_{\star}}}{f_G} := 
\EE[G_\star]{ \Dfisher{f_{G_{\star}}(\cdot \mid X)}{f_G(\cdot \mid X)}} = \EE[G_\star]{\abs{s_{G_\star}(W) - s_{G}(W)}^2}, 
\end{equation}
where $s_G(\cdot)$ is the conditional score in~\eqref{eq:conditional_score}.
Combining formula~\eqref{eq:regret_as_posterior_mean_mse}, and the Eddington/Tweedie formula in~\eqref{eq:general_tweedie} it follows that:
\begin{equation}
\label{eq:regret_as_fn_fisher_divergence}
\Regret(G_{\star}, G) = \EE[G_\star]{ \sigma^4 \Dfisher{f_{G_{\star}}(\cdot \mid X)}{f_G(\cdot \mid X)}}.
\end{equation}
When $\sigma_i^2=1$ for all $i$, the above simplifies as $\Regret(G_{\star}, G) \equiv \Dfisher{f_{G_{\star}}}{f_G}$ (recall the definition in~\eqref{eq:fisher_divergence_conditional}).

\subsection{Partial integration for SM and SURE}
\label{subsec:partial_integration}

The argument given by~\citet{cox1985penalty} and~\citet{hyvarinen2005estimation} for score matching proceeds by first writing $s_\star(W_i) = (\partial f_{\star}(Z_i \mid X_i)/\partial z) / f_{\star}(Z_i \mid X_i)$, where $f_{\star}$ is the (conditional) density of $Z_i \mid X_i$ in the data generating distribution, and arguing that,
$$
\begin{aligned}
\EE{s_\star(W_i)s(W_i) \mid X_i} &= \int \frac{ \partial }{\partial z} f_{\star}(z \mid X_i) s(z, X_i) dz \\ 
&= -\int \frac{ \partial }{\partial z}s(z, X_i) f_{\star}(z \mid X_i)dz = -\EE{\frac{\partial}{\partial z}s(Z_i, X_i) \mid X_i}.
\end{aligned}
$$
Going from the first to the second line, we use a partial integration argument, assuming sufficient regularity and that $s(z,X_i)f_{\star}(z, X_i)$ vanishes at $z = \pm \infty$. This argument does not rely in any way on the Gaussian noise structure in~\eqref{eq:EB_heterosc_side_info}. In our setting, the partial integrations for SURE and score matching are tightly tied together. Using the Eddington/Tweedie formula in~\eqref{eq:general_tweedie}, we have that
$$
\EE{(Z_i-\mu_i)s(W_i) \mid W_i} = \EE{Z_i - \mu_i  \mid W_i} s(W_i) = - \sigma_i^2 s_{\star}(W_i) s(W_i),
$$
and so we can recover the partial integration argument for score matching via Stein's lemma.

\subsection{Methods for the homoscedastic Gaussian means problem}
\label{subsec:homosc_comparison}
The rate in Theorem~\ref{theo:rate_homoscedastic} for the empirical Bayes regret is minimax optimal up to logarithmic factors~\citep{li2005convergence,polyanskiy2021sharp}.
\citet{li2005convergence} propose an F-modeling approach for which they prove an upper bound on the regret of \smash{$O(n^{-1}\log^8 n)$}. In a bit more detail, they estimate \smash{$f_{G_\star}(\cdot)$} and its derivative \smash{$\partial f_{G_{\star}}/\partial z$} by kernel density estimates (with a suitably tuned bandwidth and a higher order kernel), plug these estimates into the Eddington/Tweedie formula in~\eqref{eq:score_and_marginal}, and finally they truncate the estimated Bayes rule. 
Perhaps the most well-known method (and result), however, is due to the seminal paper of~\citet{jiang2009general}. They propose a G-modeling approach (as we do here) that estimates $G$ by the nonparametric maximum likelihood estimator~\citep[NPMLE]{robbins1950generalization, kiefer1956consistency},
$
\widehat{G}^{\text{NPMLE}} \in \argmax_G \cb{ \sum_{i=1}^n  \log\p{f_{G}(Z_i)}}.
$
When $G_{\star} \in \mathcal{P}(M)$,~\citet{jiang2009general} prove an upper bound of $O(n^{-1}\log^5 n)$ on the regret. 

More recently, \citet[Corollaries~5--6]{chen2026sharp} prove an expected regret bound of $O(n^{-1}\log^2 n/\log\log n)$ for the NPMLE constrained to $\mathcal{P}(M)$, for fixed $M$. They also show that a method  based on the Birg\'e--Le Cam estimator achieves $O(n^{-1}(\log n/\log\log n)^2)$, matching the minimax lower bound of \citet{polyanskiy2021sharp}. This latter construction is computationally intractable.

In this specific setting, we do not advocate replacing the NPMLE with our proposed SURE-trained method. For instance, under discretization, the NPMLE reduces to a convex programming problem~\citep{koenker2014convex} while the SURE objective remains non-convex. Nonetheless, our theoretical guarantees and empirical results (in Supplement~\ref{subsec:homosc_simulations}) demonstrate that SURE-training achieves comparable statistical and practical performance to the state-of-the-art method (the NPMLE) in this well-understood setting. We also refer to~\citet{ritov2024no} for further connections of the NPMLE and SURE in the homoscedastic Gaussian means problem. 

\subsection{SURE differs from MLE}
\label{subsec:sure_vs_mle}
It is instructive to compare the $\SURE$ objective in~\eqref{eq:sure_regression} to the marginal (log) maximum likelihood (MLE) objective:
\begin{equation}
\ell_{\text{MLE}}(G) := \frac{1}{n} \sum_{i=1}^n \cb{\frac{(Z_i - m(X_i))^2}{\sigma_i^2+A(X_i)} + \log(\sigma_i^2 + A(X_i))}. \label{eq:mle_regression}
\end{equation}
In special cases, for instance when $\sigma_i^2$ is the same for all $i$ and when $A(X_i)$ is constant as a function of $X_i$, then the two losses have the same minimizer as they imply the same first-order optimality conditions. 
The strong denoising guarantees under misspecification in Theorem~\ref{theo:reg} do not hold in general for the MLE. We discuss this more in the next remark when $A_{\star}$ is fixed and known.

\begin{rema}[SURE vs MLE with known $A_{\star}$]\label{rema:sure-beats-mle}
To better illustrate the differences between \eqref{eq:sure_regression} and \eqref{eq:mle_regression}, consider the special case where $A_{\star}(X_i)$, equivalently $\lambda_{\star}(X_i)$, is known but the regression function is unknown and possibly misspecified. In this case, $m(X_i)$ is the only unknown so optimizing the SURE objective reduces to minimizing
\begin{equation}\label{eqn:sure-fixed-A}
\frac{1}{n} \sum_{i=1}^n  \sigma_i^4 \frac{(Z_i - m(X_i))^2}{(\sigma_i^2 +A_{\star}(X_i))^2} = \frac{1}{n} \sum_{i = 1}^n \left(\frac{\Var{Z_i \mid \mu_i}}{\Var{Z_i \mid X_i}}\right)^2 (Z_i - m(X_i))^2
\end{equation}
and the MLE objective reduces to minimizing
\begin{equation}\label{eqn:mle-fixed-A}
\frac{1}{n} \sum_{i=1}^n \frac{(Z_i - m(X_i))^2}{\sigma_i^2+A_{\star}(X_i)} = \frac{1}{n} \sum_{i = 1}^n \frac{1}{\Var{Z_i \mid X_i}} (Z_i - m(X_i))^2 
\end{equation}
so the objectives correspond to two different versions of weighted least squares.

As we will now elaborate, the weights in \eqref{eqn:sure-fixed-A}, i.e., for SURE, are much better behaved when the model is misspecified and our goal is to denoise the $Z_i$. Consider the behavior in a joint limit where both $\sigma_1 \to 0$ and $A_{\star}(X_1) \to 0$ with everything else fixed. For the MLE \eqref{eqn:mle-fixed-A}, the weight for $(X_1,Z_1)$ will go to infinity, so the model will be forced to choose $m$ minimizing $(Z_1 - m(X_1))^2$ above all else. 

Intuitively, this is undesirable for at least two reasons: (1) in a misspecified model the pair $(X_1,Z_1)$ could be some type of outlier which we (as well as the oracle) do not necessarily need to fit well, and (2) as $\sigma_1\to 0$, the trivial estimate $Z_1$ is already a very good estimate for $\mu_1$. If moreover $\sigma_1^2/A_{\star}(X_1)\to 0$, then $\lambda_{\star}(X_1)\to 0$, so learning $m(X_1)$ becomes unimportant for denoising. The MLE does not directly target the denoising MSE, and Example~\ref{exam:mle-suboptimal} shows that it need not converge to the oracle denoiser under misspecification. In contrast, in the same limit $\sigma_1,A_{\star}(X_1) \to 0$ the weight of the pair $(X_1,Z_1)$ in the SURE objective \eqref{eqn:sure-fixed-A} will always stay bounded by $1$ (in fact, the weight of every point is at most $1$ due to the law of total variance). So SURE is well-behaved in the same limit, and as Theorem~\ref{theo:reg} shows, it will indeed be a consistent estimator for the oracle denoising function.
\end{rema}

We next provide more details for Example~\ref{exam:mle-suboptimal}.

\paragraph{Details for Example~\ref{exam:mle-suboptimal}.}
Let $\bar Z_1$ and $\bar Z_2$ be the sample means in the groups with $(\mu_i,\sigma_i^2)=(3,2)$ and $(\mu_i,\sigma_i^2)=(0,8)$, respectively. For the working model with $A=4$, the same weighted least squares calculations as in~\eqref{eqn:sure-fixed-A} and~\eqref{eqn:mle-fixed-A} give
$$
\hat m_{\mathrm{MLE}}=\frac{2\bar Z_1+\bar Z_2}{3},\,\,\qquad
\hat m_{\mathrm{SURE}}=\frac{\bar Z_1+4\bar Z_2}{5}.
$$
Since $\bar Z_1=3+O_{\mathbb P}(n^{-1/2})$ and $\bar Z_2=O_{\mathbb P}(n^{-1/2})$, these give the expressions for $\hat m_{\mathrm{MLE}}$ and $\hat m_{\mathrm{SURE}}$ in the example. For fixed $m$, using $\lambda(X_i)=1/3$ in the first group and $\lambda(X_i)=2/3$ in the second group, we get
$$
\frac{1}{n}\sum_{i=1}^n\EE[\boldmu]{\cb{\mu_i-\hat\mu_i(m)}^2}
=\frac{8}{9}+\frac{(m-3)^2+4m^2}{18}
=\frac{58}{45}+\frac{5}{18}\left(m-\frac{3}{5}\right)^2.
$$
Thus, $m_{\oracle}=3/5$. For either estimator, its variance and its covariance with each $Z_i$ are of order $O(n^{-1})$. Expanding the squared errors after plugging in $\hat m$ therefore gives in-sample MSE $11/6+O(n^{-1})$ for the MLE and $58/45+O(n^{-1})$ for SURE, proving the stated limits.

\subsection{Further examples for Theorem~\ref{theo:reg}}
\label{subsec:further_further_examples}

\begin{exam}[Pure regression setting.]
Suppose we take $\mathcal{L}:=\cb{1}$, where our notation here identifies a constant function with its value. Then minimizing $\SURE(1,b)$ in~\eqref{eq:SURE_gauss} over $b \in \mathcal{B}$ for a class $\mathcal{B}$ is equivalent to minimizing the least squares objective $n^{-1} \sum_{i=1}^n  \cb{ Z_i - b(X_i)}^2$.
Theorem~\ref{theo:reg} provides a high-probability bound of the form,
$
n^{-1}\sum_{i=1}^n \{\hat b(X_i) - b_\oracle(X_i)\}^2 \le t_*^2,\;\; \text{ where}\;\;  b_{\oracle} \in \argmin_{b \in \mathcal B}\{ n^{-1}\sum_{i=1}^n ( \mu_i - b(X_i))^2\},
$
and $Z_i=\mu_i + \xi_i$, $\EE{\xi_i}=0$.
Thus in this special case, Theorem~\ref{theo:reg} is essentially a standard oracle inequality for fixed-design regression, see e.g.,~\Citet{geer2000empirical}, stated in terms of local Gaussian complexity.
\end{exam}

\begin{exam}[Group-linear estimators and CLOSE-Gauss.] Here we also consider the setting wherein $X_i = \sigma_i^2$. In this setting,~\citet{weinstein2018grouplinear} explain 
why the best shrinkage rule of the form~\eqref{eq:linear_shrinkage} is conceptually desirable; for instance it allows for the shrinkage patterns to be different for units with different variances.~\citet{chen2026empirical} calls this model ``CLOSE-Gauss,'' where CLOSE is an acronym for ``\underline{c}onditional \underline{lo}cation-\underline{s}cale \underline{e}mpirical Bayes. \citet{weinstein2018grouplinear} propose the group-linear (GL) estimator $\grouplinear_i$ which has the following risk properties:
\begin{enumerate}[label=(\roman*.)]
\item $n^{-1}\sum_{i=1}^n \EE{ (\mu_i - \grouplinear_i)^2 \mid \sigma_i^2}\, <\, n^{-1}\sum_{i=1}^n \sigma_i^2$, that is, $\grouplinear_i$ dominates $\hat{\mu}_i=Z_i$;
\item $ \limsup_{n \to \infty} n^{-1}\sum_{i=1}^n \EE{ (\mu_i - \grouplinear_i)^2} \leq \inf_{\lambda, b}\cb{ \EE{ \{\mu_i - b(\sigma_i^2) - (1-\lambda(\sigma_i^2))Z_i\}^2}}$ if $\sigma_i^2 \mapsto \EEInline{\mu_i \mid \sigma_i^2}$ and $\sigma_i^2 \mapsto \VarInline{\mu_i \mid \sigma_i^2}$ are uniformly continuous (and further technical conditions);
\item $n^{-1}\sum_{i=1}^n \EE{ (\mu_i - \grouplinear_i)^2 \mid \sigma_i^2} = O(n^{-2/3})$ when $\VarInline{\mu_i \mid \sigma_i^2}=0$ and $\sigma_i^2 \mapsto \EE{\mu_i \mid \sigma_i^2}$ is $L$-Lipschitz.
\end{enumerate}
Our SURE-tuning approach forgoes exact dominance as in (i\.), with the upshot of upgrading regret results such as (ii\.) to be uniformly valid over a class of functions (rather than pointwise) and relaxing the assumption of Gaussian noise. Moreover, if we take
\begin{equation}
\mathcal{L}:=\cb{1},\, \mathcal{B} \equiv \mathcal{B}^{\text{Lip}} := \cb{b(\cdot) \in [-B,B] \mbox{ and $L$-Lipschitz in }  \sigma^2},
\label{eq:simple_Lipschitz}
\end{equation}
for some $L,B>0$,
then 
using results on the local Gaussian complexity of Lipschitz functions~\citep[Example 13.10]{wainwright2019highdimensional}, 
Theorem~\ref{theo:reg} yields the rate \smash{$t_*^2 = \tilde{O}( n^{-2/3})$}, which matches (up to logarithmic factors) the rate in (iii\.) above. The result continues to hold for more interesting choices of $\mathcal{L}$, for instance for $\mathcal{L} = \mathcal{B}^{\text{Lip}}$. By contrast to the results in~\citet{weinstein2018grouplinear}, Theorem~\ref{theo:reg} has the important advantage that it provides a a regret rate guarantee in case the conditions in (iii\.) do not hold.
\end{exam}

\section{Uniform convergence and Rademacher complexity}
\label{sec:uniform_convergence_rademacher}
The guarantees previewed in Section~\ref{subsec:uniform_convergence}
can be derived through population Rademacher complexity control. For any function class $\mathcal{H}$, we define this complexity as
\begin{equation}
\label{eq:population_rademacher}
\poprademacher(\mathcal{H}) := \EE{ \sup\cb{\abs{\frac{1}{n}\sum_{i=1}^n \varepsilon_i h(W_i)}\,:\, h \in \mathcal{H}}},
\end{equation}
where the $\varepsilon_i$ are iid Rademacher (that is, $\varepsilon_i = \pm 1$ with probability $1/2$) and independent of the $W_i=(Z_i,X_i)$. Note that above we are also taking an expectation over the $Z_i$ (but treat $X_i$ as fixed as in the rest of the paper).\footnote{
The overline in $\poprademacher$ signifies that the expectation is also taken with respect to the $Z_i$.
} We have the following initial result which we state in the well-specified setting wherein $G_{\star} \in \mathcal{G}$.

\begin{prop}[Uniform control by Rademacher complexity]
\label{prop:uniform_control_rademacher} 
Suppose that $G_{\star}  \in \mathcal{G}$.
\begin{itemize}
\item Let $\widehat{G} \in \argmin \cb{\SURE(G): G \in \mathcal{G}}$. Then:
$$
\begin{aligned}
\EE{\Regret(G_{\star}, \hG)} \leq 4 \poprademacher\!\!\p{\cb{h_G(\cdot)\,:\, h_G(w) = \sigma^4 \p{ \score_{G}(w)^2  + 2\frac{\partial}{\partial z}\score_{G}(w)}, \,\; G \in \mathcal{G}}}.
\end{aligned}
$$
\item  Let $\widehat{G} \in \argmin \cb{\mathrm{SM}(G): G \in \mathcal{G}}$. Then:
$$
\begin{aligned}
&\EE{\Dfisher{f_{G_{\star}}}{f_{\hG} }} \leq 4 \poprademacher\!\!\p{\cb{h_G(\cdot)\,:\, h_G(w) =   \score_{G}(w)^2  + 2\frac{\partial}{\partial z}\score_{G}(w), \,\; G \in \mathcal{G}}}.
\end{aligned}
$$
\end{itemize}
\end{prop}
\begin{proof}[Proof sketch.]
For $G \in \mathcal{G}$, write $h_G(w) := \sigma^4 \p{ \score_{G}(w)^2  + 2\frac{\partial}{\partial z}\score_{G}(w)}$. By  symmetrization,
$$\EE[G_{\star}]{\sup_{G \in \mathcal{G}}\abs{\SURE(G) - \EE[G_{\star}]{ (\mu - \EE[G]{\mu \mid W})^2}}} \leq  2\poprademacher\p{\cb{h_G(\cdot)\,:\, G \in \mathcal{G}}}.$$
Using the fact that $\SURE(\widehat{G}) \leq \SURE(G_{\star})$ and the (above) uniform convergence of $\SURE$ to the true risk, we can conclude using standard arguments for M-estimators, e.g.,~\citet[Chapter 4.1]{wainwright2019highdimensional}. For completeness, we provide remaining proof details in Supplement~\ref{subsec:proof_prop_uniform_control_rademacher}. 
\end{proof} 
For score matching, a similar result to that of Proposition~\ref{prop:uniform_control_rademacher} is presented in~\citet[Theorem 1]{koehler2023statistical}.

We now turn to SURE-based training in misspecified settings. Unlike many alternative approaches (such as marginal maximum likelihood), SURE-minimization can provide strong guarantees on the risk even when the true prior $G_{\star}$ lies outside our class $\mathcal{G}$ of candidate priors. To make this point, in what follows we also treat $\boldmu$ in~\eqref{eq:EB_heterosc_side_info} as fixed, so that the only randomness we account for is given by $Z_i \sim \mathrm{N}(\mu_i, \sigma_i^2)$. 
In analogy to the Rademacher complexity in~\eqref{eq:population_rademacher}, we define the Rademacher complexity with fixed $\boldmu$ of a class $\mathcal{H}$ of functions that takes as input $(w,\mu)$:
\begin{equation}
\label{eq:population_rademacher_fixed}
\poprademacherfixedX(\mathcal{H}) := \EE[\boldmu]{ \sup\cb{\abs{\frac{1}{n}\sum_{i=1}^n \varepsilon_i h(W_i, \mu_i)}\,:\, h \in \mathcal{H}}}.
\end{equation}
The expectation above only integrates over randomness in $Z_i \sim \mathrm{N}(\mu_i, \sigma_i^2)$ and the Rademacher $\varepsilon_i$. Our next result establishes that the compound MSE of the SURE-tuned estimator nearly matches the risk of the oracle prior (which will be different from the data-generating prior in case $G_{\star} \notin \mathcal{G}$),
$$
G_{\oracle} \in \argmin\cb{\EE[\boldmu]{\frac{1}{n} \sum_{i=1}^n\cb{Z_i + \sigma_i^2s_G(W_i) - \mu_i}^2}\,:\, G \in \mathcal{G}}, 
$$
as long as the following centered class has low Rademacher complexity:
$$\mathcal{M}_{\oracle} := \cb{h_G(\cdot)\,:\, h_G(w,\mu) = - \sigma^2 \p{z-\mu}\p{s_{G_{\oracle}}(w)-s_G(w)} + \sigma^4 \frac{\partial}{\partial z}\cb{s_{G_{\oracle}}(w)-s_G(w)} , \, G \in \mathcal{G}}.
$$
\begin{prop}[Compound risk control]
\label{prop:uniform_control_rademacher_compound}
 Let $\widehat{G} \in \argmin \cb{\SURE(G): G \in \mathcal{G}}$. Then:
$$
\begin{aligned}
    \EE[\boldmu]{\frac{1}{n} \sum_{i=1}^n \cb{Z_i + \sigma_i^2 \score_{\widehat{G}}(W_i) - \mu_i}^2\,-\,\frac{1}{n} \sum_{i=1}^n\cb{Z_i + \sigma_i^2s_{G_{\oracle}}(W_i) - \mu_i}^2} \, \leq\,  4 \poprademacherfixedX(\mathcal{M}_{\oracle}). 
\end{aligned}
$$
\end{prop}
In a broad class of problems, $\poprademacherfixedX(\mathcal{M}_{\oracle})$ will converge to $0$ as $n \to \infty$, and so our proposal has a guarantee on its frequentist MSE for denoising the $\mu_i$: even if $G_{\star} \notin \mathcal{G}$,  asymptotically we will perform at least as well as the denoiser associated to the best possible notional prior $G_{\oracle} \in \mathcal{G}$. A property as that of Proposition~\ref{prop:uniform_control_rademacher_compound} is not true for estimators $\widehat{G}$ based on another principle (e.g., marginal maximum likelihood) instead of SURE.

Taken together, Propositions~\ref{prop:uniform_control_rademacher} and~\ref{prop:uniform_control_rademacher_compound} demonstrate that SURE-tuning is a broadly applicable useful strategy for denoising in Gaussian sequence models.

\section{Implementation details and optimization strategy}

\subsection{Implementation details}
\label{subsec:imp_details}
All computations are done using PyTorch. For the optimization, in all cases we use full batch gradient descent using the Adam optimizer~\citep{kingma2015adam} with learning rate $0.01$; tolerance $10^{-6}$; and 2000 iterations.\\

\noindent \textbf{SURE-PM implementation.}
We fix $m=0$. The other parameters are learned by gradient descent and initialized as: $\tilde{\pi}_j =1$, $\tilde{u}_j=1$, and $s=\log(\text{IQR}_{0.95}(Z))$ where $\text{IQR}_{0.95}$ denotes the 95\% interquantile range. This corresponds to initializing with equal weights on evenly-spaced points spanning the 95\% interquantile range of the $Z_i$. The total number of parameters is equal to $2K$. We set $K=100$, finding no substantial improvement for larger values.\\

\noindent \textbf{SURE-LS implementation.} We use a two-layer network $h_{\theta} : \mathcal{X} \to \mathbb{R}^2$ with two hidden layers of size 8 per layer and ReLU activations. The network is given as input the standard deviation $\sigma_i$ instead of the variance $\sigma_i^2$. \\

\noindent \textbf{SURE-THING implementation.} 
In our experiments, we use a two-layer feedforward neural network $h_\theta: \mathcal{X} \to \mathbb{R}^{2K}$ with 8 neurons per hidden layer and ReLU activations. The network parameters $\theta$ are initialized by the default initialization process, which uses $\text{Unif}(-1/\sqrt{b}, 1/\sqrt{b})$ for any linear layer with input size $b$. We fix $m$ as $\text{median}(Z_1,\dotsc,Z_n)$ and we take $K = 100$ for the number of atoms. The network is given the standard deviation $\sigma_i$  as a covariate (instead of the variance).

Although we do not do so here, hyperparameters such as the network size can be selected by evaluating SURE on held-out observations, as described in Supplement~\ref{sec:cv}.

\subsection{Convexity}
\label{subsec:optimization_details}

\subsubsection{Non-convexity of SURE-PM}
\label{subsubsec:pm_nonconvexity}
We first give a simple counter-example to show that SURE-PM is not convex in the weights (probabilities) even if we fix the locations of the discretization.
Consider $n=1$, $\sigma_1^2=1$, $Z_1=0$, and priors $G_\pi = \pi\delta_{-1} + (1-\pi)\delta_1$ for $\pi \in [0,1]$. At $Z_1=0$, the posterior mean and variance are $1-2\pi$ and $4\pi(1-\pi)$, respectively. Hence,
$$
\SURE(G_\pi) = -1 + (1-2\pi)^2 + 8\pi(1-\pi),
$$
which is not convex in $\pi$. 

\subsubsection{Convexity of SURE-LS and isotonic regression}
\label{subsubsec:ls_convexity}
Dropping the constant $n^{-1}\sum_{i=1}^n\sigma_i^2$ from~\eqref{eq:SURE_gauss}, we can write the SURE-LS optimization problem as
\begin{equation}
\label{eq:sure_ls_convex_program}
\begin{array}{ll}
\underset{\lambda(\cdot),\,b(\cdot)}{\text{minimize}} & \displaystyle \frac{1}{n}\sum_{i=1}^n\cb{\lambda(X_i)Z_i-b(X_i)}^2 - \frac{2}{n}\sum_{i=1}^n\sigma_i^2\lambda(X_i) \\[1em]
\text{subject to} & \lambda(\cdot)\in\mathcal{L},\quad b(\cdot)\in\mathcal{B}.
\end{array}
\end{equation}
The objective in~\eqref{eq:sure_ls_convex_program} is a sum of squares of affine functions of $(\lambda,b)$ and a linear term. It is therefore jointly convex in $(\lambda,b)$, and the feasible set $\mathcal{L}\times\mathcal{B}$ is convex whenever both classes are convex. For the isotonic estimator of~\citet{xie2012sure} in Example~\ref{exam:isotonic_sure}, take $\mathcal{L}=\mathcal{L}^{\text{iso}}$ and $\mathcal{B}=\{0\}$. Suppose all $Z_i\ne0$ (which holds almost surely under the Gaussian model). Completing the square gives, up to a term independent of $\lambda$,
\[
\frac{1}{n}\sum_{i=1}^n Z_i^2\p{\lambda(\sigma_i^2)-\frac{\sigma_i^2}{Z_i^2}}^2.
\]
This is weighted isotonic regression with responses $\sigma_i^2/Z_i^2$, weights $Z_i^2$, and fitted values restricted to $[0,1]$. It can be solved by the pool-adjacent-violators algorithm \citep{robertson1988order}, followed by truncation of the fitted values to $[0,1]$.

\subsection{Initializing using the NPMLE}
\label{subsubsec:fitted_comparison}

We can use the NPMLE to initialize SURE-PM. Recall that the parameterization of Section~\ref{sec:computation} requires $u_1+u_K=2m$. Keeping $m$ fixed, we take the smallest interval centered at $m$ that contains the NPMLE grid.

We retain at most $K-2$ particles with positive NPMLE weights, keeping those with the largest weights if necessary, and renormalize their weights. This leaves room to add the two endpoints of the interval if they are not already among the retained particles. If there are fewer than $K$ particles, we successively add a particle at the midpoint of the largest gap until there are $K$. Let $u_1<\dotsb<u_K$ denote the resulting particles and $\hat{\pi}_j$ their weights (as returned by the NPMLE), with weight zero assigned to added particles. We initialize
$$
s=\log(u_K-u_1),\qquad
\tilde{u}_j=\log(u_{j+1}-u_j),\qquad
\tilde{\pi}_j=\log(0.99\hat{\pi}_j+0.01/K).
$$
The term $0.01/K$ gives every particle positive initial weight so that these particles can be revived if needed during optimization (i.e., not be stuck at zero weight throughout the optimization). We then can optimize the particle locations, probabilities, and scale.

\subsection{SURE cross validation}
\label{sec:cv}

Suppose that the estimator $G$ has hyperparameters like the size of the neural network. We detail how to use $K$-fold cross validation with SURE to decide the value of the hyperparameters.
\begin{enumerate}
\item Randomly partition the data into $K$ folds $\mathcal{I}_1, \mathcal{I}_2, \ldots, \mathcal{I}_K$.
\item For each value of $k = 1, 2, \ldots, K$:
\begin{enumerate}
\item Create the train set $\displaystyle \ical_{-k} := \cup_{\ell \neq k} \, \mathcal{I}_\ell$ and the holdout set $\ical_k$.
\item For each possible hyperparameter value:
\begin{enumerate}
\item Estimate $\widehat{G}_{-k}$ using data from the train set $\ical_{-k}$ (we keep the dependence on the hyperparameters implicit here)
\item Compute the cross-validated SURE of the holdout set $\ical_k$,
$$
\operatorname{CV-SURE}_k := \frac{1}{\abs{\ical_k}} \sum_{i \in \ical_k} \sqb{\sigma_i^2 + \sigma_i^4 \cb{ s_{\widehat{G}_{-k}}(W_i)^2 +2 \frac{\partial}{\partial z} s_{\widehat{G}_{-k}}(W_i)}},$$
\end{enumerate}
\end{enumerate}
\item Choose the hyperparameter value that minimizes the average cross-validated SURE values $\tfrac{1}{K} \sum_{k = 1}^K \operatorname{CV-SURE}_k$.
\end{enumerate}

\section{Further simulations in the homoscedastic setting}
\label{subsec:homosc_simulations}

We examine the homoscedastic Gaussian means problem of~\eqref{eq:gaussian_EB} in which $\sigma_i^2=1$ for all $i$ and there is no side-information. Theorem~\ref{theo:rate_homoscedastic} establishes theoretical guarantees for this problem. For each simulation setting, we generate data with $n = 1000$ observations. We consider two data-generating processes (DGPs) for the $\mu_i$ and $Z_i$; each DGP is parameterized by further hyperparameters.

We first consider a DGP with a Gaussian prior parameterized by $A_\star$:
\begin{equation}
\label{eq:normal_model}
\mu_i \simiid G_\star = \mathrm{N}(10, A_\star),\qquad   A_\star \in \{ 0.1, 1, 5 \}.
\end{equation}
We also consider a standard DGP in the compound decision theory literature~\citep{jiang2009general, koenker2014convex} wherein $\boldmu$ is fixed (and we only regenerate $Z_i$ across Monte Carlo replicates) and we set:
\begin{equation}
\mu_i = 
\begin{cases}
m_\star & \text{,  if } i \leq k_\star \\
0 & \text{,  if } i > k_\star
\end{cases},\qquad m_\star \in \{3, 4, 5, 7\}, \qquad k_\star \in \{5, 50, 500\}.
\label{eq:fixed_mu_i_simulation}
\end{equation}
The DGP parameters $m_{\star}$ and $k_{\star}$ specify the strength of non-null signals and their total number.

We consider $B=50$ Monte Carlo replicates for each setting of each DGP. We evaluate the performance of three estimators:
\begin{itemize}[leftmargin=*]
    \item the Bayes estimator for the Gaussian prior DGP and the best simple separable estimator~\citep{jiang2009general} for the compound DGP (both are denoted as ``Bayes'' in the tables);
    \item the NPMLE implemented following~\citet{koenker2014convex};
    \item SURE-PM (our proposal) as described in Section~\ref{sec:computation}.
\end{itemize}
\begin{table}
\caption{In-sample MSE of the estimators over 50 simulations in the homoscedastic problem without side-information and Gaussian prior as in~\eqref{eq:normal_model}.}
\setlength\tabcolsep{0pt}
\setlength\extrarowheight{2pt}
\begin{tabular*}{\linewidth}{@{\extracolsep{\fill}}*{4}{c}}
\hline
$A_{\star}$ & 0.1 & 1 & 5 \\ \hline
SURE-PM & 0.095 & 0.514 & 0.853 \\ 
NPMLE & 0.095 & 0.512 & 0.847 \\ 
Bayes & 0.090 & 0.501 & 0.828 \\ \hline
\end{tabular*}
\label{tab:sim_homo_normal}
\end{table} 

\begin{table}
\caption{In-sample MSE of the estimators over 50 simulations in the homoscedastic problem without side-information and fixed $\boldmu$ specified in~\eqref{eq:fixed_mu_i_simulation}.}
\setlength\tabcolsep{0pt}
\setlength\extrarowheight{2pt}
\begin{tabular*}{\linewidth}{@{\extracolsep{\fill}}*{13}{c}}
\hline
$k_\star$ & \multicolumn{4}{c}{5} & \multicolumn{4}{c}{50} & \multicolumn{4}{c}{500} \\ \cline{2-5} \cline{6-9} \cline{10-13} \hspace{0.01cm}
$m_\star$ & 3 & 4 & 5 & 7 & 3 & 4 & 5 & 7 & 3 & 4 & 5 & 7 \\ 
SURE-PM & 0.037 & 0.031 & 0.020 & 0.008 & 0.153 & 0.116 & 0.057 & 0.014 & 0.461 & 0.291 & 0.128 & 0.016 \\ 
NPMLE & 0.034 & 0.029 & 0.019 & 0.006 & 0.152 & 0.114 & 0.053 & 0.012 & 0.458 & 0.289 & 0.127 & 0.015 \\ 
Bayes & 0.027 & 0.022 & 0.012 & 0.001 & 0.144 & 0.105 & 0.046 & 0.003 & 0.449 & 0.283 & 0.118 & 0.006\\ \hline
\end{tabular*}
\label{tab:sim_homo_binary}
\end{table} 

Tables~\ref{tab:sim_homo_normal} and~\ref{tab:sim_homo_binary} report the results of the simulation.  Both SURE-PM and NPMLE show comparable performance, with modest gaps relative to the oracle Bayes estimator. This aligns with Theorem~\ref{theo:rate_homoscedastic} of this paper and~\citet{jiang2009general}, which together establish that both approaches achieve the minimax rate optimal empirical Bayes regret (up to logarithmic factors). 
These results suggest that while we do not advocate replacing the NPMLE with SURE-PM in this specific well-specified homoscedastic setting (where NPMLE enjoys computational advantages due to convexity), SURE-based methods remain competitive. In our main simulations of Section~\ref{sec:simulations} we see that the NPMLE can be highly suboptimal under misspecification.

We note that a simulation study similar to that reported in Table~\ref{tab:sim_homo_binary} was also conducted by~\citet[Section 6.1]{zhao2021simultaneous}, also see our discussion in Section~\ref{subsec:related_work} on related work. The main difference in implementation lies in the optimization method we use for computing $\hG$.\footnote{
\citet{zhao2021simultaneous} optimizes over priors of the form $G = \sum_{j=1}^K \delta_{\mu_j}/K$ with $\mu_j \in \RR$ and $K \in \mathbb N$. Then, optimization proceeds by coordinate descent optimizing over a single $\mu_j$ at a time. 
} Moreover, our simulation is accompanied by the sharp rate in Theorem~\ref{theo:rate_homoscedastic}, while~\citet{zhao2021simultaneous} does not provide a theoretical analysis of this estimator.

\section{Data fission details}
\label{sec:data_fission_details}

Recall that we use the data fission scheme
\begin{equation*}
\varepsilon_i^{(b)} \sim \mathrm{N}(0, \sigma_i^2), \qquad
\Zone = Z_i +  \varepsilon_i^{(b)}, \qquad
\Ztwo = Z_i - \varepsilon_i^{(b)},
\end{equation*}
which yields two independent data points with distribution
\begin{equation*}
\left. 
\begin{bmatrix} \Zone \\ \Ztwo \end{bmatrix} 
\, \middle| \, \mu_i, \sigma_i^2   \sim 
\mathrm{N} \left(
\begin{bmatrix} \mu_i \\ \mu_i \end{bmatrix}, 
\begin{bmatrix} 2\sigma_i^2 & 0 \\ 0 & 2\sigma_i^2 \end{bmatrix} \right) \right. 
\end{equation*}
For the $b$-th replicate, let $\hat{\mu}_i^{(b)}$ denote the $i^{\text{th}}$ shrinkage rule trained on the first data fold, $\boldZ^{(1, b)} :=(Z_1^{(1,b)},\ldots, Z_n^{(1,b)})$. The second data fold $\boldZ^{(2, b)} := (Z_1^{(2,b)}, \ldots, Z_n^{(2,b)})$ will be used for evaluation, which takes two steps.
\begin{enumerate}
\item Estimate the MSE in predicting $Z_i^{(2,b)}$ using $\hat{\boldmu}^{(b)}$. The MSE estimator will be denoted as $\displaystyle \widehat{\text{FMSE}}^{(b)}$ for the fission MSE. 
\begin{equation}\label{eqn:FMSE-b}
\widehat{\text{FMSE}}^{(b)} := \frac{1}{n} \sum_{i = 1}^n \bigl( \hat{\mu}_i^{(b)} - \Ztwo \bigr)^2.
\end{equation}
\item Estimate the Relative Improvement (RI) in the MSE between $\hat{\boldmu}^{(b)}$ and the MLE $\boldZ^{(1, b)}$ normalized by the error reduction of the NPMLE.  This describes the error reduction of an estimator compared to the baseline NPMLE estimator. The RI can be interpreted as the percent of the NPMLE error reduction achieved by the shrinkage rule $\hat{\boldmu}^{(b)}$.
\begin{equation} \label{eqn:relative-MSE-b}
\text{RI}^{(b)} := \frac{ \widehat{\text{FMSE}}^{(b)}_{\text{MLE}} - \widehat{\text{FMSE}}^{(b)} }{\widehat{\text{FMSE}}^{(b)}_{\text{MLE}} - \widehat{\text{FMSE}}^{(b)}_{\text{NPMLE}}}  
\end{equation}
for which the FMSE of the MLE and NPMLE are given by:
\begin{align}
\widehat{\text{FMSE}}^{(b)}_{\text{MLE}} &= \frac{1}{n} \sum_{i = 1}^n \bigl( \Zone - \Ztwo \bigr)^2,\\ 
\widehat{\text{FMSE}}^{(b)}_{\text{NPMLE}} &= \frac{1}{n} \sum_{i = 1}^n \bigl(  \hat{\mu}_{i \text{, NPMLE}}^{(b)} - \Ztwo \bigr)^2.
\end{align}
\end{enumerate}
The above describes the evaluation for a single replicate $b$, and we can repeat this process $B$ times and aggregate the results. In particular we are interested in the mean and standard error of the RI. These are the metrics reported in Table~\ref{tab:data_fission} (where the standard error only incorporates uncertainty due to the data fission scheme).
\begin{align}
\text{RI} &:=  \frac{1}{B} \sum_{i = 1}^B \text{RI}^{(b)},  \label{eq:ri} \\ 
\text{SE}_{\text{RI}} &:= \left[ \frac{1}{B(B-1)} \sum_{i = 1}^B \left( \text{RI}^{(b)} - \text{RI} \right)^2 \right]^{1/2} \label{eq:se_ri}.
\end{align}
We note that our metric is directly inspired by the analysis in~\citet{chen2026empirical}. \citet{chen2026empirical} normalizes performance by a different empirical Bayes baseline called ``Independent-Gauss'' that posits the working model $G(\cdot \mid X_i) = \mathrm{N}(u, \tau^2)$ for unknown $u$ and $\tau^2$ (that is, Independent-Gauss imposes both Gaussianity of the prior and that $\mu_i \indep (\sigma_i^2, X_i)$). \citet{chen2026empirical} defines the value of basic empirical Bayes as the MSE reduction using Independent-Gauss relative to the naive MLE. In the context of this paper, we prefer to think of the NPMLE as defining the value of basic empirical Bayes.

\section{Empirical process preliminaries}

We use the following version of Talagrand's concentration inequality for empirical processes; it does not require the summands to be identically distributed. See \cite{boucheron2013concentration} for an extensive discussion of related results. 
\begin{theo}[Theorem 3 of \cite{massart2000constants}]\label{theo:talagrand}
Let $b,n \ge 0$, 
suppose that $\xi_1,\ldots,\xi_n$ are independent random variables valued in the same
measurable space, and $\mathcal F$ is a separable family of $[-b,b]$-valued measurable functions
on this space. Suppose that either
\[ Z = \sup_{f \in \mathcal F} \left|\sum_{i = 1}^n (f(\xi_i) - \EE{f(\xi_i)})\right| \]
or
\[ Z = \sup_{f \in \mathcal F} \left|\sum_{i = 1}^n f(\xi_i)\right|. \]
Let
\[ \sigma^2 =  \sup_{f \in \mathcal F} \sum_{i = 1}^n \Var{f(\xi_i)}. \]
Then for any $\varepsilon,x > 0$,
$$\PP{Z \ge (1 + \varepsilon)\EE{Z} + \sigma \sqrt{8 x} + (2.5 + 32/\varepsilon)bx} \le \exp(-x)$$
and
$$\PP{Z \le (1 - \varepsilon)\EE{Z} - \sigma \sqrt{11 x} - (2.5 + 45/\varepsilon)bx} \le \exp(-x).
$$
\end{theo}
We also state a convenient application of this result combined with truncation.
\begin{lemm}[Special case of Lemma 4.10 of \cite{chen2022online}]\label{lem:truncation-shift}
Suppose that $X$ is a mean-zero random variable and $\sigma_q = \|X\|_{L_q} < \infty$ for some $q \ge 2$. Then
\[ |\EE{X \ind(|X| < s|)} \le 2s (\sigma_q/s)^q. \]
\end{lemm}
\begin{lemm}[Section 2.5 and 2.7 of \cite{vershynin2018high}]\label{lemm:moment-bound}
If $X$ is a mean-zero random variable which is $1$-subgaussian, then $\|X\|_{L_q} \lesssim \sqrt{q}$. If instead $X$ is $1$-subexponential, then $\|X\|_{L_q} \lesssim q$.
\end{lemm}
\begin{lemm}\label{lemm:sub-talagrand}
Let $n \ge 2$ and suppose that
$$ Z = \sup_{a \in \mathcal A} \left|\sum_{i = 1}^n a_i \xi_i \right|$$
where the $\xi_i$ are independent, mean zero, $1$-sub-Gaussian, and $|a_i| \le 1$. Let
$$ \sigma^2 =  \sup_{a \in \mathcal A} \|a\|_2^2. $$
Then for any $x \ge 1, \varepsilon > 0$, with probability at least $1 - e^{-x}$
$$ Z - (1 + \varepsilon)\EE{Z} \lesssim  \sigma \sqrt{x} + ( \varepsilon + 1/\varepsilon)\sqrt{\log n}x^{3/2}. $$
If instead the $\xi_i$ are assumed to be $1$-sub-exponential, then
$$ Z - (1 + \varepsilon)\EE{Z} \lesssim  \sigma \sqrt{x} + ( \varepsilon + 1/\varepsilon)\log n\,x^2. $$
\end{lemm}
\begin{proof}
We first consider the sub-Gaussian case.
Let $s > 0$, define $\rho_s(x) = \min(\max(x,-s),s)$, and define
$$ Z_s = \sup_{a \in \mathcal A} \left|\sum_{i = 1}^n a_i \rho_s(\xi_i) \right|. $$
By sub-Gaussian concentration and the union bound, we have that with probability at least $1 - e^{-x}$, provided that $s \ge \sqrt{2(x+\log(2n))}$,
$$ Z_s = Z.$$
Also, for any $q \ge 2$,
$$ \EE{|Z_s - Z|} \le \sum_{i = 1}^n \EE{|\xi_i| \ind(|\xi_i| > s)} \le ns \max_i(\|\xi_i\|_{L_q}/s)^q $$
and by Theorem~\ref{theo:talagrand}, with probability at least $1 - e^{-x}$,
\begin{align*} 
Z_s 
&\le  (1 + \varepsilon)\EE{Z_s} + \sigma \sqrt{8 x} + (2.5 + 32/\varepsilon)sx \\
&\le (1 + \varepsilon)(\EE{Z} + \EE{|Z_s - Z|}) + \sigma \sqrt{8 x} + (2.5 + 32/\varepsilon)sx \\
&\le (1 + \varepsilon)\EE{Z} + (1 + \varepsilon)ns\max_i (\|\xi_i\|_{L_q}/s)^q + \sigma \sqrt{8 x} + (2.5 + 32/\varepsilon)sx.
\end{align*}
By Lemma~\ref{lemm:moment-bound}, $\max_i\|\xi_i\|_{L_q} \le C \sqrt{q}$ for some absolute constant $C \ge 1$, so letting $s = 2C\sqrt{q}$ for some $q \ge 2$ we find
\begin{align*} 
Z_s - (1 + \varepsilon)\EE{Z}
&\le (1 + \varepsilon)n 2C\sqrt{q}2^{-q} + \sigma \sqrt{8 x} + (2.5 + 32/\varepsilon)2C\sqrt{q}x \\
&\lesssim \sigma \sqrt{x} + (\varepsilon + 1/\varepsilon)\sqrt{\log n}x^{3/2}
\end{align*}
by taking $q = c x\log n$ for a sufficiently large absolute constant $c$. This ensures that $q \ge 2$, $n2^{-q} \le e^{-x}$, and $s$ satisfies the lower bound above. Combining the two events via the union bound gives failure probability at most $2e^{-x}$. Replacing $x$ by $x+\log 2$ gives the stated bound, since $x \ge 1$.

The result in the sub-exponential case follows analogously, using $\max_i\|\xi_i\|_{L_q} \le Cq$ and $s = 2Cq$ instead; here Markov's inequality and the union bound give $\PP{\max_i|\xi_i|>s} \le n2^{-q} \le e^{-x}$.
\end{proof}
We also use a high-probability version of the sub-Gaussian comparison theorem, which is a consequence of Talagrand's majorizing measures theorem.
\begin{theo}[Exercise 8.6.5 of \cite{vershynin2018high}]
There exists an absolute constant $C > 0$ so that if $(X_t)_{t \in T}$ is a mean-zero separable stochastic process and $(Y_t)_{t \in T}$ is a mean-zero separable Gaussian process, with $X_{t_0}=Y_{t_0}=0$ almost surely for some $t_0 \in T$, and
$$\EE{\exp\{u(X_t-X_s)\}} \le \exp\left\{\frac{u^2}{2}\Var{Y_t-Y_s}\right\}, \qquad s,t \in T,\; u \in \RR,$$
then for any $\delta \in (0,1)$, with probability at least $1 - \delta$,
\begin{equation} \frac{1}{C}\sup_{t \in T} X_t \le \EE{\sup_{t \in T} Y_t} + \sqrt{\sup_t \Var{Y_t} \log(2/\delta)}. \label{eq:majorizing}
\end{equation}
\end{theo}

\begin{lemm}[Symmetrization]\label{lemm:symmetrization_argument}
Let $W_i$, $i = 1, \ldots, n$ be $n$ independent (not necessarily identically distributed) random variables. Then, for any set $\mathcal{H}$ of functions $h(\cdot)$ such that $\EE{\abs{h(W_i)}} < \infty$, 
\begin{equation}\label{eq:radamacher_inequality}
    \begin{aligned}
        \EE{\sup_{h \in \mathcal{H}}\abs{\frac{1}{n} \sum_{i=1}^n \p{h(W_i) - \EE{h(W_i)}}}} \le 2 \; \EE{\sup_{h \in \mathcal{H}} \abs{\frac{1}{n} \sum_{i=1}^n \varepsilon_i h(W_i)}},  
    \end{aligned}
\end{equation}
where $\varepsilon_i$ are iid Rademacher random variables (that is, $\varepsilon_i = \pm 1$ with probability $1/2$), independent of $W_i$. 
\end{lemm}

\section{Proofs for Supplement~\ref{sec:uniform_convergence_rademacher} and Section~\ref{subsec:basic_ineq}}

\subsection{Proof of Proposition~\ref{prop:uniform_control_rademacher} \label{subsec:proof_prop_uniform_control_rademacher}}

\begin{proof}
Let $G_\star \in \mathcal{G}$ and $\widehat{G} \in \argmin \cb{\SURE(G): G \in \mathcal{G}}$, which implies that: 
\begin{equation}
\begin{aligned}
    & \SURE(\widehat{G}) \le \SURE(G_\star), \\
    \implies & \frac{1}{n} \sum_{i=1}^n \sqb{\sigma_i^2 + \sigma_i^4 \cb{\score_{\widehat{G}}(W_i)^2 + 2\frac{\partial}{\partial z} \score_{\widehat{G}}(W_i)}} \\
    & \;\;\; \le \frac{1}{n} \sum_{i=1}^n \sqb{\sigma_i^2 + \sigma_i^4 \cb{\score_{G_\star}(W_i)^2 + 2\frac{\partial}{\partial z} \score_{G_\star}(W_i)}}, \\
    \implies & \frac{1}{n} \sum_{i=1}^n h_{\widehat{G}}(W_i) \le \frac{1}{n} \sum_{i=1}^n h_{G_\star}(W_i), 
\end{aligned}\label{eq:basic_inequality_fixedX}
\end{equation}
where $h_G(w) = \sigma^4 \p{ \score_{G}(w)^2  + 2\frac{\partial}{\partial z}\score_{G}(w)}$, for $G \in \mathcal{G}$. 

Now, using Stein's lemma, we can rewrite $\EE{\Regret(G_{\star}, \hG)}$ as: 
$$
\begin{aligned}
    & \EE{\Regret(G_{\star}, \hG)} \\
    & \;\;\; = \EE{\p{\mu - (Z + \sigma^2\score_{\hG}(W))}^2 - \p{\mu - (Z + \sigma^2\score_{G_\star}(W))}^2} \\
    & \;\;\; = \EE{\sigma^4 s_{\hG}(W)^2 + 2 \sigma^2 (Z - \mu) \score_{\hG}(W) - \sigma^4 s_{G_\star}(W)^2 - 2 \sigma^2 (Z - \mu) \score_{G_\star}(W)} \\
    & \;\;\; = \EE{\sigma^4 s_{\hG}(W)^2 + 2 \sigma^4 \frac{\partial}{\partial z} s_{\hG}(W)} - \EE{\sigma^4 s_{G_\star}(W)^2 + 2 \sigma^4 \frac{\partial}{\partial z} s_{G_\star}(W)}. 
\end{aligned}
$$
Here $W = (Z, X)$ is the fresh draw of Section~\ref{sec:general_setting}, independent of the training sample $W_{1:n} := (W_1, \ldots, W_n)$. Since $\widehat{G}$ is a function of $W_{1:n}$ only, conditionally on $W_{1:n}$ the score $s_{\widehat{G}}$ is a fixed function of the fresh observation, so the Eddington/Tweedie formula~\eqref{eq:general_tweedie} (used in the first equality) and Stein's lemma, $\EE{(Z - \mu) s_G(W) \mid X, W_{1:n}} = \sigma^2 \EE{\frac{\partial}{\partial z} s_G(W) \mid X, W_{1:n}}$ (used in the last equality, for $G \in \{\widehat{G}, G_\star\}$), both apply.

We note that, $X_i$'s are fixed here and $X \sim \text{Unif}(X_1, \ldots, X_n)$, so if we take expectation over $X$ in $\Regret(G_{\star}, \hG)$, then given fixed $X_i$'s, 
$$
\begin{aligned}
\EE{h_G(W)} = \frac{1}{n} \sum_{i=1}^n \EE{h_G(Z_i, X_i)} = \frac{1}{n} \sum_{i=1}^n \EE{h_G(W_i)}.
\end{aligned}
$$
Then, adding and subtracting $\frac{1}{n}\sum_{i=1}^n h_{\hG}(W_i)$, using~\eqref{eq:basic_inequality_fixedX}, and applying Lemma~\ref{lemm:symmetrization_argument}, we have:
$$
\begin{aligned}
    & \EE{\Regret(G_{\star}, \hG)} \\
    & \;\; = \EE{\sigma^4 s_{\hG}(W)^2 + 2 \sigma^4 \frac{\partial}{\partial z} s_{\hG}(W)} - \EE{\sigma^4 s_{G_\star}(W)^2 + 2 \sigma^4 \frac{\partial}{\partial z} s_{G_\star}(W)} \\
    & \;\; \le 2 \EE{\sup_{G \in \mathcal{G}} \abs{\EE{h_{G}(W)} - \frac{1}{n} \sum_{i=1}^n h_{G}(W_i)}} \\
    & \;\; \le 4 \EE{\sup_{G \in \mathcal{G}} \abs{\frac{1}{n} \sum_{i=1}^n \varepsilon_i h_G(W_i)}} \\
    & \;\; = 4 \poprademacher\!\!\p{\cb{h_G(\cdot)\,:\, h_G(w) = \sigma^4 \p{ \score_{G}(w)^2  + 2\frac{\partial}{\partial z}\score_{G}(w)}, \,\; G \in \mathcal{G}}}. 
\end{aligned}
$$
The argument for score matching is analogous, with $h_G(w) = \score_{G}(w)^2  + 2\frac{\partial}{\partial z}\score_{G}(w)$ and $\SM$ in place of $\SURE$; in place of the regret identity above one uses that, for every fixed $G \in \mathcal{G}$, $\EE{h_G(W)} - \EE{h_{G_\star}(W)} = \EE{\cb{s_G(W) - s_{\star}(W)}^2} = \Dfisher{f_{G_{\star}}}{f_{G}}$ by the partial integration argument of Supplement~\ref{subsec:partial_integration} and~\eqref{eq:fisher_divergence_conditional}, applied conditionally on $W_{1:n}$ with $G = \hG$.
\end{proof}

\subsection{Proof of Proposition~\ref{prop:uniform_control_rademacher_compound}}

\begin{proof} 
First, we rewrite $\SURE(G)$ in a different way: 
\begin{equation*}
    \begin{aligned}
        & \SURE(G) \\
        & \;\; = \frac{1}{n} \sum_{i=1}^n \sqb{\sigma^2_i + \sigma_i^4 \p{s_G(W_i)^2 + 2 \frac{\partial}{\partial z} s_G(W_i)}} \\
        & \;\; = \frac{1}{n} \sum_{i=1}^n \sqb{\sigma^2_i + \cb{(Z_i - \mu_i + \sigma_i^2 s_G(W_i)) - (Z_i - \mu_i)}^2 + 2 \sigma_i^4 \frac{\partial}{\partial z} s_G(W_i)} \\
        & \;\; = \frac{1}{n} \sum_{i=1}^n \sqb{\sigma^2_i - (Z_i - \mu_i)^2 + (Z_i - \mu_i + \sigma_i^2 s_G(W_i))^2 - 2 (Z_i - \mu_i) \sigma_i^2 s_G(W_i) + 2 \sigma_i^4 \frac{\partial}{\partial z} s_G(W_i)}. 
    \end{aligned}
\end{equation*}
From the above, applied to $G = \hG$ and to $G = G_{\oracle}$, and using the fact that $\SURE(\hG) \le \SURE(G_{\oracle})$ (since $G_{\oracle} \in \mathcal{G}$), we can write:
\begin{equation*}
    \begin{aligned}
        \frac{1}{n} \sum_{i=1}^n &\cb{\p{Z_i + \sigma_i^2 \score_{\widehat{G}}(W_i) - \mu_i}^2} \,-\, \frac{1}{n} \sum_{i=1}^n\cb{\p{Z_i + \sigma_i^2 \score_{G_{\oracle}}(W_i) - \mu_i}^2} \\ 
        \leq & \;\; \frac{2}{n} \sum_{i=1}^n\sqb{- \sigma_i^2\p{Z_i - \mu_i}\cb{\score_{G_{\oracle}}(W_i) - \score_{\widehat{G}}(W_i)} +  \sigma_i^4 \frac{\partial}{\partial z} \cb{\score_{G_{\oracle}}(W_i) - \score_{\widehat{G}}(W_i)}} \\
        = & \; \; \frac{2}{n} \sum_{i=1}^n h_{\hG}(W_i, \mu_i) \\
        \leq & \; \; 2 \sup \cb{\abs{\frac{1}{n} \sum_{i=1}^n h_{G}(W_i, \mu_i)} : G \in \mathcal{G}}, 
    \end{aligned}
\end{equation*}
by the definition of $h_G$. 

We further note that, for each fixed $G \in \mathcal{G}$, $\EE[\boldmu]{h_G(W_i, \mu_i)} = 0$, by Stein's lemma. Then, using Lemma~\ref{lemm:symmetrization_argument}, we can write for $h_G(W_i, \mu_i)$ that: 
$$
\begin{aligned}
    &\EE[\boldmu]{\frac{1}{n} \sum_{i=1}^n \cb{Z_i + \sigma_i^2 \score_{\widehat{G}}(W_i) - \mu_i}^2 - \frac{1}{n} \sum_{i=1}^n\cb{Z_i + \sigma_i^2s_{G_{\oracle}}(W_i) - \mu_i}^2} \\ 
    &\;\;\;\; \leq  4 \; \EE[\boldmu]{\sup \cb{\abs{\frac{1}{n} \sum_{i=1}^n \varepsilon_i h_{G}(W_i, \mu_i)} : G \in \mathcal{G}}} \\
    &\;\;\;\; =  4 \; \poprademacherfixedX(\mathcal{M}_{\oracle}).
\end{aligned}
$$
Hence proved. 
\end{proof}

\subsection{Proof of Proposition~\ref{prop:basic_inequality} \label{subsec:proof_basic_inequality}}

\begin{proof}
We prove the following more general version, valid for arbitrary noise variances $\sigma_i^2 > 0$; the proposition follows by taking $\sigma_i^2 = 1$ for all $i$. We know, by definition in~\eqref{eq:SURE}, 
$$
\SURE(G) = \frac{1}{n} \sum_{i=1}^n \p{\sigma_i^2 + \sigma_i^4 \p{s_G^2(W_i) + 2\frac{\partial}{\partial z} s_G(W_i)}}. 
$$
Since $\hG$ minimizes $\SURE(G)$, we have $\SURE(\hG) \leq \SURE(G_\star)$, i.e., 
\begin{equation*}
    \begin{aligned}
        & \frac{1}{n} \sum_{i=1}^n \p{\sigma_i^2 + \sigma_i^4 \p{\score_{\hG}^2(W_i) + 2\frac{\partial}{\partial z} \score_{\hG}(W_i)}} \leq \frac{1}{n} \sum_{i=1}^n \p{\sigma_i^2 + \sigma_i^4 \p{\score_\star^2(W_i) + 2\frac{\partial}{\partial z} \score_\star(W_i)}}, \text{ i.e.,} \\
        & \frac{1}{n} \sum_{i=1}^n \sigma_i^4 \p{\score_{\hG}(W_i) - \score_\star(W_i)}^2 \leq \frac{2}{n} \sum_{i=1}^n \sigma_i^4 \p{\score_\star^2(W_i) - \score_\star(W_i)s_{\hG}(W_i) + \frac{\partial}{\partial z} \p{\score_\star(W_i) - \score_{\hG}(W_i)}}. 
    \end{aligned}
\end{equation*}
Rearranging the equation above, we can write: 
\begin{equation}
\begin{aligned}
&\frac{1}{n}\sum_{i=1}^n  \sigma_i^4 \cb{\score_{\hG}(W_i)-\score_{\star}(W_i)}^2 \\ 
&\;\;\;\;\;\;\;\;\;\leq\;  \frac{2}{n}\sum_{i=1}^n \sigma_i^4 \sqb{ \score_{\star}(W_i)\cb{\score_{\star}(W_i)-\score_{\hG}(W_i)}  + \frac{\partial}{\partial z}\cb{\score_{\star}(W_i)-\score_{\hG}(W_i)}}.
\end{aligned} \label{eq:basic_inequality_heteroscedastic}
\end{equation}
Hence proved. 
\end{proof}

\subsection{Proof of Lemma~\ref{lemm:deterministic_inequality} \label{subsec:proof_deterministic_inequality}}

\begin{proof}
As in the proof of Proposition~\ref{prop:basic_inequality}, we prove a more general version allowing for heteroscedasticity: suppose $\sigma_i^2 > 0$ for all $i$, $G_\star \in \mathcal{G}$, and $\hG$ minimizes $\SURE(G)$ over $\mathcal{G}$. From the basic inequality~\eqref{eq:basic_inequality_heteroscedastic}, we have: 
\begin{equation}
\begin{aligned}
&\frac{1}{n}\sum_{i=1}^n  \sigma_i^4 \cb{\score_{\hG}(W_i)-\score_\star(W_i)}^2 \\ 
&\;\;\;\;\;\;\;\;\;\leq\; \frac{2}{n}\sum_{i=1}^n \sigma_i^4 \sqb{ \score_\star(W_i)\cb{\score_\star(W_i)-\score_{\hG}(W_i)}  + \frac{\partial}{\partial z}\cb{\score_\star(W_i)-\score_{\hG}(W_i)}} \\
&\;\;\;\;\;\;\;\;\; = \; \frac{2}{n}\sum_{i=1}^n h_{\hG}(W_i). 
\end{aligned}\label{eq:basic_ineq}
\end{equation}
where, for ease of writing, we define $$h_G(W_i) = \sigma_i^4 \sqb{ \score_\star(W_i)\cb{\score_\star(W_i)-\score_{G}(W_i)}  + \frac{\partial}{\partial z}\cb{\score_\star(W_i)-\score_{G}(W_i)}}$$ for $G \in \mathcal{G}$. In this proof, we use the following weighted versions of $\widehat W$ and $\hat r$, which coincide with those in the statement of the lemma when $\sigma_i^2 = 1$ for all $i$:
\begin{equation*}
\begin{aligned}
\widehat W(r) := &\sup\Bigg\{ \frac{2}{n}\sum_{i=1}^n h_G(W_i): G \in \mathcal{G}\, \text{ s.t. }\, \frac{1}{n}\sum_{i=1}^n\sigma_i^4 \cb{\score_\star(W_i) - s_G(W_i)}^2 \leq r^2 \Bigg\}, \\
\hat r := &\sup \cb{ r \ge 0 \;:\; r^2 \le \widehat W(r) }.
\end{aligned}
\end{equation*}
We will show that $\frac{1}{n}\sum_{i=1}^n \sigma_i^4 \cb{\score_{\hG}(W_i)-\score_\star(W_i)}^2 \le \hat r^2$, which is the claim of the lemma when $\sigma_i^2 = 1$. Define $r_0 := \sqb{\frac{1}{n}\sum_{i=1}^n \sigma_i^4 \cb{\score_{\hG}(W_i)-\score_\star(W_i)}^2}^{1/2} \ge 0$ (the case $r_0 = 0$ is allowed). Then, by~\eqref{eq:basic_ineq} and the definition of $\widehat W$: 
\begin{equation}
    \begin{aligned}
        & \frac{1}{n}\sum_{i=1}^n  \sigma_i^4 \cb{\score_{\hG}(W_i)-\score_\star(W_i)}^2 \\
        & \;\; = \; r_0^2 \\
        & \;\; \le \; \frac{2}{n}\sum_{i=1}^n h_{\hG}(W_i) \\
        & \;\; \le \; \sup \cb{\frac{2}{n} \sum_{i=1}^n h_G(W_i) : G \in \mathcal{G} \text{ s.t. } \frac{1}{n}\sum_{i=1}^n\sigma_i^4 \cb{\score_\star(W_i) - s_G(W_i)}^2 \leq r_0^2} \\
        & \;\; =  \; \widehat W(r_0), 
    \end{aligned}\label{eq:bi_hatW_ineq}
\end{equation}
since $\hG \in \{G \in \mathcal{G} : \frac{1}{n}\sum_{i=1}^n\sigma_i^4 \cb{\score_\star(W_i) - s_G(W_i)}^2 \leq r_0^2\}$. Also, from~\eqref{eq:bi_hatW_ineq}, we have, 
$$r_0 \in \cb{r \ge 0 \,:\, r^2 \le \widehat W(r)}, $$ 
And, since $\hat r = \sup \{ r \ge 0\,:\, r^2 \le \widehat W(r) \}$, so, we must have, $r_0 \le \hat r$. Thus, we have from~\eqref{eq:bi_hatW_ineq}, 
$$
\frac{1}{n}\sum_{i=1}^n  \sigma_i^4 \cb{\score_{\hG}(W_i)-\score_\star(W_i)}^2 \le \hat r^2. 
$$
Hence proved. 
\end{proof}

\section{Proofs for SURE-PM (Section~\ref{sec:fast_rates})}
\label{sec:proofs_sure_homoscedastic}
\subsection{Complexity of scores and their derivatives}

Following~\citet{zhang1997empirical, jiang2009general}, throughout our proof it is convenient to consider the regularized score defined as,
\begin{equation}
\label{eq:regularized_score}
\regscore_G(z) := \frac{\partial }{\partial z} f_G(z) \cdot \frac{1}{f_G(z) \lor \rho} = \frac{f_G'(z)}{f_G(z) \lor \rho} ,\;\;\;\rho >0,
\end{equation}
where $a \lor b := \max\cb{a,b}$ for $a, b \geq 0$ and $h'(z) = \partial h(z) / \partial z $ denotes the derivative of $h(\cdot)$ with respect to $z$.
Notice that when $f_G(z) \geq \rho$, then $\regscore_G(z) = \score_G(z)$ (and in particular, for $\rho=0$, $\regscore_G(z)=\score_G(z)$ for all $z$).

We will need the following function:

\begin{equation}
\tilde{L}(\rho) := \sqrt{-\log(2\pi \rho^2)}.
\label{eq:L_rho_inverse}
\end{equation}
Now let us give some properties of the regularized score and also its derivative,
$$ \frac{\partial}{\partial z} \regscore_G(z) = \begin{cases} \frac{f_G''(z)}{\rho}, & f_G(z) < \rho \\ 
\frac{f_G''(z)}{f_G(z)} - \p{ \frac{f_G'(z)}{f_G(z)}}^2, & f_G(z) > \rho.
\end{cases}$$
First some global bounds that scale as $\log(1/\rho)$ (instead of trivial upper bound $1/\rho$).
\begin{lemm}[Proposition 1 in~\citet{jiang2009general}]
\label{lemm:proposition_1_jiang_zhang}
For $0< \rho < (2\pi e^3)^{-1/2}$, we have that:
$$
\begin{aligned}
&\abs{\regscore_G(z)} \leq \tilde{L}(\rho), \\ 
-1 \leq &\frac{\partial}{\partial z} \regscore_G(z) \leq \tilde{L}^2(\rho), 
\end{aligned} 
$$
where the second inequality holds for every $z$ with $f_G(z) \neq \rho$ (since $f_G$ is real-analytic and non-constant, the set $\{z : f_G(z) = \rho\}$ is discrete, hence Lebesgue-null).
\end{lemm}

Given $\varepsilon>0$ a class of functions $\mathcal{H}$ and a (semi)-norm $\Norm{\cdot}$, a $(\varepsilon, \mathcal{H}, \Norm{\cdot})$-covering is defined as a finite set of functions $h_1,\dotsc,h_B$ with the following property: for any $h \in \mathcal{H}$, there exists $j \in \cb{1,\dotsc,B}$ such that $\Norm{h-h_j} \leq \varepsilon$. The covering number $N(\varepsilon, \mathcal{H}, \Norm{\cdot})$ is defined as the smallest $B \in \mathbb N$ such that there exists a $(\varepsilon, \mathcal{H}, \Norm{\cdot})$-covering with $B$ elements.

Proposition 3 in~\citet{jiang2009general} provides control of the covering number of the family $(\cb{\regscore_G(\cdot) \,:\, G },\; \Norm{\cdot}_{\infty})$. Our next lemma compactly states a corollary of that proposition and extends it to one further derivative.

\begin{lemm}
\label{lemm:second-covering}
\noindent Consider the families of functions, 
$$\mathscr{T}^{(1)}_{\rho} := \cb{\regscore_G(\cdot) = \frac{f_G^{\prime}(\cdot)}{f_G(\cdot) \vee \rho}: G \in \mathcal{P}(M) },\,\,\,\,\,\mathscr{T}^{(2)}_{\rho} := \cb{\frac{f_G^{\prime\prime}(\cdot)}{f_G(\cdot) \vee \rho}: G \in \mathcal{P}(M) }, 
$$ for $0 < \rho \leq 1/ (\sqrt{2\pi}e)$, where $\mathcal{P}(M) = \cb{G: G \text{ has support on } [-M, M]}$ and $M \geq 0$ is arbitrary. Then, for any $0 < \varepsilon < \frac{1}{e}$, 
     \begin{equation} \label{eq:first-covering-bound}
        \log N(\varepsilon, \mathscr{T}^{(1)}_{\rho}, \Norm{\cdot}_{\infty}) \lesssim \p{1 + M^2 + \abs{\log(\varepsilon\rho)}}^2, 
    \end{equation}
    and
    \begin{equation} \label{eq:second-covering-bound}
        \log N(\varepsilon, \mathscr{T}^{(2)}_{\rho}, \Norm{\cdot}_{\infty}) \lesssim \p{1 + M^2 + \abs{\log(\varepsilon\rho)}}^2.
    \end{equation}
The above covering numbers also hold verbatim if we center these two classes, i.e., if we consider the following classes instead for some fixed $G_{\star}$:
$$\mathscr{T}^{c,(1)}_{\rho} := \cb{\regscore_G(\cdot) - \regscore_{G_{\star}}(\cdot) \,:\, G \in \mathcal{P}(M) },\,\,\,\,\,\mathscr{T}^{c,(2)}_{\rho} := \cb{\frac{f_G^{\prime\prime}(\cdot)}{f_G(\cdot) \vee \rho}-\frac{f_{G_\star}^{\prime\prime}(\cdot)}{f_{G_{\star}}(\cdot) \vee \rho}\,:\, G \in \mathcal{P}(M)}.$$
\end{lemm}

\begin{proof}

Our Lemma below builds on arguments similar to the proofs of Proposition 3 in~\citet{jiang2009general} and Theorem 3.1. in~\citet{ghosal2001entropies}.
Let us first recall the function $\Tilde{L}(y) = \sqrt{-\log(2\pi y^2)}$ defined in~\eqref{eq:L_rho_inverse} which is the inverse of the standard Gaussian pdf $\varphi(x)$ over $\mathbb R_+$ We will repeatedly use that $\Tilde{L}^2(y)$ is decreasing in $y$, that $\Tilde{L}^2(y) \leq 2\log\p{1/y}$ for $0 < y < 1$, and that $y \log\p{1/y} \leq e^{-1}$ for $0 < y < 1$.

The proof for the covering number bound on the first derivative follows along the lines of Proposition 3 in~\citet{jiang2009general}. However, that proof holds for a semi-norm on a different support for the prior compared to our case. Here, we already prove the covering number bound for the second derivative below, and it follows along the lines similarly for the first derivative as well, defining appropriate bound on $\Norm{{f_G^{\prime}}/\{f_G \vee \rho\} - {f_H^{\prime}}/\{f_H \vee \rho\}}_{\infty}$, for example, taking $\eta = C (e\varepsilon \rho)^2$, for some constant $C < 1$, we can prove that 
\begin{equation*}
\Norm{\frac{f_G^{\prime}}{f_G \vee \rho} - \frac{f_H^{\prime}}{f_H \vee \rho}}_{\infty} \lesssim \frac{\eta}{\rho} \Tilde{L}(\eta) \lesssim \frac{1}{\sqrt{e}} \frac{\eta^{1/2}}{\rho} \lesssim \varepsilon,
\end{equation*} 
using the fact that $\eta^{1/2} \Tilde{L}(\eta) \leq \frac{1}{\sqrt{e}}$. Tracking constants exactly as in the second-derivative proof below, the implied constants in this display depend on $M$ only through the choice $C \asymp (1+M^2)^{-1}$, which enters the covering-number count only through $\log(1/\eta) \lesssim \log\frac{1}{\varepsilon\rho} + \log(1+M^2)$; hence the first-derivative bound holds with the same absolute constant as the second-derivative bound. We are not going into details for this proof, since we are already proving the same for the second derivative. 
    
For proving the covering number bound for the second derivative, let $\eta = \frac{1}{C} e \varepsilon \rho$ (where $C > 1$ is a constant to be defined later), which implies that $\eta < \rho$ and $\Tilde{L}^2(\eta) \geq \Tilde{L}^2(\rho) \geq 2$, since $\Tilde{L}^2(y)$ is decreasing in $y^2$, where $\Tilde{L}(y) = \sqrt{-\log{(2\pi y^2)}}$. 
    
    Now, if $f_G(x) \geq \rho$, where $0 < \rho \leq (2\pi e^2)^{-\frac{1}{2}}$, we have 
    \begin{align*}
        \tilde{L}^2(\rho) & \geq \tilde{L}^2\p{f_G(x)} \text{ and } \\
        \tilde{L}^2(\rho) & \geq \tilde{L}^2\p{(2\pi e^2)^{-\frac{1}{2}}} = 2, 
    \end{align*} 
    since $\tilde{L}^2(y)$ is decreasing in $y^2$. 
    
    From Lemma A.1 in~\citet{jiang2009general}, 
    
    \begin{align*}
        & 0 \leq \frac{f_G^{\prime\prime}(x)}{f_G(x)} + 1 \leq \tilde{L}^2(f_G(x)), \\
        \implies & -1 \leq \frac{f_G^{\prime\prime}(x)}{f_G(x)} \leq \tilde{L}^2(f_G(x)) - 1 \leq \tilde{L}^2(\rho) - 1, \\ 
        \implies & \abs{\frac{f_G^{\prime\prime}(x)}{f_G(x)}} \leq \max{\{\tilde{L}^2(\rho) - 1, 1\}} = \tilde{L}^2(\rho) - 1,
    \end{align*}
    for $f_G(x) \geq \rho$, and given that $\tilde{L}^2(\rho) - 1 \geq 1$. Thus, we can write, 
    \begin{align*}
        \abs{\frac{f_G^{\prime\prime}(x)}{f_G(x) \vee \rho} - \frac{f_H^{\prime\prime}(x)}{f_H(x) \vee \rho}} \leq & \frac{1}{\rho} \abs{f_G^{\prime\prime}(x) - f_H^{\prime\prime}(x)} + \frac{\tilde{L}^2(\rho)-1}{\rho} \abs{f_G(x) - f_H(x)}. 
    \end{align*}
    Now, 
    \begin{align*}
        & f_G(x) = \int \varphi(x-u) d(G(u)), \\
        \implies & f^{\prime}_G(x) = - \int (x - u) \varphi(x-u) d(G(u)), \\
        \implies & f^{\prime\prime}_G(x) = - f_G(x) + \int (x-u)^2 \varphi(x-u) d(G(u)). 
    \end{align*}
    That means,  
    \begin{align*}
        \abs{f^{\prime\prime}_G(x) - f^{\prime\prime}_{H}(x)} \leq \abs{ f_G(x) - f_{H}(x)} + \abs{\int (x-u)^2 \varphi(x-u) \{d(G(u)) - d(H(u))\}}. 
    \end{align*}
    Then, we have 
    \begin{align*}
        & \abs{\frac{f_G^{\prime\prime}(x)}{f_G(x) \vee \rho} - \frac{f_H^{\prime\prime}(x)}{f_H(x) \vee \rho}} \\
        & \;\;\; \leq \frac{1}{\rho} \abs{\int (x-u)^2 \varphi(x-u) d(G(u) - H(u))} + \frac{\tilde{L}^2(\rho)}{\rho} \abs{\int \varphi(x-u) d(G(u) - H(u)) } \\
        & \;\;\; \leq \frac{1}{\rho} \abs{\int (x-u)^2 \varphi(x-u) d(G(u) - H(u)) } + \frac{\tilde{L}^2(\eta)}{\rho} \abs{\int \varphi(x-u) d(G(u) - H(u))}, 
    \end{align*}
    since $0 < \eta < \rho$. 
    
    We define $t^* = \max\{2M, \sqrt{8}\Tilde{L}(\eta)\}$. We note that $x^2\varphi(x)$ is decreasing in $x$ for $x \geq \sqrt{2}$, since $\frac{d}{dx}x^2\varphi(x) = x \varphi(x) (2-x^2) \leq 0$ for $x \geq \sqrt{2}$. Then, 
    \begin{align*}
        \sup_{\abs{x} \geq t^*} \abs{\int_{-M}^M (x-u)^2 \varphi(x-u) d(G(u) - H(u))} & \leq 2 \cdot (t^* - M)^2 \cdot \varphi(t^* - M) \\
        & \leq 2 \cdot \p{\frac{t^*}{2}}^2 \cdot \varphi\p{\frac{t^*}{2}} \\
        & \leq 2 \cdot (\sqrt{2}\Tilde{L}(\eta))^2 \cdot \varphi\p{\sqrt{2}\Tilde{L}(\eta)} \\
        & = 4 \cdot \Tilde{L}^2(\eta) \cdot \frac{1}{\sqrt{2\pi}} \exp{\p{-\frac{1}{2} \cdot 2\Tilde{L}^2(\eta)}} \\
        & = 4 \cdot \Tilde{L}^2(\eta) \cdot \frac{1}{\sqrt{2\pi}} \exp{\p{\log(2\pi\eta^2)}} \\
        & \leq 12 \, \eta^2 \Tilde{L}^2(\eta), 
    \end{align*} 
    since for $\abs{x} \geq t^*$ and $\abs{u} \leq M$, $\abs{x - u} \geq t^* - M \geq t^* - \frac{t^*}{2} = \frac{t^*}{2} \geq \sqrt{2} \Tilde{L}(\eta) \geq \sqrt{2}$, using $t^* \geq 2M$, $t^* \geq \sqrt{8}\Tilde{L}(\eta)$ and $\Tilde{L}^2(\eta) \geq 2$.
    
    Similarly, we also have: 
    \begin{align*}
        \sup_{\abs{x} \geq t^*} \abs{\int_{-M}^M \varphi(x-u) d(G(u) - H(u)) } & \leq 2 \cdot \varphi(t^* - M) \\
        & \leq 2 \cdot \varphi\p{\frac{t^*}{2}} \\
        & \leq 2 \cdot \varphi\p{\sqrt{2}\Tilde{L}(\eta)} \\
        & = 2 \cdot \frac{1}{\sqrt{2\pi}} \exp{\p{-\frac{1}{2} \cdot 2\Tilde{L}^2(\eta)}} \\
        & = 2 \cdot \frac{1}{\sqrt{2\pi}} \exp{\p{\log(2\pi\eta^2)}} \\
        & \leq 6 \, \eta^2, 
    \end{align*}
    since $\varphi(x)$ is decreasing for $x \geq 0$. 
    
    Now, we bound the values of the two integrals when $| x | \leq t^*$. We use Taylor's formula to compose the Gaussian density into a polynomial term and a remainder term. Recall that for some $v \in (-x^2/2, 0)$, we can write
    \begin{equation*}
    e^{-x^2/2} = \sum_{j = 0}^{k-1} \left( \frac{-x^2}{2} \right)^j  \frac{1}{j!} + \left( \frac{-x^2}{2} \right)^k \frac{1}{k!} \, e^v.
    \end{equation*} Then, we can bound the remainder term by
    \begin{align}
        \abs{\varphi(x) -  \frac{1}{\sqrt{2\pi}}  \sum_{j = 0}^{k-1} \left( \frac{-x^2}{2} \right)^j  \frac{1}{j!} } & = \frac{1}{\sqrt{2\pi}} \abs{ \frac{x^{2k} }{2^k}\frac{1}{k!} e^v} \nonumber \\
        &\leq \frac{1}{\sqrt{2\pi}} \abs{ \frac{x^{2k} }{2^k} \left( \frac{e}{k} \right)^k} \nonumber \\ 
        &= \frac{\p{e \frac{x^2}{2}}^k}{\sqrt{2\pi} k^k} = \frac{(\sqrt{e/2})^{2k}}{\sqrt{2\pi} k^k} x^{2k}.
        \label{eq:normal-remainder}
    \end{align} 
    where the inequality comes from $e^v \leq 1$ and Stirling's formula, namely that $1/k! < (e/k)^k$ for $k \geq 2$. We have yet to define $k$, but it will be chosen to be at least 2. 
    Then, we have: 
    \begin{align*}
        & \sup_{\abs{x} \leq t^*} \abs{\int_{-M}^M (x-u)^2 \varphi(x-u) d(G(u) - H(u))}  \\ 
        & \quad =  \sup_{\abs{x} \leq t^*} \abs{\int_{-M}^M (x-u)^2 \p{\varphi(x-u) \pm \sum_{j=0}^{k-1} \frac{(-1)^j (x-u)^{2j}}{\sqrt{2\pi} j! 2^j}}  d(G(u) - H(u))} \\
        & \quad \leq \quad (1) + (2), 
    \end{align*} 
    where
    \begin{align*}
        (1) &:= \sup_{\abs{x} \leq t^*} \abs{\int_{-M}^M \sum_{j=0}^{k-1} \frac{(-1)^j (x-u)^{2j+2}}{\sqrt{2\pi} j! 2^j}  d(G(u) - H(u))} \\
        (2) & := \sup_{\abs{ x } \leq t^*} \abs{\int_{-M}^M (x-u)^2 \p{\varphi(x-u) - \sum_{j=0}^{k-1} \frac{(-1)^j (x-u)^{2j}}{\sqrt{2\pi} j! 2^j}} d(G(u) - H(u))} \\ 
        &\leq \sup_{\abs{ x } \leq t^*} \abs{ \frac{(\sqrt{e/2})^{2k}}{\sqrt{2\pi} k^k} \int_{-M}^M (x-u)^{2k+2}  d\abs{G(u) - H(u)}},
    \end{align*}
    for which we apply \eqref{eq:normal-remainder} to term (2). Similarly, we also have:
    \begin{align*}
        & \sup_{\abs{x} \leq t^*} \abs{\int_{-M}^M \varphi(x-u) d(G(u) - H(u)) } \\
        & \quad = \sup_{\abs{x} \leq t^*} \abs{\int_{-M}^M \left( \varphi(x-u)   \pm \sum_{j=0}^{k-1} \frac{(-1)^j (x-u)^{2j}}{\sqrt{2\pi} j! 2^j}  \right)  d(G(u) - H(u)) } \\ 
        & \quad  \leq \quad (3) + (4), 
    \end{align*}
    where
    \begin{align*}
    (3) &:= \sup_{\abs{x} \leq t^*} \abs{ \int_{-M}^M \sum_{j=0}^{k-1} \frac{(-1)^j (x-u)^{2j}}{\sqrt{2\pi} j! 2^j}  d(G(u) - H(u)) } \\
    (4) &:= \sup_{\abs{ x } \leq t^*} \abs{\int_{-M}^M  \p{\varphi(x-u) - \sum_{j=0}^{k-1} \frac{(-1)^j (x-u)^{2j}}{\sqrt{2\pi} j! 2^j}} d(G(u) - H(u))} \\
    &\leq \sup_{\abs{ x } \leq t^*} \abs{\frac{(\sqrt{e/2})^{2k}}{\sqrt{2\pi} k^k} \int_{-M}^M (x-u)^{2k}  d\abs{G(u) - H(u)}}.
    \end{align*} 
    Now, if $\int u^l dG(u) = \int u^l dH(u)$, for $l = 1, 2, \ldots, 2k$, then the terms (1) and (3) from the above inequalities vanish. To bound the terms (2) and (4), we reason as follows. The integrands in (2) and (4) are some even-degree polynomials of  $(x - u)$, which are increasing for $x - u \geq 0$. So we upper bound $|x - u|$ when $|x | \leq t^*$ and $|u| \leq M$. 
    \begin{align*}
        \abs{x - u} \leq t^* + M \leq 3M + \sqrt{8}\,\Tilde{L}(\eta) =: T, \qquad \text{so that} \qquad T^2 \leq 18M^2 + 16\Tilde{L}^2(\eta),
    \end{align*}
    because $t^* = \max\{2M, \sqrt{8}\Tilde{L}(\eta)\} \leq 2M + \sqrt{8}\Tilde{L}(\eta)$. Note that $T$ is additive in $M$ and $\Tilde{L}(\eta)$.
    
    Then with $a = \frac{e}{2} T^2$, the upper bounds of (2) and (4) can be written as:
    \begin{align*}
    (2) &\leq \sup_{\abs{ x } \leq t^*} \abs{ \frac{(\sqrt{e/2})^{2k}}{\sqrt{2\pi} k^k} \int_{-M}^M T^{2k+2}  d\abs{G(u) - H(u)}} \\
    &\leq 2 \frac{(\sqrt{e/2})^{2k}}{\sqrt{2\pi} k^k}  T^{2k+2} \\
    & = \frac{2}{\sqrt{2\pi}} \, T^2 \, \frac{a^k}{k^k} \\
    (4) &\leq \frac{2}{\sqrt{2\pi}} \frac{a^k}{k^k}
    \end{align*}
    We choose $k$ as the smallest integer exceeding $e \cdot a + \Tilde{L}^2(\eta)$. Notice that $k \geq 2$. The reason for this choice is so that the shared term in (2) and (4) can be bounded by 
    \begin{align*}
        \frac{a^k}{k^k} 
        &= \exp{\p{-k\p{\log{(k)} - \log(a)}}} \\
        &\leq \exp(-k) \leq \exp \left(- \Tilde{L}^2 (\eta) \right), 
    \end{align*}
    since $k \geq e \cdot a$ gives $\log(k) - \log(a) \geq 1$, and $k \geq \Tilde{L}^2(\eta)$. Finally, we can write 
    \begin{align*}
        (2) &= \sup_{\abs{x} \leq t^*} \abs{ \int_{-M}^M (x-u)^2 \varphi(x-u) d(G(u) - H(u)) } \\
        & \leq \frac{2}{\sqrt{2\pi}} \, T^2 \, \exp\p{-\Tilde{L}^2(\eta)} \\
        & = \frac{2}{\sqrt{2\pi}} \, 2\pi\eta^2 \, T^2 \\
        & \leq 6\, \eta^2 \p{18M^2 + 16\Tilde{L}^2(\eta)} \\
        & \leq (54M^2 + 96)\,\eta^2 \Tilde{L}^2(\eta), \text{ and } \\
        (4) &= 
        \sup_{\abs{x} \leq t^*} \abs{ \int_{-M}^M \varphi(x-u) d(G(u) - H(u)) } \\
        & \leq \frac{2}{\sqrt{2\pi}} \exp\p{-\Tilde{L}^2(\eta)}\\
        &= \frac{2}{\sqrt{2\pi}} \, 2\pi \eta^2\\
        & \leq 6 \, \eta^2.
    \end{align*} 
    where the last inequality for (2) uses the fact that $\Tilde{L}^2(\eta) \geq 2$. Defining $C_1 = 6$ and $C_2 = \max\{12,\, 96 + 54M^2\} = 96 + 54M^2$, we can write, combining all the results above: 
    \begin{align*}
         \Norm{\frac{f_G^{\prime\prime}}{f_G \vee \rho} - \frac{f_H^{\prime\prime}}{f_H \vee \rho}}_{\infty} \leq (C_1 + C_2) \frac{\eta^2}{\rho} \Tilde{L}^2(\eta), 
    \end{align*}
    for any $G, H \in \mathcal{P}(M)$ whose moments agree up to order $2k$. Now fix $G \in \mathcal{P}(M)$. By Lemma A.1 in~\citet{ghosal2001entropies} (Carath\'eodory's theorem applied to the moment curve), there exists a discrete distribution $G_m \in \mathcal{P}(M)$ with at most $N := 2k+1$ support points in $[-M, M]$ whose moments agree with those of $G$ up to order $2k$. Let $G_{m,\eta}$ be obtained from $G_m$ by moving each support point to the nearest point of the grid $\{0, \pm \eta^2, \pm 2\eta^2, \ldots\} \cap [-M-\eta^2, M+\eta^2]$ (so each point moves by at most $\eta^2$), keeping the weights, and let $\Tilde{G}_{m, \eta}$ have the same support points as $G_{m,\eta}$ and weights coming from an $\eta^2$-net over the $N$-dimensional simplex for the $l_1$ norm. Using the bound on $\abs{f_G^{\prime\prime}/(f_G \vee \rho) - f_H^{\prime\prime}/(f_H \vee \rho)}$ from the beginning of the proof (which holds for arbitrary $G, H$), we can write:
    \begin{align*}
        \Norm{\frac{f_G^{\prime\prime}}{f_G \vee \rho} - \frac{f_{G_m}^{\prime\prime}}{f_{G_m} \vee \rho}}_{\infty} & \leq (C_1 + C_2) \frac{\eta^2}{\rho} \Tilde{L}^2(\eta), \\
        \Norm{\frac{f_{G_m}^{\prime\prime}}{f_{G_m} \vee \rho} - \frac{f_{G_{m,\eta}}^{\prime\prime}}{f_{G_{m,\eta}} \vee \rho} }_{\infty} & \leq \frac{\eta^2}{\rho} \p{\Norm{ \varphi^{\prime\prime\prime}}_{\infty} + \Tilde{L}^2(\eta)\Norm{ \varphi^{\prime}}_{\infty}} \\
        & = \frac{\eta^2}{\rho} \p{\sqrt{\frac{9-3\sqrt{6}}{\pi}} e^{-\frac{3-\sqrt{6}}{2}} + \frac{1}{\sqrt{2\pi e}}\Tilde{L}^2(\eta)} \\
        & \leq C_3 \, \frac{\eta^2}{\rho} \Tilde{L}^2(\eta), \\
        \Norm{\frac{f_{G_{m,\eta}}^{\prime\prime}}{f_{G_{m,\eta}} \vee \rho} - \frac{f_{\Tilde{G}_{m, \eta}}^{\prime\prime}}{f_{\Tilde{G}_{m, \eta}} \vee \rho}}_{\infty} & \leq \frac{\eta^2}{\rho} \p{\Norm{ \varphi^{\prime\prime}}_{\infty} + \Tilde{L}^2(\eta)\Norm{ \varphi}_{\infty}} \\
        &= \frac{\eta^2}{\rho} \p{\frac{1}{\sqrt{2\pi}} + \frac{1}{\sqrt{2\pi}} \Tilde{L}^2(\eta)} \\
        & \leq C_4 \, \frac{\eta^2}{\rho} \Tilde{L}^2(\eta), 
    \end{align*}
    for $C_3 = \frac{1}{2} \sqrt{\frac{9-3\sqrt{6}}{\pi}} e^{-\frac{3-\sqrt{6}}{2}} + \frac{1}{\sqrt{2\pi e}}$ and $C_4 = \frac{3}{2\sqrt{2\pi}}$ (since $\Tilde{L}^2(\eta) \geq 2$). Thus, combining all results, we can write:
    \begin{align*}
        & \Norm{\frac{f_G^{\prime\prime}}{f_G \vee \rho} - \frac{f_{\Tilde{G}_{m, \eta}}^{\prime\prime}}{f_{\Tilde{G}_{m, \eta}} \vee \rho}}_{\infty} \leq C_5 \frac{\eta^2}{\rho} \Tilde{L}^2(\eta) = C_5 \frac{\eta}{\rho} \eta \Tilde{L}^2(\eta) \leq C_5 \frac{\eta}{\rho} \frac{2}{e} = \frac{2C_5}{e} \frac{e \varepsilon \rho}{C\rho} = \varepsilon \\
        \implies & \Norm{\frac{f_G^{\prime\prime}}{f_G \vee \rho} - \frac{f_{\Tilde{G}_{m, \eta}}^{\prime\prime}}{f_{\Tilde{G}_{m, \eta}} \vee \rho}}_{\infty} \leq \varepsilon,
    \end{align*}
    by definition of $\eta$, for some constant $C_5 = C_1 + C_2 + C_3 + C_4 \leq 104 + 54M^2$ (which depends on $M$) and we take $C := 2C_5 \leq 208 + 108M^2 \leq 352(1+M^2)$, since $\Tilde{L}^2(\eta) \leq 2 \log(\frac{1}{\eta})$ and $y \log(\frac{1}{y}) \leq \frac{1}{e}$ for any $0 < y < 1$; clearly $C > 1$. Thus the functions $\frac{f_{\Tilde{G}_{m, \eta}}^{\prime\prime}}{f_{\Tilde{G}_{m, \eta}} \vee \rho}$, as $\Tilde{G}_{m, \eta}$ ranges over all distributions with at most $N$ support points on the grid and weights in the $\eta^2$-net, form an $\varepsilon$-cover of $\mathscr{T}^{(2)}_{\rho}$ in $\Norm{\cdot}_{\infty}$. The grid $\{0, \pm\eta^2, \pm 2\eta^2, \ldots\} \cap [-M - \eta^2, M + \eta^2]$ has at most $\frac{2M}{\eta^2} + 3 \leq \frac{2M + 3}{\eta^2}$ points (as $\eta^2 < 1$), and by Lemma A.4 in~\citet{ghosal2001entropies}, an $\eta^2$-net over the $N$-dimensional simplex for the $l_1$ norm has cardinality at most $(\frac{5}{\eta^2})^N$. Therefore, 
    \begin{equation} \label{eq:raw-count}
        \log N\p{\varepsilon, \mathscr{T}^{(2)}_{\rho}, \Norm{\cdot}_{\infty}} \leq N \p{\log(2M+3) + \log 5 + 4 \log\p{\frac{1}{\eta}}}.
    \end{equation}
    Write $\Lambda := 1 + M^2 + \log\p{\frac{1}{\varepsilon\rho}}$; note $\Lambda \geq 1$ since $\varepsilon\rho < 1$. First, by the definition of $C$,
    \begin{align*}
        \log\p{\frac{1}{\eta}} & = \log\p{\frac{1}{\varepsilon\rho}} + \log C - 1 \leq \log\p{\frac{1}{\varepsilon\rho}} + \log 352 + \log(1+M^2) \\
        & \leq \log\p{\frac{1}{\varepsilon\rho}} + 6 + M^2 \leq 6\Lambda,
    \end{align*}
    using $\log(1+M^2) \leq M^2$ and $\log 352 < 6$. Hence $\Tilde{L}^2(\eta) \leq 2\log(1/\eta) \leq 12 \Lambda$. Second, by the definition of $k$ and $T$, and using $\Tilde{L}^2(\eta) \geq 2$ to absorb the $+1$,
    \begin{align*}
        k & \leq \frac{e^2}{2} T^2 + \Tilde{L}^2(\eta) + 1 \leq 9e^2 M^2 + (8e^2 + 1)\Tilde{L}^2(\eta) + 1 \\
        & \leq 67 M^2 + 61 \Tilde{L}^2(\eta) \leq 67 M^2 + 732 \Lambda \leq 800 \Lambda,
    \end{align*}
    so that $N = 2k + 1 \leq 1601 \Lambda$. Third, $\log(2M+3) + \log 5 \leq \log 15 + \log(1 + M) \leq 3 + M \leq 3 + \frac{1 + M^2}{2} \leq 4\Lambda$. Plugging into \eqref{eq:raw-count},
    \begin{align*}
        \log N\p{\varepsilon, \mathscr{T}^{(2)}_{\rho}, \Norm{\cdot}_{\infty}} \leq 1601 \Lambda \p{4\Lambda + 24 \Lambda} \leq 5 \cdot 10^4 \p{1 + M^2 + \abs{\log(\varepsilon\rho)}}^2,
    \end{align*}
    which is \eqref{eq:second-covering-bound} with an absolute constant.

    Finally, if $\{h_1, \ldots, h_B\}$ is an $\varepsilon$-cover of $\mathscr{T}^{(j)}_{\rho}$, then $\{h_1 - h_{\star}, \ldots, h_B - h_{\star}\}$, with $h_{\star}$ being the element of $\mathscr{T}^{(j)}_{\rho}$ corresponding to $G_{\star}$, is an $\varepsilon$-cover of $\mathscr{T}^{c,(j)}_{\rho}$ of the same cardinality. Hence, the bounds hold verbatim for the centered classes. Hence proved. 
\end{proof}

\subsection{Further technical lemmata}

The first lemma allows us to just think of regularized scores on an event of very high probability. We introduce the notations $t^*$ and $\rho^*$ that are dependent on $M$. In the remainder of the univariate proofs, we will use the fact that $\abs{ \log (\rho^*)}$ grows no faster than $\left(1 + M\right)^2 \log n$. 

\begin{lemm}
\label{lemm:compact_z_whp}
For every $C \geq 2$ the following are true.\footnote{We have the specific choice, $C=10$, in mind for what follows.} Let 
\begin{equation}\label{eq:tstar} 
t^* := t^*_M = M + \sqrt{2C \log n},\quad \rho^* := \rho^*_M = \frac{1}{2 \sqrt{2 \pi}} e^{-4M^2} n^{-2C}. 
\end{equation}
Define the sample event,
$$A_n := \cb{ \max_i\abs{Z_i} \leq t^*}.$$
The following hold regarding the event $A_n$.
\begin{enumerate}
\item $\PP{A_n} \geq 1- 2 n^{1-C}.$
\item On $A_n$ it holds that $\abs{s_G(Z_i)} \leq 2M + \sqrt{2C \log n}$ for all $i=1,\dotsc,n$ for all $G \in \mathcal{P}(M)$.
\item On $A_n$ it holds that:
$$ \min_{i=1,\dotsc,n} \inf_{G \in \mathcal{P}(M)} f_G(Z_i) \geq \frac{1}{\sqrt{2\pi}} \exp\p{ -4M^2} n^{-2C}.$$
\item On $A_n$, for any $\rho \leq \rho^*$ it holds that $s_G(Z_i) = s_G^{\rho}(Z_i)$ for all $i=1,\dotsc,n$ and in particular also $s_{\star}(Z_i) = s_{\star}^{\rho}(Z_i)$.
\end{enumerate}

\end{lemm}
\begin{proof}
For the first part, we can write $Z_i = \mu_i + \xi_i$, where $\xi_i \sim \mathrm{N}(0, 1)$ and $|\mu_i| \leq M$ by definition. Applying Hoeffding's inequality gives that
$$\PP{ \max_i\abs{Z_i} \geq t^*} \leq \sum_{i = 1}^n  \PP{ \abs{Z_i} \geq t^*} \leq  \sum_{i = 1}^n  \PP{ \abs{\xi_i} \geq \sqrt{2C \log n}} \leq 2 n^{1-C}.$$
For the second result, we just write $s_G(Z_i) = \EE[G]{\mu_i \mid Z_i} - Z_i$. Since the prior is supported on $[-M,M]$, the first term is bounded by $M$ and the second is bounded by $t^*$ on $A_n$. For the third result, suppose without loss of generality that $Z_i \geq 0$, then:
$$
\begin{aligned}
f_G(Z_i) &\geq \frac{1}{\sqrt{2\pi}} \exp\p{ -(Z_i +M)^2/2}  \\   
&\geq \frac{1}{\sqrt{2\pi}} \exp\p{ -(t^* +M)^2/2} \\ 
&\geq \frac{1}{\sqrt{2\pi}} \exp\p{ -(2 C \log n + 4M^2)} =  \frac{1}{\sqrt{2\pi}} \exp\p{ -4M^2} n^{-2C} = 2 \rho^*.
\end{aligned}
$$
This lower bound explains the choice of $\rho^*$. Regularizing by $\rho^*$ gives strict inactivity of the regularization (part four).
\end{proof}

Our proof relies crucially on controlling the complexity of the following classes and their star hull. For any class $\mathcal{F}$ of functions and any $\delta > 0$ we also define the localized population Rademacher complexity~\citep[Equation (14.3)]{wainwright2019highdimensional}:
\begin{equation}
\label{eq:localized_population_rademacher}
\localizedpoprademacher{\delta}{\mathcal{F}} = \EE{ \sup\cb{\abs{\frac{1}{n}\sum_{i=1}^n \varepsilon_i f(W_i)}\,:\, f \in \mathcal{F},\,\, \Norm{f}_2 \leq \delta}},
\end{equation}
where the $\varepsilon_i$ are iid Rademacher and independent of the $Z_i$.
Throughout Supplement~\ref{sec:proofs_sure_homoscedastic} we use the notation $\Norm{\cdot}_2$ to denote the $L^2(f_{\star},\RR)$ norm, unless stated otherwise. Thus, $\Norm{h}_2^2 = \EE[G_\star]{h(Z)^2} = \int_\RR h(z)^2 f_\star(z)\,dz$.
We first consider the class of centered regularized scores,
\begin{equation}
\mathcal{S}^{\rho}_c = \cb{ s_G^{\rho}(\cdot) - s_{\star}^{\rho}(\cdot) \,:\, G \in \mathcal{P}(M)}.
\label{eq:score_class_center_rho_rho}
\end{equation}
We also define the star hull of a set of functions $\mathcal{H}$ as 
\begin{equation}
\label{eq:starhull}
\StarHull(\mathcal{H}) = \cb{ \lambda \cdot h(\cdot)\,:\, h \in \mathcal{H},\, \lambda \in [0,1]}.
\end{equation}

\begin{lemm}[Complexity of centered regularized score class]
\label{lemm:complexity_regularized_score}
We have the following results for the centered score class and its star hull.
\begin{enumerate}
\item For $0< \rho < (2\pi e^3)^{-1/2}$, we have that $$\sup_{h \in \mathcal{S}^{\rho}_c} \Norm{h(\cdot)}_{\infty} = \sup_{h \in \StarHull(\mathcal{S}^{\rho}_c)} \Norm{h(\cdot)}_{\infty} \leq 2 \sqrt{ - \log \left( 2 \pi \rho^2 \right) } \lesssim \sqrt{\abs{\log(\rho)}}.$$
\item $\log N(\eta, \mathcal{S}^{\rho}_c, \Norm{\cdot}_{\infty}) \leq \log N(\eta, \StarHull(\mathcal{S}^{\rho}_c), \Norm{\cdot}_{\infty}) \lesssim_C \left( 1 + M^2 + \abs{ \log(\eta \rho)}\right)^2$.
\item Let $\rho = \rho^*$ be as in~\eqref{eq:tstar} (which depends on $C, M >0$).
Then it holds that:
\begin{equation}
\localizedpoprademacher{\delta}{\StarHull \left(\mathcal{S}^{\rho^*}_c\right)} \lesssim_{C} \frac{\delta}{\sqrt{n}} \bigl( \abs{ \log(1/\delta)}  + \abs{ \log \rho^*} \bigr) + \frac{\sqrt{\abs{\log \rho^*}}}{n}\bigl( \abs{ \log(1/\delta)}  + \abs{ \log \rho^*} \bigr)^2.
\label{eq:rademacher_bound_center_score}
\end{equation}
\item Continuing with $\rho = \rho^*$,
there exists a constant $C' >0$ such that $$\delta = C'\frac{\abs{ \log \rho^*}^{1.5}}{\sqrt{n}},$$
satisfies
$$
\localizedpoprademacher{\delta}{\StarHull(\mathcal{S}^{\rho^*}_c)}  \leq \frac{\delta^2}{\sqrt{ \abs{ \log (\rho^*) } }}.
$$
\end{enumerate}
\end{lemm}

\begin{proof}
Part 1: This result follows from Lemma~\ref{lemm:proposition_1_jiang_zhang}.

Part 2: By Lemma~\ref{lemm:second-covering},
$\log N(\eta, \mathcal{S}^{\rho}_c, \Norm{\cdot}_{\infty}) \lesssim \left( 1 + M^2 + \abs{ \log(\eta \rho)}\right)^2$. 
To upgrade this to a result for the star hull we can use the argument of~\citet[Lemma 4.5]{mendelson2002improving}. That is, first let $N$ be the size of a $\eta/2$ cover of  $\mathcal{S}^{\rho}_c$ and call the centers $h_1,\dotsc,h_N$. Our strategy is to construct a $\eta/2$ cover of each of:
$$
\cb{ \lambda h_j(\cdot): \lambda \in [0,1]}, \, j=1,\dotsc,n.
$$
Why does this suffice? Any element in the star hull may be represented as $\lambda h(\cdot)$ for some $h(\cdot) \in \mathcal{S}^{\rho}_c$. Now first pick the closest element $h_j$ in the cover of  $\mathcal{S}^{\rho}_c$ and then pick $\lambda_{jk}$ in the cover of $\cb{ \lambda h_j(\cdot): \lambda \in [0,1]}$. We get:
$$
\Norm{\lambda h(\cdot) - \lambda_{jk} h_j(\cdot)}_{\infty} \leq \Norm{h(\cdot) - h_j(\cdot)}_{\infty} + \Norm{ \lambda h_j(\cdot) -  \lambda_{jk} h_j(\cdot) }_{\infty}  \leq \eta/2 + \eta/2 = \eta.
$$
It only remains to count how many elements our cover has. Notice that for each $j$ it suffices to cover $[0,1]$ at $\eta /  \bigl(4 \sqrt{ \abs{\log(\rho)} } \bigr)$, which we can do with order $ \bigl(4 \sqrt{ \abs{\log(\rho)} } \bigr)/\eta$ elements. In total,
\begin{align*}
\log N(\eta, \StarHull(\mathcal{S}^{\rho}_c), \Norm{\cdot}_{\infty}) &\lesssim \log N(\eta/2, \mathcal{S}^{\rho}_c, \Norm{\cdot}_{\infty}) + \log\left( \bigl(4 \sqrt{ \abs{\log(\rho)} } \bigr)/\eta \right) \\
&\lesssim \left( 1 + M^2 + \abs{ \log(\eta \rho)}\right)^2.
\end{align*}
For part 3, we argue as follows. First let $N$ and $h_1,\dotsc,h_N$ be a proper $(\eta, \Norm{\cdot}_{\infty})$-cover of the $\delta$-localized function class $\cb{h \in \StarHull(\mathcal{S}_c^{\rho^*})\,:\, \Norm{h}_2 \leq \delta}$. By proper cover we mean that $h_j \in \StarHull(\mathcal{S}_c^{\rho^*})$ and $\Norm{h_j}_2 \leq \delta$. We will pick $\eta$ later. By standard arguments, we just use the covering number of the whole class (up to constants), that is $\log N \lesssim  \left( 1 + M^2 + \abs{ \log(\eta \rho^*)}\right)^2$.

Now take $h \in \cb{h \in \StarHull(\mathcal{S}_c^{\rho})\,:\, \Norm{h}_2 \leq \delta}$ and let $h_{\hat{j}}$ be the nearest element to it in the cover. Then:

$$
\begin{aligned}
\abs{\frac{1}{n}\sum_{i=1}^n \varepsilon_i h(Z_i)} &\leq \abs{\frac{1}{n}\sum_{i=1}^n \varepsilon_i \cb{h(Z_i)-h_{\hat{j}}(Z_i)}} + \abs{\frac{1}{n}\sum_{i=1}^n \varepsilon_i h_{\hat{j}}(Z_i)} \\ 
& \leq \eta \,+\, \frac{1}{n}\max_{j=1}^N  \abs{\sum_{i=1}^n \varepsilon_i h_{j}(Z_i)}
\end{aligned}
$$
Now argue as follows. First note that for each $j$, $\varepsilon_i h_j(Z_i)$ has the following properties: first, its absolute value is upper bounded by $\sqrt{ \abs{ \log \rho^* }}$ and second its variance is as follows:
$$\Var{\varepsilon_i h_{j}(Z_i)} = \EE{h_{j}^2(Z_i)} \leq \delta^2,$$
since we picked a proper cover. This means that $\varepsilon_i h_{j}(Z_i)$ is a sub-Gamma random variable, $\Gamma(\delta^2, \sqrt{ \abs{ \log \rho^* }})$, where we use the definition of sub-Gamma in~\citet[Chapter 2.4]{boucheron2013concentration}. By tensorization, this means that $\sum_{i=1}^n \varepsilon_i h_j(X_i)$ is sub-Gamma, $\Gamma(n \delta^2 , \sqrt{ \abs{ \log \rho^* }})$. Applying the maximal inequality in Corollary 2.6 of~\citet{boucheron2013concentration} we thus find that:
$$ 
\EE{\max_{j=1}^N \abs{\sum_{i=1}^n \varepsilon_i h_{j}(Z_i)}} \leq \delta \sqrt{2 n \log(2N)} \, + \, \sqrt{ \abs{ \log \rho^* }} \log(2N).
$$
Taking $\eta = \delta/\sqrt{n}$ and noting that $\abs{\log \rho^*} \lesssim \left( 1 + M \right)^2 \log n$, we find that:
\begin{align*}
\localizedpoprademacher{\delta}{\left(\mathcal{S}^{\rho^*}_c\right)} &\lesssim \eta + \frac{1}{n} \left( \sqrt{n} \delta  \left( 1 + M^2 + \abs{ \log(\eta \rho^*)}\right) + \sqrt{\abs{ \log ( \rho^* ) }} \left( 1 + M^2 + \abs{ \log(\eta \rho^*)}\right)^2   \right) \\
&\lesssim \frac{\delta}{\sqrt{n}} \left( 1 + M^2 + \abs{ \log(1/\delta)}  + \abs{ \log \rho^*} \right)  \\
&\qquad  \qquad + \frac{\sqrt{\abs{\log \rho^*}}}{n}\left( 1 + M^2 + \abs{ \log(1/\delta)}  + \abs{ \log \rho^*} \right)^2 \\
&\lesssim \frac{\delta}{\sqrt{n}} \bigl( \abs{ \log(1/\delta)}  + \abs{ \log \rho^*} \bigr) + \frac{\sqrt{\abs{\log \rho^*}}}{n}\bigl( \abs{ \log(1/\delta)}  + \abs{ \log \rho^*} \bigr)^2
\end{align*}

For Part 4: 
Taking
$$\delta = C'\frac{\abs{ \log \rho^*}^{1.5}}{\sqrt{n}},$$
for some $C' > 1$,
 the right hand side of~\eqref{eq:rademacher_bound_center_score} becomes 
\begin{equation*}
C' \frac{\abs{ \log \rho^*}^{2.5}}{n}  + \frac{\abs{ \log \rho^*}^{2.5}}{n} , 
\end{equation*}
which in turn, for a constant $C'$ that is sufficiently large, is at most
$$(C')^2 \frac{\abs{ \log \rho^*}^{3}}{n} \cdot \frac{1}{\sqrt{ \abs{ \log (\rho^*) } }}  = \frac{\delta^2}{\sqrt{ \abs{ \log (\rho^*) } }}.$$
\end{proof}

Similarly to Lemma~\ref{lemm:complexity_regularized_score}, we now study the complexity of the following class.

\begin{equation}
\label{eq:noise_process}
\mathcal{M}^{\rho}_c = \cb{h_G(\cdot) = \regscore_{\star}(\cdot)\cb{\regscore_{\star}(\cdot)-\regscore_{G}(\cdot)}  + \frac{\partial}{\partial z}\cb{\regscore_{\star}(\cdot)-\regscore_G(\cdot)}\,:\, G \in \mathcal{P}(M)}.
\end{equation}
\begin{lemm}[Complexity of centered regularized noise process class]
\label{lemm:noise_process_complexity}
We have the following results for the centered noise process class and its star hull.
\begin{enumerate}
\item For $0< \rho < (2\pi e^3)^{-1/2}$, we have that $$\sup_{h \in \mathcal{M}^{\rho}_c} \Norm{h(\cdot)}_{\infty} = \sup_{h \in \StarHull(\mathcal{M}^{\rho}_c)} \Norm{h(\cdot)}_{\infty} \lesssim \abs{\log(\rho)}.$$
\item Let $t^*$, $\rho^*$ be defined as in \eqref{eq:tstar} (both of which depend on $C,M$). Further define
$$\|f\|_{t^*,\infty} := \|f\|_{L_{\infty}[-t^*,t^*]}.$$
then
$$\log N(\eta, \mathcal{M}^{\rho^*}_c, \Norm{\cdot}_{t^*,\infty}) \leq \log N(\eta, \StarHull(\mathcal{M}^{\rho^*}_c), \Norm{\cdot}_{t^*,\infty}) \lesssim_{C} \left( \abs{ \log(\eta)} + \abs{\log( \rho^*)} \right)^2.$$
\item Continuing with $\rho^*$ as above, then it holds that:
\begin{align*}
\localizedpoprademacher{\delta}{\StarHull(\mathcal{M}^{\rho^*}_c)} &\lesssim \frac{\delta}{\sqrt{n}}\left( \abs{ \log(1/\delta)} + \abs{\log \rho^*} \right) \,+\, \frac{\abs{\log \rho^* } }{n} \left( \abs{\log(1/\delta)} + \abs{\log \rho^* } \right)^2\\
&\qquad   + \abs{ \log \rho^*} 2n^{1-C}.
\end{align*}

\item Take any $h \in \mathcal{M}_c^{\rho^*}$. Then:
$$ \abs{\EE[G_{\star}]{h(Z_i)}} \lesssim_{C}  n^{(1-C)/2}\abs{\log \rho^*}^{1/2}(   \abs{\log \rho^*}^{1/2} + M) .$$
\end{enumerate}
\end{lemm}

\begin{proof}
For part 1, we again directly call upon Lemma~\ref{lemm:proposition_1_jiang_zhang}, applying it to both $\regscore_G(z)$ and its derivative.

For part 2, we start by noting that by Lemma~\ref{lemm:compact_z_whp}, for any $\abs{z} \leq t^*$, it holds that:
$$\frac{\partial}{\partial z}\score^{\rho^*}_G(z) = \frac{f_G''(z)}{f_G(z) \lor \rho^*} - \left(\frac{f'_G(z)}{f_G(z) \lor \rho^*}\right)^2.$$
It then suffices to separately cover
$$\cb{\frac{f_G''(z)}{f_G(z) \lor \rho^*}\,:\, G \in \mathcal{P}(M)},\quad \text{ and }\quad  \cb{\left(\frac{f'_G(z)}{f_G(z) \lor \rho^*}\right)^2\,:\, G \in \mathcal{P}(M)}.$$
We already constructed a cover of the first class in Lemma~\ref{lemm:second-covering} (it is the class $\mathscr{T}^{(2)}_{\rho^*}$ therein). Meanwhile, we can cover the second class in the display equation above by noting that its elements are the squares of elements in $\mathscr{T}^{(1)}_{\rho^*}$ of Lemma~\ref{lemm:second-covering}. Using the upper bound in Lemma~\ref{lemm:proposition_1_jiang_zhang}, it follows that:
$$\log N(\eta, \mathcal{M}^{\rho^*}_c, \Norm{\cdot}_{t^*,\infty})  \lesssim_{C} \left( 1 + M^2 + \abs{ \log(\eta \rho^*)}\right)^2 \lesssim_{C} \left( \abs{ \log(\eta)} + \abs{\log( \rho^*)} \right)^2$$
since $\abs{\log\rho^*} \lesssim (1 + M)^2 \log n$. The argument for the star hull of  $\mathcal{M}^{\rho^*}_c$ is identical to the analogous argument in Lemma~\ref{lemm:complexity_regularized_score}, and so, omitted. Proving part 3 is similar to the part 3 proof in Lemma~\ref{lemm:complexity_regularized_score}, but the covering number result above uses a clipped norm for $|Z_i| < t^*$. 
\begin{align*}
&\localizedpoprademacher{\delta}{\StarHull(\mathcal{M}^{\rho^*}_c)} \\
&  \qquad = \mathbb{E} \left[ \sup_{h \in \StarHull(\mathcal{M}^{\rho^*}_c), \, \| h \|_2 \leq \delta } \abs{ \frac{1}{n} \sum_{i = 1}^n \varepsilon_i h(Z_i) }  \right] \\
&  \qquad  \leq \mathbb{E} \left[ \sup_{h \in \StarHull(\mathcal{M}^{\rho^*}_c) , \, \| h \|_2 \leq \delta  }  \left| \ind(A_n) \frac{1}{n} \sum_{i = 1}^n \varepsilon_i h(Z_i) \right|  \right] + \mathbb{E} \left[ \sup_{h \in \StarHull(\mathcal{M}^{\rho^*}_c)}  \left| \ind(A_n^c) \frac{1}{n} \sum_{i = 1}^n \varepsilon_i h(Z_i) \right|  \right] \\
&  \qquad  \lesssim \frac{\delta}{\sqrt{n}}\left( \abs{ \log(1/\delta)} + \abs{\log \rho^*} \right) \,+\, \frac{\abs{\log \rho^* } }{n} \left( \abs{\log(1/\delta)} + \abs{\log \rho^* } \right)^2  + \abs{ \log \rho^*} 2n^{1-C}
\end{align*}
In the second step, we separate the first term from the tail term using the event $A_n$. For the first term of the last step, we apply the exact argument as in the proof of part 3 of Lemma~\ref{lemm:complexity_regularized_score} to the first term. The only difference here is that $h \in \mathcal{M}_c^{\rho^*}$ is bounded by $\abs{ \log \rho^*}$ (instead of $\sqrt{\abs{ \log \rho^*}}$ as in the other Lemma. For the second term of the last step, we use $\| h( \cdot ) \|_\infty \lesssim \abs{ \log \rho^*}$ and the probability of $A_n^c$ given in part 1 of Lemma~\ref{lemm:compact_z_whp}.

Let us prove part 4. By Stein's lemma, we have that:
$$
\EE[G_\star]{\score_{\star}(Z_i)\cb{\regscorestar_{\star}(Z_i)-\regscorestar_{G}(Z_i)}  + \frac{\partial}{\partial z}\cb{\regscorestar_{\star}(Z_i)-\regscorestar_G(Z_i)}} = 0.$$
Thus, 
$$
\begin{aligned}
\abs{\EE[G_\star]{h(Z_i)}} &= \abs{\EE[G_\star]{\regscorestar_{\star}(Z_i)\cb{\regscorestar_{\star}(Z_i)-\regscorestar_{G}(Z_i)}  + \frac{\partial}{\partial z}\cb{\regscorestar_{\star}(Z_i)-\regscorestar_G(Z_i)}}} \\ 
&=\abs{\EE[G_{\star}]{\cb{\regscorestar_{\star}(Z_i)-s_{\star}(Z_i)}\cb{\regscorestar_{\star}(Z_i)-\regscorestar_{G}(Z_i)}}} \\
&= \abs{\EE[G_{\star}]{\ind(A_n^c) \cb{\regscorestar_{\star}(Z_i)-s_\star(Z_i)}\cb{\regscorestar_{\star}(Z_i)-\regscorestar_{G}(Z_i)}}} \\ 
& \stackrel{(i)}{\leq} \sqrt{\PP[G_{\star}]{A_n^c}}\sqrt{ \EE{\cb{\regscorestar_{\star}(Z_i)-s_\star(Z_i)}^2\cb{\regscorestar_{\star}(Z_i)-\regscorestar_{G}(Z_i)}^2}} \\ 
& \stackrel{(ii)}{\lesssim} \sqrt{\PP[G_\star]{A_n^c}} \abs{\log \rho^*}^{1/2} \sqrt{ \EE{\cb{\regscorestar_{\star}(Z_i)-s_\star(Z_i)}^2}} \\
& \stackrel{(iii)}{\lesssim} \sqrt{\PP[G_\star]{A_n^c}} \abs{\log \rho^*}^{1/2}(   \abs{\log \rho^*}^{1/2} + M) \\ 
& \stackrel{(iv)}{\lesssim } n^{(1-C)/2}\abs{\log \rho^*}^{1/2}(   \abs{\log \rho^*}^{1/2} + M).
\end{aligned}
$$ 
In $(i)$ we use Cauchy-Schwarz, in $(ii)$ we use the uniform bound on the regularized scores (\ref{lemm:complexity_regularized_score} part 1), and in $(iii)$ we additionally use the inequality $(x + y)^2 \le 2x^2 + 2y^2$ and  that $\abs{s_\star(Z_i)} \leq \abs{Z_i} + M \leq \abs{\xi_i} + 2M$ with $\xi_i \sim \mathrm{N}(0,1)$, and so $\EE[G_\star]{s_\star(Z_i)^2} \leq  1+ 4M + 4 M^2$. Finally,  in step $(iv)$ we use Part 1 of Lemma~\ref{lemm:compact_z_whp} to control $\PP[G_\star]{A_n^c}$. 
\end{proof}

For the next lemma, as well as for the main proof of Theorem~\ref{theo:rate_homoscedastic}, we use the following standard notation in empirical process theory. Given a function $h: \RR \to \RR$, we write
$$\hEE{h} := \frac{1}{n}\sum_{i=1}^n h(Z_i).$$

\begin{lemm}
\label{lemm:l2_to_l2pn}
There is an event $A_n'$ with the following properties:
\begin{enumerate}
\item $\PP{A_n'} \geq  1- \exp\left(-c_1 \abs{ \log \rho^* }^2 \right)$ for a positive constant $c_1$. 
\item On the event $A_n \cap A_n'$ and for $\rho^*$ and $M,C$ defined in Lemma~\ref{lemm:compact_z_whp}, it holds that:
$$\Dfisher{f_{\star}}{f_G} \lesssim \hEE{ \p{\regscorestar_G  - \regscorestar_{\star}}^2} + 4M^2 n^{1-C} +  \frac{\abs{ \log \rho^* }^3}{n}  \text{ for any } G \in \mathcal{P}(M).$$
\end{enumerate}
\end{lemm}

\begin{proof}
We first define the event $A_n'$
as the high probability event in~\citet[Theorem 14.1]{wainwright2019highdimensional}, which we apply for the class $\StarHull(\mathcal{S}_c^{\rho^*})$ studied in~Lemma~\ref{lemm:complexity_regularized_score}. In particular, by the former lemma, we may take $t=\delta = C' \abs{ \log (\rho^*) }^{1.5} / \sqrt{n} $, which we derive probability bounds on $A_n'$.
 
Theorem 14.1 of~\citet{wainwright2019highdimensional} then yields that on the event $A_n'$ the following holds for any $h \in \StarHull(\mathcal{S}_c^{\rho^*})$:

$$\abs{ \EE{h(Z_i)^2} -  \hEE{h^2}} \leq \frac{1}{2} \EE{h(Z_i)^2} + \frac{1}{2} t^2.$$
Hence:
$$\EE{h(Z_i)^2} =  \EE{h(Z_i)^2}  - \hEE{h^2}  + \hEE{h^2} \leq    \frac{1}{2} \EE{h(Z_i)^2} + \frac{1}{2} t^2 + \hEE{h^2}.$$
And by rearranging:
$$
\EE[G_{\star}]{h(Z_i)^2} \leq t^2 + 2\hEE{h^2}.
$$
Notice that the above holds for any $h \in \StarHull(\mathcal{S}_c^{\rho})$, so in particular it also holds for $h_G \in \mathcal{S}_c^{\rho^*}$, where the subscript $G$ corresponds to the indexing $G \in \mathcal{P}(M)$ in the definition~\eqref{eq:score_class_center_rho_rho} and $\rho^*$ is as in \eqref{eq:tstar}.

Making the notation more explicit, $\EE[G_{\star}]{h_G(Z_i)^2} = \EE[G_{\star}]{ \cb{ \regscore_G(Z_i) - \regscore_{\star}(Z_i)}^2}$. 
We also have that,
$$
\begin{aligned}
\Dfisher{f_{\star}}{f_G} &= \EE[G_\star]{ \cb{ \score_G(Z_i) - \score_{\star}(Z_i)}^2} \\ 
&= \EE[G_\star]{ \cb{ \score_G(Z_i) - \score_{\star}(Z_i)}^2 \ind(A_n)} \, + \,  \EE[G_\star]{ \cb{ \score_G(Z_i) - \score_{\star}(Z_i)}^2 \ind(A_n^c)} \\ 
& = \EE[G_\star]{ \cb{ \regscorestar_G(Z_i) - \regscorestar_{\star}(Z_i)}^2 \ind(A_n)}  \,+\, \EE[G_\star]{ \cb{ \EE[G]{\mu_i \mid Z_i} - \EE[G_{\star}]{\mu_i \mid Z_i}}^2 \ind(A_n^c)} \\ 
& \leq \EE[G_\star]{ \cb{ \regscorestar_G(Z_i) - \regscorestar_{\star}(Z_i)}^2}\,+\, 4M^2 \PP[G_{\star}]{A_n^c}.
\end{aligned}
$$
\end{proof}

\subsection{Putting everything together: Proof of Theorem~\ref{theo:rate_homoscedastic}}

\begin{proof}[Proof of Theorem~\ref{theo:rate_homoscedastic}] Take $\delta = n^{-2}$. Throughout this proof we use $\rho^*$ defined in~\eqref{eq:tstar} (which is a function of $n$, $C>0$ to be specified at the end of the proof and $M$). 
Recall the following from Lemma~\ref{lemm:deterministic_inequality}. There is a (random) complexity function
$$\widehat W(r) := \sup_{G \in \mathcal{P}(M) : \hEE{ (\regscorestar_G - \regscorestar_{\star})^2} \le r^2} 2\, \hEE{\regscorestar_{\star}\cb{\regscorestar_\star - \regscorestar_G} + \frac{\partial}{\partial z}(\regscorestar_{\star} - \regscorestar_G)}.$$
and to upper bound the squared loss, it suffices to upper bound $(\hat r)^2$ where $\hat r$ is the greatest postfixed point
$$ 
\hat r = \sup \{ r \ge 0 \mid r^2 \le  \widehat W(r) \}. 
$$
To do this, we will apply a recursive localization argument at a carefully chosen deterministic sequence of radii
$r_0,r_1,\ldots,r_I$ for $I \ge 0$ to be chosen later. This will let us iteratively upper bound $\widehat W(\hat r)$ (and thus $(\hat r)^2$) using that $\widehat W$ is a monotone function, where for each radius we apply concentration of measure and we combine the good events using the union bound. (Phrased differently, we will with high probability be able to use the fact that $\hat r \le r_0$ to argue that $\hat r \le r_1$, and so on which will ultimately let us argue $\hat r \le r_I$.)

To start with, we consider the analysis for a single fixed $r$. We require that $r^2 \gtrsim \abs{\log \rho^*}^3 /n$ (throughout the rest of the argument). We will only consider priors in the class 
\begin{equation*}
\mathscr{G}_r = \left\{ G \in \mathcal{P}(M) :
\hEE{\p{\regscorestar_G - \regscorestar_{\star}}^2} \leq r^2 \right\}.
\end{equation*}
On our high probability event $A_n \cap A_n'$, priors belonging to the above class also belong in $\mathscr{G}_r'$, which is defined as
\begin{equation*}
\mathscr{G}'_r = \left\{ G \in \mathcal{P}(M) :
\EE{ \left( \regscorestar_\star \left( \regscorestar_\star - \regscorestar_G \right) + \frac{\partial}{\partial z } \left(\regscorestar_\star - \regscorestar_{G}  \right)\right)^2 } \leq C' (1 +M)^2 \bigl( r \log n \bigr)^2 \right\}
\end{equation*} 
for some large enough $C' > 0$. The claim that, on $A_n \cap A_n'$, $\mathscr{G}_r \subseteq \mathscr{G}_r'$  can be shown in a few lines:
\begin{align}
\label{eqn:expand-noise-process}
\mathbb{E} \Biggl[ \biggl( \regscorestar_\star \left( \regscorestar_\star - \regscorestar_G \right) \biggr. \biggr. &+ \left. \left. \frac{\partial}{\partial z } \left(\regscorestar_\star - \regscorestar_{G}  \right)\right)^2 \right]  \notag \\
&\lesssim \EE{\left( \regscorestar_\star \left( \regscorestar_\star - \regscorestar_G \right)\right)^2} + \EE{\left(\frac{\partial}{\partial z } \left(\regscorestar_\star - \regscorestar_{G}  \right) \right)^2 } \\
&\lesssim  \abs {\log \rho^* } r^2 + \left(1+M^2\right) \log^2 n \, r^2 \notag \\ 
&\lesssim \left(1+M\right)^2 \,  r^2 \log^2 n \notag
\end{align}
In \eqref{eqn:expand-noise-process}, each expectation is separated into a center and tail term. Let $A = [-t^*, t^*]$ be a deterministic interval. Then, for some $Z \sim f_{G_\star}$, the probability $\PP{Z \in A^c} \leq 2n^{-C}$ using the same proof idea as part 1 of Lemma~\ref{lemm:compact_z_whp}. We begin with the first term:
\begin{align*}
&\EE{\left( \regscorestar_\star \left( \regscorestar_\star - \regscorestar_G \right)\right)^2} \\
& \qquad \leq \left\| \regscorestar_\star \right\|^2_\infty \, \EE{ \left( \regscorestar_\star(z) - \regscorestar_G(z) \right)^2} \\
& \qquad \lesssim \abs {\log \rho^* }  \, \left(  \EE{ \ind\{z \in A\} \left( \score_\star(z) - \score_G(z) \right)^2} + \EE{ \ind\{z \in A^c\} \left( \regscorestar_\star(z) - \regscorestar_G(z) \right)^2} \right) \\
&\qquad \lesssim  \abs {\log \rho^* } \, \left( r^2 + 4M^2 n^{1-C} + \frac{ \abs{\log \rho^* }^3}{n} + 2n^{-C} \abs{ \log \rho^*}  \right)   \\
& \qquad \lesssim\abs {\log \rho^* } r^2
\end{align*}
In the first and third inequality, Lemma~\ref{lemm:proposition_1_jiang_zhang} is used to bound the regularized score. In the second inequality, the expectation that conditions on $A$ is bounded using Part 2, Lemma~\ref{lemm:l2_to_l2pn}. In the last inequality, we use that $\abs{\log \rho^* }^3 / n \lesssim r^2$. As for the second term:
\begin{align*}
&\mathbb{E} \left[  \left(\frac{\partial}{\partial z } \left(\regscorestar_\star(z) - \regscorestar_{G}(z)  \right) \right)^2 \right] \\
&\qquad \qquad  \stackrel{(i)}{=}  \mathbb{E} \left[ \ind\{ z \in A\} \left(\frac{\partial}{\partial z } \left(\score_\star(z) - \score_{G}(z)  \right) \right)^2 \right] + \mathbb{E} \left[ \ind\{ z \in A^c \} \left(\frac{\partial}{\partial z } \left(\regscore_\star(z) - \regscore_{G}(z)  \right) \right)^2 \right] \\
&\qquad \qquad \stackrel{(ii)}{\lesssim} (1 + M)^2 \left[ 1 + \log^2_+ \left( \frac{C_0 M^2}{\Dfisher{f_{\star}}{f_G} } \right) \right] \Dfisher{f_{\star}}{f_G} + \abs{ \log\rho^*}^2 \left(2n^{-C}\right) \\
&\qquad \qquad \stackrel{(iii)}{\lesssim} (1 + M)^2 \left[ 1 + \log^2_+ \left( \frac{C_0 M^2}{ r^2 } \right) \right] r^2 + \abs{ \log\rho^*}^2 \left(2n^{-C}\right)  \\
&\qquad \qquad \stackrel{(iv)}{\lesssim} (1 + M)^2\, r^2  \log^2 n.
\end{align*}
In $(i)$, we know from Lemma~\ref{lemm:compact_z_whp} that conditioning on $A$ deactivates the regularization. In $(ii)$ we use Lemma~\ref{theo:apply-gaussian-strip-interpolation} for the first term and Lemma~\ref{lemm:proposition_1_jiang_zhang} for the second term. For $(iii)$, we know $t \mapsto [1 + \log^2_+(C_0 M^2/t) ] t$ is non-decreasing, and Lemma~\ref{lemm:l2_to_l2pn} gives $\Dfisher{f_{\star}}{f_G}\lesssim r^2$ on $A_n'$. In $(iv)$, $C_0 (1 + M)^2 / r^2 \lesssim C_0 (1+M)^2 n / \abs{ \log \rho^* }^3 \lesssim n$, and the second term is absorbed into the radius floor. 

Now we can begin to bound the complexity function for this fixed $r$ on $A_n \cap A_n'$.
\begin{align*} 
\widehat W(r) 
&= \sup_{G \in \mathscr{G}_r}  2\, \hEE{\regscorestar_\star\cb{\regscorestar_\star - \regscorestar_G} + \frac{\partial}{\partial z} (\regscorestar_\star - \regscorestar_G)}  \\
&\lesssim 
\sup_{G \in \mathscr{G}_r}  \EE{\regscorestar_\star\cb{\regscorestar_\star - \regscorestar_G} + \frac{\partial}{\partial z} (\regscorestar_\star - \regscorestar_G)}  \\
&+ \sup_{G \in \mathscr{G}_r} \Bigg\{\hEE{\regscorestar_\star\cb{\regscorestar_\star - \regscorestar_G} + \frac{\partial}{\partial z}(\regscorestar_\star - \regscorestar_G)} - \EE{\regscorestar_\star\cb{\regscorestar_\star - \regscorestar_G} + \frac{\partial}{\partial z} (\regscorestar_\star - \regscorestar_G)}  \Bigg\}.
\end{align*}
By Lemma~\ref{lemm:noise_process_complexity} part (4), the first term can be bounded by
$$
\sup_{G \in \mathscr{G}_r} \EE{\regscorestar_\star\cb{\regscorestar_\star - \regscorestar_G} + \frac{\partial}{\partial z} (\regscorestar_\star - \regscorestar_G)} \lesssim  n^{\frac{1-C}{2}}\abs{\log \rho^*}^{1/2}(   \abs{\log \rho^*}^{1/2} + M).
$$
The second term can be bounded using the fact that $\mathscr{G}_r \subseteq \mathscr{G}_r'$. 
\begin{equation*}
\sup_{G \in \mathscr{G}_r} \left[\hEE{\regscorestar_\star\cb{\regscorestar_\star - \regscorestar_G} + \frac{\partial}{\partial z} (\regscorestar_\star - \regscorestar_G)} -\EE{\regscorestar_\star\cb{\regscorestar_\star - \regscorestar_G} + \frac{\partial}{\partial z}(\regscorestar_\star - \regscorestar_G)} \right] \leq Z_r,
\end{equation*}
where $Z_r$ is defined as 
\begin{align*}
Z_r := & \sup_{G \in \mathscr{G}_r'} \Bigg\{ \abs{ \hEE{\regscorestar_\star\cb{\regscorestar_\star - \regscorestar_G} + \frac{\partial}{\partial z} (\regscorestar_\star - \regscorestar_G)} - \EE{\regscorestar_\star\cb{\regscorestar_\star - \regscorestar_G} + \frac{\partial}{\partial z}(\regscorestar_\star - \regscorestar_G)} } \Bigg\}. 
\end{align*}
By Theorem~\ref{theo:talagrand} (and the boundedness statement from Lemma~\ref{lemm:noise_process_complexity}, part 1), we have 
\[ Z_r \lesssim \EE{Z_r} +  \left(1 + M\right)   r \log n   \sqrt{\frac{\log(2/\tau)}{n}} + \frac{\abs{\log \rho^*} \log(2/\tau)}{n} \]
with probability at least $1 - \tau$. By symmetrization, we can upper bound $ \EE{Z_r}$ by
$$2\EE{\sup_{G \in \mathscr{G}_r' } 
\left|\hEE{\varepsilon \p{ \regscorestar_\star\cb{\regscorestar_\star - \regscorestar_G} + \frac{\partial}{\partial z}(\regscorestar_\star - \regscorestar_G)}}\right|},
$$
which is upper bounded by the localized Rademacher complexity of the centered regularized noise process class, $\localizedpoprademacher{ (1 + M) r  \log n  }{\mathcal{M}^{\rho^\star}_c}$, from 
Lemma~\ref{lemm:noise_process_complexity}, if we recall that in general (not conditioning on $A_n \cap A_n'$),
\[ \localizedpoprademacher{\delta}{\mathcal{F}} = \EE{ \sup\cb{\abs{\frac{1}{n}\sum_{i=1}^n \varepsilon_i f(W_i)}\,:\, f \in \mathcal{F},\, \Norm{f}_2 \leq \delta}}. \]
Combining part (3) of
Lemma~\ref{lemm:noise_process_complexity}, $\abs{ \log \rho^*} \lesssim \left(1 +M\right)^2 \log n$, and $r \gtrsim \abs{ \log \rho^*}^{1.5}/\sqrt{n}$,
\begin{align*}
\localizedpoprademacher{ \left(1 + M\right) \,  r \log n   }{\mathcal{M}^{\rho^*}_c} 
&\lesssim \frac{ \left(1 + M\right) \,  r \log n  }{\sqrt{n}}\abs{ \log \rho^*}  + \frac{\abs{\log \rho^* }^3}{n} + \abs{ \log \rho^*} 2n^{1-C} \\
&\lesssim \frac{r}{\sqrt{n}} \left( \left(1 + M\right)^3 \log^2 n  + \abs{ \log \rho^*}^{1.5} \right) + \abs{ \log \rho^*} 2n^{1-C} \\
&\lesssim \frac{r \, \left( 1 + M \right)^3 \log^2 n}{\sqrt{n}}.
\end{align*}
The tail term is negligible for large $C$ because of the radius floor. In summary, for fixed $r$ we can show with probability at least $1 - \tau$ that
\[ Z_r \lesssim \frac{\left( 1 + M \right)^3 \, r \log^2 n}{\sqrt{n}} + r \, (1 + M) \log n \sqrt{\frac{\log(2/\tau)}{n}} + \frac{\abs{\log \rho^* } \log(2/\tau)}{n} \]
and
\begin{align}\label{eqn:w-singler-bound} 
\widehat W(r) &\lesssim  \frac{\left(1 + M \right)^3 \, r \log^2 n}{\sqrt{n}} +  (1 + M) \, r \log n \sqrt{\frac{\log(2/\tau)}{n}} + \frac{\abs{\log \rho^* } \log(2/\tau)}{n} \notag \\
& \qquad \qquad +  n^{\frac{1-C}{2}}\abs{\log \rho^*}^{1/2}(   \abs{\log \rho^*}^{1/2} + M). 
\end{align}
This concludes the description of the analysis at a fixed radius $r$.

Now we define the following deterministic sequence of radii for constants $K > 1$ and $\tau > 0$ to be chosen later. Let
$$r_0=C'\sqrt{\abs{\log\rho^*}},$$
where $C'$ is sufficiently large that $\widehat W(r)\leq r_0^2$ for every $r\geq0$, by part (1) of Lemma~\ref{lemm:noise_process_complexity}. We take the radius floor to be $K\abs{\log\rho^*}^3/n$, with $K$ large enough for~\eqref{eqn:w-singler-bound} to apply above this floor. For each $i \ge 0$, define
\begin{align*}
r^2_{i+1} &= K \left( \frac{r_i \, \left( 1 + M \right)^3 \log^2 n}{\sqrt{n}} + r_i \left( 1 + M \right) \,  \log n \sqrt{\frac{\log(2/\tau)}{n}} + \frac{\abs{\log \rho^* } \log(2/\tau)}{n}  
\right. 
\\
& \Biggl. \qquad \qquad +  n^{\frac{1-C}{2}}\abs{\log \rho^*}^{1/2}(   \abs{\log \rho^*}^{1/2} + M) \Biggr) 
\end{align*}
By the AM-GM inequality, we have $ab = (a/\sqrt{2})(b\sqrt{2}) \le a^2/4 + b^2$, so
\begin{align*}
r_{i + 1}^2 &\le \frac{r_i^2}{4} + \frac{K^2}{n} \left( \left( 1 + M\right)^3 \log^2 n + (1 + M)  \log n \sqrt{ \log (2 / \tau)} \right)^2 \\
& \qquad \qquad + K \left( \frac{\abs{ \log \rho^* }  \log(2 / \tau) }{ n } + n^{\frac{1-C}{2}}\abs{\log \rho^*}^{1/2}(   \abs{\log \rho^*}^{1/2} + M)  \right) 
\end{align*}
So for each $i$, either
\begin{align}
\label{eqn:ri-terminate}
r^2_{i+1} &\leq  \frac{2K^2}{n} \left( \left( 1 + M\right)^3 \log^2 n + (1 + M)  \log n \sqrt{ \log (2 / \tau)} \right)^2 \notag \\
& \qquad \qquad + 2K \left( \frac{\abs{ \log \rho^* }  \log(2 / \tau) }{ n } + n^{\frac{1-C}{2}}\abs{\log \rho^*}^{1/2}(   \abs{\log \rho^*}^{1/2} + M)  \right) 
\end{align}
or $r_{i + 1}^2 \le 2r_i^2/4 = r_i^2/2$. How many contractions will occur before \eqref{eqn:ri-terminate} is satisfied? Let 
\begin{equation*}
I_{\max} := \max \left\{ 1, \, \left\lceil \log_2 \left( r_0^2 \, \frac{n}{K\abs{ \log \rho^* }^3}\right) \right\rceil + 1 \right\}.
\end{equation*}
Take $\tau=\delta/(2I_{\max})$. If $r_0^2\leq K\abs{\log\rho^*}^3/n$, set $I=0$. Otherwise, let $I\geq1$ be the first index for which~\eqref{eqn:ri-terminate} holds with $i=I-1$, or $r_I^2\leq K\abs{\log\rho^*}^3/n$. This index is deterministic. If neither condition held in the first $I_{\max}$ steps, every step would be a contraction, giving
$$r_{I_{\max}}^2\leq r_0^2 2^{-I_{\max}}\leq K\abs{\log\rho^*}^3/n,$$
a contradiction. Thus $I\leq I_{\max}$.

Now we can proceed with the induction argument. The choice of $r_0$ and the definition of $\hat r$ give the base case $\hat r^2\leq\widehat W(\hat r)\leq r_0^2$. For $0\leq i<I$, the radius $r_i$ is above the floor by the definition of $I$. Assuming that $\hat{r}^2 \leq r_i^2$,
\begin{equation*}
r^2_{i+1} \overset{(i)}{\geq} \widehat{W} \left( r_i \right) \overset{(ii)}{\geq} \widehat{W} \left( \hat{r} \right) \overset{(iii)}{\geq} \hat{r}^2,
\end{equation*}
where $(i)$ follows from~\eqref{eqn:w-singler-bound} and the choice of $K$, with failure probability at most $\tau$ for each radius; $(ii)$ is by the inductive hypothesis and $\widehat{W}$ being monotonically increasing; and $(iii)$ is by definition of $\hat{r}$. This proves $\hat r^2\leq r_I^2$ on the intersection of these events and $A_n\cap A_n'$.

Observe that the above induction fails with probability at most
\begin{equation*}
\PP{A_n^c} + \PP{A_n'^c} + I_{\max} \tau \leq 2n^{1-C} + \exp\left(-c_1 \abs{\log \rho^*}^2\right) + \frac{\delta}{2} \leq \delta
\end{equation*}
using Lemmas~\ref{lemm:compact_z_whp} and \ref{lemm:l2_to_l2pn}; along with a union bound for each step of the induction. If we stop at the radius floor, use $\abs{\log\rho^*}^3\lesssim(1+M)^6\log^4 n$; otherwise, use~\eqref{eqn:ri-terminate}. Thus, with probability at least $1 - \delta$, we have
\begin{align*}
\hat{r}^2 &\leq r_I^2  \\
&\lesssim \frac{ \left( 1 + M\right)^6 \log^4 n + (1 + M)^2  \log^2 n \log (2 / \tau) + \abs{ \log \rho^* }  \log(2 / \tau)}{n} \\
&\qquad \qquad  + n^{\frac{1-C}{2}}\abs{\log \rho^*}^{1/2}(   \abs{\log \rho^*}^{1/2} + M)  
\end{align*}
assuming that $C$ was chosen sufficiently large ($C=10$ suffices). Since $\log(1/\tau)$ grows as $\log n$, we arrive at our final rate. So from Lemma~\ref{lemm:l2_to_l2pn}, with probability at least $1 - \delta$, we find an upper bound on the squared loss, which if all other parameters are fixed is a rate of $(1 + M)^6 \log^4 n/n$ in the number of samples.
\begin{equation*}
\Dfisher{f_{\star}}{f_{\widehat{G}}} \lesssim \hat{r}^2 + 4M^2 n^{1-C} +  \frac{\abs{ \log \rho^* }^3}{n}  \lesssim (1 + M)^6 \log^4 n/n
\end{equation*}
\end{proof}

\subsection{Proof of Theorem~\ref{theo:apply-gaussian-strip-interpolation}.}

We first present a generic Gaussian strip interpolation lemma for functions that are bounded and analytic on a complex strip. The proof combines exponential decays of Fourier tails with localization under the Gaussian measure. We will then tailor the result to our setting in the proof of Theorem~\ref{theo:apply-gaussian-strip-interpolation}.
\begin{lemm}
\label{lemm:1d-gaussian-strip-interpolation}
Let $\delta > 0$ and $h: \CC \to \CC$ such that $h(\RR) \subseteq \RR$. Let $h$ be holomorphic on the strip $\{ z \in \CC  : | \Im(z) | < 2 \delta\}$ and bounded: for $B < \infty$, $\displaystyle \sup_{x \in \RR} \sup_{ |y | < \delta } | h(x + iy ) |  \leq B$. There exists some universal $C > 0$ such that
\begin{equation*}
\int | h'(x) |^2 \varphi(x) dx \leq C \left( 1 + \frac{1}{\delta^2} \right) \left[ 1 + \log_+^2 \left( \frac{C B^2}{\int |h(x)|^2 \varphi(x) dx} \right) \right] \int |h(x)|^2 \varphi(x) dx,
\end{equation*}
where $\log_+ (x) = \log x \vee 0$ and $\varphi$ is the standard Gaussian density.
\end{lemm}

\begin{proof}
In this proof, $\Norm{\cdot}_2$ denotes the $L^2(\RR)$ norm with respect to Lebesgue measure: $\Norm{g}_2^2 = \int_\RR |g(x)|^2\,dx$. We begin with the following notations:
\begin{itemize}
\item $v(z) = h(z) \exp(-z^2 / 4)$,
\item $\rho = 1 \wedge \delta$,
\item $v_{+\rho}(x) = h(x + i \rho) \exp\bigl( - (x + i \rho)^2 / 4 \bigr)$,
\item $v_{-\rho}(x) = h(x - i \rho) \exp\bigl( - (x - i \rho)^2 / 4 \bigr)$, and
\item $\displaystyle \widehat{v}(\xi) = (2 \pi)^{-1/2} \int e^{- i x \xi} v(x) dx$, the unitary Fourier transformation of $v$.
\end{itemize}
We observe two properties of the shifted $v_{+\rho}$ and $v_{-\rho}$. The first property is that their respective Fourier transforms can be written in terms of $\widehat{v}(\xi)$:
\begin{equation*} 
\widehat{v_{+\rho}}(x) = e^{- \rho \xi} \widehat{v}(\xi) \qquad \text{and} \qquad \widehat{v_{-\rho}}(x) = e^{\rho \xi} \widehat{v}(\xi).
\end{equation*}
The transforms can be derived by using the holomorphicity and boundedness of $v$ and Cauchy's integral theorem. The second property is that they are both $L_2$ integrable over the real line:
\begin{equation*} 
\| v_{+\rho}\|_2^2 \leq c B^2 \qquad \text{and} \qquad \| v_{-\rho} \|_2^2 \leq cB^2.
\end{equation*}
With these notations and facts, we can start the bounds:
\begin{equation*}
\int_\RR | h'(x) |^2 \varphi(x) dx  = (2 \pi)^{-1/2} \int_\RR | h'(x) e^{-x^2/4} | ^2 dx \lesssim \| v ' \|^2_2 + \| v(x) x \|^2_2,
\end{equation*}
since $v'(x) = h'(x) e^{-x^2/4} - v(x) \tfrac{x}{2}$. Now we can bound the two $L_2$ terms separately.
\begin{align*}
\| v ' \|^2_2 &= \int_\RR | \widehat{v'}(\xi) |^2 d\xi &&\text{Plancherel theorem} \\
&= \int | i \xi \, \widehat{v} (\xi) |^2 d \xi && \text{integration by parts} \\
&= \int_{|\xi| \leq K} \xi^2 | \widehat{v} (\xi) |^2 d \xi + \int_{|\xi |> K} \xi^2 | \widehat{v} (\xi) |^2 d \xi 
\end{align*}
for some $K>0$. The first integral term can be bounded above by $K^2 \| v \|^2_2$ using Plancherel again, and the bound for the second term begins with noting $e^2 (\rho \xi)^2 \leq 4 e^{\rho \abs{ \xi} }$ for all $\xi \in \RR$. 
\begin{align*}
\int_{|\xi |> K} \xi^2 | \widehat{v} (\xi) |^2 d \xi &\leq \int_{|\xi |> K} \left( \frac{4}{(e \rho)^2 } e^{\rho \abs{\xi}} \right) \frac{e^{\rho \abs{ \xi}} }{e^{\rho \abs{ \xi}} }  | \widehat{v} (\xi) |^2 d \xi  \\
&\leq C \frac{e^{-\rho K} }{\rho^2} \int_{|\xi| > K } e^{2 |\xi|  \rho} | \widehat{v}(\xi) | ^2 d \xi \\
&\leq  C \frac{e^{-\rho K} }{\rho^2} \Biggl(\, \, \underbrace{\int_\RR | e^{- \xi \rho} \widehat{v} (\xi)|^2  d \xi}_{\| v_{+\rho} \|_2^2} + \underbrace{\int_\RR | e^{\xi \rho} \widehat{v} (\xi)|^2  d \xi}_{ \| v_{-\rho} \|_2^2} \, \,  \Biggr) \leq C B^2  \frac{e^{-\rho K} }{\rho^2} 
\end{align*}
Selecting $\displaystyle K = \frac{1}{\rho} \log_+ \left( \frac{CB^2}{ \| v \|^2_2}\right)$, we have
\begin{equation} \label{eqn:v'-l2-bound}
\| v ' \|_2^2 \leq \frac{1}{\rho^2} \left[ 1 + \log_+^2 \left( \frac{CB^2}{ \| v \|^2_2}\right)\right] \| v \|^2_2.
\end{equation}
Now, we proceed to bound the second term for some $R\geq 1$:
\begin{align}
\label{eqn:vx-l2-bound}
\| v(x) x \|^2_2 &= \int_{| x | \leq R} | v(x) x |^2 dx + \int_{|x| > R} | v(x) x |^2 dx \notag \\ 
&\lesssim R^2 \|v\|^2_2 + B^2 \int_{|x| > R} e^{-x^2 / 4} dx \notag \\
&\leq R^2 \|v\|^2_2 + B^2 e^{-R^2/4} \notag \\
&\leq C \left[ 1 + \log_+ \left( \frac{CB^2}{ \| v \|^2_2}\right) \right] \| v \|_2^2
\end{align}
where, for the last line, we set $\displaystyle R = 2 \sqrt{\log_+ \left( \frac{C B^2}{\|v \|^2_2} \right) }$. 

\noindent We arrive at the result by combining \eqref{eqn:v'-l2-bound} and \eqref{eqn:vx-l2-bound}, and noting that $1 + \rho^{-2} \leq 2 + \delta^{-2}$.
\end{proof}

\begin{proof}[Proof of Theorem~\ref{theo:apply-gaussian-strip-interpolation}]
Fix a $\mu \in [-M, M]$. Define the  translated function $
h_\mu(z) := \score_{G_\star}(z + \mu) - \score_G(z + \mu)$. For convenience, we can rewrite $h_\mu$ in terms of another function $r_G$:
\begin{align*}
r_G(x) &= \int_\RR \exp \left( \langle \mu, x  \rangle - \mu^2 / 2 \right) dG(\mu) \\
h_\mu(z) &= \biggl( (z + \mu) + \score_{G_\star}(z + \mu)  \biggr) - \biggl( (z + \mu) + \score_{G}(z + \mu)  \biggr) \\
&= \frac{ r'_{G_\star} (x + \mu) }{ r_{G_\star} (x + \mu) } - \frac{ r'_{G} (x + \mu) }{ r_{G} (x + \mu) }
\end{align*}
Define $\delta_M = \tfrac{\pi}{4M}$. The remainder of the proof is organized as follows. First, we establish the holomorphy and boundedness conditions of $h_\mu$. Then, we apply Lemma~\ref{lemm:1d-gaussian-strip-interpolation} to $h_\mu$ and adjust the inequality to arrive at the Lemma result.

We begin by showing $h_\mu$ is holomorphic on $\{ z \in \CC : | \Im(z) | < 2 \delta_M \}$. On the strip, we know $0 < | \mu y | \leq M |y| < \tfrac{\pi}{2 } $, so $\cos(\mu y ) \geq \cos \bigl( M |y| \bigr) > 0$ since cosine is even. Thus,
\begin{equation*}
\Re \bigl( r_G(x + iy) \bigr) = \int \exp \left( \langle \mu, x \rangle - \mu^2 /2 \right) \cos (\mu y) dG(\mu) \geq \cos \bigl( M |y| \bigr) r_G(x) > 0.
\end{equation*}
Now we show boundedness of $h_\mu(x + iy)$ when $|y| < \delta_M$. 
\begin{align*}
&\left| r_G'(x + iy ) \right| \leq M \int \exp \bigl( \mu x - \mu^2/2 \bigr) \bigl| \exp( i \mu y ) \bigr| d G(\mu) = M r_G(x) \\
&\left| r_G(x + iy ) \right| \geq \Re \bigl( r_G(x + iy ) \bigr) \geq \cos \bigl( M |y| \bigr) r_G(x) \\ 
&\left| \frac{r_G'(x + iy )}{r_G(x + iy )} \right| \leq \frac{M}{\cos \bigl( M |y| \bigr) } \leq \sqrt{2}M \\ \\
&\Rightarrow \sup_{x \in \RR} \sup_{|y| \leq \delta_M} | h_\mu(x + iy) |  \leq 2 \sqrt{2}M
\end{align*}
Using the notations 
\begin{equation*}
\Psi(t) := t \left[ 1 + \log_+^2 \left( \frac{CB^2}{t} \right) \right] \qquad \text{and} \qquad
d_\mu := \int \bigl| h_\mu(x) \bigr|^2 \varphi(x) dx,
\end{equation*}
we can apply Lemma~\ref{lemm:1d-gaussian-strip-interpolation} to $h_\mu$ with $\delta = \delta_M = \tfrac{\pi}{4M}$ and $B = 2 \sqrt{2}M$:
\begin{equation*}
e_\mu := \int \bigl| h'_\mu(x) \bigr|^2 \varphi(x)dx \leq C ( 1 + \delta_M^{-2} ) \Psi(d_\mu) \leq C( 1 + M^2 ) \Psi(d_\mu).
\end{equation*}
Notice that $\Psi$ is concave over its relevant domain. Since $\displaystyle | h_\mu(z) | = \bigl| \mathbb{E}_{\mu \sim G_\star}\bigl[ \theta \mid Z = z + \mu \bigr] - \mathbb{E}_{\mu \sim G}\bigl[ \theta \mid Z = z + \mu \bigr] \bigr| \leq 2M$, then $0 \leq d_\mu \leq 4M^2$. So, for $ 0 \leq t \leq 4M^2$ and a large enough constant $C$ such as $C \geq 4e$,
\begin{equation}
\Psi''(t) = \frac{2 \bigl( 1 - \log_+(CB^2 / t)  \bigr)  }{t} \leq 0.
\end{equation}
Finally, we can apply Fubini, change of variable, and Jensen's inequality:
\begin{align*}
\int \left( \frac{\partial}{\partial z} \score_{G_\star}(z) - \frac{\partial}{\partial z}  \score_{G}(z)  \right)^2 f_{G_\star}(z)dz &= \int_\mu \, \, \int_z | h'_0(z) |^2 \varphi(z - \mu) dz \, \, d G_\star(\mu) &  \\
&\leq \int_\mu \, \, C(1 + M^2) \Psi(d_\mu) d G_\star(\mu) \\
&\leq C(1 + M^2) \Psi\left(  \Dfisher{f_{G_\star}}{f_G} \right) \\
&\lesssim (1 + M^2) \left[ 1 + \log_+^2 \left( \frac{C M^2}{ \Dfisher{f_{G_\star}}{f_G}} \right) \right] \Dfisher{f_{G_\star}}{f_G}.
\end{align*}
We interpret the righthand side as zero when the Fisher divergence is zero. 
\end{proof}

\subsection{Proof of Corollary~\ref{coro:rate_subgaussian}}

\begin{proof}
Let $G_{\star,n}:=G_\star(\cdot\mid[-M_n,M_n])$ and $A_n:=\{\max_{1\leq i\leq n}|\mu_i|\leq M_n\}$. By~\eqref{eq:prior_exponential_moment},
\begin{align*}
G_\star\{|u|>M_n\}&\leq2\exp(-M_n^2/\tau^2)=2n^{-4}\leq n^{-3},\\
\int_{|u|>M_n}u^2\,G_\star(du)&\leq2M_n^2\exp(-M_n^2/\tau^2)\leq M_n^2n^{-3}.
\end{align*}
For the second inequality, we use that $x^2\exp(-x^2/\tau^2)$ is decreasing for $x\geq M_n\geq\tau$. A union bound gives $\PP[G_\star]{A_n^c}\leq n^{-2}$. Conditional on $A_n$, we have $\mu_i\simiid G_{\star,n}$ and $Z_i\mid\mu_i\simindep\mathrm{N}(\mu_i,1)$. Hence, by Theorem~\ref{theo:rate_homoscedastic}, with conditional probability at least $1-n^{-2}$,
\begin{equation*}
\Regret(G_{\star,n},\hG)\leq C\frac{(1+M_n)^6\log^4 n}{n}.
\end{equation*}

It remains to bound the regret under $G_\star$. For any $G\in\mathcal{P}(M_n)$, write $t_G(z):=\EE[G]{\mu\mid Z=z}$ and $t_\star:=t_{G_\star}$. For a fresh draw $(\mu,Z)$ from the original model, we have
\begin{equation}
\label{eq:truncation_regret_transfer}
\begin{aligned}
\Regret(G_\star,G)
&=\EE[G_\star]{\cb{\mu-t_G(Z)}^2-\cb{\mu-t_\star(Z)}^2}\\
&\leq\Regret(G_{\star,n},G)+\EE[G_\star]{\cb{\mu-t_G(Z)}^2\ind\{|\mu|>M_n\}}\\
&\leq\Regret(G_{\star,n},G)+4M_n^2n^{-3}.
\end{aligned}
\end{equation}
The first inequality follows by restricting the oracle risk to $|\mu|\leq M_n$ and using optimality of the posterior mean under $G_{\star,n}$. The second uses $|t_G(z)|\leq M_n$ and the tail second moment bound above. Taking $G=G_{\star,n}$ also gives $\Dfisher{f_{G_\star}}{f_{G_{\star,n}}}\leq4M_n^2n^{-3}$.

Putting the above results together, with probability at least $1-2n^{-2}$,
\begin{equation*}
\Regret(G_\star,\hG)\leq C\frac{(1+M_n)^6\log^4 n}{n}+\frac{4M_n^2}{n^3}.
\end{equation*}
Substituting $M_n=2\tau\sqrt{\log n}$ yields~\eqref{eq:rate_exponential_tails}. The argument for subexponential priors is analogous.
\end{proof}

\section{Proofs for SURE-LS (Section~\ref{sec:regression})}
\label{sec:appendix_regression_proofs}

\subsection{Deterministic argument}
\label{subsec:deterministic_argument_misspecified}

We now prove the key deterministic inequality from Proposition~\ref{prop:regression-basic}.

\begin{proof}[Proof of Proposition~\ref{prop:regression-basic}]
Write $\theta=(\lambda,b)$, $\theta^{\oracle}=(\lambda^{\oracle},b^{\oracle})$, and
$d=\hat\theta-\theta^{\oracle}$. The segment from $\hat\theta$ to $\theta^{\oracle}$ is feasible by star-shapedness. First-order optimality along this segment gives
$$
0\le\langle\nabla\SURE(\hat\theta),-d\rangle.
$$
The Hessian $\mathcal H=\nabla^2\SURE$ is a constant $2n\times2n$ matrix (conditional on the observations), and
$$
\nabla\SURE(\hat\theta)=\nabla\SURE(\theta^{\oracle})+\mathcal H d,
\;\;
\frac12\langle d,\mathcal H d\rangle
=\frac1n\sum_{i=1}^n\{(\hat\lambda_i-\lambda_i^{\oracle})Z_i-(\hat b_i-b_i^{\oracle})\}^2.
$$
Since the expected SURE objective equals the population risk, first-order optimality at the oracle gives
$$
\left\langle\EE{\nabla\SURE(\theta^{\oracle})},\theta^{\oracle}-\theta\right\rangle\le0
\;\;\text{for every feasible }\theta.
$$
This inequality also holds at the realized $\hat\theta$. Thus
\begin{align*}
\frac1n\sum_{i=1}^n\{(\hat\lambda_i-\lambda_i^{\oracle})Z_i-(\hat b_i-b_i^{\oracle})\}^2
&\le\frac12\langle\nabla\SURE(\theta^{\oracle}),-d\rangle\\
&\le\frac12\left\langle\nabla\SURE(\theta^{\oracle})-\EE{\nabla\SURE(\theta^{\oracle})},-d\right\rangle.
\end{align*}
Expanding the centered gradient using $Z_i=\mu_i+\xi_i$ shows that the last expression equals
\[
\frac1n\sum_{i=1}^n\left[
\{(2\lambda_i^{\oracle}\mu_i-b_i^{\oracle})(\lambda_i^{\oracle}-\hat\lambda_i)
-\lambda_i^{\oracle}(b_i^{\oracle}-\hat b_i)\}\xi_i
+\lambda_i^{\oracle}(\lambda_i^{\oracle}-\hat\lambda_i)(\xi_i^2-\sigma_i^2)
\right],
\]
as required.
\end{proof}
\subsection{Lower isometry bound}
\label{subsec:lower_isometry}

We start with the proof of Lemma~\ref{lemm:isometry_bounded_noise}, before extending it to the case with unbounded noise.

\begin{lemm}[Lemma~\ref{lemm:isometry_bounded_noise} restated]
Suppose $\xi_1,\ldots,\xi_n$ are independent, \smash{$\abs{\EE{\xi_i}} \le \sqrt{\EE{\xi_i^2}}/2$}, each satisfy $\EE{\xi_i^2} \ge 1$, and are valued in $[-\upperbound, \upperbound]$ for $\upperbound \ge 1$.
Consider a separable set $\mathcal H \subset \mathbb{R}^n \times [-1,1]^n$ with squared radius, resp. Rademacher complexity, 
$$r^2 := \sup_{(c,f) \in \mathcal H} \frac{1}{n}\sum_{i=1}^n f_i^2,\;\quad \rademacher(\mathcal{H}) := \EE{\sup_{(c,f) \in \mathcal H} \abs{ \frac{1}{n}\sum_{i =1}^n  \varepsilon_i c_i} + B\sup_{(c,f) \in \mathcal H} \abs{ \frac{1}{n}\sum_{i =1}^n  \varepsilon_i f_i}},$$
where the right-hand side expectation is taken over iid Rademacher random variables $\varepsilon_i$.

Then, for any $x\ge0$, with probability at least $1-e^{-x}$, uniformly over $(c,f)\in\mathcal H$, we have
$$\frac{1}{n}\sum_{i = 1}^n (c_i + f_i \xi_i)^2 \ge \frac{1}{8n}\sum_{i = 1}^n (c_i^2 + f_i^2) - 72 B\rademacher(\mathcal{H}) - 17 \upperbound r \sqrt{\frac{\max_{i=1}^n\EE{\xi_i^2} x}{n}} - 180\frac{\upperbound^2 x}{n}.$$
\end{lemm}
\begin{proof}
We start by separating out the effect of large $c_i$. Write
$a_i=\rho_{2B}(c_i)$, where $\rho_s(u)=\min(\max(u,-s),s)$.
For every $|y|\le B$,
\begin{equation}\label{eq:isometry-clipping}
(c_i+y)^2\ge(a_i+y)^2+\frac12(c_i^2-a_i^2).
\end{equation}
Indeed, there is equality when $|c_i|\le2B$; otherwise the difference between the two sides is
$(|c_i|-2B)\{(|c_i|+2B)/2+2\operatorname{sign}(c_i)y\}\ge0$.
The moment assumptions and the arithmetic--geometric mean inequality also give
\[
\EE{(a_i+f_i\xi_i)^2}
\ge\frac12\{a_i^2+f_i^2\EE{\xi_i^2}\}
\ge\frac12(a_i^2+f_i^2).
\]
Define the absolute supremum over all coordinates,
\[
Z=\sup_{(c,f)\in\mathcal H}\left|
\sum_{i=1}^n\{(a_i+f_i\xi_i)^2-\EE{(a_i+f_i\xi_i)^2}\}\right|.
\]
Combining the preceding two bounds yields, uniformly over $\mathcal H$,
\begin{equation}\label{eq:isometry-clipped-mean}
\frac1n\sum_{i=1}^n(c_i+f_i\xi_i)^2
\ge\frac1{2n}\sum_{i=1}^n(c_i^2+f_i^2)-\frac Zn.
\end{equation}
Since $|a_i+f_i\xi_i|\le3B$, the squared map is $6B$-Lipschitz on the relevant interval. Symmetrization and the contraction principle give
\begin{align*}
\EE{Z}
&\le2\EE{\sup_{(c,f)\in\mathcal H}\left|\sum_{i=1}^n\varepsilon_i(a_i+f_i\xi_i)^2\right|}\\
&\le24B\,\EE{\sup_{(c,f)\in\mathcal H}\left|\sum_{i=1}^n\varepsilon_i(a_i+f_i\xi_i)\right|}\\
&\le24B\left\{\EE{\sup_{(c,f)\in\mathcal H}\left|\sum_{i=1}^n\varepsilon_i c_i\right|}
+B\,\EE{\sup_{(c,f)\in\mathcal H}\left|\sum_{i=1}^n\varepsilon_i f_i\right|}\right\}\\
&=24Bn\,\rademacher(\mathcal H).
\end{align*}
For the last inequality, odd clipping does not increase absolute Rademacher complexity: apply contraction to the class together with its negative. Conditional on the noise, contraction also bounds the $f_i\xi_i$ term by $B$ times the complexity of $f$.
Moreover,
\[
\sup_{(c,f)\in\mathcal H}\sum_{i=1}^n\Var{(a_i+f_i\xi_i)^2}
\le36B^2nr^2\max_i\EE{\xi_i^2}.
\]
Apply Theorem~\ref{theo:talagrand} to the independent variables $(i,\xi_i)$ with $\varepsilon=2$ and $b=9B^2$. With probability at least $1-e^{-x}$,
\[
Z\le3\EE{Z}+17B\sqrt n\,r\sqrt{\max_i\EE{\xi_i^2}\,x}+180B^2x.
\]
Substitution in~\eqref{eq:isometry-clipped-mean}, and weakening its leading coefficient from $1/2$ to $1/8$, proves the stated bound.
\end{proof}
\begin{lemm}
\label{lemm:subgaussian-isometry}
Suppose that $\xi_1,\ldots,\xi_n$ are independent, mean-zero, $\sigma$-sub-Gaussian random variables with $\EE{\xi_i^2}\ge2$ and $\sigma^2\ge2$; thus
$\EE{\exp(u\xi_i)}\le\exp(\sigma^2u^2/2)$ for all $u\in\mathbb R$.
Fix $\delta\in(0,1/2]$ and set
\[
B=\sigma\sqrt{8\log(16n\sigma^2/\delta)},\qquad x=\log(2/\delta).
\]
For a separable set $\mathcal H\subset\mathbb R^n\times[-1,1]^n$, define
\begin{align*}
r^2&=\sup_{(c,f)\in\mathcal H}\frac1n\sum_{i=1}^n f_i^2,\\
\rademacher(\mathcal H)
&=\EE{\sup_{(c,f)\in\mathcal H}\left|\frac1n\sum_{i=1}^n\varepsilon_i c_i\right|}
+B\,\EE{\sup_{(c,f)\in\mathcal H}\left|\frac1n\sum_{i=1}^n\varepsilon_i f_i\right|},
\end{align*}
where the expectations are over iid Rademacher variables. With probability at least $1-\delta$, uniformly over $(c,f)\in\mathcal H$,
\begin{align*}
\frac1n\sum_{i=1}^n(c_i+f_i\xi_i)^2
\ge\frac1{8n}\sum_{i=1}^n(c_i^2+f_i^2)
-72B\,\rademacher(\mathcal H)-20B\sigma r\sqrt{\frac{x}{n}}-180\frac{B^2x}{n}.
\end{align*}
\end{lemm}
\begin{proof}
Sub-Gaussian tails and tail integration give
\begin{align*}
p_i:=\PP{|\xi_i|>B}
&\le2e^{-B^2/(2\sigma^2)}\le\min\{\delta/(2n),1/4\},\\
v_i:=\EE{\xi_i^2\ind(|\xi_i|>B)}
&\le2(B^2+2\sigma^2)e^{-B^2/(2\sigma^2)}\le1/2.
\end{align*}
Consequently,
\[
\frac32\le\EE{\xi_i^2\mid |\xi_i|\le B}\le\frac43\sigma^2,
\]
and, using $\EE{\xi_i}=0$ and Cauchy--Schwarz,
\[
\left|\EE{\xi_i\mid |\xi_i|\le B}\right|^2
\le\frac{p_iv_i}{(1-p_i)^2}\le\frac29
\le\frac14\EE{\xi_i^2\mid |\xi_i|\le B}.
\]
The event $E=\{\max_i|\xi_i|\le B\}$ has probability at least $1-\delta/2$. Conditional on $E$, independence is preserved and the preceding estimates verify every moment assumption of Lemma~\ref{lemm:isometry_bounded_noise}. Apply that lemma conditionally with $x=\log(2/\delta)$ and use $17\sqrt{4/3}\le20$. The conditional failure probability is at most $\delta/2$, so the total failure probability is at most $\delta$.
\end{proof}
\subsection{Proof of Theorem~\ref{theo:reg}}
\label{subsec:proof_regression_theorem}

We first give a localized bound retaining the distribution-dependent complexity terms. We then compare these terms with the local Gaussian complexities in Theorem~\ref{theo:reg}.
\begin{theo}\label{theo:reg-localized}
Let $n\ge2$ and $\delta\in(0,1/2]$. Suppose that $\mathcal L\subset[0,1]^n$ and $\mathcal B\subset\mathbb R^n$ are star-shaped about the chosen oracle minimizers, and that the chosen SURE minimizer exists and is measurable. Assume that the $\xi_i$ are independent, mean-zero, $K\sigma_i$-sub-Gaussian random variables with variances $\sigma_i^2$ and $\sigma_i\in[\sqrt2,\sigma_{\text{max}}]$.
Write $\sigma=K\sigma_{\text{max}}$ and $w_i=2\lambda_i^{\oracle}\mu_i-b_i^{\oracle}$. Let $B=\sigma\sqrt{8\log(48n\sigma^2/\delta)}$ be the truncation level in Lemma~\ref{lemm:subgaussian-isometry} with failure probability $\delta/3$.
For independent $\zeta_i\sim\mathrm N(0,\sigma_i^2)$, define
\begin{align*}
G_1(r)&=\EE{\sup_{\substack{b\in\mathcal B\\\|b-b^{\oracle}\|_2\le r}}
\frac1n\sum_{i=1}^n\lambda_i^{\oracle}(b_i-b_i^{\oracle})\zeta_i},\qquad
G_2(r)=\EE{\sup_{\substack{\lambda\in\mathcal L\\\|\lambda-\lambda^{\oracle}\|_2\le r}}
\frac1n\sum_{i=1}^n w_i(\lambda_i^{\oracle}-\lambda_i)\zeta_i},\\
T_3(r)&=\EE{\sup_{\substack{\lambda\in\mathcal L\\\|\lambda-\lambda^{\oracle}\|_2\le r}}
\left|\frac1n\sum_{i=1}^n\lambda_i^{\oracle}(\lambda_i^{\oracle}-\lambda_i)(\xi_i^2-\sigma_i^2)\right|}.
\end{align*}
Let $
\mathcal K_{r_1,r_2}=\{(b,\lambda)\in\mathcal B\times\mathcal L:
\|b-b^{\oracle}\|_2\le r_1,\ \|\lambda-\lambda^{\oracle}\|_2\le r_2\}$, and define
\begin{align*}
\rademacher(r_1,r_2)
&=\frac1n\EE{\sup_{\mathcal K_{r_1,r_2}}
\left|\sum_{i=1}^n\varepsilon_i\{(\lambda_i-\lambda_i^{\oracle})\mu_i-(b_i-b_i^{\oracle})\}\right|
+B\sup_{\mathcal K_{r_1,r_2}}
\left|\sum_{i=1}^n\varepsilon_i(\lambda_i-\lambda_i^{\oracle})\right|},\\
\Delta_\delta(r_1,r_2)
&=72B\,\rademacher(r_1,r_2)
+20B\sigma\frac{r_2}{\sqrt n}\sqrt{\frac{\log(6/\delta)}{n}}
+180\frac{B^2\log(6/\delta)}{n}.
\end{align*}
Here the suprema are over $(b,\lambda)\in\mathcal K_{r_1,r_2}$, and the $\varepsilon_i$ are iid Rademacher.

For $r_1(t)=4\sqrt{\|\mu\|_\infty^2+2}\,t\sqrt n$ and $r_2(t)=t\sqrt{8n}$, define
\begin{align*}
W(t)
&=C_0K\{G_1(r_1(t))+G_2(r_2(t))\}+3T_3(r_2(t))
+\Delta_\delta(r_1(t),r_2(t))\\
&\quad+C_0\frac{\{\sigma(r_1(t)+\|w\|_\infty r_2(t))+\sigma^2r_2(t)\}\sqrt{\log(12/\delta)}}{n}
+C_0\frac{\sigma^2\log n\,\log^2(12/\delta)}{n},
\end{align*}
where $C_0$ is a sufficiently large universal constant.
Then both conclusions of Theorem~\ref{theo:reg} hold with probability at least $1-\delta$, with $t_*$ replaced by
\begin{equation}\label{eq:reg-localized-radius}
t_{\mathrm{loc}}=\inf\{t\ge0:t^2\ge2W(t)\}.
\end{equation}
\end{theo}
\begin{rema}
The constant $C_0$ accounts for Gaussian comparison and sub-exponential concentration. The complexity terms can be bounded by symmetrization and contraction, as in the deduction of Theorem~\ref{theo:reg} below.
\end{rema}
\begin{proof}
Fix $t>0$ such that $W(t)\le t^2/2$, and write
$r_1=r_1(t)$ and $r_2=r_2(t)$. Define the boundary
$$
\mathcal G_{r_1,r_2}=\{(b,\lambda)\in\mathcal K_{r_1,r_2}:
\|b-b^{\oracle}\|_2=r_1\ \text{or}\ \|\lambda-\lambda^{\oracle}\|_2=r_2\}.
$$
If $(\hat b,\hat\lambda)\in\mathcal K_{r_1,r_2}$, let $(b',\lambda')=(\hat b,\hat\lambda)$. Otherwise, let $(b',\lambda')$ be the intersection of $\mathcal G_{r_1,r_2}$ with the segment from the oracle to $(\hat b,\hat\lambda)$. This point is feasible by star-shapedness.

For a feasible pair, write
\begin{align*}
Q(b,\lambda)&=\frac1n\sum_{i=1}^n\{(\lambda_i-\lambda_i^{\oracle})Z_i-(b_i-b_i^{\oracle})\}^2,\\
N(b,\lambda)&=\frac1n\sum_{i=1}^n\left[
\{w_i(\lambda_i^{\oracle}-\lambda_i)+\lambda_i^{\oracle}(b_i-b_i^{\oracle})\}\xi_i
+\lambda_i^{\oracle}(\lambda_i^{\oracle}-\lambda_i)(\xi_i^2-\sigma_i^2)\right].
\end{align*}
Proposition~\ref{prop:regression-basic} gives $Q(\hat b,\hat\lambda)\le N(\hat b,\hat\lambda)$. Along the segment from the oracle, $Q$ scales quadratically and $N$ linearly. Since $N(\hat b,\hat\lambda)\ge Q(\hat b,\hat\lambda)\ge0$, this also gives
\begin{equation}\label{eq:reg-interpolated-basic}
Q(b',\lambda')\le N(b',\lambda').
\end{equation}
Define the localized processes
\begin{align*}
Z_1(r)&=\sup_{\substack{b\in\mathcal B\\\|b-b^{\oracle}\|_2\le r}}
\frac1n\sum_{i=1}^n\lambda_i^{\oracle}(b_i-b_i^{\oracle})\xi_i,\qquad
Z_2(r)=\sup_{\substack{\lambda\in\mathcal L\\\|\lambda-\lambda^{\oracle}\|_2\le r}}
\frac1n\sum_{i=1}^n w_i(\lambda_i^{\oracle}-\lambda_i)\xi_i,\\
Z_3(r)&=\sup_{\substack{\lambda\in\mathcal L\\\|\lambda-\lambda^{\oracle}\|_2\le r}}
\left|\frac1n\sum_{i=1}^n\lambda_i^{\oracle}(\lambda_i^{\oracle}-\lambda_i)(\xi_i^2-\sigma_i^2)\right|.
\end{align*}
In particular, $\EE{Z_3(r)}=T_3(r)$ and
$N(b',\lambda')\le Z_1(r_1)+Z_2(r_2)+Z_3(r_2)$.\\

\noindent\textbf{Lower isometry and parameter control.}
Apply Lemma~\ref{lemm:subgaussian-isometry} with failure probability $\delta/3$ and
$$
f_i=\lambda_i-\lambda_i^{\oracle},\qquad
c_i=(\lambda_i-\lambda_i^{\oracle})\mu_i-(b_i-b_i^{\oracle}),
\qquad (b,\lambda)\in\mathcal K_{r_1,r_2}.
$$
The normalized radius in that lemma is at most $r_2/\sqrt n$, not $r_2$.
Since $|f_i|\le1$, its other assumptions hold as well. Hence, with probability at least $1-\delta/3$, uniformly on $\mathcal K_{r_1,r_2}$,
\begin{equation}\label{eq:reg-local-isometry}
Q(b,\lambda)\ge\frac1{8n}\left\{\|\lambda-\lambda^{\oracle}\|_2^2
+\|(\lambda-\lambda^{\oracle})\mu-(b-b^{\oracle})\|_2^2\right\}
-\Delta_\delta(r_1,r_2),
\end{equation}
where products of vectors are coordinatewise. Put $H(r_1,r_2)=Z_1(r_1)+Z_2(r_2)+Z_3(r_2)+\Delta_\delta(r_1,r_2)$.
Combining~\eqref{eq:reg-interpolated-basic} and~\eqref{eq:reg-local-isometry} gives
\begin{equation}\label{eqn:l-to-h}
\frac1{8n}\|\lambda'-\lambda^{\oracle}\|_2^2\le H(r_1,r_2).
\end{equation}
For the other parameter, the triangle inequality followed by Cauchy--Schwarz gives
\begin{align}\label{eq:H-to-simplify}
\|b'-b^{\oracle}\|_2^2
&\le(\|\mu\|_\infty^2+1)\left\{\|\lambda'-\lambda^{\oracle}\|_2^2
+\|(\lambda'-\lambda^{\oracle})\mu-(b'-b^{\oracle})\|_2^2\right\}\nonumber\\
&\le8n(\|\mu\|_\infty^2+1)H(r_1,r_2).
\end{align}
In particular,
\begin{equation}\label{eqn:b-to-h}
\frac{\|b'-b^{\oracle}\|_2^2}{16n(\|\mu\|_\infty^2+2)}\le H(r_1,r_2).
\end{equation}

\noindent\textbf{Upper bound on $H$.}
The processes defining $Z_1$ and $Z_2$ are centered and vanish at their respective oracles. Thus~\eqref{eq:majorizing}, with failure probability $\delta/6$ for each process, gives
\begin{align*}
Z_1(r_1)&\le C_0K G_1(r_1)+C_0\frac{\sigma r_1\sqrt{\log(12/\delta)}}{n},\\
Z_2(r_2)&\le C_0K G_2(r_2)+C_0\frac{\sigma \|w\|_\infty r_2\sqrt{\log(12/\delta)}}{n}.
\end{align*}
Here we used $|\lambda_i^{\oracle}|\le1$ and $|w_i|\le\|w\|_\infty$ to bound the comparison variances.
Also, $\xi_i^2-\sigma_i^2$ has sub-exponential scale at most $C\sigma^2$, while the coefficients
$a_i=\lambda_i^{\oracle}(\lambda_i^{\oracle}-\lambda_i)$ satisfy $|a_i|\le1$ and $\|a\|_2\le r_2$.
Apply Lemma~\ref{lemm:sub-talagrand} after rescaling the noise, with $\varepsilon=2$ and tail parameter $\log(3/\delta)\le \log(12/\delta)$. With failure probability at most $\delta/3$, including the truncation in that lemma,
$$
Z_3(r_2)\le3T_3(r_2)+C_0\sigma^2\left\{\frac{r_2\sqrt{\log(12/\delta)}}{n}
+\frac{\log n\,\log^2(12/\delta)}{n}\right\}.
$$
The union bound over these three events and the lower-isometry event gives total failure probability at most
$\delta/6+\delta/6+\delta/3+\delta/3=\delta$.
On their intersection, $H(r_1,r_2)\le W(t)\le t^2/2$.\\

\noindent\textbf{Localization and conclusion.}
If $(\hat b,\hat\lambda)$ were outside $\mathcal K_{r_1,r_2}$, then $(b',\lambda')\in\mathcal G_{r_1,r_2}$. By the choices of $r_1$ and $r_2$, the left side of at least one of~\eqref{eqn:l-to-h} and~\eqref{eqn:b-to-h} would equal $t^2$. This would imply
$t^2\le H(r_1,r_2)\le t^2/2$, a contradiction. Thus $(b',\lambda')=(\hat b,\hat\lambda)$, and
$$
Q(\hat b,\hat\lambda)\le Z_1(r_1)+Z_2(r_2)+Z_3(r_2)
\le H(r_1,r_2)\le t^2.
$$
The second conclusion follows from the triangle inequality.

The positive constant term in $W(t)$ ensures that $t_{\mathrm{loc}}>0$, and its terms grow at most linearly in $t$, so the feasible set is nonempty. The argument holds for every feasible $t$; taking a sequence of feasible radii decreasing to $t_{\mathrm{loc}}$ and passing to the limit in the scalar tail probabilities proves the conclusions at the infimum. No attainment of this infimum or global compactness of the parameter classes is needed: the proof uses only the chosen minimizers, star-shapedness, and bounded localized evaluation sets.
\end{proof}

\begin{proof}[Deduction of Theorem~\ref{theo:reg}]
All constants below depend only on $K,\sigma_{\text{max}}$, and $M$.
Choose $c\ge\max\{4\sqrt{M^2+2},\sqrt8\}$.
At $r_1=r_1(t)$ and $r_2=r_2(t)$, Gaussian comparison with the bounded coordinate multipliers gives
$$
G_1(r_1)+G_2(r_2)\le C\{\mathscr G_n(ct;\mathcal B_{\oracle})+\mathscr G_n(ct;\mathcal L_{\oracle})\}.
$$
To bound absolute Rademacher suprema, first compare with absolute Gaussian suprema. Each shifted localized class contains zero, so its absolute Gaussian width is at most twice its ordinary Gaussian width. Splitting $c_i=(\lambda_i-\lambda_i^{\oracle})\mu_i-(b_i-b_i^{\oracle})$ and applying contraction therefore gives
\begin{align*}
\rademacher(r_1,r_2)
&\le C\{\mathscr G_n(ct;\mathcal B_{\oracle})+(\|\mu\|_\infty+B)\mathscr G_n(ct;\mathcal L_{\oracle})\}\\
&\le C\left\{\mathscr G_n(ct;\mathcal B_{\oracle})+\sqrt{\log(n/\delta)}\,\mathscr G_n(ct;\mathcal L_{\oracle})\right\}.
\end{align*}
Here $B\le C\sqrt{\log(n/\delta)}$ because the noise parameters are fixed.

For the quadratic-noise term, symmetrization followed by contraction conditional on the noise gives
\begin{align*}
T_3(r_2)
&\le2\EE{\max_i|\xi_i^2-\sigma_i^2|}\,
\EE{\sup_{\substack{\lambda\in\mathcal L\\\|\lambda-\lambda^{\oracle}\|_2\le r_2}}
\left|\frac1n\sum_{i=1}^n\varepsilon_i(\lambda_i-\lambda_i^{\oracle})\right|}\le C\log n\,\mathscr G_n(ct;\mathcal L_{\oracle}).
\end{align*}
We used $|\lambda_i^{\oracle}|\le1$ and
$\EE{\max_i|\xi_i^2-\sigma_i^2|}\le C\sigma^2\log n$, which follows from the sub-exponential tail bound and $n\ge2$.

Finally, $r_2/\sqrt n=\sqrt8\,t$, and for $n\ge2$ and $\delta\le1/2$,
$$
\sqrt{\log(n/\delta)\log(1/\delta)}\le C|\log\delta|\sqrt{\log n},
\;\;
\log(n/\delta)\log(1/\delta)\le C|\log\delta|^2\log n.
$$
Substituting these estimates into $W(t)$ shows that
\begin{align}\label{eq:reg-gaussian-fixed-point}
W(t)\le C\Bigg\{
&\sqrt{\log(n/\delta)}\,\mathscr G_n(ct;\mathcal B_{\oracle})
+\log(n/\delta)\,\mathscr G_n(ct;\mathcal L_{\oracle})\nonumber\\
&+t|\log\delta|\sqrt{\frac{\log n}{n}}
+|\log\delta|^2\frac{\log n}{n}\Bigg\}.
\end{align}
Choose the constant in Theorem~\ref{theo:reg} at least twice the constant in this display. Every radius feasible for its defining infimum then satisfies $t^2\ge2W(t)$, so $t_{\mathrm{loc}}\le t_*$. Theorem~\ref{theo:reg-localized} now gives both claimed inequalities.
\end{proof}

\section{Proofs for multivariate SURE-PM}

In this section, we prove Theorem~\ref{theo:rate_homoscedastic_multivariate}. We work throughout under its assumptions and use the notation of Section~\ref{subsec:multivariate_sure_pm}. The proof follows the univariate argument in Supplement~\ref{sec:proofs_sure_homoscedastic}. We first give the corresponding basic inequality and a multivariate extension of Theorem~\ref{theo:apply-gaussian-strip-interpolation}. We write $s_\star=s_{G_\star}$.

\paragraph{SURE and regret.}
For any fixed $G\in\mathcal{P}([-M,M]^d)$, Gaussian integration by parts gives
$$
\EE[G_\star]{(Z-\mu)^\top\Sigma s_G(Z)}=\EE[G_\star]{\tr\p{\Sigma\nabla s_G(Z)\Sigma}}.
$$
Expanding the squared error and using $\EE[G_\star]{\|Z-\mu\|^2}=\tr(\Sigma)$ shows that $\SURE(G)$ in~\eqref{eq:multivariate_SURE} is unbiased for $\EE[G_\star]{\|\mu-\EE[G]{\mu\mid Z}\|^2}$. Moreover, by the defining property of conditional expectation and~\eqref{eq:multivariate_score_and_marginal},
\begin{equation}
\label{eq:regret_as_posterior_mean_mse_multivariate}
\Regret(G_\star,G)=\EE[G_\star]{\|\EE[G]{\mu\mid Z}-\EE[G_\star]{\mu\mid Z}\|^2}=\Dfishersigma{f_{G_\star}}{f_G}.
\end{equation}

\subsection{Proof elements}

The following is the multivariate version of Proposition~\ref{prop:basic_inequality}.
\begin{prop}[Multivariate basic inequality]
\label{prop:basic_inequality_multivariate}
Let $G_\star$ and $\hG$ be as in Theorem~\ref{theo:rate_homoscedastic_multivariate}. Then,
\begin{align}
\label{eq:basic_inequality_multivariate}
\frac{1}{n} \sum_{i=1}^n &\| \Sigma \score_{\widehat{G}}(Z_i) - \Sigma \score_{\star}(Z_i) \|^2  \notag \\
&\leq \frac{2}{n} \sum_{i=1}^n \tr \left(   \Sigma \biggl(  \score_\star(Z_i) \bigl( \score_\star(Z_i) - \score_{\widehat{G}}(Z_i) \bigr)^\top  + \bigl( \nabla \score_\star(Z_i) - \nabla \score_{\widehat{G}}(Z_i) \bigr) \biggr) \Sigma \right) 
\end{align}
\end{prop}
\begin{proof}
The inequality follows by expanding $\SURE(\hG)-\SURE(G_\star)\leq0$ and completing the square.
\end{proof}

As in Lemma~\ref{lemm:deterministic_inequality}, the basic inequality gives the following bound in terms of a local complexity functional.

\begin{lemm}[Multivariate deterministic inequality]
\label{lemm:deterministic_inequality_multivariate} Suppose that $G_{\star} \in \mathcal{P} \left( [-M, M]^d \right) $ and that we estimate $\hG$ by minimizing $\SURE(G)$ over $\mathcal{P} \left( [-M, M]^d \right) $. Define the local complexity functional as
\begin{equation*}
\begin{aligned}
\widehat W_\Sigma(r) := &\sup\Bigg\{ \frac{2}{n}\sum_{i=1}^n \tr \left( \Sigma \left( \score_\star(Z_i)  \left( \score_\star(Z_i) - \score_G(Z_i) \right)^\top + \nabla \left( \score_\star(Z_i) - \score_G(Z_i) \right)  \right) \Sigma \right) : \\ 
&\;\;\;\;\;\;\;\;\;\;\;\;\;\;\;\;\;\;\;  G \in \mathcal{P} \left( [-M, M]^d \right) \,\, \text{ s.t. }\,\, \frac{1}{n}\sum_{i=1}^n \| \Sigma \left( \score_\star(Z_i) - \score_G(Z_i) \right) \|^2  \leq r^2 \Bigg\}.
\end{aligned}
\end{equation*}
Define the post-fixed point as $\hat{r} =\sup \{ r \geq 0 : \widehat{W}_\Sigma(r) \geq r^2 \}$. Then, it holds that
\begin{equation*}
\frac{1}{n}\sum_{i=1}^n \| \Sigma \left( \score_\star(Z_i) - \score_{\widehat{G}}(Z_i) \right) \|^2 \leq \hat{r}^2.
\end{equation*}
\end{lemm}
\begin{proof}
Write $r_{\widehat{G}}^2 := \frac{1}{n}\sum_{i=1}^n \|\Sigma\p{s_\star(Z_i)-s_{\widehat{G}}(Z_i)}\|^2$. By~\eqref{eq:basic_inequality_multivariate}, $r_{\widehat{G}}^2\leq\widehat W_\Sigma(r_{\widehat{G}})$. Thus, $r_{\widehat{G}}\leq\hat r$ by the definition of $\hat r$.
\end{proof}

To control the right-hand side of the basic inequality, we need to control differences of score derivatives in terms of differences of scores. The following result extends Theorem~\ref{theo:apply-gaussian-strip-interpolation} to the multivariate setting. Its proof is given in Supplement~\ref{subsec:proof_multivariate_derivative}.

\begin{lemm}[Multivariate version of Theorem~\ref{theo:apply-gaussian-strip-interpolation}]
\label{lemm:multivariate-apply-gaussian-strip-interpolation}
Let $G, G_\star \in \mathcal{P}\left( [-M, M]^d \right)$. There is a $C_{\underline{k}, \overline{k}, d, M}  > 0$ and $C_0 > 4 d M^2 e $ (that is a function of $M$ and $d$) such that
\begin{align*}
\int \bigl| \Sigma_i^\top \bigr.& \nabla \left( \score_{G_\star}- \score_G  \right)(z)    \bigl. \Sigma_i \bigr|^2 f_{G_\star}(z)  dz \\
&\leq C_{\underline{k}, \overline{k}, d, M} \left(1 + M^2 \right) \left[ 1 + \log^2_+ \left( \frac{C_0 }{\Dfishersigma{f_{G_{\star}}}{f_{G}} } \right) \right] \Dfishersigma{f_{G_{\star}}}{f_{G}} 
\end{align*}
for $i=1,\dotsc,d$, where $\Sigma_i$ is the $i$th column of $\Sigma$.
\end{lemm}

\subsection{Complexity of scores and their derivatives}

As in Supplement~\ref{sec:proofs_sure_homoscedastic}, we bound the complexity using regularized scores. The following result gives the covering numbers needed for the empirical process argument. Its proof builds on~\citet{saha2020nonparametric} and extends their bounds to Hessians.

\begin{lemm}[Covering numbers for regularized scores and Hessians]
\label{lemm:multivariate-second-covering}
\noindent Consider the families of functions,
$$\mathscr{T}^{(1)}_{\rho} := \cb{\frac{ \nabla f_G(\cdot)}{f_G(\cdot) \vee \rho}: G \in \mathcal{P}\left([-M,M]^d\right) },\,\,\,\, \mathscr{T}^{(2)}_{\rho} := \cb{\frac{\nabla^2 f_G(\cdot)}{f_G(\cdot) \vee \rho}: G \in \mathcal{P}\left([-M,M]^d\right) }.$$
Let $0 < \rho < (2 \pi)^{-d/2} e^{-1/2}$ and $\| \cdot   \|_{F, \infty} = \sup_{z \in \RR^d} \| \cdot  \|_F$. There exists a constant $C_d'$ depending on $d$ alone such that for any $ 0 < \eta \leq \tfrac{C_d'}{ (2\pi)^{d/2} \sqrt{e} }$,  
     \begin{equation} 
     \label{eq:multivariate_regularized_score_covering}
        \log N \left( \eta, \mathscr{T}_\rho^{(1)} , \| \cdot \|_{F, \infty} \right) \lesssim_{M, d, \overline{k}, \underline{k} } |\log(\eta \rho)|^{d+1}.
    \end{equation}
    and
    \begin{equation} 
    \label{eq:multivariate_regularized_hessian_covering}
        \log N \left( \eta, \mathscr{T}_\rho^{(2)} , \| \cdot \|_{F, \infty} \right) \lesssim_{M, d, \overline{k}, \underline{k} } |\log(\eta \rho)|^{d+1}.
    \end{equation}
\end{lemm}

We first collect bounds on the scores and their derivatives, then use moment matching to construct a cover.

\begin{lemm}[Lemma F.1 in~\citet{saha2020nonparametric} adapted for general $\Sigma$]
\label{lemm:auxiliary-result}
Let $G \in \mathcal{P}([-M,M]^d)$, $z \in  \RR^d$, and $\underline{k} I_d \preceq \Sigma \preceq \overline{k} I_d$. Then, the following inequalities hold:
\begin{equation} \label{eqn:saha-F1-gradient-norm}\left( \frac{\| \nabla f_G(z) \|}{f_G(z)} \right)^2 \leq \text{trace} \left( \Sigma^{-1} + \frac{\nabla^2 f_G(z)}{ f_G(z) } \right) \leq \underline{k}^{-1} \log \left( \frac{(2 \pi \underline{k} )^{-d}  }{f^2_G(z)} \right) 
\end{equation}
For $0 < \rho \leq (2 \pi \underline{k})^{-d/2} e^{-1/2}$,
\begin{equation}
\label{eq:saha-F2-gradient-norm}
\frac{ \| \nabla f_G(z) \| }{ f_G(z) \vee \rho} \leq \sqrt{\underline{k}^{-1} \log \frac{(2 \pi \underline{k} )^{-d}  }{\rho^2}}.
\end{equation}
For $0 < \rho \le (2\pi\underline{k})^{-d/2}e^{-1}$ and for every $z$ with $f_G(z) \neq \rho$, 
\begin{equation} \label{eq:score-gradient-trace}
\abs{\tr\p{\Sigma \, \nabla s_G^\rho(z) \Sigma }} \le \max\left\{d\overline{k}, \,\overline{k}^2 \underline{k}^{-1} \log\p{\frac{(2\pi\underline{k})^{-d}}{\rho^2}}\right\}, 
\end{equation}
where $s_G^\rho(z) = \frac{ \nabla f_G(z) }{ f_G(z) \vee \rho}$. 
\end{lemm}

\begin{proof}
Differentiating the marginal density in~\eqref{eq:multivariate_score_and_marginal} gives
\begin{align*}
\nabla f_G(z) &= - \Sigma^{-1}  \int (z - \mu)  \varphi_d(z - \mu; \Sigma) G(d\mu) \\
\nabla^2 f_G(z) &= - \Sigma^{-1} f_G(z) + \int \Sigma^{-1} (z - \mu) (z - \mu)^\top \Sigma^{-1} \varphi_d(z - \mu; \Sigma) G(d\mu). 
\end{align*}
Equivalently, using the posterior distribution of $\mu$ given $Z=z$,

\begin{align*}
\frac{\nabla f_G(z) } {f_G(z) } &= \Sigma^{-1} \bigl( \EE{\mu - Z \mid Z = z} \bigr) \\
\frac{\nabla^2 f_G(z) } {f_G(z) } &= - \Sigma^{-1} + \Sigma^{-1} \EE{ (Z - \mu) (Z - \mu)^\top \mid Z = z } \Sigma^{-1}.
\end{align*}
Decomposing the variance and noting that $\Var{Z - \mu \mid Z = z} \succeq 0$,
\begin{align}
\label{eqn:psd-relation}
\Sigma^{-1} + \frac{\nabla^2 f_G(z) } {f_G(z) } &= \Sigma^{-1}  \EE{ (Z - \mu) (Z - \mu)^\top \mid Z = z } \Sigma^{-1} \notag \\
&\succeq \Sigma^{-1}\left( \EE{ Z - \mu \mid Z = z} \right) \left( \EE{ Z - \mu \mid Z = z} \right)^\top\Sigma^{-1} \notag \\
&= \frac{\nabla f_G(z) } {f_G(z) } \left( \frac{\nabla f_G(z) } {f_G(z) } \right)^\top
\end{align}
Now, recall that $x \mapsto \exp(x/2)$ is convex on $\RR$, so that by Jensen's inequality $\EE{X} \le 2 \log\p{\EE{e^{X/2}}}$ for any real random variable $X$. 
\begin{align} \label{eqn:trace-relation}
\tr\p{\Sigma^{-1} + \frac{\nabla^2 f_G(z)}{f_G(z)}} & = \tr\p{\Sigma^{-1} \EE{(Z-\mu)(Z-\mu)^\top \mid Z=z} \Sigma^{-1}}\notag \\ 
& = \tr\p{\Sigma^{-1/2} \EE{(Z-\mu)(Z-\mu)^\top \mid Z=z} \Sigma^{-1/2} \Sigma^{-1}}\notag \\
& \le \lambda_{\max}(\Sigma^{-1}) \tr\p{\Sigma^{-1/2} \EE{(Z-\mu)(Z-\mu)^\top \mid Z=z} \Sigma^{-1/2}}\notag \\ 
& \le \underline{k}^{-1} \EE{(Z-\mu)^\top \Sigma^{-1}(Z-\mu) \mid Z=z}\notag \\
& \le 2 \underline{k}^{-1} \log \EE{e^{\frac{1}{2}(Z-\mu)^\top \Sigma^{-1}(Z-\mu)} \mid Z = z}\notag \\
& = 2 \underline{k}^{-1} \log\p{\frac{(2\pi)^{-d/2} \text{det}(\Sigma)^{-1/2}}{f_G(z)}}\notag \\
& \le 2 \underline{k}^{-1} \log\p{\frac{(2\pi\underline{k})^{-d/2}}{f_G(z)}}\notag \\
\implies \tr\p{\Sigma^{-1} + \frac{\nabla^2 f_G(z)}{f_G(z)}} & \le \underline{k}^{-1} \log\p{\frac{(2\pi\underline{k})^{-d}}{\p{f_G(z)}^2}}. 
\end{align}
Here, the first equality follows from~\eqref{eqn:psd-relation}, second one from $\tr(PQ) = \tr(QP)$, third inequality using $\tr\p{PQ} \le \lambda_{\max}(P) \tr\p{Q}$ for positive semi-definite $P$ and $Q$ with $\lambda_{\max}(P)$ being the maximum eigen-value of $P$, fourth inequality from $\Sigma \succeq \underline{k}I_d$ and $\EE{\tr(A)} = \tr\p{\EE{A}}$, fifth inequality follows from the Jensen's inequality mentioned above, sixth equality follows from the definition of the posterior density of $\mu \mid Z = z$ with the Gaussian kernel in the likelihood is exactly the inverse of the expression we integrate over, and finally the last inequality follows from $\text{det}(\Sigma) \ge \underline{k}^d$ since $\Sigma \succeq \underline{k}I_d$. Combining \eqref{eqn:psd-relation}  with \eqref{eqn:trace-relation} yields \eqref{eqn:saha-F1-gradient-norm}. The proof for the second result \eqref{eq:saha-F2-gradient-norm} is identical to that for F.2 in~\citet{saha2020nonparametric}. 

For the third result, we start with: 
\begin{equation*}
    \nabla s_G^\rho(z) = 
    \begin{cases}
        \frac{\nabla^2 f_G(z)}{f_G(z)} - \frac{\nabla f_G(z)\nabla f_G(z)^\top}{\p{f_G(z)}^2}, & \text{if } f_G(z) > \rho, \\
        \frac{\nabla^2 f_G(z)}{\rho}, & \text{if } f_G(z) < \rho. 
    \end{cases}
\end{equation*} 

\noindent \textbf{Case 1: } If $f_G(z) > \rho$, by~\eqref{eqn:psd-relation}, we can write: 
\begin{equation*}
    \begin{aligned}
        & \Sigma^{-1} + \frac{\nabla^2 f_G(z)}{f_G(z)} - \frac{\nabla f_G(z)\nabla f_G(z)^\top}{\p{f_G(z)}^2} \succeq 0 \\
        \implies & \Sigma^{-1} + \nabla s_G^\rho(z) \succeq 0. 
    \end{aligned}
\end{equation*}
Next, we see that, from~\eqref{eqn:saha-F1-gradient-norm}: 
\begin{equation*}
    \begin{aligned}
        & \tr\p{\Sigma^{-1} + \nabla s_G^\rho(z)} \le \tr\p{\Sigma^{-1} + \frac{\nabla^2 f_G(z)}{f_G(z)}} \le \underline{k}^{-1} \log\p{\frac{(2\pi\underline{k})^{-d}}{\p{f_G(z)}^2}} \\
        \implies & \tr\p{\Sigma^{-1} + \nabla s_G^\rho(z)} \le \underline{k}^{-1} \log\p{\frac{(2\pi\underline{k})^{-d}}{\rho^2}}, 
    \end{aligned}
\end{equation*}
since $\log\p{\frac{(2\pi\underline{k})^{-d}}{y^2}}$ is decreasing in $y$ and $f_G(z) > \rho$. Then, 
\begin{equation*}
    \begin{aligned}
        \tr\p{\Sigma \nabla s_G^\rho(z) \Sigma} & = \tr\p{\Sigma \p{\Sigma^{-1} + \nabla s_G^\rho(z)} \Sigma} - \tr\p{\Sigma} \\
        & \le \tr\p{\Sigma \p{\Sigma^{-1} + \nabla s_G^\rho(z)} \Sigma} \\
        & \le \overline{k}^2 \underline{k}^{-1} \log\p{\frac{(2\pi\underline{k})^{-d}}{\rho^2}}. 
    \end{aligned}
\end{equation*}
Here, the first inequality drops the term $-\tr\p{\Sigma} \le 0$, and the second uses $\tr\p{PQ} = \tr\p{QP}$, $\tr\p{PQ} \le \lambda_{\max}(P) \tr\p{Q}$ for positive semi-definite $P$ and $Q$, $\lambda_{\max}(\Sigma^2) \le \overline{k}^2$ since $\Sigma^2 \preceq \overline{k}^2 I_d$. 

\textbf{Case 2: } If $f_G(z) < \rho$, using~\eqref{eqn:psd-relation} again, 
\begin{equation*}
    \begin{aligned}
        \Sigma^{-1} + \frac{\nabla^2 f_G(z)}{\rho} & = \Sigma^{-1} + \frac{f_G(z)}{\rho} \p{\Sigma^{-1} + \frac{\nabla^2 f_G(z)}{f_G(z)} - \Sigma^{-1}} \\
        & = \p{1-\frac{f_G(z)}{\rho}} \Sigma^{-1} + \frac{f_G(z)}{\rho} \p{\Sigma^{-1} + \frac{\nabla^2 f_G(z)}{f_G(z)}} \\
        & \succeq 0, 
    \end{aligned}
\end{equation*}
which implies: 
\begin{equation*}
    \Sigma^{-1} + \nabla s_G^\rho(z) \succeq 0. 
\end{equation*}
Multiplying $\Sigma^{-1} + \nabla s_G^\rho(z)$ by $\Sigma$ on both sides and taking traces, and then using~\eqref{eqn:saha-F1-gradient-norm}, 
\begin{equation*}
    \begin{aligned}
        \tr\p{\Sigma \nabla s_G^\rho(z) \Sigma} + \tr\p{\Sigma} & = \p{1-\frac{f_G(z)}{\rho}} \tr\p{\Sigma} + \frac{f_G(z)}{\rho} \tr\p{\Sigma^2 \p{\Sigma^{-1} + \frac{\nabla^2 f_G(z)}{f_G(z)}}} \\
        \implies \tr\p{\Sigma \nabla s_G^\rho(z) \Sigma} & \le \p{-\frac{f_G(z)}{\rho}} \tr\p{\Sigma} + \frac{f_G(z)}{\rho} \overline{k}^2 \underline{k}^{-1} \log\p{\frac{(2\pi\underline{k})^{-d}}{\p{f_G(z)}^2}} \\
        & \le \frac{f_G(z)}{\rho} \overline{k}^2 \underline{k}^{-1} \log\p{\frac{(2\pi\underline{k})^{-d}}{\p{f_G(z)}^2}}\\
        & \le \overline{k}^2 \underline{k}^{-1} \log\p{\frac{(2\pi\underline{k})^{-d}}{\rho^2}}. 
    \end{aligned}
\end{equation*}
Here, the equality uses $\tr(PQ) = \tr(QP)$, while the first inequality uses the facts that $\tr\p{PQ} \le \lambda_{\max}(P) \tr\p{Q}$ for positive semi-definite $P$ and $Q$ and $\lambda_{\max}(\Sigma^2) \le \overline{k}^2$ since $\Sigma^2 \preceq \overline{k}^2 I_d$, second one drops $\p{-\frac{f_G(z)}{\rho}} \tr\p{\Sigma} \le 0$, and the final one uses that $y\log\p{\frac{(2\pi\underline{k})^{-d}}{y^2}}$ is increasing in $y$ for $0 < y \le (2\pi\underline{k})^{-d/2}e^{-1}$ and $f_G(z) < \rho$. 

Combining both cases, we have the following: 
\begin{equation*}
    \begin{aligned}
        \Sigma^{-1} + \nabla s_G^\rho(z) & \succeq 0, \\
        \tr\p{\Sigma \nabla s_G^\rho(z) \Sigma} & \le \overline{k}^2 \underline{k}^{-1} \log\p{\frac{(2\pi\underline{k})^{-d}}{\rho^2}}. 
    \end{aligned}
\end{equation*}
Now, we can write: 
\begin{equation*}
    \begin{aligned}
        & \tr\p{\Sigma \nabla s_G^\rho(z) \Sigma} = \tr\p{\Sigma \p{\Sigma^{-1} + \nabla s_G^\rho(z)} \Sigma} - \tr\p{\Sigma} \\
        \implies & \tr\p{\Sigma \nabla s_G^\rho(z) \Sigma} \ge -\tr\p{\Sigma} \ge -d\overline{k}, 
    \end{aligned}
\end{equation*}
using the facts that $PQP \succeq 0$ implying $\tr\p{PQP} \ge 0$, for positive semi-definite $P$ and $Q$, together with $\tr\p{\Sigma} \le d\overline{k}$. Thus,
\begin{equation*}
    \abs{\tr\p{\Sigma \nabla s_G^\rho(z) \Sigma}} \le \max\left\{d\overline{k}, \,\overline{k}^2 \underline{k}^{-1} \log\p{\frac{(2\pi\underline{k})^{-d}}{\rho^2}}\right\}.
\end{equation*}
Hence proved. 
\end{proof}

We now approximate the Hessians by moment matching. For $z\in\RR^d$ and $a>0$, write $B(z,a)$ and $\overset{\circ}{B}(z,a)$ for the closed and open Euclidean balls of radius $a$ centered at $z$, and $N(a,S)$ for the $a$-covering number of a set $S$. We also write
$$
\overset{\circ}{S^a}=(-M-a,M+a)^d,\qquad S^a=[-M-a,M+a]^d.
$$

\begin{lemm}[Moment matching lemma]
\label{lemm:moment-matching-mvn}
Let $G, G' \in \mathcal{P}([-M,M]^d)$. Fix $z \in \RR^d$, $a \geq \max\{1,\sqrt{2\overline{k}}\}$, and $c \geq 1$. Define $A \subseteq \RR^d$ such that $\displaystyle
\overset{\circ}{B}(z, a) \subseteq A \subseteq B(z, ca)$. Let $m \geq 1$ be an integer. For all $k_1, k_2, \ldots, k_d \in \{ 0, 1, \ldots, 2m+2\}$ suppose that $G$ and $G'$ have identical moments:
\begin{equation}
\label{eq:moment-matching-assumption}
\int_A \mu_1^{k_1} \ldots \mu_d^{k_d} G(d\mu) =\int_A \mu_1^{k_1} \ldots \mu_d^{k_d} G'(d\mu).
\end{equation}
Then, 
\begin{align*}
\| \nabla^2 f_G(z)&- \nabla^2 f_{G'}(z) \|_F \\
&\leq \frac{2\sqrt{d} \underline{k}^{-1}}{(2 \pi \underline{k})^{d/2}} \left(  e^{-a^2 / (2 \overline{k}) }  +   \frac{1}{\sqrt{2 \pi} }  (ca)^2 \left( \frac{  e (ca)^2 }{2 \underline{k} (m + 1)}\right)^{m+1}    \right)  \\
& \qquad \qquad + \frac{2d \underline{k}^{-2} }{(2 \pi \underline{k})^{d/2} } \left( a^2 e^{- a^2 / (2 \overline{k} ) } +  \frac{1}{\sqrt{2 \pi} }  (ca)^2 \left( \frac{  e (ca)^2 }{2 \underline{k} (m + 1)}\right)^{m+1}\right)
\end{align*}
\end{lemm}

\begin{proof}
We can bound the difference in the Hessians as follows:
\begin{align}
\label{eq:expand-hessian-difference}
\| \nabla^2 f_G(z) &- \nabla^2 f_{G'}(z) \|_F \nonumber \\
&\leq \sqrt{d}\underline{k}^{-1}|f_G(z)-f_{G'}(z)| \nonumber \\
&\quad + \left\|\int \Sigma^{-1}(z-\mu)(z-\mu)^\top\Sigma^{-1}\varphi_d(z-\mu;\Sigma)\bigl(G(d\mu)-G'(d\mu)\bigr)\right\|_F.
\end{align}
We split the second integral over $A^c$ and $A$. When integrating over $A^c$, we use
$$
\sup_{\mu\in A^c}\|z-\mu\|^2 e^{-\|z-\mu\|^2/(2\overline{k})}
\leq \sup_{u\geq a^2}u e^{-u/(2\overline{k})}
= a^2 e^{-a^2/(2\overline{k})}.
$$
The last step is because $u\mapsto u e^{-u/(2\overline{k})}$ is decreasing for $u\geq a^2\geq 2\overline{k}$. Since $|G-G'|(\RR^d)\leq2$, the contribution of $A^c$ is at most $2d\underline{k}^{-2}(2\pi\underline{k})^{-d/2}a^2e^{-a^2/(2\overline{k})}$.

Now, we bound the integral over $A$. Taylor's theorem gives
$$
e^u=\sum_{i=0}^m\frac{u^i}{i!}+\frac{u^{m+1}}{(m+1)!}e^v
$$
for some $v$ between $u$ and $0$. Take $u=-\tfrac12(z-\mu)^\top\Sigma^{-1}(z-\mu)$. The matrix $\Sigma^{-1}(z-\mu)(z-\mu)^\top\Sigma^{-1}\sum_{i=0}^m u^i/i!$ has polynomial entries of degree at most $2m+2$ in $\mu$, so its integral against $G-G'$ over $A$ vanishes by~\eqref{eq:moment-matching-assumption}. Therefore,
\begin{align*}
&\left\|\int_A \Sigma^{-1}(z-\mu)(z-\mu)^\top\Sigma^{-1}e^u\bigl(G(d\mu)-G'(d\mu)\bigr)\right\|_F \\
&\quad=\left\|\int_A \Sigma^{-1}(z-\mu)(z-\mu)^\top\Sigma^{-1}\left(e^u-\sum_{i=0}^m\frac{u^i}{i!}\right)\bigl(G(d\mu)-G'(d\mu)\bigr)\right\|_F \\
&\quad\leq 2d\underline{k}^{-2}(ca)^2\frac{\{(ca)^2/(2\underline{k})\}^{m+1}}{(m+1)!} \\
&\quad\leq \frac{2d\underline{k}^{-2}}{\sqrt{2\pi}}(ca)^2\left(\frac{e(ca)^2}{2\underline{k}(m+1)}\right)^{m+1}.
\end{align*}
Here we used $|u|\leq(ca)^2/(2\underline{k})$ on $A$, $e^v\leq1$, and Stirling's formula. The same argument without the matrix factor bounds $|f_G(z)-f_{G'}(z)|$, as in Lemma 3 of~\citet{soloff2024multivariate}. Using $(ca)^2\geq1$ and substituting both bounds into~\eqref{eq:expand-hessian-difference} yields the statement.
\end{proof}

\begin{lemm}[Approximation by mixtures of discrete mixing measures]
\label{lemm:approx-by-mixture-of-discrete-mixing-measures}
Let $G \in \mathcal{P}([-M,M]^d)$ and fix $a \geq \max\{1,\sqrt{2\overline{k}}\}$. Then there exists a discrete probability measure $G' \in \mathcal{P}([-M,M]^d)$ with at most $\ell$ atoms such that
\begin{equation*}
\left\| \nabla^2 f_G(z) - \nabla^2 f_{G'}(z) \right\|_F \leq \frac{2e^{-a^2 / (2 \overline{k} ) }}{ (2 \pi \underline{k})^{d/2}} \left[  \sqrt{d} \underline{k}^{-1}  \left(  1  +   \frac{1}{\sqrt{2 \pi} }  9a^2    \right) +  d \underline{k}^{-2}  \left( a^2 +  \frac{1}{\sqrt{2 \pi} }  9a^2 \right) \right]  ,
\end{equation*} uniformly over $z\in\RR^d$, where $\ell$ is defined as
\begin{equation}
\label{eq:number-of-atoms-in-mixing-measure}
\ell := (2 \left \lfloor{13.5 a^2 / \underline{k} } \right \rfloor  +3)^d N(a, S^a)  + 1.
\end{equation}
\end{lemm}

\begin{proof}
Let $\overset{\circ}{S^a} = (- M - a, M + a)^d$, and let $L = N \bigl(a, \overset{\circ}{S^a}\bigr)$ be the $a$-covering number of that set. Let $b_1, \ldots, b_L$ be its cover and $B_i = \left\{ u \in \RR^d : \| b_i - u \| \leq a \right\}$. Then, denote $E_1, \ldots, E_L$ as the standard disjointification of $B_i$, so $E_i := B_i \setminus \cup_{j < i} B_j$.

Let $m := \left \lfloor{13.5 a^2 / \underline{k} } \right \rfloor$ and $G'\in\mathcal P([-M,M]^d)$ be such that, for $i \in [L]$, and $k_1, k_2, \ldots, k_d \in \{ 0, 1, \ldots, 2m+2\}$ suppose that $G$ and $G'$ have identical moments:
\begin{equation}
\int_{E_i} \mu_1^{k_1} \ldots \mu_d^{k_d} G(d\mu) =\int_{E_i} \mu_1^{k_1} \ldots \mu_d^{k_d} G'(d\mu).
\end{equation}
For $z\in\RR^d$, let $F=\{i\in[L]:E_i\cap\overset{\circ}{B}(z,a)\neq\varnothing\}$ and set
$$
A=\overset{\circ}{B}(z,a)\cup\bigcup_{i\in F}E_i.
$$
The diameter of each $E_i$ is at most $2a$, so
$$
\overset{\circ}{B}(z,a)\subseteq A\subseteq\overset{\circ}{B}(z,3a).
$$
The part of $A$ outside $\bigcup_{i\in F}E_i$ has zero mass under both $G$ and $G'$, since the $E_i$ cover $[-M,M]^d$. Thus $G$ and $G'$ have matching moments over $A$.
Now, we apply Lemma~\ref{lemm:moment-matching-mvn} using $c=3$:
\begin{align*}
\| \nabla^2 f_G(z) &- \nabla^2 f_{G'}(z) \|_F  \\
&\leq \frac{2\sqrt{d} \underline{k}^{-1}}{(2 \pi \underline{k})^{d/2}} \left(  e^{-a^2 / (2 \overline{k}) }  +   \frac{1}{\sqrt{2 \pi} }  9a^2 \left( \frac{  9 e a^2 }{2 \underline{k} (m + 1)}\right)^{m+1}    \right)  \\
& \qquad \qquad + \frac{2d \underline{k}^{-2} }{(2 \pi \underline{k})^{d/2} } \left( a^2 e^{- a^2 / (2 \overline{k} ) } +  \frac{1}{\sqrt{2 \pi} }  9a^2 \left( \frac{  9e a^2 }{2 \underline{k} (m + 1)}\right)^{m+1}\right) \\
&\leq \frac{2\sqrt{d} \underline{k}^{-1} e^{-a^2 / (2 \overline{k} ) }}{(2 \pi \underline{k})^{d/2}} \left(  1  +   \frac{1}{\sqrt{2 \pi} }  9a^2    \right) +  \frac{2d \underline{k}^{-2} e^{-a^2 / (2 \overline{k} )  }} {(2 \pi \underline{k})^{d/2} } \left( a^2 +  \frac{1}{\sqrt{2 \pi} }  9a^2 \right)
\end{align*}
The second inequality is due to the definition of $m$ and that $(e/3)^6 \leq e^{-1/2}$,
\begin{equation*}
\left( \frac{9ea^2}{2 \underline{k} (m+1)} \right)^{m+1}  \leq \left(  \frac{e}{3 } \right)^{m+1} \leq \left( e^{-1/2} \right)^{ (m+1)/6} \leq e^{-a^2/( 2 \overline{k})}.
\end{equation*}
This gives the bound for the Frobenius distance between the two Hessians; what is left is to count the number of atoms $\ell$. Such a $G'$ with $\ell$ support points exists by Caratheodory's theorem. Consider
\begin{align*}
T =\left\{  \left( \int \mu_1^{k_1} \mu_2^{k_2} \ldots \mu_d^{k_d} \, 1\{ \mu \in E_i\} dG(\mu), \right. \right. & \biggl. \, \, \, k_1, \ldots, k_d \in \{0, 1, \ldots, 2m+2\}, \, \, \, i \in [L] \biggr)   \\
& \biggl. : G \in \mathcal{P}\p{[-M,M]^d} \biggr\}.
\end{align*}
Then $T$ is a convex subset of $\RR^{(2m+3)^d L}$. Also, $T$ is a convex hull of 
\begin{align*}
C = \left\{ \left(  \mu_1^{k_1} \mu_2^{k_2} \ldots \mu_d^{k_d} \, 1\{ \mu \in E_i\} , \right. \right. & \biggl. \, \, \, k_1, \ldots, k_d \in \{0, 1, \ldots, 2m+2\}, \, \, \, i \in [L] \biggr)   \\
& \biggl. : \mu \in [-M,M]^d \biggr\}.
\end{align*}
Every element of $T$ can be written as a convex combination of at most $(2m+3)^dL + 1$ elements in $C$. Thus we choose $G'$ from the discrete probability distribution where the support points are from elements in $C$ and the weight comes from the convex combination. Finally, $ L = N(a, \overset{\circ}{S^a}) \leq N(a, S^a)$ yielding \eqref{eq:number-of-atoms-in-mixing-measure}.
\end{proof}

\begin{lemm}[Covering number for family of Hessians, extension of Theorem 4.1 in~\citet{saha2020nonparametric}]
\label{lemm:second-covering-Hessian}
Let $\mathscr{H} := \cb{\nabla^2 f_G(\cdot) : G \in \mathcal{P}\left([-M,M]^d\right)}$. There exists a constant $C'_d$ depending only on $d$ such that the following holds. For $\displaystyle 0 < \eta  \leq \frac{C_d'}{(2 \pi)^{d/2} \sqrt{e}}$,
\begin{equation*}
\log N \bigl(\eta, \mathscr{H}, \| \cdot \|_{F,\infty} \bigr) \lesssim_{d, M,\underline{k},\overline{k}} \left| \log \left( \eta \right) \right|^{d+1}.
\end{equation*}
\end{lemm}
\begin{proof}
For $0<\alpha<1$ to be chosen below, let
\begin{equation} 
\label{eq:definition-a}
a := \sqrt{ 2 \overline{k} \log \left( \frac{1}{\alpha} \right) }.
\end{equation}
Let $z \in \RR^d$ and $G \in \mathcal{P} \bigl( [-M, M]^d \bigr)$. By Lemma~\ref{lemm:approx-by-mixture-of-discrete-mixing-measures}, there exists a $G'$ with $\ell$ support points such that
\begin{align}
\label{eq:G-G'-difference}
\left\| \nabla^2 f_G(z) - \nabla^2 f_{G'}(z) \right\|_F &\leq \frac{2e^{-a^2 / (2 \overline{k} ) }}{ (2 \pi \underline{k})^{d/2}} \left[  \sqrt{d} \underline{k}^{-1}  \left(  1  +   \frac{1}{\sqrt{2 \pi} }  9a^2    \right) +  d \underline{k}^{-2}  \left( a^2 +  \frac{1}{\sqrt{2 \pi} }  9a^2 \right) \right] \notag\\
&\leq \frac{a^2 \alpha}{ (2 \pi \underline{k})^{d/2}} \left[ 2\left(  \sqrt{d} \underline{k}^{-1} + d \underline{k}^{-2} \right) \left(  1  +   \frac{9}{\sqrt{2 \pi} }      \right)  \right]
\end{align} 
where $\ell$ is defined in \eqref{eq:number-of-atoms-in-mixing-measure}. In the second inequality, we use the definition of $a$ from~\eqref{eq:definition-a} and the fact that $a \geq 1$ (based on how we define $\alpha$ later on). Denote the support points of $G'$ as $a_1, \ldots, a_\ell$ and the corresponding weights as $w_1, \ldots, w_\ell$ so that we can write  $G' = \sum_{i = 1}^\ell w_i \delta_{a_i}$. Now, let $S^a = [-M-a, M+a]^d$, and $s_1, \ldots, s_D$ be its corresponding cover (with $D = N\left(\alpha, S^a \right)$ being the $\alpha$-covering number of $S^a$). Define $G'' = \sum_{i = 1}^\ell w_i \delta_{b_i}$, where $b_i \in \{s_1, \ldots, s_D \}$ and $\| a_i - b_i \| \leq \alpha$ for all $i \in [\ell]$. The distance between the Hessians induced by these two priors can be bounded:
\begin{align}
\label{eq:expand-G'-G''-difference}
\left\| \nabla^2 f_{G'}(z) \right. & \left. - \nabla^2 f_{G''}(z) \right\|_F \nonumber \\
&= \left\| \sum_{i = 1}^\ell w_i \left[ \Sigma^{-1} (z - a_i) (z - a_i)^\top \Sigma^{-1} \varphi_d(z - a_i; \Sigma)  - \Sigma^{-1} f_{G'}(z) \right]  \right. \nonumber  \\
&\qquad \qquad \left. - \sum_{i = 1}^\ell w_i \left[ \Sigma^{-1}  (z - b_i) (z - b_i)^\top \Sigma^{-1} \varphi_d(z - b_i; \Sigma) - \Sigma^{-1} f_{G''}(z) \right] \right\|_F  \nonumber  \\
& \leq \left\|  \sum_{i = 1}^\ell w_i \left( - \nabla \varphi_d(z - a_i; \Sigma) (z - a_i)^\top \Sigma^{-1}  + \nabla \varphi_d(z - b_i; \Sigma) (z - b_i)^\top \Sigma^{-1} \right) \right\|_F \nonumber \\
&\qquad \qquad + \| \Sigma^{-1} \|_F \left| f_{G'}(z) - f_{G''}(z) \right|, \nonumber \\
& \leq \sqrt{d} \underline{k}^{-1}  \left\| \sum_{i = 1}^\ell w_i \left(  - \nabla \varphi_d(z - a_i; \Sigma) (z - a_i)^\top   + \nabla \varphi_d(z - b_i; \Sigma) (z - b_i)^\top \right)  \right\|_F \\
&\qquad \qquad + \sqrt{d} \underline{k}^{-1} \left| f_{G'}(z) - f_{G''}(z) \right|. \nonumber 
\end{align}
The first inequality uses $\nabla\varphi_d(z;\Sigma)=-\Sigma^{-1}z\varphi_d(z;\Sigma)$; the second uses $\|\Sigma^{-1}\|_F\leq\sqrt d\,\underline{k}^{-1}$. We continue bounding~\eqref{eq:expand-G'-G''-difference}. For $u_i(t)=(1-t)(z-a_i)+t(z-b_i)$, the fundamental theorem of calculus gives
$$
\nabla\varphi_d(z-b_i;\Sigma)-\nabla\varphi_d(z-a_i;\Sigma)
=\int_0^1\nabla^2\varphi_d(u_i(t);\Sigma)(a_i-b_i)\,dt.
$$
Thus the first term is bounded using
\begin{align*}
&\left\|\nabla\varphi_d(z-b_i;\Sigma)(z-b_i)^\top-\nabla\varphi_d(z-a_i;\Sigma)(z-a_i)^\top\right\|_F \\
&\quad\leq\left\|\nabla\varphi_d(z-a_i;\Sigma)(a_i-b_i)^\top\right\|_F
+\int_0^1\left\|\nabla^2\varphi_d(u_i(t);\Sigma)(a_i-b_i)(z-b_i)^\top\right\|_F\,dt.
\end{align*}
The norm of those two terms can be bounded as follows.
\begin{align}
\label{eqn:bound-hess-term-1}
\left\| \nabla \varphi_d(z - a_i; \Sigma)  (a_i - b_i)^\top   \right\|_F 
&=  \|  a_i - b_i \| \, \| \nabla \varphi_d(z - a_i;\Sigma) \|  \notag \\
&\leq \alpha \left( 2 \pi \underline{k} \right)^{-d/2} e^{- \| z - a_i \|^2 / (2 \overline{k} ) } \left\| \Sigma^{-1} (z - a_i) \right\|_F \notag \\
&\leq \alpha \left( 2 \pi \underline{k} \right)^{-d/2}  \sqrt{d} \underline{k}^{-1} \sup_{u > 0} u e^{-u^2 / (2 \overline{k} ) } \notag \\
&=\alpha \left( 2 \pi \underline{k} \right)^{-d/2}  \sqrt{d \overline{k} / e} \, \, \underline{k}^{-1} 
\end{align}
To bound the Hessian term, let $u_i$ be any point on this segment. We know $\nabla^2 \varphi_d(u_i; \Sigma) =  - \Sigma^{-1} \varphi_d(u_i; \Sigma)  + \Sigma^{-1} u_i u_i^\top \Sigma^{-1} \varphi_d(u_i; \Sigma)$. Also, by definition of $u_i$, we can write $z - b_i = u_i + \zeta_i$ where $\| \zeta_i \| \leq \alpha$ so $\| z - b_i \| \leq \| u_i \| + \alpha$. Then,
\begin{align}
\label{eqn:bound-hess-term-2}
\left\| \right.   \nabla^2 \varphi_d& (u_i; \Sigma  )  (a_i    - b_i)  (z - b_i)^\top \left. \right\|_F \notag \\
&\leq \| a_i - b_i \| \, \| z - b_i \| \, \varphi_d(u_i;\Sigma)\| -\Sigma^{-1}+\Sigma^{-1}u_i u_i^\top\Sigma^{-1}\|_F  \notag \\
&\leq \alpha \left( \| u_i \| + \alpha \right) (2 \pi \underline{k} )^{-d/2} e^{- \| u_i \|^2 / (2 \overline{k})} \left( \frac{\sqrt{d}}{ \underline{k}} + \frac{d}{\underline{k}^{2}} \| u_i \|^2 \right) \notag\\
&\leq \alpha (2 \pi \underline{k})^{-d/2}  \left( \alpha \frac{\sqrt{d}}{\underline{k}} \sup_{t \geq 0} e^{-t^2 / (2 \overline{k})} + \frac{\sqrt{d}}{\underline{k}} \sup_{t \geq 0} t e^{-t^2 / (2 \overline{k})} \right. \notag \\ 
& \qquad \qquad \qquad \qquad \left. {}+ \alpha \frac{d}{\underline{k}^2} \sup_{t \geq 0} t^2  e^{-t^2 / (2 \overline{k})} +  \frac{d}{\underline{k}^2} \sup_{t \geq 0} t^3  e^{-t^2 / (2 \overline{k})} \right) \notag \\
&= \alpha (2 \pi \underline{k})^{-d/2}  \left( \alpha \frac{\sqrt{d}}{\underline{k}} + \frac{\sqrt{d}}{\underline{k}} \left(\frac{ \overline{k}}{e} \right)^{1/2} + \alpha \frac{d}{\underline{k}^2} \left( \frac{2 \overline{k}}{e} \right) +  \frac{d}{\underline{k}^2} \left( \frac{3 \overline{k}}{e} \right)^{3/2} \right) \notag \\
&\leq \alpha (2 \pi \underline{k})^{-d/2} c_{d, \underline{k}, \overline{k}}.
\end{align}
Later, we will choose $\alpha \leq 1$ so that's why we can define some $c_{d, \underline{k}, \overline{k}}$ independently of $\alpha$. Now we can bound the second term in \eqref{eq:expand-G'-G''-difference}. The argument is the same MVT argument as in the beginning of the proof of Theorem 4.1 in \cite{saha2020nonparametric}.
\begin{align}
\label{eqn:bound-hess-term-3}
| f_{G'}(z) - f_{G''}(z) | &\leq \sum_{i = 1}^\ell w_i | \varphi_d(z - a_i; \Sigma) - \varphi_d(z - b_i; \Sigma) | \notag \\
&\leq \sum_{i = 1}^\ell w_i \sup_{t \in \RR^d} \| \nabla \varphi_d(t; \Sigma) \| \alpha \leq \alpha (2 \pi \underline{k})^{-d/2} \sqrt{d \overline{k}/e} \, \underline{k}^{-1}
\end{align}
Combining \eqref{eqn:bound-hess-term-1}, \eqref{eqn:bound-hess-term-2}, and \eqref{eqn:bound-hess-term-3} back into \eqref{eq:expand-G'-G''-difference}, we achieve another intermediate bound:
\begin{equation}
\label{eq:G'-G''-difference}
\left\|  \nabla^2 f_{G'}(z) -  \nabla^2 f_{G''}(z) \right\|_F 
\leq \frac{\alpha }{( 2 \pi \underline{k} )^{d/2} } \sqrt{d} \underline{k}^{-1}  \left(2 \sqrt{d \overline{k}/e} \, \underline{k}^{-1}  + c_{d, \underline{k}, \overline{k}} \right) 
\end{equation}
Define $G''' = \sum_{i =1}^\ell w_i' \delta_{b_i}$ where $w_1', \ldots, w_\ell'$ is another probability vector to be defined later. The distance between the Hessians induced by $G''$ and $G'''$ is bounded by
\begin{align}
\label{eq:G''-G'''-difference}
\| \nabla^2 f_{G''}(z) - \nabla^2 f_{G'''}(z) \|_F  &\leq \sum_{i = 1}^\ell | w_i - w_i' |  \, \varphi_d(z - b_i; \Sigma) \, \| \Sigma^{-1} \left( - I_d +  (z - b_i) (z - b_i)^\top \Sigma^{-1} \right) \|_F \nonumber \\ 
&\leq (2 \pi \underline{k})^{-d/2}  \sum_{i = 1}^\ell | w_i - w_i' | \sup_{u\geq0} e^{-u/(2\overline{k})} \left(\sqrt{d}\underline{k}^{-1}+d\underline{k}^{-2}u\right) \nonumber \\
&\leq (2 \pi \underline{k})^{-d/2}  \sum_{i = 1}^\ell | w_i - w_i' | \left(  \sqrt{d} \underline{k}^{-1} + d \underline{k}^{-2} \left( \frac{2\overline{k}}{e} \right) \right).
\end{align}
Now, set $ \sum_{i = 1}^\ell | w_i - w_i' | \leq \alpha$ and apply triangle inequality on \eqref{eq:G-G'-difference}, \eqref{eq:G'-G''-difference}, and \eqref{eq:G''-G'''-difference}:
\begin{equation*}
\| \nabla^2 f_G(z) - \nabla^2 f_{G'''}(z) \|_F 
\leq  \frac{ \alpha}{ (2\pi \underline{k} )^{d/2}  } \biggl( a^2 C_1 + C_2 + C_3 \biggr).
\end{equation*}
The constants (dependent on $d$, $\underline{k}$ and $\overline{k}$) are defined as 
\begin{align*}
&C_1 = 2\left(  \sqrt{d} \underline{k}^{-1} + d \underline{k}^{-2} \right) \left(  1  +   \frac{9}{\sqrt{2 \pi} }      \right), &
C_2 &=  \sqrt{d} \underline{k}^{-1}  \left(2 \sqrt{d \overline{k}/e} \, \underline{k}^{-1}  + c_{d, \underline{k}, \overline{k}} \right), \\
&\text{and } C_3 =   \sqrt{d} \underline{k}^{-1} + d \underline{k}^{-2} \left( \frac{2\overline{k}}{e} \right).
\end{align*} 
Now, we define $\alpha$ and $\xi$ as
\begin{equation*}
\alpha = \frac{1}{3} \min \left\{ \frac{1}{C_1}, \, \frac{1}{C_2}, \, \frac{1}{C_3}, \, 1, \, e^{-1}, \, e^{- 1 / (2 \overline{k})} \right\} \min\{\xi,1\} \qquad \text{and} \qquad \xi = \eta (2 \pi \underline{k})^{d/2}.
\end{equation*}
As promised, $\alpha\leq e^{-1}\wedge e^{-1/(2\overline{k})}$, so $a\geq\max\{1,\sqrt{2\overline{k}}\}$. All together,
\begin{equation*}
\| \nabla^2 f_G(z) - \nabla^2 f_{G'''}(z) \|_F 
\leq \frac{ a^2 \alpha}{ (2\pi \underline{k} )^{d/2}  } \biggl( C_1 + C_2 + C_3 \biggr) \leq a^2 \eta.
\end{equation*}
Whats left is counting the number of ways to choose $G'''$. 
\begin{itemize}
\item For support points, there are at most $D^\ell$ ordered choices, allowing repeated support points. A grid on $S^a$ gives
$$
D=N(\alpha,S^a)\leq\left(1+\frac{2\sqrt d(M+a)}{\alpha}\right)^d
\leq C_{d,M,\underline{k},\overline{k}}\eta^{-3d/2},
$$
where the last inequality uses $a^2=2\overline{k}\log(1/\alpha)$ and $\alpha\asymp_{d,\underline{k},\overline{k}}\eta$.
\item For the weights, an $\ell_1$-covering at radius $v=\alpha$ of all $\ell$-dimensional probability vectors is given by
\begin{equation*}
\left(1 + \frac{2}{v} \right)^\ell.
\end{equation*}
This can be bounded:
\begin{equation*}
1 + \frac{2}{v} \leq \frac{3}{v} \leq C_{d,\underline{k},\overline{k}}\eta^{-1}.
\end{equation*}
\end{itemize}
So, for $\displaystyle A := D\left(1 + \frac{2}{v} \right)$, we have $A \leq C_{d,M,\underline{k},\overline{k}} \eta^{-3d/2 - 1}$. Replacing any covering center outside $\mathscr H$ by an element of $\mathscr H$ in its covering ball at most doubles the radius. Thus,
\begin{align*} 
\log N \left( 2a^2 \eta, \mathscr{H}, \| \cdot \|_{F,\infty} \right) &\leq \ell \log A \\
&\leq C_{d,M,\underline{k},\overline{k}}\ell \abs { \log \eta } \\
&\leq C_{d,M,\underline{k},\overline{k}}a^{2d} N(a, S^a) \abs{ \log \eta} \lesssim_{d, \underline{k}, \overline{k}, M} \abs{ \log \eta}^{d+1}.
\end{align*}
Here $N(a,S^a)$ is bounded by a constant depending only on $d$ and $M$, since $a\geq1$. Also, $a^2\asymp_{d,\underline{k},\overline{k}}|\log\eta|$. Applying the preceding bound with $\eta$ replaced by $\eta/(C|\log\eta|)$ for a sufficiently large $C$ gives
$$
\log N\left(\eta,\mathscr H,\|\cdot\|_{F,\infty}\right)
\lesssim_{d,\underline{k},\overline{k},M}|\log\eta|^{d+1}.
$$
\end{proof}

\begin{proof}[Proof of Lemma~\ref{lemm:multivariate-second-covering}]
We follow the proof strategy of Corollary D.1 in~\citet{saha2020nonparametric}. The construction in Lemma~\ref{lemm:second-covering-Hessian} also gives the same uniform covering bound for densities and gradients, using polynomial factors of degrees zero and one. We give the details for $\mathscr{T}_\rho^{(2)}$; the argument for $\mathscr{T}_\rho^{(1)}$ is the same with first derivatives. For all $G, H \in \mathcal{P}([-M, M]^d)$, we have that
\begin{align*}
\left\| \frac{\nabla^2 f_G(z) }{\rho \vee f_G(z) } \right. &- \left. \frac{\nabla^2 f_H(z) }{\rho \vee f_H(z)} \right\|_F = \left\| \frac{\nabla^2 f_G(z) }{\rho \vee f_G(z) } - \frac{\nabla^2 f_H(z) }{\rho \vee f_H(z)} \pm \frac{\nabla^2 f_G(z) }{\rho \vee f_H(z)} \right\|_F \\
&\leq \frac{\| \nabla^2 f_G(z) \|_F }{\rho \vee f_G(z) } \left(  \frac{| 
\, f_H(z) \vee \rho - f_G(z) \vee \rho \, | }{f_H(z) \vee \rho}  \right) + \frac{1}{\rho} \| \nabla^2 f_G(z) - \nabla^2 f_H(z) \|_F \\
&\leq  \left( \frac{1}{\rho} \sqrt{d} \underline{k}^{-1} + \frac{1}{\rho^2 } (2 \pi \underline{k})^{-d/2} d \underline{k}^{-2} \left( \frac{2 \overline{k} } {e} \right) \right) \left| f_G(z) - f_H(z) \right| \\
& \qquad \qquad \qquad + \frac{1}{\rho} \| \nabla^2 f_G(z) - \nabla^2 f_H(z) \|_F,
\end{align*} 
where the last step is because for all $z$,
\begin{align*}
\frac{\| \nabla^2 f_G(z) \|_F }{\rho \vee f_G(z) } &\leq \| \Sigma^{-1} \|_F \frac{f_G(z) } {f_G(z) \vee \rho} + \frac{1}{\rho \vee f_G(z) } (2 \pi \underline{k})^{-d/2} d \underline{k}^{-2}  \sup_{u\geq0}u e^{-u/(2\overline{k})} \\
&\leq \sqrt{d} \underline{k}^{-1} + \frac{1}{\rho } (2 \pi \underline{k})^{-d/2} d \underline{k}^{-2} \left( \frac{2 \overline{k}}{e} \right)
\end{align*}
Thus, to construct an $\eta$-level cover, it suffices to cover the two classes
\begin{equation*} 
\left\{ f_G(\cdot): G \in \mathcal{P}([-M, M]^d) \right\} \qquad \text{ and } \qquad \mathscr{H} = \left\{ \nabla^2 f_G(\cdot): G \in \mathcal{P}([-M, M]^d) \right\}.
\end{equation*} 
We cover the first class at level $\eta\left(\tfrac{1}{\rho}\sqrt d\,\underline{k}^{-1}+\tfrac{1}{\rho^2}(2\pi\underline{k})^{-d/2}d\underline{k}^{-2}(2\overline{k}/e)\right)^{-1}/4$, and the second class at level $\eta\rho/4$. Choose one prior from each nonempty intersection of covering balls for the two classes. Any other prior in the same intersection is within twice each covering radius, so the preceding bound is at most $\eta$.

For both classes, the logarithm of the number of covering balls is bounded by a constant depending on $d$, $M$, $\underline{k}$, and $\overline{k}$ times $|\log(\eta\rho)|^{d+1}$. Multiplying the two covering numbers yields~\eqref{eq:multivariate_regularized_hessian_covering}.
\end{proof}

\subsection{Empirical process bounds}

We first show that the regularized and unregularized scores agree on the sample with high probability. We then bound the localized complexities of the score and noise process classes and compare empirical and population squared errors.

\begin{lemm}[Multivariate version of Lemma \ref{lemm:compact_z_whp}]
\label{lemm:compact_z_whp_multivariate}
There exists $c>0$ (that depends on $M$ and $d$), such that for every $C \geq 2$ the following are true.\footnote{We have the specific choice, $C=10 \vee 10/\overline{k}$, in mind for what follows.} Let 
\begin{equation}\label{eq:tstar_multivariate} 
t^* := M + \sqrt{2C \overline{k} \log n },\quad \rho^* := \frac{(2 \pi \overline{k} )^{-d/2}}{2} \exp \left( \frac{-4dM^2}{\underline{k}} \right) n^{-2 C d \overline{k} / \underline{k}}
\end{equation}
With $\| Z_i \|_\infty = \sup_{j \in [d]} | Z_{ij} |$, define the event,
$$A_n := \cb{ \max_i \| Z_i \|_\infty \leq t^*}.$$
The following hold regarding the event $A_n$.
\begin{enumerate}
\item $\PP{A_n} \geq 1- 2d n^{1-C}.$
\item On $A_n$, $\| s_G(Z_i)\|_\infty \leq \tfrac{\sqrt{d}}{\underline{k}} \left( 2M + \sqrt{2C \overline{k} \log n} \right)$ for all $i=1,\dotsc,n$ for all $G \in \mathcal{P}(M)$.
\item On $A_n$ it holds that:
$$ \min_{i=1,\dotsc,n} \inf_{G \in \mathcal{P}(M)} f_G(Z_i) \geq c_{d, M, \underline{k}, \overline{k}} \, n^{-2 d C \overline{k} / \underline{k}}.$$
\item On $A_n$, for any $\rho \leq \rho^*$ it holds that $s_G(Z_i) = s_G^{\rho}(Z_i)$ for all $i=1,\dotsc,n$ and in particular also $s_{\star}(Z_i) = s_{\star}^{\rho}(Z_i)$.
\end{enumerate}

\end{lemm}

\begin{proof}
For the first part, we can write $Z_{ij} = \mu_{ij} + \sqrt{\Sigma_{jj}} \xi_{ij}$, where $\xi_{ij} \sim \mathrm{N}(0, 1)$. We can use that $\abs{\mu_{ij}} \leq M$ and $\Sigma_{jj} \leq \overline{k}$ to see that, for some $C$,
\begin{align*}
1 - \PP{A_n} &= \PP{ \max_{i \in [n]} \max_{j \in [d]} | Z_{ij} | \geq t^*} \\
&\leq \sum_{i=1}^n \sum_{j=1}^d \PP{ | Z_{ij} | > t^*} \\
&\leq \sum_{i=1}^n \sum_{j=1}^d \PP{ \abs{ \mu_{ij}} + \sqrt{\Sigma_{jj}} \abs{\xi_{ij}} > t^*} \\
&\leq \sum_{i=1}^n \sum_{j=1}^d \PP{ \abs{\xi_{ij}} >  \sqrt{2C \log n } }\\
&\leq 2d n^{1-C}.
\end{align*}

For the second part, we can use the fact that the prior $G$ is supported along each dimension on $[-M, M]$ and that $\| Z_i \|_\infty$ is bounded by $t^*$ on $A_n$.  
\begin{equation*}
\| s_G(Z_i) \|_\infty = \left\| \Sigma^{-1} \left( \EE[G]{ \mu_i \mid Z_i } - Z_i\right)  \right\|_\infty  \leq \frac{\sqrt{d}}{\underline{k}} \left( 2M + \sqrt{2C \overline{k} \log n} \right)
\end{equation*}

For the third part:
\begin{align*}
f_G(Z_i) &= \int (2 \pi)^{-d/2} | \Sigma|^{-1/2} \exp{ \left( - \frac{1}{2} (Z_i - \mu)^\top \Sigma^{-1} (Z_i - \mu)  \right) } dG(\mu) \\
&\geq (2 \pi \overline{k} )^{-d/2} \int \exp{ \left( - \frac{1}{2 \underline{k} } \| Z_i- \mu \|^2  \right) } dG(\mu) \\
&\geq (2 \pi \overline{k} )^{-d/2} \exp{ \left( - \frac{1}{2 \underline{k}} d \left( t^* + M\right)^2  \right) }  \\
&\geq (2 \pi \overline{k} )^{-d/2} \exp\left( \frac{d}{\underline{k}} \left( -4M^2 - 2 C \overline{k} \log n \right) \right) \\
&\geq (2 \pi \overline{k} )^{-d/2} \exp \left( \frac{-4dM^2}{\underline{k}} \right) n^{-2 d C \overline{k} / \underline{k}} = 2 \rho^*
\end{align*}
where $c$ and $C$ are dependent on $M$, $d$, and $\overline{k}$. For the fourth part, setting $\rho \leq \rho^*$ means that the marginal will always be at least the value of $\rho$ (on $A_n$), so the regularized score for some prior $G$ will be the same as the unregularized score; this is also true for $G_\star$. 
\end{proof}

Next, we want to control the complexity of two classes of functions: the centered regularized score class and the centered noise process class. We use the same localized population Rademacher complexity as defined \eqref{eq:localized_population_rademacher}.

Since we are in the multivariate setting, we will consider the scalar-valued, $j$-th component of the centered regularized scores:
\begin{equation}
\label{eq:multivariate_score_class_center_rho}
( \mathcal{S}_c^\rho)_j = \left\{ \bigl(  \regscore_G(\cdot) - \regscore_\star (\cdot) \bigr)_j : G \in \mathcal{P} \left( [ -M, M]^d \right) \right\}.
\end{equation}

\begin{lemm}[Multivariate version of Lemma \ref{lemm:complexity_regularized_score}] 
\label{lemm:complexity_regularized_score_multivariate}
For each $j \in [d]$, the following hold.
\begin{enumerate}
\item For $0< \rho < (2\pi)^{-d/2} e^{-1/2}$, we have that $$\sup_{h \in {( \mathcal{S}_c^\rho )_j} } \| h (\cdot ) \|_{\infty} = \sup_{h \in \StarHull\bigl( ( \mathcal{S}^{\rho}_c)_j \bigr)} \Norm{h(\cdot)}_{\infty}  \lesssim_{d, \underline{k}} \sqrt{ \abs{ \log \rho } }.$$
\item $\log N(\eta, (\mathcal{S}^{\rho}_c)_j, \Norm{\cdot}_{\infty}) \leq \log N \bigl(\eta, \StarHull\bigl( (\mathcal{S}^{\rho}_c)_j \bigr), \Norm{\cdot}_{\infty} \bigr) \lesssim_{d, M, \overline{k}, \underline{k} } \abs{ \log(\eta \rho)}^{d+1}$.
\item Let $\rho = \rho^*$ be as in~\eqref{eq:tstar_multivariate} (which depends on $C, M >0$).
Then it holds that:
\begin{align}
&\localizedpoprademacher{\delta}{\StarHull\bigl( (\mathcal{S}^{\rho^*}_c)_j \bigr)} \notag \\ 
&  \qquad \qquad \lesssim_{C,M, \overline{k}, \underline{k} } \frac{\delta}{\sqrt{n}}\p{ \log(1/\delta) + \log n}^{\frac{d+1}{2}} \,+\, \frac{\sqrt{\log n}}{n}\p{ \log(1/\delta) + \log n}^{d+1}.
\label{eq:rademacher_bound_center_score_entrywise}
\end{align}
\item Furthermore,
there exists a constant $C' = C'(C,M,d,\underline{k},\overline{k})>0$ such that $$\delta_d = C'\frac{\log^{\frac{d+2}{2}} n} {\sqrt{n}},$$
satisfies
$$
\localizedpoprademacher{\delta_d}{\StarHull \bigl( (\mathcal{S}^{\rho^*}_c)_j \bigr) }  \leq \frac{\delta_d^2}{ \sup_{h \in \StarHull\bigl( ( \mathcal{S}^{\rho}_c)_j \bigr)} \Norm{h(\cdot)}_{\infty}  }.
$$
\end{enumerate}
\end{lemm}
\begin{proof}
For the first part, we use Lemma \ref{lemm:auxiliary-result} to write
\begin{equation*}
\sup_{h \in {( \mathcal{S}_c^\rho )_j} } \| h (\cdot ) \|_{\infty} \leq 2 \sup_{G \in \mathcal{P} \left( [-M, M]^d \right) } \, \sup_{z \in \RR^d}  \left| \frac{\nabla f_G(z)_j }{f_G(z) \vee \rho }  \right|\leq 2\sqrt{  \underline{k}^{-1}  \log \frac{(2 \pi \underline{k} )^{-d}  }{\rho^2}} \lesssim_{d, \underline{k}} \sqrt{ \abs{ \log \rho } }.
\end{equation*}
For the second part, we follow the similar argument as in part two of Lemma \ref{lemm:complexity_regularized_score}. First, construct an $\eta/2$-level cover of $(\mathcal{S}_c^\rho)_j$ using $N$ elements, and call the centers $h_1, \ldots, h_N$. Second, for each $k \in [N]$, construct a $\eta / 2$-cover of the class $\displaystyle \{ \lambda h_k(\cdot) : \lambda \in [0, 1] \}$; equivalently, this is a $\displaystyle (c \sqrt{ \abs{ \log \rho } } )^{-1} \eta / 2$-level cover of $[0, 1]$). Finally, the logarithm of the total number of elements in this cover $\{ \lambda_{k \ell} h_k(\cdot) \}$  is 
\begin{equation*}
\log N \bigl(\eta, \StarHull \bigl( (\mathcal{S}_c^\rho )_j \bigr), \| \cdot \|_{ \infty} ) \lesssim | \log ( \eta \rho) |^{d+1} + \log \left( \frac{2 c \sqrt{ \abs{ \log \rho } } }{\eta} \right) \lesssim_{d, M, \overline{k}, \underline{k} } | \log ( \eta \rho) |^{d+1}
\end{equation*}
using \eqref{eq:multivariate_regularized_score_covering} of Lemma~\ref{lemm:multivariate-second-covering}. That result has the covering number in terms of the Frobenius norm of the entire score vector, which can be adapted (up to a factor of $d$) for the component-wise scores through Cauchy-Schwarz.

For the third part, we use a similar argument as Lemma \ref{lemm:complexity_regularized_score}. Take $\rho = \rho^*$ and let $h_1, \ldots, h_N$ be a proper $\eta$-level cover of the $\delta$-localized class, $\displaystyle \{ h \in \StarHull\bigl( ( \mathcal{S}^{\rho}_c)_j \bigr) : \| h \|_2 \leq \delta \}$. From the second part of this lemma, we know $\log N \lesssim | \log \eta \rho |^{d+1}$. For some $h$ in that class, let $h_{\hat{j}}$ be the closest element of the cover to it. Then,
\begin{equation*}
\left| \frac{1}{n} \sum_{i = 1}^n \varepsilon_i h(Z_i) \right|  = \left| \frac{1}{n} \sum_{i = 1}^n \varepsilon_i \bigl( h(Z_i) \pm h_{\hat{j}}(Z_i) \bigr) \right| \leq \eta + \frac{1}{n} \max_{j = 1}^N \left| \sum_{i = 1}^n \varepsilon_i h_j(Z_i) \right|.
\end{equation*}  
Each $\epsilon_i h_j(Z_i)$ is a sub-Gamma variable because it is bounded by $c \sqrt{ \abs{ \log \rho } }$ and has variance at most $\delta^2$ from localization. By tensorization of sub-Gamma random variables, their sum is also sub-Gamma, $\Gamma\bigl(n \delta^2, c \sqrt{\abs{ \log \rho } } \bigr)$. Corollary 2.6 of~\citet{boucheron2013concentration} gives that
\begin{equation*}
\EE { \max_{j = 1}^N \left| \sum_{i = 1}^n \varepsilon_i h_j(Z_i) \right|}  \leq  \delta \sqrt{2n \log(2N)} + c \sqrt{ \abs{ \log \rho } } \log(2N).
\end{equation*}
The bound in the lemma comes from plugging $\eta = \delta/\sqrt{n}$ into the combined results:
\begin{equation*}
\localizedpoprademacher{\delta}{\StarHull((\mathcal{S}^{\rho^*}_c)_j)} \lesssim \frac{\delta}{\sqrt{n}} + \frac{\delta}{\sqrt{n}} | \log (\eta \rho) | ^{\frac{d+1}{2}} + \frac{\sqrt{ \abs{ \log \rho } }}{n} | \log (\eta \rho) |^{d+1}.
\end{equation*} 

For part 4, take $\displaystyle \delta_d = C' \frac{ \log^{\frac{d+2}{2}} n }{\sqrt{n}}$ into the right hand side of \eqref{eq:rademacher_bound_center_score_entrywise}:
\begin{align*}
\localizedpoprademacher{\delta_d}{\StarHull \bigl( (\mathcal{S}^{\rho^*}_c)_j \bigr)} &\lesssim C' \frac{\log^\frac{d+2}{2} n}{n} \left( \log \frac{1}{\delta_d} + \log n \right)^\frac{d+1}{2} + \frac{ \sqrt{ \log n}  }{n} \left( \log \frac{1}{\delta_d} + \log n \right)^{d+1} \\
&\lesssim \left( C' \frac{\log^{\frac{2d+3}{2}} n}{n} + \frac{\log^{\frac{2d+3}{2}} n }{n} \right) \frac{\sqrt{ \abs{ \log \rho } }}{ \sup_{ h \in \StarHull{ \bigl( ( \mathcal{S}_c^{\rho^*})_j \bigr) }} \| h ( \cdot ) \|_\infty}  \\
&\leq \frac{\delta_d^2}{\sup_{h \in  \StarHull{ \bigl( ( \mathcal{S}_c^{\rho^*})_j \bigr) }} \Norm{h(\cdot)}_{\infty}  },
\end{align*}
where the last step uses that $\abs{\log\rho}$ grows as $O(\log n)$ and takes $C'$ sufficiently large.
\end{proof}

Now, we study the complexity of the multivariate centered regularized noise process. 
\begin{equation}
\label{eq:multivariate_noise_process}
\mathcal{M}^{\rho}_c = \cb{h_G(\cdot) = \tr \left( \Sigma \left( \regscore_{\star}(\cdot) \bigl( \regscore_{\star}(\cdot)-\regscore_{G}(\cdot) \bigr)^\top  + \nabla \bigl( \regscore_{\star}(\cdot)-\regscore_G(\cdot) \bigr)\right) \Sigma \right) \,:\, G \in \mathcal{P}\left([-M, M]^d \right)}.
\end{equation}

\begin{lemm}[Multivariate version of Lemma \ref{lemm:noise_process_complexity}]
\label{lemm:noise_process_complexity_multivariate}
We have the following results for the centered noise process class and its star hull. Constants may depend on $C,d,M,\underline{k},\overline{k}$.
\begin{enumerate}
\item For $0< \rho < (2\pi)^{-d/2} e^{-1/2}$, we have that $$\sup_{h \in \mathcal{M}^{\rho}_c} \Norm{h(\cdot)}_{\infty} = \sup_{h \in \StarHull(\mathcal{M}^{\rho}_c)} \Norm{h(\cdot)}_{\infty} \lesssim_{d, \underline{k}, \overline{k}} \abs{ \log \rho } .$$
\item Let $t^*$, $\rho^*$ be defined as in \eqref{eq:tstar_multivariate} (both of which depend on $C,M$, and $\rho^*$ depends on $d$). Further  define
$$\|f\|_{t^*, \infty} := \|f\|_{L_{\infty}([-t^*,t^*]^d)}.$$
then
$$\log N(\eta, \mathcal{M}^{\rho^*}_c, \Norm{\cdot}_{t^*,\infty}) \leq \log N(\eta, \StarHull(\mathcal{M}^{\rho^*}_c), \Norm{\cdot}_{t^*,\infty})  \lesssim_{d, M, \underline{k}, \overline{k}} | \log ( \eta \rho^*) |^{d+1}.$$
\item Continuing with $\rho^*$ as above, then it holds that:
\begin{align*}
\label{eq:rademacher-of-noise-process}
&\localizedpoprademacher{\delta}{\StarHull\left(\mathcal{M}^{\rho^*}_c\right)} \\
&\qquad \lesssim \frac{\delta}{\sqrt{n}} \left( \abs{\log\delta} + \log n \right)^{\frac{d+1}{2}} +  \frac{\log n}{n} \left( \abs{\log\delta} + \log n \right)^{d+1} + 2dn^{1-C} \abs{ \log \rho^*}.
\end{align*}

\item Take any $h \in \mathcal{M}_c^{\rho^*}$. Then:
$$ \abs{\EE[G_{\star}]{h(Z_i)}} \lesssim_{C,d,M,\underline{k},\overline{k}}  n^{\frac{1 - C}{2}}  \sqrt{| \log \rho^* | } \left( \sqrt{ | \log \rho^* | } +  M \right).$$
\end{enumerate}
\end{lemm}

\begin{proof}
For the first part:
\begin{align*}
\| h( \cdot ) \|_\infty &\leq \left\| \tr \left( \Sigma \left( \regscore_{\star}(\cdot) \bigl( \regscore_{\star}(\cdot)-\regscore_{G}(\cdot) \bigr)^\top  \right) \Sigma \right) \right\|_\infty + \left\| \tr \bigl( \Sigma \left( \nabla \bigl( \regscore_{\star}(\cdot)-\regscore_G(\cdot) \bigr)\right) \Sigma \bigr) \right\|_\infty \\
&\stackrel{(i)}{\leq} \overline{k}^2 \, \biggl\| \| \regscorestar_\star \| \, \| \regscorestar_\star - \regscorestar_G \| \biggr\|_\infty + 2\max\left\{d\overline{k}, \,\overline{k}^2 \underline{k}^{-1} \log\p{\frac{(2\pi\underline{k})^{-d}}{\rho^2}}\right\} \\
&\lesssim_{d, \underline{k}, \overline{k}} \abs{ \log \rho } 
\end{align*}
In $(i)$, we used the bounded eigenvalues of $\Sigma$ and Cauchy Schwarz for the first term; and we used \eqref{eq:score-gradient-trace} of Lemma~\ref{lemm:auxiliary-result} for the second term. 
Then, in $(ii)$, we bound the terms with the score with the inequality \eqref{eq:saha-F2-gradient-norm} from Lemma \ref{lemm:auxiliary-result} and part 1 of Lemma~\ref{lemm:complexity_regularized_score}.

For the second part: we take $\rho = \rho^*$ (so that $\rho < (2 \pi)^{-d/2} e^{-1/2}$) and $\eta > 0$. Note that for $\|z\|_\infty \leq t^*$, Lemma~\ref{lemm:compact_z_whp_multivariate} gives that
\begin{equation*}
\nabla  \regscore_G (z) = \frac{\nabla^2 f_G(z) }{ f_G(z) \vee \rho } -  \frac{\nabla f_G(z) \nabla f_G(z)^\top }{ (f_G(z) \vee \rho)^2 } = \frac{\nabla^2 f_G(z) }{ f_G(z) \vee \rho } - \regscore_G(z) \regscore_G(z)^\top.
\end{equation*}
For $G, \, H \in \mathcal{P}\left( [-M, M]^d \right)$, the corresponding noise processes are different at most by
\begin{align*}
&\| h_G - h_H  \|_{t^*, \infty} = \left\| \tr \left( \Sigma \left(  \regscore_{\star} \left( \regscore_H - \regscore_G \right)^\top + \nabla (\regscore_H - \regscore_G) \right) \Sigma \right)  \right\|_{t^*, \infty}
\\
&\leq \overline{k}^2 \, \biggl\|    \| \regscore_\star\| \, \| \regscore_H - \regscore_G \|   \biggr\|_{t^*, \infty} +  \left\| \tr \left( (\Sigma^2)^\top \left( \frac{\nabla^2 f_H }{ f_H \vee \rho } -  \frac{\nabla^2 f_G }{ f_G \vee \rho }\right)  \right) \right\|_{t^*, \infty}  \\
&\qquad \qquad  + \left\| \tr \biggl( \Sigma^2 \left( \regscore_G(z) \regscore_G(z)^\top - \regscore_H(z) \regscore_H(z)^\top\right)  \biggr) \right\|_{t^*, \infty}  \\
&\stackrel{(i)}{\lesssim}_{d, \underline{k}} \overline{k}^2 \sqrt{ \abs{\log \rho}} \sup_{\|z\|_\infty \leq t^*} \left\| \frac{\nabla f_H(z) }{ f_H(z) \vee \rho } - \frac{\nabla f_G(z) }{ f_G(z) \vee \rho } \right\|  + \overline{k}^2 \sup_{\|z\|_\infty \leq t^*} \left\| \frac{\nabla^2 f_G(z)  }{ f_G(z) \vee \rho } -  \frac{\nabla^2 f_H(z) }{ f_H(z) \vee \rho } \right\|_F \\
& \qquad \qquad + \overline{k}^2 \left\|  \bigl( \| \regscore_G \| + \| \regscore_H \| \bigr) \| \regscore_G - \regscore_H \| \right\|_{t^\star, \infty}   \\
&\stackrel{(ii)}{\lesssim}_{d, \underline{k}, \overline{k}} \left( \sqrt{ \abs{\log \rho}} + 1 \right) \sup_{\|z\|_\infty \leq t^*} \left\| \frac{\nabla f_H(z) }{ f_H(z) \vee \rho } - \frac{\nabla f_G(z) }{ f_G(z) \vee \rho } \right\| +  \sup_{\|z\|_\infty \leq t^*} \left\| \frac{\nabla^2 f_G(z)  }{ f_G(z) \vee \rho } -  \frac{\nabla^2 f_H(z) }{ f_H(z) \vee \rho } \right\|_F
\end{align*}
In $(i)$, we use the regularized score norm bound from Lemma~\ref{lemm:auxiliary-result}, \eqref{eq:saha-F2-gradient-norm}; $\| \Sigma^2 \|_F \leq \sqrt{d} \, \overline{k}^2$ and Cauchy Schwarz; and $\| aa^\top - bb^\top \|_F \leq ( \| a \| + \|b \|) \| a - b\|$. In $(ii)$, we apply~\eqref{eq:saha-F2-gradient-norm} to the third term. Now, it suffices to cover $\mathscr{T}_\rho^{(1)}$ (the space of regularized scores) and $\mathscr{T}_\rho^{(2)}$ (the space of regularized Hessians) at a smaller radius to account for the displayed factors. Applying Lemma~\ref{lemm:multivariate-second-covering}, we can bound the first two terms using a product of the two covers, whose log number of elements is a factor of $| \log (\eta \rho) |^{d+1}$ elements. Thus the logarithm of the covering number of $\mathcal{M}_c^\rho$ is of order $|\log(\eta\rho)|^{d+1}$. Passing to the star hull uses the same argument as in Lemma~\ref{lemm:complexity_regularized_score_multivariate}.

For the third part: we first present the analysis on $A_n$, taking $\rho=\rho^*$. For $\eta, \delta > 0$, let $N = N\left(\eta, \| \cdot \|_{t^*, \infty} \right)$ be the covering number of the delta-localized noise processes: $\left\{ h \in \StarHull(\mathcal{M}_c^\rho) : \| h \|_2 \leq \delta \right\}$, where $\| h \|_2^2 = \int h^2(z) f_{G_\star}(z) dz$. Let $h_1, \ldots, h_N$ be a proper cover of the class (so that each $h_j$ belongs to this localized class). From part two of this lemma, we know $\log N \lesssim_{d, M, \underline{k}, \overline{k}} | \log (\eta \rho) |^{d+1}$, since choosing centers in the localized class at most doubles the covering radius.

Take $h$ in the delta-localized noise process, and let $h_{\hat{j}}$ denote its nearest element from the cover. Then, on $A_n$,
\begin{align*}
\left|  \ind \left(A_n \right) \frac{1}{n} \sum_{i = 1}^n \varepsilon_i h(Z_i) \right|  &\leq \left|  \ind \left(A_n \right)  \frac{1}{n} \sum_{i = 1}^n \varepsilon_i \bigl( h(Z_i) - h_{\hat{j}}(Z_i) \bigr)  \right| + \left| \ind \left(A_n \right)   \frac{1}{n} \sum_{i = 1}^n \varepsilon_i  h_{\hat{j}}(Z_i) \right|    \\
&\leq \eta + \frac{1}{n} \max_{j = 1}^N \left| \ind \left(A_n \right)  \sum_{i = 1}^n  \varepsilon_i  h_j(Z_i) \right|.
\end{align*}
Observe the following two facts. By part one of this lemma, $| h_j(Z_i) | \lesssim_{d,\underline{k},\overline{k}} \abs{ \log \rho }$. And due to delta localization, $\Var{ \varepsilon_i h_j(Z_i) } = \EE{h_j(Z_i)^2} \leq \delta^2$. This means that $\varepsilon_i h_{j}(Z_i)$ is a sub-Gamma random variable, $\Gamma(\delta^2, c \abs{ \log \rho })$, where we use the definition of sub-Gamma in~\citet[Chapter 2.4]{boucheron2013concentration}. By tensorization, $\sum_{i=1}^n \varepsilon_i h_j(Z_i)$ is sub-Gamma, $\Gamma(n \delta^2 , c \abs{ \log \rho })$. Applying the maximal inequality in Corollary 2.6 of~\citet{boucheron2013concentration} we thus find that:
$$ 
\EE{\max_{j=1}^N \abs{\ind \left(A_n \right)  \sum_{i=1}^n \varepsilon_i h_{j}(Z_i)}} \leq \delta \sqrt{2 n \log(2N)} \, + \, C_{d,\underline{k},\overline{k}} \abs{ \log \rho } \log(2N).
$$
Taking the supremum over the localized class and using the covering number bound with $\eta = \delta / \sqrt{n}$ and $\rho = \rho^*$, we find that:
\begin{align*}
\EE{ \sup_h \left| \ind \left(A_n \right) \frac{1}{n} \sum_{i = 1}^n \varepsilon_i h(Z_i) \right|} &\lesssim \eta + \frac{\delta}{\sqrt{n}}\abs{ \log(\eta \rho )}^{\frac{d+1}{2}} + \frac{\abs{\log(\rho)} \abs{ \log(\eta \rho )}^{d+1}}{n} \\
&\lesssim \frac{\delta}{\sqrt{n}} \left( \abs{\log\delta} + \log n \right)^{\frac{d+1}{2}} +  \frac{\log n}{n} \left( \abs{\log\delta} + \log n \right)^{d+1}
\end{align*}
Now, we include the tail event $A_n^c$ in our analysis: 
\begin{align*}
\localizedpoprademacher{\delta}{\StarHull\left(\mathcal{M}^{\rho^*}_c\right)} 
&\leq \mathbb{E} \left[ \sup_{h \in \StarHull(\mathcal{M}^{\rho^*}_c), \, \| h \|_2 \leq \delta }  \left| \ind(A_n) \frac{1}{n} \sum_{i = 1}^n \varepsilon_i h(Z_i) \right|  \right] \\
&\qquad + \mathbb{E} \left[ \sup_{h \in \StarHull(\mathcal{M}^{\rho^*}_c),\, \| h \|_2 \leq \delta}  \left| \ind(A_n^c) \frac{1}{n} \sum_{i = 1}^n \varepsilon_i h(Z_i) \right|  \right] \\
&  \lesssim \frac{\delta}{\sqrt{n}} \left( \abs{\log\delta} + \log n \right)^{\frac{d+1}{2}} +  \frac{\log n}{n} \left( \abs{\log\delta} + \log n \right)^{d+1} + 2dn^{1 - C} \abs{ \log \rho^*}.
\end{align*}

For the fourth part, recall $\rho = \rho^*$, and start with this matrix equality from Stein's lemma:
\begin{equation*}
\EE[G_\star]{ s_\star(Z_i) \left( \regscore_\star(Z_i) - \regscore_G(Z_i) \right)^\top + \left(  \nabla \regscore_\star - \nabla \regscore_G \right) } = 0.
\end{equation*}
Then,
\begin{align*}
\bigl| \EE[G_\star]{\right. &h(Z_i) \left. } =  \left| \EE[G_\star] { \tr \left( \Sigma \left(  \left( \regscore_\star(Z_i) - s_\star(Z_i) \right) \left( \regscore_\star(Z_i) - \regscore_G(Z_i)  \right)^\top \right) \Sigma \right) }  \right|\\
&\stackrel{(i)}{\leq}  \left| \overline{k}^2 \,  \EE[G_\star] {  \left\|  \regscore_\star(Z_i) - s_\star(Z_i) \right\|  \,\left\| \regscore_\star(Z_i) - \regscore_G(Z_i)  \right\|  }  \right| \\
&=  \left| \overline{k}^2 \, \EE[G_\star] { \ind(A_n^c)  \left\|  \regscore_\star(Z_i) - s_\star(Z_i) \right\|  \,\left\| \regscore_\star(Z_i) - \regscore_G(Z_i)  \right\|  }  \right| \\
&\stackrel{(ii)}{\leq} \overline{k}^2 \,  \sqrt{ \PP[G_\star]{A_n^c} } \,  \left(    \EE[G_\star]{ \,   \left\| \regscore_\star(Z_i) - \regscore_G(Z_i) \right\|^2  \left\| \regscore_\star(Z_i) - s_\star(Z_i) \right\|^2 \,  } \right)^{1/2} \\
&\stackrel{(iii)}{\leq} 2\overline{k}^2 \,  \sqrt{ \PP[G_\star]{A_n^c} } \,   \sqrt{  \underline{k}^{-1} \log \frac{(2 \pi \underline{k} )^{-d}  }{\rho^2}} \, \left(  \EE[G_\star]{   \sum_{j =1}^d \bigl( \regscore_\star(Z_i) - s_\star(Z_i)\bigr)_j  ^2 } \right)^{1/2} \\
&\stackrel{(iv)}{\lesssim} \overline{k}^2 \,  \sqrt{ \PP[G_\star]{A_n^c} } \,   \sqrt{ \underline{k}^{-1} \log \frac{(2 \pi \underline{k} )^{-d}  }{\rho^2}}\, \biggl(  d \underline{k}^{-1}  \log \frac{(2 \pi \underline{k} )^{-d}  }{\rho^2}  + d\underline{k}^{-1}  \biggr)^{1/2} \\
&\stackrel{(v)}{\lesssim}_{d, M, \underline{k}, \overline{k}} n^{\frac{1 - C}{2}}  \sqrt{ \abs{ \log \rho } } \left( \sqrt{\abs{ \log \rho } } +  M \right).
\end{align*}
For each step, we used $(i)$ Cauchy-Schwarz; $(ii)$ Cauchy-Schwarz on the expectation; $(iii)$  inequality \eqref{eq:saha-F2-gradient-norm} from Lemma \ref{lemm:auxiliary-result};  $(iv)$ the fact that $(x + y)^2 \leq 2x^2 + 2y^2$, inequality \eqref{eq:saha-F2-gradient-norm}, and $\EE[G_\star]{\|s_\star(Z_i)\|^2} \leq d\underline{k}^{-1}$; and $(v)$ Part 1 of Lemma \ref{lemm:compact_z_whp_multivariate}.
\end{proof}

Before moving to the next lemma and Theorem \ref{theo:rate_homoscedastic_multivariate}, recall the standard notation in empirical process theory. Given a function $h: \RR \to \RR$, its empirical L2 norm can be denoted by
$$\hEE{h} := \frac{1}{n}\sum_{i=1}^n h(Z_i).$$

\begin{lemm}[Multivariate extension of Lemma \ref{lemm:l2_to_l2pn}]
\label{lemm:l2_to_l2pn_multivariate}
There is an event $A_n'$ with the following properties:
\begin{enumerate}
\item $\PP{A_n'} \geq  1- C d \exp(-c_1 \log^{d+1} n)$ for a positive constant $c_1$.
\item On the event $A_n \cap A_n'$, for $\rho^*$ and $M,C$ defined in Lemma~\ref{lemm:compact_z_whp_multivariate}, it holds for any $G \in \mathcal{P}\left([-M, M]^d \right)$:
$$\Dfishersigma{f_{\star}}{f_G} \leq 2\left(\frac{\overline{k}}{\underline{k}}\right)^2 \hEE{  \left\| \Sigma \left( \regscore_G  - \regscore_{\star} \right) \right\| ^2} + 8d^2M^2 n^{1-C} + C' \frac{\log^{d+2} n}{n},$$
where $C'$ depends on $C,M,d,\underline{k},\overline{k}$.
\end{enumerate}
\end{lemm}

\begin{proof}
For each $j \in [d]$, we can define the event $A_{n, j}'$ as the high probability event in~\citet[Theorem 14.1]{wainwright2019highdimensional} applied to the  class $\StarHull{ \bigl( ( \mathcal{S}_c^{\rho^*})_j \bigr) }$ studied in~Lemma~\ref{lemm:complexity_regularized_score_multivariate}:
\begin{equation*}
A_{n, j}' = \left\{ \left|  \EE{ h(Z_i)^2} - \hEE{ h^2 } \right| \leq \frac{1}{2} \EE{ h(Z_i)^2} + \frac{t^2}{2} \, : \,  h \in \StarHull{ \bigl( ( \mathcal{S}_c^{\rho^*})_j \bigr) } \right\}.
\end{equation*}
Taking $t = \delta_d$, Lemma \ref{lemm:complexity_regularized_score_multivariate} gives that $\PP{ A_{n, j}'} \geq 1 - C\exp(-c_1 \log^{d+1} n )$. Then, we can define $A_n' = \bigcap_{j = 1}^d A_{n, j}'$. Applying Bonferroni's inequality on $A_n'$, we derive the bound in part one:
\begin{equation*}
\PP{ A_{n}' } \geq 1 - \sum_{j = 1}^d ( 1 - \PP{ A_{n, j}' }) \geq 1 - Cd  \exp(-c_1 \log^{d+1} n ).
\end{equation*} 

For part two, we begin by applying the same arithmetic as in Lemma \ref{lemm:l2_to_l2pn} to see that on $A_n'$,
$$
\EE[G_{\star}]{h(Z_i)^2} \leq t^2 + 2\hEE{h^2}.
$$
Notice that the above holds for any $h \in \StarHull\bigl( (\mathcal{S}_c^{\rho^*})_j \bigr)$, so in particular it also holds for $h_G \in (\mathcal{S}_c^{\rho^*})_j$, where the subscript $G$ corresponds to the indexing $G \in \mathcal{P}\left([-M, M]^d \right)$ in the definition~\eqref{eq:multivariate_score_class_center_rho} and $\rho^*$ is as in \eqref{eq:tstar_multivariate}. To make the notation more explicit, the above bound from Theorem 14.1 of Wainwright is for $\EE[G_{\star}]{h_G(Z_i)^2} = \EE[G_{\star}]{ \bigl( \regscore_G(Z_i) - \regscore_{\star}(Z_i)\bigr)^2_j}$. Now we can bound the multivariate Fisher divergence (taking $\rho = \rho^*$). The expectations involving $A_n$ below are over the full sample.
\begin{align*}
\Dfishersigma{f_{\star}}{f_G} &= \EE[G_\star]{  \left\| \Sigma \bigl( \score_{\star}(Z_i) - \score_G(Z_i) \bigr) \right\|^2} \\ 
&= \EE[G_\star]{  \left\| \Sigma \bigl(   \score_{\star}(Z_i) - \score_G(Z_i) \bigr) \right\|^2 \ind(A_n)}  \, + \, \EE[G_\star]{  \left\| \Sigma \bigl(  \score_{\star}(Z_i) - \score_G(Z_i)  \bigr) \right\|^2 \ind(A_n^c) }  \\ 
&= \EE[G_\star]{  \left\| \Sigma \bigl( \regscore_\star(Z_i) - \regscore_{G}(Z_i)\bigr) \right\|^2 \ind(A_n)}  \\
&\qquad \qquad + \EE[G_\star]{  \bigl\|    \EE[G_\star]{u_i \mid Z_i}  - \EE[G]{u_i \mid Z_i}  \bigr\|^2 \ind(A_n^c) }  \\
&\leq \EE[G_\star]{  \left\| \Sigma \bigl( \regscore_\star(Z_i) - \regscore_{G}(Z_i)\bigr) \right\|^2 \ind(A_n) }  \, + \, 4 d M^2 \PP[G_\star]{A_n^c }  \\
&\stackrel{(i)}{\leq} \overline{k}^2 \sum_{j = 1}^d \EE[G_\star]{ \bigl(\regscore_\star(Z_i) - \regscore_G(Z_i) \bigr)_j^2 } + 4dM^2\PP[G_\star]{A_n^c} \\
&\leq \overline{k}^2 \left( dt^2 + 2 \hEE{ \| \regscore_\star - \regscore_G\|^2 } \right) + 8d^2M^2 n^{1-C} \\
&\leq 2\left(\frac{\overline{k}}{\underline{k}}\right)^2 \hEE{\|\Sigma(\regscore_\star-\regscore_G)\|^2} + d\overline{k}^2t^2 + 8d^2M^2 n^{1-C}.
\end{align*}
Taking $t = \delta_d$, we arrive at the lemma bound.
\end{proof}

\subsection{Proof of Theorem~\ref{theo:rate_homoscedastic_multivariate}}

\begin{proof}[Proof of Theorem~\ref{theo:rate_homoscedastic_multivariate}]
Take $\delta = n^{-2}$ and $\rho^*$ as in~\eqref{eq:tstar_multivariate}, with $C$ chosen sufficiently large, so that $\abs{\log\rho^*}\asymp\log n$. Throughout this proof constants may depend on $d,M,\underline{k},\overline{k}$. Our analysis will be on the high probability event $A_n \cap A_n'$, for which we showed in Lemma~\ref{lemm:compact_z_whp_multivariate} that regularizing by $\rho^*$ is equivalent to not regularizing, so the score $\score$ are interchangeable on the sample $Z_1, \ldots, Z_n$. Now we set up some notations. We denote an element of the centered regularized noise process as $h_G( \cdot ) \in \mathcal{M}^{\rho^*}_c$:
\begin{equation*}
h_G(z) = \tr \left( \Sigma \left( \regscorestar_{\star}(z) \bigl( \regscorestar_{\star}(z)-\regscorestar_{G}(z) \bigr)^\top  + \nabla \bigl( \regscorestar_{\star}(z)-\regscorestar_G(z) \bigr)\right) \Sigma \right).
\end{equation*}
For some positive radius $r$, consider the class
\begin{equation*}
\mathscr{G}_r = \left\{ G \in \mathcal{P}\left([-M, M]^d\right) : \hEE{ \left\| \Sigma\left( \regscorestar_G - \regscorestar_\star \right) \right\|^2 } \leq r^2 \right\}.
\end{equation*}
Denote the multivariate random complexity function and its post-fixed point respectively as
\begin{equation*}
\widehat{W}_\Sigma(r) = \sup_{G \in \mathscr{G}_r } 2 \hEE{\, h_G\, } \qquad \text{and} \qquad \hat{r} = \sup \{ r \geq 0 : r^2 \leq \widehat{W}_\Sigma(r) \}.
\end{equation*}
On $A_n$, compactness of $\mathcal{P}([-M,M]^d)$ and continuity of the sample scores and their derivatives imply that the supremum defining $\hat r$ is attained, so $\widehat{W}_\Sigma(\hat r)\geq\hat r^2$. For the remainder of the argument we require $r^2 \gtrsim \log^{d+2} n/n$ and only consider priors belonging to $\mathscr{G}_r$. Why $\mathscr{G}_r$? In Lemma~\ref{lemm:l2_to_l2pn_multivariate}, $\Dfishersigma{f_{G_\star}}{f_{\widehat{G}}}$ is upper bounded by some terms including the empirical squared error (weighted by $\Sigma$), which is then upper bounded by $\hat{r}^2$ via the multivariate deterministic inequality in Lemma~\ref{lemm:deterministic_inequality_multivariate}. So an upper bound on $\hat{r}^2$ yields an upper bound on the Fisher divergence. We apply a recursive localization argument to deterministic radii $r_0,r_1,\ldots,r_I$, successively bounding $\hat r$ by each radius. At each iteration, we apply concentration of measure to bound $\widehat{W}_\Sigma$.

We first present the analysis for a single fixed $r$ before working through the recursion. We claim that on the high probability event $A_n \cap A_n'$, $\mathscr{G}_r \subseteq \mathscr{G}_r'$, where 
\begin{equation*}
\mathscr{G}_r' = \left\{ G \in \mathcal{P}\left([-M, M]^d\right) : \EE{ \, h_G^2(z) \, } \leq C'r^2 \, \log^2 n  \right\}
\end{equation*}
for some large enough $C' > 0$. This is because
\begin{align}
\label{eqn:expand-noise-process-multivariate}
\EE{ \, h_G^2(z) \, } &\lesssim \EE{\tr^2 \left( \Sigma  \regscorestar_{\star}(z) \bigl( \regscorestar_{\star}(z)-\regscorestar_{G}(z) \bigr)^\top \Sigma \right) } + \EE{ \tr^2 \left( \Sigma \nabla \bigl( \regscorestar_{\star} - \regscorestar_G \bigr) (z) \Sigma \right)  }  \\
&\lesssim_{d, \overline{k}, \underline{k}, M} r^2 \,  \log^2 n. \notag
\end{align}
Each term~\eqref{eqn:expand-noise-process-multivariate} is separated into a center and tail term. Denote the deterministic cube $A = [-t^*, t^*]^d$. Then, for some $Z \sim f_{G_\star}$, the probability $\PP{Z \in A^c} \leq 2dn^{-C}$ using the same proof idea as part 1 of Lemma~\ref{lemm:compact_z_whp_multivariate}. Lets expand the first term:
\begin{align*}
&\EE{\tr^2 \left( \Sigma  \regscorestar_{\star}(z) \bigl( \regscorestar_{\star}(z)-\regscorestar_{G}(z) \bigr)^\top \Sigma \right) } \\
&\qquad   \stackrel{(i)}{\leq}\left( d \overline{k}^4 \right)   \EE{ \| \regscorestar_\star(z) \|^2 \, \|  \regscorestar_{\star}(z)-\regscorestar_{G}(z) \|^2 } \\
&\qquad \stackrel{(ii)}{\lesssim}_{d, \underline{k}, \overline{k}} \log n \, \EE{ \ind\{ z \in A \} \| \Sigma (\regscorestar_{\star}(z)-\regscorestar_{G}(z)) \|^2 + \ind\{ z \in A^c \} \| \Sigma (\regscorestar_{\star}(z)-\regscorestar_{G}(z)) \|^2 } \\
&\qquad \stackrel{(iii)}{\lesssim}_{d, \underline{k}, \overline{k}, M} \log n \left( r^2 + n^{1-C} + \frac{\log^{d+2}n}{n} + (2dn^{-C}) \log n \right) \\
&\qquad \stackrel{(iv)}{\lesssim}_{d, \underline{k}, \overline{k}} r^2 \log n
\end{align*} 
Inequality $(i)$ uses Cauchy-Schwarz on trace and the fact $\| \Sigma^\top \Sigma \|_F \leq \sqrt{d}
\overline{k}^2$. In $(ii)$, we apply Lemma~\ref{lemm:auxiliary-result} coordinate-wise to $\| \regscorestar_\star(z) \|^2 $ (and $\abs{\log \rho^*} \lesssim \log n$); also, $\| \regscorestar_{\star}(z)-\regscorestar_{G}(z) \| \leq \tfrac{1}{\underline{k}} \| \Sigma (\regscorestar_{\star}(z)-\regscorestar_{G}(z)) \|$. In $(iii)$, we use part 2 of Lemma~\ref{lemm:l2_to_l2pn_multivariate}, the fact that $G \in \mathscr{G}_r$, and the radius floor combined to see $\Dfishersigma{f_{\star}}{f_G} \lesssim r^2 + \log^{d+2} n/n \lesssim r^2$. In $(iv)$, the remaining terms are absorbed into the radius floor $r^2 \gtrsim \log^{d+2}n/n$.

As for the second term in~\eqref{eqn:expand-noise-process-multivariate}:
\begin{align*}
&\mathbb{E} \biggl[ \tr^2  \left( \Sigma \nabla \bigl( \regscorestar_{\star} - \regscorestar_G \bigr) (z) \Sigma \right) \biggr]\\
&\qquad \stackrel{(i)}{=} \EE{ \ind\{z \in A\} \tr^2 \left( \Sigma \, \nabla \bigl( \score_{\star} - \score_G \bigr) (z) \,  \Sigma \right)  } + \EE{ \ind\{z \in A^c\} \tr^2 \left( \Sigma \, \nabla \bigl( \regscorestar_{\star} - \regscorestar_G \bigr) (z) \,  \Sigma \right)  } \\
& \qquad \stackrel{(ii)}{\leq} d \sum_{j = 1}^d  \EE{  \,    \left( \Sigma_j^\top \, \nabla \left( \score_\star(z) - \score_G(z)  \right) \, \Sigma_j \right)^2} + C_{d, \overline{k}, \underline{k}} ( 1 + \abs{\log \rho^*} )^2 (2dn^{-C}) \\
& \qquad \stackrel{(iii)}{\leq} d^2\, C_{\underline{k}, \overline{k}, d, M} \left(1 + M^2 \right) \left[ 1 + \log^2_+ \left( \frac{C_0}{\Dfishersigma{f_{G_{\star}}}{f_{G}} } \right) \right] \Dfishersigma{f_{G_{\star}}}{f_{G}}  \\
& \qquad \qquad + C_{d, \overline{k}, \underline{k}} ( 1 + \abs{\log \rho^*} )^2 (2dn^{-C}) \\
&\qquad \stackrel{(iv)}{\lesssim}_{\underline{k}, \overline{k}, d, M}  \left[ 1 + \log^2_+ \left( \frac{C_0}{r^2} \right) \right] r^2 + \log^4 n (2dn^{-C}) \\
&\qquad \stackrel{(v)}{\lesssim}_{\underline{k}, \overline{k}, d, M}  r^2  \log^2 n
\end{align*}
In $(i)$, recall that regularization is inactive on the cube $A$, as shown in the proof of Lemma~\ref{lemm:compact_z_whp_multivariate}. In $(ii)$ for the first term, we use Cauchy-Schwarz on $\left( \sum_{j = 1}^d  \Sigma_j^\top \, \nabla \left( \score_\star(z) - \score_G(z)  \right) \, \Sigma_j \right)^2$. In $(ii)$ for the second term, we use Lemma~\ref{lemm:auxiliary-result}, \eqref{eq:score-gradient-trace}. Lemma~\ref{lemm:multivariate-apply-gaussian-strip-interpolation} is used in $(iii)$. In $(iv)$, we use the fact that $t \mapsto [1 + \log^2_+(C_0 /t) ] t$ is non-decreasing, combined with Lemma~\ref{lemm:l2_to_l2pn_multivariate}: on $A_n'$, for $G \in \mathscr{G}_r$,
\begin{equation*}
\Dfishersigma{f_{\star}}{f_G}  \lesssim_{d, \overline{k}, \underline{k}, M} \hEE{ \left\| \Sigma \left( \regscorestar_G - \regscorestar_\star\right) \right\|^2} + \frac{\log^{d+2} n}{n} \lesssim_{d, \overline{k}, \underline{k}, M} r^2.
\end{equation*} 
Finally in $(v)$, we use that $C_0/r^2 \lesssim C_0 n / \log^{d+2} n \lesssim n$.

Now, bound the complexity functional on $A_n\cap A_n'$. For each fixed $r$ above the radius floor, the following holds outside an event of probability at most $\tau$. We apply concentration to the deterministic class $\mathscr{G}'_r$.
\begin{align}
\label{eqn:single_r_W_sigma_bound}
\widehat{W}_\Sigma(r) &\lesssim \sup_{G \in \mathscr{G}_r} \EE{ h_G} + \sup_{G \in \mathscr{G}_r} \biggl( \hEE{h_G} - \EE{h_G(z) } \biggr) \notag \\
&\stackrel{(i)}{\lesssim}_{M, d, \underline{k}, \overline{k}} n^{\frac{1 - C}{2}}  \abs{\log \rho^*} + \sup_{G \in \mathscr{G}'_r} \biggl( \hEE{h_G} - \EE{h_G(z) } \biggr) \notag\\
&\stackrel{(ii)}{\lesssim}_{M, d, \underline{k}, \overline{k}} n^{\frac{1 - C}{2}}  \abs{\log \rho^*} + \notag\\
& \qquad  \qquad \left(  \mathbb{E} \left[ \sup_{G \in \mathscr{G}'_r} \abs{ \hEE{h_G} - \EE{h_G(z) } } \right]  + r \log n \sqrt{ \frac{\log(2 / \tau)}{n} } + \frac{ \abs{ \log \rho^* } \log(2 / \tau) }{n}  \right) \notag \\
&\stackrel{(iii)}{\lesssim}_{M, d, \underline{k}, \overline{k}} n^{\frac{1 - C}{2}}  \abs{\log \rho^*} +\notag \\
& \qquad  \qquad \left( 2 \mathbb{E} \left[ \sup_{G \in \mathscr{G}'_r} \abs{ \frac{1}{n} \sum_{i = 1}^n \varepsilon_i h_G(z_i)} \right] + r \log n \sqrt{ \frac{\log(2 / \tau)}{n} } + \frac{ \abs{ \log \rho^* } \log(2 / \tau) }{n}  \right) \notag \\
&\stackrel{(iv)}{\lesssim}_{M, d, \underline{k}, \overline{k}} n^{\frac{1 - C}{2}}  \abs{\log \rho^*} + \notag \\
& \qquad  \qquad \left( \frac{r \, \log^\frac{d+3}{2} n }{\sqrt{n}} + r \, \log n \sqrt{ \frac{\log(2 / \tau)}{n} } + \frac{ \abs{ \log \rho^* } \log(2 / \tau) }{n}  \right)  
\end{align}
where $(i)$ is from Lemma~\ref{lemm:noise_process_complexity_multivariate} part 4 for the first term and $\mathscr{G}_r \subseteq \mathscr{G}_r'$ for the second term. Then, $(ii)$ is from Theorem~\ref{theo:talagrand} (taking $f = \frac{1}{n} h_G$ and $\sigma = \sqrt{C'} r \log n /\sqrt{n}$ with boundedness from Lemma~\ref{lemm:noise_process_complexity_multivariate}); $(iii)$ is symmetrization. $(iv)$ is upper bounded by the localized Rademacher complexity in part 3 of Lemma~\ref{lemm:noise_process_complexity_multivariate} with $\delta = \sqrt{C'} r \, \log n$ (for the radii $r\lesssim\sqrt{\log n}$ used below):
\begin{align*}
\localizedpoprademacher{\sqrt{C'} r \, \log n}{\StarHull(\mathcal{M}^{\rho^*}_c)} &\lesssim \frac{r \log n}{\sqrt{n}} \log^{\frac{d+1}{2}} n +  \frac{\log n}{n} \log^{d+1} n + 2dn^{1-C} \abs{ \log \rho^*} \\
&\lesssim \frac{r \, \log^\frac{d+3}{2} n}{\sqrt{n}}  +  \frac{r \, \log^\frac{d+2}{2} n}{\sqrt{n} }.
\end{align*} 
The second term is absorbed using $r^2\gtrsim\log^{d+2}n/n$, and the tail term is negligible for large $C$. This concludes the analysis for a fixed radius. We will use the upper bound on the complexity function, \eqref{eqn:single_r_W_sigma_bound}, throughout the following recursion argument. We define a deterministic sequence of radii for constants $K > 1$ and $\tau > 0$ to be chosen later. Let $r_0=C'\sqrt{|\log\rho^*|}$, with $C'$ sufficiently large that $\widehat W_\Sigma(r)\leq r_0^2$ for every $r\geq0$, by part 1 of Lemma~\ref{lemm:noise_process_complexity_multivariate}. We take the radius floor to be $K\log^{d+2}n/n$, with $K$ large enough for~\eqref{eqn:single_r_W_sigma_bound} to apply above this floor. Define
\begin{align*}
r_{i + 1}^2 &= K \left( n^{\frac{1 - C}{2}}  \abs{\log \rho^*} + \frac{r_i \, \log^\frac{d+3}{2} n }{\sqrt{n}} + r_i \, \log n \sqrt{ \frac{\log(2 / \tau)}{n} } + \frac{ \abs{ \log \rho^*} \log(2 / \tau) }{n}  \right) \\
&\leq \frac{r^2_i}{4} + \frac{K^2}{n} \left( \log^\frac{d+3}{2} n  +\log n \sqrt{\log(2 / \tau)} \right)^2  + K \left(  n^{\frac{1 - C}{2}} \abs{\log \rho^*} + \frac{ \abs{ \log \rho^*} \log(2 / \tau) }{n} \right) 
\end{align*}
for $i \geq 0$, where the preceding inequality is from AM-GM. So for $i \geq 0$, at least one of the following holds:
\begin{align}
r^2_{i + 1} &\leq \frac{r^2_i}{2} \label{eqn:contraction}  \\
r^2_{i + 1} &\leq \frac{4K^2}{n} \left( \log^{d+3} n  +\log^2 n \log(2 / \tau) \right) + 2K \left(  n^{\frac{1 - C}{2}} \abs{\log \rho^*} + \frac{ \abs{ \log \rho^*} \log(2 / \tau) }{n} \right). \label{eqn:bad-event}
\end{align} 
How many times can \eqref{eqn:contraction} occur before \eqref{eqn:bad-event} necessarily holds? Let 
\begin{equation*}
I_{\max} := \max \left\{ 1, \left\lceil \log_2 \left( r_0^2 \, \frac{n}{K\log^{d+2} n}  \right) \right\rceil + 1 \right\}.
\end{equation*}
Take $\tau=\delta/(2I_{\max})$. If $r_0^2\leq K\log^{d+2}n/n$, set $I=0$. Otherwise, let $I\geq1$ be the first index for which~\eqref{eqn:bad-event} holds with $i=I-1$, or $r_I^2\leq K\log^{d+2}n/n$. This index is deterministic. If neither condition held in the first $I_{\max}$ steps, every step would be a contraction, giving
$$r_{I_{\max}}^2\leq r_0^2 2^{-I_{\max}}\leq K\log^{d+2}n/n,$$
a contradiction. Thus $I\leq I_{\max}$.

Now we can proceed with the induction argument. The choice of $r_0$ and the definition of $\hat r$ give the base case $\hat r^2\leq\widehat W_\Sigma(\hat r)\leq r_0^2$. For $0\leq i<I$, the radius $r_i$ is above the floor by the definition of $I$. Assuming that $\hat{r}^2 \leq r_i^2$,
\begin{equation*}
r^2_{i+1} \overset{(i)}{\geq} \widehat{W}_\Sigma \left( r_i \right) \overset{(ii)}{\geq} \widehat{W}_\Sigma \left( \hat{r} \right) \overset{(iii)}{\geq} \hat{r}^2,
\end{equation*}
where $(i)$ follows from~\eqref{eqn:single_r_W_sigma_bound} and the choice of $K$, with failure probability at most $\tau$ for each radius; $(ii)$ is by the inductive hypothesis and $\widehat{W}_\Sigma$ being monotonically increasing; and $(iii)$ is by definition of $\hat{r}$. This proves $\hat r^2\leq r_I^2$ on the intersection of these events and $A_n\cap A_n'$.

Observe that the above induction fails with probability at most
\begin{equation*}
\PP{A_n^c} + \PP{A_n'^c} + I_{\max} \tau \leq 2dn^{1-C}+C'd\exp(-c_1 \log^{d+1} n) + \frac{\delta}{2} \leq \delta
\end{equation*}
for sufficiently large $n$, using Lemmas~\ref{lemm:compact_z_whp_multivariate} and \ref{lemm:l2_to_l2pn_multivariate}, along with a union bound for each step of the induction. If we stop at the radius floor, use $\log^{d+2}n\leq\log^{d+3}n$; otherwise, use~\eqref{eqn:bad-event}. Thus, with probability at least $1 - \delta$, we have
\begin{align*}
\hat{r}^2 &\leq r_I^2  \\
&\lesssim \frac{ \log^{d+3} n  +\log^2 n \log(2 / \tau)  + \abs{ \log \rho^*} \log(2 / \tau)}{n} +    n^{\frac{1 - C}{2}} \abs{\log \rho^*}  
\end{align*}
assuming that $C$ was chosen sufficiently large ($C=10$ suffices). Since $\log(1/\tau)$ grows as $\log n$, we arrive at our final rate. So from Lemma~\ref{lemm:l2_to_l2pn_multivariate}, with probability at least $1 - \delta$, we find an upper bound on the squared loss, which if all other parameters are fixed is a rate of $ \log^{d+3} n/n$ in the number of samples.
\begin{equation*}
\Dfishersigma{f_{\star}}{f_{\widehat{G}}} \lesssim_{d,M,\underline{k},\overline{k}} \hat{r}^2 + n^{1-C} + \frac{\log^{d+2} n}{n} \lesssim_{d,M,\underline{k},\overline{k}} \log^{d+3} n/n.
\end{equation*}
For the remaining finitely many $n\geq2$, enlarge the constant using $\Dfishersigma{f_\star}{f_G}\leq4dM^2$.
\end{proof}

\subsection{Proof of Lemma~\ref{lemm:multivariate-apply-gaussian-strip-interpolation}}
\label{subsec:proof_multivariate_derivative} 

We first prove a multivariate Gaussian strip interpolation for general covariance, then tailor the result specifically for Lemma~\ref{lemm:multivariate-apply-gaussian-strip-interpolation}. 

\begin{lemm}[Extension of Lemma~\ref{lemm:1d-gaussian-strip-interpolation}]
\label{lemm:multivariate-general-gaussian-strip-interpolation}
Let $\delta > 0$ and $h: \CC^d \to \CC$ be such that $h(\RR^d) \subseteq \RR$; $h$ is holomorphic on the strip $\left\{ z \in \CC^d: \abs{ \Im \left(z_j \right) }< 2 \delta, \text{ for all $j$} \right\}$; and $h$ is bounded by some $B$:
\begin{equation*}
\sup_{x \in \RR^d} \sup_{| y_j | < \delta, \, j \in [d] } | h(x + i y ) | < B.
\end{equation*}
Let $u \in \mathbb{S}^{d-1}$ be a unit vector. Then there is some constant $C$ such that 
\begin{align*}
\int | \partial_u &h(z) |^2 \varphi_d(z; \Sigma) dz \\
& \leq \frac{C}{\underline{k}}\left( 1 + \frac{\overline{k} }{\delta^2} \right) \left[ 1 + \log_+^2 \left( \frac{CB^2 }{ \int_{\RR^d} | h(z) |^2 \varphi_d(z; \Sigma) dz } \right) \right] \int_{\RR^d} | h(z) |^2 \varphi_d(z; \Sigma) dz 
\end{align*}
where the directional derivative is defined as $\partial_u h(z) = \langle u, \nabla h(z) \rangle$.
\end{lemm}

\begin{proof}
Recall that when $\Sigma = I_d$ or is otherwise spherical, $\varphi_d$ is rotationally invariant. Rotational invariance would allow us to prove the statement with just $u = e_1$, the first standard basis vector, without loss of generality. This is because for some arbitrary direction $u$, we can find some rotation matrix $R$ such that $R u = e_1$. Then,
\begin{itemize}
\item $\displaystyle \int \left| h\left(R^{-1} t\right)  \right|^2 \varphi_d(t; I_d) dt = \int \bigl| h (t) \bigr|^2 \varphi_d(t; I_d ) dt$ by change of variable;
\item $\partial_u h \left( t \right) \bigl|_{t = R^{-1}x} = \langle u, \nabla_t  h (t) \rangle \bigr|_{t = R^{-1}x}$  is equivalent to $\displaystyle \partial_{e_1}  h \left(R^{-1} x\right) = \left\langle e_1, \nabla_x  h \left(R^{-1} x\right) \right\rangle$;
\item and each slice $t_1 \mapsto h\left(R^{-1}(t_1,t')\right)$, for fixed $t'\in\RR^{d-1}$, is holomorphic for $|\Im t_1|<2\delta$ and bounded by $B$ for $|\Im t_1|<\delta$, since $\|R^{-1}e_1\|_\infty\leq1$.
\end{itemize}
However, $\Sigma$ is a general covariance matrix with bounded eigenvalues. So, define 
\begin{equation*}
\tilde{h}(t) = h\left(\Sigma^{1/2} t\right),
\end{equation*}
and observe the properties
\begin{itemize}
\item $\displaystyle
\int | h (z) |^2 \varphi_d(z ; \Sigma) dz = \int |  \tilde{h} (t) |^2 \varphi_d(t ; I_d) dt$ (by change of variable), and 
\item $\displaystyle \nabla_t \tilde{h}(t) = \Sigma^{1/2} \nabla_z h(z) \bigl|_{z = \Sigma^{1/2} t}$, so $\displaystyle |\partial_u h(z)| = \left|\left\langle u, \Sigma^{-1/2} \nabla_t \tilde{h}(t) \right\rangle\right| = \| \Sigma^{-1/2} u \| |\partial_{v} \tilde{h}(t)| \leq \underline{k}^{-1/2} |\partial_{v} \tilde{h}(t)|$, where $v$ is a unit vector $v = \Sigma^{-1/2} u / \| \Sigma^{-1/2} u \|$.
\item For a rotation $R$ such that $Rv=e_1$ and fixed $t'\in\RR^{d-1}$, the univariate map $t_1 \mapsto\tilde{h}(R^{-1}(t_1, t'))$ is holomorphic for $|\Im t_1|<2\delta/\sqrt{\overline{k}}$ and bounded by $B$ for $|\Im t_1|<\delta/\sqrt{\overline{k}}$, since $\|\Sigma^{1/2}R^{-1}e_1\|_\infty\leq\sqrt{\overline{k}}$.
\end{itemize}
With these ideas, we can bound the left-hand side.
\begin{align*}
\int | \partial_u &h (z) |^2 \varphi_d(z ; \Sigma) dz \stackrel{(i)}{\leq}  \underline{k}^{-1} \int | \partial_{v} \tilde{h} (t) |^2 \varphi_d(t ; I_d) \, dt \\
\stackrel{(ii)}{=}  &\underline{k}^{-1} \int | \partial_{t_1}[\tilde{h}(R^{-1}t)] |^2 \varphi_d(t ; I_d) \, dt \\
\stackrel{(iii)}{=} &\underline{k}^{-1} \int_{t' \in \RR^{d-1}} \left( \int_{t_1 \in \RR} | \partial_{t_1}[\tilde{h}(R^{-1}t)] |^2 \varphi \left(t_1 \right) d t_1 \right) \varphi_{d-1} \left(t' ; I_{d-1}\right) \, dt' \\
\stackrel{(iv)}{\leq} &\underline{k}^{-1} C \left( 1 + \overline{k} \delta^{-2} \right) \int_{t' \in \RR^{d-1}} \Psi \left( \int_{\RR} | \tilde{h}(R^{-1}t) |^2 \varphi(t_1) d t_1 \right) \varphi_{d-1} \left(t' ; I_{d-1}\right) \, dt' \\
\stackrel{(v)}{\leq} &\underline{k}^{-1} C \left( 1 + \frac{\overline{k} }{\delta^{2}} \right) \Psi \left( \int_{t \in \RR^d} | \tilde{h}(R^{-1}t) |^2 \varphi_d(t; I_d ) dt \right) \\
= &\underline{k}^{-1}    C \left( 1 + \frac{\overline{k} }{\delta^2} \right) \left[ 1 + \log_+^2 \left( \frac{CB^2 }{ \int_{\RR^d} | h(z) |^2 \varphi_d(z; \Sigma) dz } \right) \right] \int_{\RR^d} | h(z) |^2 \varphi_d(z; \Sigma) dz 
\end{align*}
In $(i)$, we do a change of variable $z = \Sigma^{1/2} t$. In $(ii)$, we apply the rotation $R$ such that $Rv = e_1$. In $(iii)$, we write $t = (t_1, t')^\top$, in which $t_1$ is the first coordinate and $t'$ is the remaining $(d-1)$ coordinates. We also factorize the $d$-dimensional joint distribution. In $(iv)$, we apply Lemma~\ref{lemm:1d-gaussian-strip-interpolation} to the function $t_1 \mapsto \tilde{h}(R^{-1}(t_1, t'))$ for some fixed $t' \in \RR^{d-1}$, where $\Psi$ is defined as before:
\begin{equation}
\label{eqn:Psi-definition}
\Psi(t) =  \left[ 1 + \log^2_+ \left( \tfrac{CB^2}{t} \right) \right] t.
\end{equation} 
In $(v)$, we apply Jensen's inequality, observing the concavity of $\Psi$ on $[0,B^2]$ when $C > e$, with $\Psi(0)=0$ by continuity.
\end{proof}

\begin{proof}[Proof of Lemma~\ref{lemm:multivariate-apply-gaussian-strip-interpolation}]
We introduce some notation. Let $v \in \mathbb{S}^{d-1}$ be a unit vector and $h_v: \RR^d \to \RR$ be 
\begin{equation*}
h_v(z) = \left\langle \score_{G_\star}(z) - \score_G(z) , v \right \rangle.
\end{equation*}
We can also express $h_v$ in terms of another function $r_G$:
\begin{equation*}
h_v(z) = \left\langle \frac{\nabla r_{G_\star}(z) }{r_{G_\star} (z)} - \frac{\nabla r_{G} (z)}{r_{G} (z)}  , v \right \rangle, \quad \text{where} \quad r_G(z) = \int \exp \left( \mu^\top \Sigma^{-1} z - \mu^\top \Sigma^{-1} \mu / 2 \right) dG(\mu).
\end{equation*}
The above fact follows from observing that
\begin{equation*}
r_G(z) = \frac{f_G(z) } {\varphi_d(z; \Sigma)} \qquad \text{and} \qquad \nabla r_G(z) = \frac{ \nabla f_G(z)}{ \varphi_d(z ; \Sigma)} + \Sigma^{-1}  \, z  r_G(z).
\end{equation*}
Finally, define
\begin{equation*}
\delta_{M, \underline{k}, d} = \frac{\pi}{4} \left(\frac{M d}{\underline{k}} \right)^{-1}.
\end{equation*}
The remainder of the proof will proceed as follows. First, we establish strip holomorphy and boundedness of $h_v(z)$. Then, we apply Lemma~\ref{lemm:multivariate-general-gaussian-strip-interpolation}.

We show that $h_v(z)$ is holomorphic on the strip $\{ z \in \CC^d : \| \Im(z) \|_\infty < 2 \delta_{M, \underline{k}, d} \}$. Let $x + iy$ be a complex vector in the strip. Then,
\begin{equation*}
\Re \left( r_G(x + i y) \right) = \int   \exp \left( \mu^\top \Sigma^{-1} x - \mu^\top \Sigma^{-1} \mu / 2 \right) \cos \left( \mu^\top \Sigma^{-1} y \right)   dG(\mu).
\end{equation*}
Since cosine is even, $\cos \left( \mu^\top \Sigma^{-1} y \right) = \cos \left(  | \mu^\top \Sigma^{-1} y | \right)$. This can be lowerbounded by observing
\begin{equation*}
0 \leq | \mu^\top \Sigma^{-1} y | \leq \frac{M d}{\underline{k}} \|y \|_\infty < \frac{\pi}{2} \qquad \implies \qquad
\cos \left( \mu^\top \Sigma^{-1} y \right) \geq \cos \left( \frac{M d}{\underline{k}} \|y \|_\infty \right) > 0.
\end{equation*}
The denominators do not vanish and the numerators are holomorphic, so this establishes strip holomorphy. Now, we can show boundedness. Let $x \in \RR^d$ and $y \in \RR^d$ such that the largest element in $y$ is bounded: $\| y \|_\infty < \delta_{M, \underline{k}, d}$. Then,
\begin{align*}
\| \nabla r_G(x + iy) \|_2 &\leq \int \left\| \Sigma^{-1} \mu \right\|_2 \exp \left( \mu^\top \Sigma^{-1} x - \mu^\top \Sigma^{-1} \mu / 2 \right) \left| \exp\left( i \mu^\top \Sigma^{-1} y \right) \right| dG(\mu) \\
&\leq \sqrt{d} M \left( \underline{k} \right)^{-1} r_G(x)
\end{align*} 
From before, we can also lower bound $| r_G(x+ iy)|$ using the fact that $0 \leq  \tfrac{M d}{\underline{k}} \|y \|_\infty < \tfrac{\pi}{4}$:
\begin{equation*}
| r_G(x+ iy)| \geq \left| \Re \left( r_G(x+ iy) \right)  \right| \geq \cos \left( \frac{M d}{\underline{k}} \|y \|_\infty \right) r_G(x) \geq \frac{\sqrt{2}}{2} r_G(x).
\end{equation*}
Finally we can conclude (using Cauchy Schwarz on $h_v$ and the same bound for $G_\star$) that
\begin{equation*}
\left\| \frac{\nabla r_G(x + iy) }{r_G(x + iy) } \right\|_2 \leq \frac{\sqrt{2d} M}{\underline{k}}, \quad \text{so} \quad \abs{ h_v ( x + iy)} \leq \frac{2 \sqrt{2d} M}{\underline{k}}.
\end{equation*}
For $u \in \mathbb{S}^{d-1}$ and $\mu\in\RR^d$, apply Lemma~\ref{lemm:multivariate-general-gaussian-strip-interpolation} to $z\mapsto h_v(z+\mu)$, taking $\delta = \delta_{M, \underline{k}, d}$, $B = 2 \sqrt{2 d} M / \underline{k}$, and $\Psi$ as in \eqref{eqn:Psi-definition}.
\begin{equation}
\label{eqn:intermediate-lemma-49-inequality}
\int_{\RR^d} \left| \partial_u h_v(z ) \right|^2 \varphi_d(z - \mu ; \Sigma) dz \leq \frac{C}{\underline{k}} \left( 1 + \overline{k}\delta_{M, \underline{k}, d}^{-2} \right) \Psi \left( \int_{\RR^d} \left|  h_v(z ) \right|^2 \varphi_d(z - \mu ; \Sigma) dz  \right).
\end{equation}
The integral $ \int_{\RR^d} \left|  h_v(z ) \right|^2 \varphi_d(z - \mu ; \Sigma) dz $ is bounded between 0 and $\tfrac{4M^2 d}{\underline{k}^2} $. We can choose a value of $C$ (in $\Psi$) so that $\Psi$ is both non-decreasing and concave over its domain. 
\begin{align*}
0 &\leq  \int_{\RR^d} \left|  h_v(z ) \right|^2 \varphi_d(z - \mu ; \Sigma) dz \\
&= \int_{\RR^d} \left| \left\langle \score_{G_\star}(z) - s_G(z), v \right\rangle \right|^2 \varphi_d(z - \mu; \Sigma) dz \\
&\leq \int_{\RR^d} \left\| \Sigma^{-1} \left( \EE[G_\star]{\mu \mid z} - \EE[G]{\mu \mid z} \right)  \right\|^2 \varphi_d(z -\mu; \Sigma) dz \leq \frac{4M^2 d}{\underline{k}^2}   \\
\Psi'(t) &=  \left( \log (CB^2 /t) - 1 \right)^2  \geq 0 \\
\Psi''(t) &= \frac{2}{t} \left( 1 - \log ( CB^2 / t) \right) \leq 0 \iff  \frac{CB^2 }{t} \geq e \\
\Rightarrow C &\geq e \left(\frac{\underline{k}^2}{8 M^2 d} \right) \left( \frac{4M^2 d}{\underline{k}^2}  \right) = \frac{ e  }{2}. 
\end{align*}
Now, we can use Fubini, integrate \eqref{eqn:intermediate-lemma-49-inequality} with respect to $d G_\star(\mu)$, and apply Jensen's inequality.
\begin{align*}
\int \left| u^\top  \right. & \nabla  \bigl( s_{G_\star} - s_G \bigr)(z) \left.  v \right|^2 f_{G_\star}(z) dz  = \int_{\mu} \int_{\RR^d} \left| \partial_u h_v(z ) \right|^2 \varphi_d(z - \mu ; \Sigma) dz \, dG_\star(\mu)\\
&\leq \int_{\mu} C_{\underline{k}, \overline{k}, d, M} \left( 1 + M^2 \right) \, \Psi \left( \int_{\RR^d} \left|  h_v(z ) \right|^2 \varphi_d(z - \mu ; \Sigma) dz \right) \, dG_\star(\mu)  \\
&\leq  C_{\underline{k}, \overline{k}, d, M} \left( 1 + M^2 \right) \, \Psi \left( \int_{\mu} \int_{\RR^d} \left|  h_v(z ) \right|^2 \varphi_d(z - \mu ; \Sigma) dz \, dG_\star(\mu)  \right) \\
&=  C_{\underline{k}, \overline{k}, d, M} \left( 1 + M^2 \right) \, \Psi \left( \EE[G_\star]{ \left| \bigl\langle \score_{G_\star}(z) - \score_G(z) , v \bigr\rangle \right|^2 } \right) \\
&\leq  C_{\underline{k}, \overline{k}, d, M} \left( 1 + M^2 \right) \, \Psi \left( \frac{1}{\underline{k}^2}  \EE[G_\star]{ \left\| \Sigma \left( \score_{G_\star}(z) - \score_G(z)\right) \right\|^2 } \right) \\
&\leq C_{\underline{k}, \overline{k}, d, M} \left( 1 + M^2 \right) \left[1 + \log_+^2 \left( \frac{C_0 }{\Dfishersigma{f_{G_{\star}}}{f_{G}}} \right)  \right] \Dfishersigma{f_{G_{\star}}}{f_{G}}
\end{align*}
In the last line, we write $C_0 = CB^2 \underline{k}^2 = 8CdM^2$ and absorb $\underline{k}^{-2}$ into the constant. Finally, taking $u = v = \Sigma_j / \| \Sigma_j \|$ (and knowing that $\| \Sigma_j \| \leq \overline{k}$, which will get absorbed in the constant), we arrive at the lemma statement.
\end{proof}

\section{Derivatives, log-Sobolev inequalities, and other proof elements }
\label{subsec:homosc_proof_elements}

As we discussed in Section~\ref{subsec:basic_ineq},
the most challenging part of our argument is to establish the approximate validity of an implication as in~\eqref{eq:wrong_isometry}. The challenge here is that $s_G(\cdot)$ depends on $\tfrac{\partial}{\partial z}f_G(\cdot)$, while $\tfrac{\partial}{\partial z} s_G(\cdot)$ which appears in SURE/SM also depends on the second derivative of $f_G(\cdot)$. The following result shows that in our setting, when the Fisher divergence between $f_{G_{\star}}$ and $f_G$ is small (for some $G \in \mathcal{P}(M)$), then the score derivatives are also close.

\begin{theo}
\label{theo:derivative_score}
Let $G, G_{\star} \in \mathcal{P}(M)$. Then, for any $\rho \in (0, 1/\sqrt{2\pi e^3})$, 
\begin{equation*}
    \begin{aligned}
        &\EE[G_{\star}]{ \cb{\frac{\partial}{\partial z} s_{G_\star}(Z) - \frac{\partial}{\partial z} s_G(Z)}^2 \ind\cb{f_{G_{\star}}(Z) \geq \rho, f_G(Z) \geq \rho}} \\ 
        &\qquad\qquad\qquad\qquad\; \lesssim_M \max\cb{\abs{\log\Dfisher{f_{G_\star}}{f_G}}^2, \abs{\log \rho}^4} \Dfisher{f_{G_\star}}{f_G}.
    \end{aligned}
\end{equation*}
\end{theo}
When $\Dfisher{f_{G_\star}}{f_G} = 0$ (equivalently, $f_G = f_{G_\star}$), the right-hand side is interpreted as $0$. Theorem~\ref{theo:derivative_score} controls the score derivatives when the marginal densities are not too small, that is, when they are lower bounded by $\rho>0$. An important feature of the result is that the right-hand side only depends on $\abs{\log \rho}^4$. This will allow us to take $\rho$ to be polynomial in $n^{-1}$.

Proving Theorem~\ref{theo:derivative_score} requires  
powerful machinery. To state this machinery and its implications, it will be convenient to consider in addition to the Fisher divergence $\Dfisher{f_{G_{\star}}}{f_G}$ between $f_{G_{\star}}$ and $f_G$, defined in~\eqref{eq:fisher_divergence}, their Kullback-Leibler divergence as well as their squared Hellinger distance:
\begin{equation*}
\DKL{f_{G_\star}}{f_G} := \int \log\p{ \frac{f_{G_\star}(z)}{f_G(z)}} f_{G_\star}(z) dz ,\,\;\; \Dhel^2(f_{G_{\star}}, f_G) := \int \p{ \sqrt{f_{G_\star}(z)}-\sqrt{f_G(z)}}^2dz.
\label{eq:hellinger}
\end{equation*}
Our first result uses and extends the breakthrough induction argument of~\citet[Lemma 1]{jiang2009general} to  control second derivatives.\footnote{Related inequalities for higher order derivatives were derived by~\citet[arXiv v4, SM10.2; removed in later versions]{chen2026empirical} with the different goal of controlling regret in estimating higher moments of $\theta_i$.
} We write $a \lor b := \max\cb{a,b}$ for $a, b \geq 0$.
\begin{lemm}
\label{lemm:blown_isometries}
For any $\rho \in (0, 1/\sqrt{2\pi e^3})$, define 
$$\Delta_1^2(\rho) := \int \frac{\cb{\tfrac{\partial}{\partial z} f_{G_\star}(z) - \tfrac{\partial}{\partial z}f_G(z)}^2}{f_{G_\star}(z)\lor \rho + f_G(z) \lor \rho} dz,\,\,\,\,\, \Delta_2^2(\rho) := \int \frac{\big\{\tfrac{\partial^2}{\partial z^2}f_{G_\star}(z) - \tfrac{\partial^2}{\partial z^2}f_G(z)\big\}^2}{f_{G_\star}(z)\lor \rho + f_G(z) \lor \rho} dz.$$
The following inequalities then hold:
\begin{enumerate}
\item $\;\;\Delta_1^2(\rho) \; \lesssim \;\Dfisher{f_{G_\star}}{f_G} + \p{\abs{\log \rho} + M^2} \Dhel^2(f_{G_\star}, f_{G}).$
\item $\;\; \Delta_2^2(\rho) \; \lesssim \; \max\cb{a^4 \Dhel^2(f_{G_\star}, f_G), \abs{\log \rho}^3 \p{\Dfisher{f_{G_\star}}{f_G} + \p{\abs{\log \rho} + M^2}\Dhel^2(f_{G_\star}, f_{G})}} $,

where $a^2 := \max\{\tilde{L}^2(\rho)+1, \abs{\log\Dhel^2(f_{G_\star}, f_G)}\}$.
\item $\;\;\displaystyle \int \p{\frac{\tfrac{\partial^2}{\partial z^2}f_{G_\star}(z)}{f_{G_\star}(z) \lor \rho} - \frac{\tfrac{\partial^2}{\partial z^2}f_{G}(z)}{f_{G}(z) \lor \rho}}^2  f_{G_\star}(z) dz \; \lesssim \; \Delta_2^2(\rho) + \abs{\log \rho}^2 \Dhel^2(f_{G_\star}, f_{G}).$
\end{enumerate}
\end{lemm}
If $f_G = f_{G_\star}$, all quantities above vanish and the three claims hold trivially (with $a$ undefined but irrelevant); the proof below therefore assumes $f_G \neq f_{G_\star}$. Lemma~\ref{lemm:blown_isometries} by itself does not suffice for the proof of Theorem~\ref{theo:derivative_score}. The reason is that the right-hand side of inequalities 1.--3. also depends on the squared Hellinger distance $\Dhel^2(f_{G_{\star}}, f_G)$. Our next step establishes that we may control $\Dhel^2(f_{G_{\star}}, f_G)$ by $\Dfisher{f_{G_{\star}}}{f_G}$ using a functional inequality for Gaussian convolutions of compactly supported measures. We first recall the definition of the logarithmic Sobolev constant. Given a probability measure $\nu$ on $\RR$, its logarithmic Sobolev constant $C_{\text{LS}} = C_{\text{LS}}(\nu) \in (0, \infty]$ is the smallest constant such that for all smooth functions $\psi: \RR \to \RR$, the following holds:
$$
\int_{\RR} \psi^2(z) \log(\psi^2(z))\nu(dz) \,-\, \int_{\RR} \psi^2(z) \nu(dz) \log\p{  \int_{\RR} \psi^2(z) \nu(dz) } \leq C_{\text{LS}}(\nu) \int_{\RR} \abs{ \frac{\partial}{\partial z} \psi(z)}^2 \nu(dz).
$$
A key result by \citet{zimmermann2016elementary} establishes that the logarithmic Sobolev constant remains uniformly bounded across all distributions formed by convolving $\mathrm{N}(0,1)$ with any compactly supported probability measure $G \in \mathcal{P}(M)$; see~\citet{bardet2018functional, chen2021dimensionfree} for further related results.
\begin{theo}[Theorem 1.1.\ in \citet{zimmermann2016elementary}]
\label{theo:logsobolev}
The convolution of a compactly supported measure and a Gaussian measure satisfies a logarithmic Sobolev inequality with
$$
\sup_{G \in \mathcal{P}(M)}  \cb{C_{\text{LS}}(f_G) }\, < \, \infty,
$$
where we identify the Lebesgue density $f_G$ with its induced probability measure.
\end{theo}
The upshot of Theorem~\ref{theo:logsobolev} is that it enables us to control the Hellinger distance via the Fisher divergence; and so it allows us to map the result of Lemma~\ref{lemm:blown_isometries} to Theorem~\ref{theo:derivative_score}.

\begin{prop}[Proposition 1 in~\citet{koehler2023statistical}]
\label{prop:KL_via_Fisher}
Let $G_{\star}, G \in \mathcal{P}(M)$. Then:
$$\Dhel^2(f_{G_\star}, f_G) \leq \DKL{f_{G_\star}}{f_G} \leq   \frac{1}{4}C_{\text{LS}}(f_G) \Dfisher{f_{G_\star}}{f_G} \lesssim_M \Dfisher{f_{G_\star}}{f_G}.$$
\end{prop}

\subsection{Proof of results in Section~\ref{subsec:homosc_proof_elements}}

\subsubsection{Proof of Theorem~\ref{theo:derivative_score}}
\begin{theo}[Restatement of Theorem~\ref{theo:derivative_score}]
Let $G, G_{\star} \in \mathcal{P}(M)$. Then, for any $0< \rho < \frac{1}{\sqrt{2\pi e^3}}$,  
\begin{equation*}
    \begin{aligned}
        &\EE[G_{\star}]{ \cb{\frac{\partial}{\partial z} s_{G_\star}(Z) - \frac{\partial}{\partial z} s_G(Z)}^2 \ind\cb{f_{G_{\star}}(Z) \geq \rho, f_G(Z) \geq \rho}} \\ 
        &\qquad\qquad\qquad\qquad\; \lesssim_M \max\cb{\abs{\log\Dfisher{f_{G_\star}}{f_G}}^2, \abs{\log \rho}^4} \Dfisher{f_{G_\star}}{f_G}.
    \end{aligned}
\end{equation*}
\end{theo}

\begin{proof}
Recalling that $s_G(z)  = \frac{f'_G(z)}{f_G(z)}$, we observe that
\begin{equation*}
         \frac{\partial}{\partial z} s_G(z) = \frac{f_G''(z)}{f_G(z)} - s_G^2(z). 
\end{equation*}
Then, using the above and the inequality $(a + b)^2 \le 2a^2 + 2b^2$, we have
\begin{equation*}
    \begin{aligned}
        \EE[G_{\star}]{ \cb{\frac{\partial}{\partial z} s_{G_\star}(Z) - \frac{\partial}{\partial z} s_G(Z)}^2 \ind\cb{f_{G_{\star}}(Z) \geq \rho, f_G(Z) \geq \rho}} \leq (1) + (2)
    \end{aligned}
\end{equation*}
where we define $(1)$ and $(2)$ below. For $(1)$,
\begin{equation*}
    \begin{aligned}
        (1) & :=  \; 2\EE[G_{\star}]{ \cb{\frac{f_{G_\star}''(Z)}{f_{G_\star}(Z)} - \frac{f_G''(Z)}{f_G(Z)}}^2 \ind\cb{f_{G_{\star}}(Z) \geq \rho, f_G(Z) \geq \rho}} \\
        & = \; 2\EE[G_{\star}]{ \cb{\frac{f_{G_\star}''(Z)}{f_{G_\star}(Z) \lor \rho} - \frac{f_G''(Z)}{f_G(Z) \lor \rho}}^2 \ind\cb{f_{G_{\star}}(Z) \geq \rho, f_G(Z) \geq \rho}} \\
        & \lesssim \; \Delta_2^2(\rho) + \abs{\log \rho}^2 \Dhel^2(f_{G_\star}, f_{G}) \\
        & \lesssim \; \Delta_2^2(\rho) + \abs{\log \rho}^2 C_{\text{LS}}(f_G) \Dfisher{f_{G_\star}}{f_G} \\
        & \lesssim_M  \; \Delta_2^2(\rho) + \abs{\log \rho}^2 \Dfisher{f_{G_\star}}{f_G}, 
    \end{aligned}
\end{equation*}
where the first inequality is due to the third claim in Lemma~\ref{lemm:blown_isometries}, and the second inequality uses Proposition~\ref{prop:KL_via_Fisher}, and the last inequality uses Theorem~\ref{theo:logsobolev}. For $(2)$,
\begin{equation*}
    \begin{aligned}
        (2) := & \; 2\EE[G_\star]{\cb{s_G^2(Z) - s^2_{G_\star}(Z)}^2 \ind\cb{f_{G_{\star}}(Z) \geq \rho, f_G(Z) \geq \rho}} \\
        = & \; 2\EE[G_\star]{\cb{\p{s_G(Z) - s_{G_\star}(Z)}\p{s_G(Z) + s_{G_\star}(Z)}}^2 \ind\cb{f_{G_{\star}}(Z) \geq \rho, f_G(Z) \geq \rho}} \\
        \leq & \; 4\EE[G_\star]{\cb{\p{s_G(Z) - s_{G_\star}(Z)}}^2 \p{s^2_G(Z) + s^2_{G_\star}(Z)} \ind\cb{f_{G_{\star}}(Z) \geq \rho, f_G(Z) \geq \rho}} \\ 
        \leq & \; 8 \Tilde{L}^2(\rho) \EE[G_\star]{\cb{s_G(Z) - s_{G_\star}(Z)}^2} \\
        \lesssim & \; \abs{\log \rho} \Dfisher{f_{G_\star}}{f_G}.
    \end{aligned} 
\end{equation*}
where we used Lemma~\ref{lemm:proposition_1_jiang_zhang} in the second inequality. 

Furthermore, by the second claim in Lemma~\ref{lemm:blown_isometries}, Proposition~\ref{prop:KL_via_Fisher}, and $M^2 \lesssim_M \abs{\log \rho}$ (recall $\abs{\log \rho} \ge 1$),
\begin{equation*}
\begin{aligned}
\Delta_2^2(\rho) & \lesssim a^4 \Dhel^2(f_{G_\star}, f_G) + \abs{\log \rho}^3 \Dfisher{f_{G_\star}}{f_G} + \abs{\log \rho}^3 \p{\abs{\log \rho} + M^2} \Dhel^2(f_{G_\star}, f_G) \\
& \lesssim_M a^4 \Dhel^2(f_{G_\star}, f_G) + \abs{\log \rho}^4 \Dfisher{f_{G_\star}}{f_G}.
\end{aligned}
\end{equation*}
To bound the first term, write $H^2 := \Dhel^2(f_{G_\star}, f_G)$ and $c := C_{\text{LS}}(f_G)/4 \lesssim_M 1$, so that $H^2 \le c\,\Dfisher{f_{G_\star}}{f_G}$ by Proposition~\ref{prop:KL_via_Fisher}. Since $a^4 \le (\tilde{L}^2(\rho)+1)^2 + \log^2(H^2)$,
$$
a^4 H^2 \le (\tilde{L}^2(\rho)+1)^2 H^2 + H^2 \log^2(H^2) \lesssim_M \abs{\log \rho}^2 \Dfisher{f_{G_\star}}{f_G} + \psi(H^2), \qquad \psi(t) := t \log^2 t .
$$
The function $\psi$ is increasing on $(0, e^{-2}]$ and bounded by $2$ on $(0, 2]$ (note $H^2 \le 2$). If $c\,\Dfisher{f_{G_\star}}{f_G} \le e^{-2}$, then 
$$
\begin{aligned}
\psi(H^2) & \le \psi(c\,\Dfisher{f_{G_\star}}{f_G}) \\
& \le 2c\,\Dfisher{f_{G_\star}}{f_G}\p{\log^2 c + \log^2 \Dfisher{f_{G_\star}}{f_G}} \\
& \lesssim_M \p{1 + \abs{\log \Dfisher{f_{G_\star}}{f_G}}^2} \Dfisher{f_{G_\star}}{f_G};
\end{aligned}
$$ 
otherwise $\psi(H^2) \le 2 < 2e^2 c\,\Dfisher{f_{G_\star}}{f_G} \lesssim_M \Dfisher{f_{G_\star}}{f_G}$. In either case,
$$
\Delta_2^2(\rho) \lesssim_M \max\cb{\abs{\log\Dfisher{f_{G_\star}}{f_G}}^2, \abs{\log \rho}^4} \Dfisher{f_{G_\star}}{f_G}.
$$
Combining everything, we have:
\begin{equation*}
    \begin{aligned}
        &\EE[G_{\star}]{ \cb{\frac{\partial}{\partial z} s_{G_\star}(Z) - \frac{\partial}{\partial z} s_G(Z)}^2 \ind\cb{f_{G_{\star}}(Z) \geq \rho, f_G(Z) \geq \rho}} \\ 
        & \;\;\; \lesssim_M \max\cb{\abs{\log\Dfisher{f_{G_\star}}{f_G}}^2, \abs{\log \rho}^4, \abs{\log \rho}} \Dfisher{f_{G_\star}}{f_G} \\
        & \;\;\; \lesssim_M \max\cb{\abs{\log\Dfisher{f_{G_\star}}{f_G}}^2, \abs{\log \rho}^4} \Dfisher{f_{G_\star}}{f_G} 
    \end{aligned}
\end{equation*}
as desired.
\end{proof}

\subsubsection{Proof of Lemma~\ref{lemm:blown_isometries}}
\begin{proof}
Let us start with the first claim. Throughout, we write $w(z) := 1/(f_{G_\star}(z)\lor \rho + f_G(z) \lor \rho)$, $f_{\max}(z) := \max\{f_{G_\star}(z), f_G(z)\}$, $f_{\min}(z) := \min\{f_{G_\star}(z), f_G(z)\}$, and let $s_{\max}(z)$ denote the score ($s_{G_\star}(z)$ or $s_G(z)$) of whichever of the two densities is larger at $z$. Since $f'_G = s_G f_G$ and $f'_{G_\star} = s_{G_\star} f_{G_\star}$, we have for every $z$
$$
f'_{G_\star}(z) - f'_G(z) = \cb{s_{G_\star}(z) - s_G(z)} f_{\min}(z) + s_{\max}(z)\cb{f_{G_\star}(z) - f_G(z)}
$$
If $f_{G_\star}(z) \ge f_G(z)$, write $s_{G_\star}f_{G_\star} - s_G f_G = (s_{G_\star} - s_G) f_G + s_{G_\star}(f_{G_\star} - f_G)$; otherwise use the symmetric identity. Hence, by $(a+b)^2 \le 2a^2 + 2b^2$,
\begin{equation}\label{eqn:claim1-pointwise}
\begin{aligned}
&\cb{f'_{G_\star}(z) - f'_G(z)}^2 w(z) \\
& \;\;\;\;\; \le 2\cb{s_{G_\star}(z) - s_G(z)}^2 f_{\min}(z)^2 w(z) + 2 s_{\max}(z)^2 \cb{f_{G_\star}(z) - f_G(z)}^2 w(z).
\end{aligned}
\end{equation}
We bound the three factors on the right-hand side separately. Firstly, $w \le 1/(f_{G_\star} + f_G) \le 1/f_{\max}$, so $f_{\min}^2 w \le f_{\min} \le f_{G_\star}$. Secondly,
$$
\cb{f_{G_\star} - f_G}^2 w \le \frac{\p{\sqrt{f_{G_\star}} - \sqrt{f_G}}^2 \p{\sqrt{f_{G_\star}} + \sqrt{f_G}}^2}{f_{G_\star} + f_G} \le 2 \p{\sqrt{f_{G_\star}} - \sqrt{f_G}}^2 .
$$
Thirdly, by Lemma A.1 in~\citet{jiang2009general}, which states that $$\p{\frac{f'_G(z)}{f_G(z)}}^2 \le \frac{f''_G(z)}{f_G(z)} + 1 \le \tilde{L}^2(f_G(z)),$$ for every $z$ and every mixing distribution $G$, we have $s_{\max}(z)^2 \le \tilde{L}^2(f_{\max}(z))$. Integrating~\eqref{eqn:claim1-pointwise} therefore gives
\begin{equation}\label{eqn:claim1-integrated}
\Delta_1^2(\rho) \le 2 \Dfisher{f_{G_\star}}{f_G} + 4 \int \tilde{L}^2\p{f_{\max}(z)} \p{\sqrt{f_{G_\star}(z)} - \sqrt{f_G(z)}}^2 dz .
\end{equation}
To control the logarithmic weight, we split according to whether $f_{\max}(z) \ge \rho$. On $\{f_{\max} \ge \rho\}$ we have $\tilde{L}^2(f_{\max}) \le \tilde{L}^2(\rho)$, since $\tilde{L}^2(\cdot)$ is decreasing. On $\{f_{\max} < \rho\}$ we use that $f_G(z) = \int \varphi(z-u) G(du) \ge \varphi(\abs{z} + M)$ for every $G \in \mathcal{P}(M)$ (as $\abs{z - u} \le \abs{z} + M$ for $\abs{u} \le M$), so that $f_{\max}(z) \ge \varphi(\abs{z}+M)$ and hence $\tilde{L}^2(f_{\max}(z)) \le \tilde{L}^2\p{\varphi(\abs{z}+M)} = (\abs{z} + M)^2 \le 2z^2 + 2M^2$. Consequently,
\begin{equation}\label{eqn:claim1-split}
\begin{aligned}
& \int \tilde{L}^2\p{f_{\max}} \p{\sqrt{f_{G_\star}} - \sqrt{f_G}}^2 dz \\
& \;\;\;\;\; \le \tilde{L}^2(\rho) \Dhel^2(f_{G_\star}, f_G) + 2 \int z^2 \p{\sqrt{f_{G_\star}(z)} - \sqrt{f_G(z)}}^2 dz + 2M^2 \Dhel^2(f_{G_\star}, f_G).
\end{aligned}
\end{equation}
It remains to bound the second-moment-weighted Hellinger distance in~\eqref{eqn:claim1-split}. We claim that
\begin{equation}\label{eqn:weighted-hellinger}
\int z^2 \p{\sqrt{f_{G_\star}(z)} - \sqrt{f_G(z)}}^2 dz \le 2 \Dfisher{f_{G_\star}}{f_G} + 2(M+1)^2 \Dhel^2(f_{G_\star}, f_G).
\end{equation}
To see this, let $h(z) := \sqrt{f_{G_\star}(z)/f_G(z)} - 1$, so that $h^2 f_G = (\sqrt{f_{G_\star}} - \sqrt{f_G})^2$, and note that $h'(z) = \tfrac{1}{2}\sqrt{f_{G_\star}(z)/f_G(z)}\cb{s_{G_\star}(z) - s_G(z)}$, whence $\int (h')^2 f_G \, dz = \tfrac{1}{4}\Dfisher{f_{G_\star}}{f_G}$. Differentiating $f_G(z) = \int \varphi(z-u)G(du)$ under the integral sign yields the identity
$$
z f_G(z) = - f'_G(z) + f_G^{(1)}(z), \qquad f_G^{(1)}(z) := \int u\, \varphi(z-u)\, G(du), \qquad \abs{f_G^{(1)}(z)} \le M f_G(z).
$$
Writing $A := \int z^2 h^2 f_G$, $B := \int (h')^2 f_G = \tfrac14 \Dfisher{f_{G_\star}}{f_G}$ and $C := \int h^2 f_G = \Dhel^2(f_{G_\star}, f_G)$, and integrating by parts (the boundary terms vanish since $\abs{z h^2 f_G} = \abs{z}(\sqrt{f_{G_\star}} - \sqrt{f_G})^2 \le 2\abs{z}(f_{G_\star} + f_G) \to 0$ as $\abs{z} \to \infty$, both mixture densities having Gaussian tails because $G, G_\star$ are supported on $[-M, M]$), we obtain 
$$
\begin{aligned}
A & = \int (z h^2)\, (z f_G)\, dz = \int (z h^2)' f_G \, dz + \int z h^2 f_G^{(1)} \, dz \\
& = \int h^2 f_G + 2\int z h h' f_G + \int z h^2 f_G^{(1)} \\
& \le C + 2\sqrt{AB} + M\sqrt{AC},
\end{aligned}
$$
where the last step uses the Cauchy--Schwarz inequality twice and $\abs{f_G^{(1)}} \le M f_G$. Thus $\sqrt{A}$ satisfies $x^2 - (2\sqrt{B} + M\sqrt{C})\,x - C \le 0$, which implies $\sqrt{A} \le 2\sqrt{B} + (M+1)\sqrt{C}$, and so $A \le 8B + 2(M+1)^2 C$. This is~\eqref{eqn:weighted-hellinger}.

Combining~\eqref{eqn:claim1-integrated}, \eqref{eqn:claim1-split} and~\eqref{eqn:weighted-hellinger}, we conclude that
\begin{equation}\label{eqn:claim1-final}
\begin{aligned}
\Delta_1^2(\rho) & \le 18\, \Dfisher{f_{G_\star}}{f_G} + 4\p{\tilde{L}^2(\rho) + 6(M+1)^2} \Dhel^2(f_{G_\star}, f_G) \\
& \lesssim \Dfisher{f_{G_\star}}{f_G} + \p{\abs{\log \rho} + M^2}\Dhel^2(f_{G_\star}, f_{G}),
\end{aligned}
\end{equation}
which is the first claim. 

For the second claim, we define for $k \ge 0$: 
$$
\Delta_k^2 := \int \frac{\cb{\frac{d^k}{dz^k}(f_{G_\star}(z) - f_G(z))}^2}{f_{G_\star}(z)\lor \rho + f_G(z) \lor \rho} dz,
$$
so that $\Delta_1, \Delta_2$ are the quantities in the statement. Each $\Delta_k$ is finite (by~\eqref{eqn:a2-analogue} below) and, since $f_G \neq f_{G_\star}$, strictly positive: if $\frac{d^k}{dz^k}(f_{G_\star} - f_G) \equiv 0$, then $f_{G_\star} - f_G$ would be an integrable polynomial, hence zero. To derive an upper bound for $\Delta_2$, we use a slightly different version of equation (A.2) from~\citet{jiang2009general}, valid for every integer $k \ge 0$ and every $a > 0$ with $a^2 \ge 2k - 1$: 
\begin{equation}\label{eqn:a2-analogue}
    \begin{aligned}
        \int \abs{D^k(f_G - f_{G_\star})}^2 & \le a^{2k} \int \abs{f_G - f_{G_\star}}^2 + \frac{4}{\pi} a^{2k-1} e^{-a^2} \\
        & \le \frac{2\sqrt{2}a^{2k}}{\sqrt{\pi}} \Dhel^2(f_{G_\star}, f_G) + \frac{4a^{2k-1}}{\pi}e^{-a^2}, 
    \end{aligned} 
\end{equation} 
since $f_G(z) = \int \varphi(z-u) G(du) = \int \frac{1}{\sqrt{2\pi}} e^{-\frac{1}{2} (z-u)^2} G(du) \le \frac{1}{\sqrt{2\pi}}$ for any $G$ and consequently $\int \abs{f_G-f_{G_\star}}^2 \le 2\sqrt{\frac{2}{\pi}} \Dhel^2(f_{G_\star}, f_G)$ using again that $(x - y)^2 = (\sqrt{x} - \sqrt{y})^2(\sqrt{x} + \sqrt{y })^2$. The restriction $a^2 \ge 2k-1$ is what makes the Gaussian tail estimate $\int_a^\infty t^{2k} e^{-t^2} dt \le a^{2k-1} e^{-a^2}$ valid (integrate by parts once and use $\int_a^\infty t^{2k-2}e^{-t^2}dt \le a^{-2}\int_a^\infty t^{2k}e^{-t^2}dt$). 

To find an upper bound for $\Delta_2$, we use the following key inequality, which was proved by \citet{jiang2009general} via integration by parts:
\begin{equation*}
    \Delta_k^2 \leq \Delta_{k-1} \Delta_{k+1} + 2 \tilde{L}(\rho) \Delta_{k-1} \Delta_k.
\end{equation*}
We define $k_0 \ge 1$ such that $k_0 \leq \tilde{L}^2(\rho)/2 < k_0+1$ and let
\[ k^* = \min\{k \ge 1: \Delta_{k+1} \leq k_0 2\tilde{L}(\rho) \Delta_k\}, \] 
with the convention that the minimum of an empty set is $\infty$. 

For any integer $k$ satisfying $1 \le k < k^*$, we have by definition that $2 \tilde{L}(\rho) \Delta_k \le \Delta_{k + 1}/k_0$ and so 
\begin{equation*}
    \Delta_k^2 \leq (1+1/k_0) \Delta_{k-1} \Delta_{k+1},
\end{equation*} 
or equivalently $\Delta_k/\Delta_{k - 1} \le (1 + 1/k_0) \Delta_{k + 1}/\Delta_k$. Also, let $a \ge 1$ be defined by 
\[ a^2 = \max \{\tilde L^2(\rho) + 1, \abs{\log \Dhel^2(f_{G_\star}, f_G)} \}. \]
With this choice, $e^{-a^2} \le \Dhel^2(f_{G_\star}, f_G)$ (if $\Dhel^2 < 1$ this is $a^2 \ge \abs{\log \Dhel^2}$; otherwise it is trivial), and $a^2 \ge \tilde{L}^2(\rho) + 1 \ge 2k_0 + 1$, so that~\eqref{eqn:a2-analogue} applies for all $k \le k_0 + 1$. Since $f_{G_\star} \lor \rho + f_G \lor \rho \ge \rho$, we obtain for all $k \le k_0 + 1$, using $a \ge 1$,
\begin{equation}\label{eqn:Delta_k_H}
\Delta_k^2 \le \frac{1}{\rho}\p{2\sqrt{\tfrac{2}{\pi}}\, a^{2k} + \tfrac{4}{\pi} a^{2k-1}} \Dhel^2(f_{G_\star}, f_G) \le \frac{3\, a^{2k}}{\rho} \Dhel^2(f_{G_\star}, f_G).
\end{equation}
Moreover $\Delta_0^2 \le 2\Dhel^2(f_{G_\star}, f_G)$, since $(f_{G_\star} - f_G)^2 = (\sqrt{f_{G_\star}} - \sqrt{f_G})^2(\sqrt{f_{G_\star}} + \sqrt{f_G})^2 \le 2(\sqrt{f_{G_\star}} - \sqrt{f_G})^2(f_{G_\star} + f_G)$ and $f_{G_\star} + f_G \le f_{G_\star}\lor\rho + f_G\lor\rho$. Finally, $(2\pi\rho^2)^{-1/(4k_0)} < e$, because $k_0 \ge 1$ and $\tilde{L}^2(\rho) < 2k_0 + 2 \le 4k_0$; hence $\rho^{-1/(2k_0)} = (2\pi)^{1/(4k_0)}(2\pi\rho^2)^{-1/(4k_0)} \le c_0 := (2\pi)^{1/4} e$. 
The analysis now splits into two cases.

\textbf{Case 1: $k^* \leq k_0$.} If $k^* = 1$, then $\Delta_2/\Delta_1 \le 2k_0\tilde{L}(\rho)$ by the definition of $k^*$. If $k^* \ge 2$, chaining the relation $\Delta_k/\Delta_{k-1} \le (1+1/k_0)\Delta_{k+1}/\Delta_k$ for $k = 2, \ldots, k^*-1$ gives $$\Delta_2/\Delta_1 \le (1+1/k_0)^{k^*-2}\,\Delta_{k^*}/\Delta_{k^*-1},$$ while the key inequality at $k = k^*$ together with $\Delta_{k^*+1} \le 2k_0\tilde{L}(\rho)\Delta_{k^*}$ gives $$\Delta_{k^*}^2 \le \Delta_{k^*-1}\Delta_{k^*}\p{2k_0\tilde{L}(\rho) + 2\tilde{L}(\rho)},$$ i.e., $\Delta_{k^*}/\Delta_{k^*-1} \le 2(k_0+1)\tilde{L}(\rho)$. In either case, using $1 + x \le e^x$, $k^* \le k_0$ and $2k_0 + 2 \le \tilde{L}^2(\rho) + 2 \le 2\tilde{L}^2(\rho)$,
\begin{equation*}
    \frac{\Delta_2}{\Delta_1} \leq e \cdot 2(k_0+1)\tilde{L}(\rho) \leq 2e \tilde{L}^3(\rho).
\end{equation*}
Also, using~\eqref{eqn:claim1-final}, we can write:
\begin{equation}
    \Delta_2^2 \le 4e^2 \tilde{L}^6(\rho) \Delta_1^2 \lesssim \abs{\log \rho}^3 \cb{\Dfisher{f_{G_\star}}{f_G} \,+\, \p{\abs{\log \rho} + M^2}\Dhel^2(f_{G_\star}, f_{G})}.
    \label{eqn:case1}
\end{equation} 

\textbf{Case 2: $k^* > k_0$.} In this case, we have for all $k \le k_0 < k^*$ that
\begin{equation*}
    \begin{aligned}
        \frac{\Delta_2}{\Delta_1} & \leq (1+1/k_0)^k \frac{\Delta_{k+1}}{\Delta_k},
    \end{aligned}
\end{equation*}
so taking the geometric mean of these inequalities yields
\begin{equation*}
     \frac{\Delta_2}{\Delta_1}  \leq \left[\prod_{k=1}^{k_0} \left(1+\frac{1}{k_0}\right)^k \frac{\Delta_{k+1}}{\Delta_k}\right]^{\frac{1}{k_0}}  = (1+1/k_0)^{\frac{k_0+1}{2}} \left(\frac{\Delta_{k_0+1}}{\Delta_1}\right)^{\frac{1}{k_0}},
\end{equation*}
where in the equality we used that $\sum_{k = 1}^{k_0} k = \frac{k_0(k_0 + 1)}{2}$. Hence,
\begin{equation*}
    \begin{aligned}
    \Delta_2 & \leq (1+1/k_0)^{\frac{k_0+1}{2}} (\Delta_{k_0+1})^{\frac{1}{k_0}} (\Delta_1)^{\frac{k_0-1}{k_0}} \\
    & \leq (1+1/k_0)^{k_0} (\Delta_{k_0+1})^{\frac{1}{k_0}} (\Delta_1)^{\frac{k_0-1}{k_0}} \\
    & \leq e \; [(\Delta_{k_0+1}) (\Delta_1)^{k_0-1}]^{\frac{1}{k_0}},
    \end{aligned}
\end{equation*}
since $k_0 \ge 1$ implying that $k_0+1 \le 2k_0$. 

The same argument applied to the ratios $\Delta_1/\Delta_0 \le (1+1/k_0)^{k}\, \Delta_{k+1}/\Delta_k$, $0 \le k \le k_0 - 1$, gives
$$
\Delta_1 \le (1+1/k_0)^{\frac{k_0-1}{2}} (\Delta_{k_0})^{\frac{1}{k_0}} (\Delta_0)^{\frac{k_0-1}{k_0}} \le e^{1/2} (\Delta_{k_0})^{\frac{1}{k_0}} (\Delta_0)^{\frac{k_0-1}{k_0}}.
$$
Writing $H := \Dhel(f_{G_\star}, f_G)$ and inserting~\eqref{eqn:Delta_k_H} for $k = k_0$ and $k = k_0+1$, $\Delta_0 \le \sqrt{2} H$, and $\rho^{-1/(2k_0)} \le c_0$, we obtain
$$
\begin{aligned}
\Delta_1 & \le e^{1/2}\, 3^{\frac{1}{2k_0}} \, 2^{\frac{k_0-1}{2k_0}}\, \rho^{-\frac{1}{2k_0}}\, a\, H \le \sqrt{6}\, e^{1/2} c_0\, a H =: c_1\, a H, \\
\Delta_2 & \le e\, \sqb{\sqrt{3}\, a^{k_0+1} \rho^{-1/2} H \cdot (c_1 a H)^{k_0-1}}^{\frac{1}{k_0}} \le \sqrt{3}\, e\, c_0 c_1\, a^2 H.
\end{aligned}
$$
Thus, in Case 2 (which includes the possibility $k^* = \infty$),
\begin{equation}
\Delta_2^2 \lesssim a^4 \Dhel^2(f_{G_\star}, f_G).
\label{eqn:case2}
\end{equation}

Combining~\eqref{eqn:case1} and~\eqref{eqn:case2}, we obtain
\[ \Delta_2^2 \lesssim \max\cb{a^4 \Dhel^2(f_{G_\star}, f_G), \abs{\log \rho}^3 \p{\Dfisher{f_{G_\star}}{f_G} + \p{\abs{\log \rho} + M^2}\Dhel^2(f_{G_\star}, f_{G})}},  \]
which is the second claim.

As for the third claim, define
\[ w_\star = 1/(f_G \lor \rho + f_{G_\star} \lor \rho) \]
and observe that
\[ 1/(f_G \lor \rho) - 2w_\star = [f_G \lor \rho + f_{G_\star} \lor \rho - 2(f_G \lor \rho)] \frac{w_\star}{f_G \lor \rho} = [f_{G_\star} \lor \rho - f_G \lor \rho] \frac{w_\star}{f_G \lor \rho}\]
so
\[ f_G''/(f_G \lor \rho) - 2f_G'' w_\star = \frac{f_G''}{f_G \lor \rho}(f_{G_\star} \lor \rho - f_G\lor \rho) w_\star. \]
We can upper bound
\[ \frac{|f_G''|}{f_G \lor \rho} = \frac{|f_G''|}{f_G} \frac{f_G}{f_G \lor \rho} 
\le 2\tilde{L}^2(\rho),  \]
using Lemma~\ref{lemm:proposition_1_jiang_zhang}. 

So the bound follows by triangle inequality, using the symmetric bound for $f_{G_\star}$.
\end{proof}

\subsubsection{Proof of Proposition~\ref{prop:KL_via_Fisher}}

As mentioned in the main text, the argument for this proposition appears in~\citet{koehler2023statistical} (see there for further related references). \citet{koehler2023statistical} states the result for the Kullback-Leibler divergence, however, the first inequality between squared Hellinger distance and Kullback-Leibler divergence is standard.
For self-containedness and because we use a slightly different definition of the log-Sobolev constant, we  provide a proof below.

\begin{proof}
Let us take $\psi(z) = \sqrt{\frac{f_{G_\star}(z)}{f_G(z)}}$ and $\nu(dz) = f_G(z) dz$ in the definition of the logarithmic Sobolev constant. Then, we have: 
\begin{equation*}
    \begin{aligned}
        \int_{\RR} \frac{f_{G_\star}(z)}{f_G(z)} \log\p{\frac{f_{G_\star}(z)}{f_G(z)}} f_G(z) dz - & \int_{\RR} \frac{f_{G_\star}(z)}{f_G(z)} f_G(z) dz \cdot \log\p{\int_{\RR} \frac{f_{G_\star}(z)}{f_G(z)} f_G(z) dz} \\
        \leq & C_{\text{LS}}(f_G) \int_{\RR}  \abs{\frac{\partial}{\partial z} \p{\sqrt{\frac{f_{G_\star}(z)}{f_G(z)}}}}^2 f_G(z) dz, 
    \end{aligned}
\end{equation*}
which implies, 
\begin{equation*}
    \begin{aligned}
        \DKL{f_{G_\star}}{f_G} & \leq  C_{\text{LS}}(f_G) \cdot \int_{\RR} \abs{ \frac{1}{2 \sqrt{\frac{f_{G_\star}(z)}{f_G(z)}}} \frac{\partial}{\partial z} \p{\frac{f_{G_\star}(z)}{f_G(z)}} }^2 f_G(z) dz \\
        & =  \frac{1}{4} C_{\text{LS}}(f_G) \int_{\RR} \abs{\frac{\partial}{\partial z} \p{\frac{f_{G_\star}(z)}{f_G(z)}}}^2 \p{\frac{f_G(z)}{f_{G_\star}(z)}}^2 f_{G_\star}(z) dz \\
        & =  \frac{1}{4} C_{\text{LS}}(f_G) \int_{\RR} \p{\frac{\partial}{\partial z} \log\p{\frac{f_{G_\star}(z)}{f_G(z)}}}^2 f_{G_\star}(z) dz \\
        & =  \frac{1}{4} C_{\text{LS}}(f_G) \int_{\RR} \p{ \frac{\partial}{\partial z} \log\p{f_{G_\star}(z)} - \frac{\partial}{\partial z} \log\p{f_{G}(z)} }^2 f_{G_\star}(z) dz \\ 
        & =  \frac{1}{4} C_{\text{LS}}(f_G) \Dfisher{f_{G_\star}}{f_G}.
    \end{aligned}
    \label{eq:prop9_second_ineq}
\end{equation*}
Thus,
\begin{equation*}
    \Dhel^2(f_{G_\star}, f_G) \leq \DKL{f_{G_\star}}{f_G} \leq \frac{1}{4} C_{\text{LS}}(f_G) \Dfisher{f_{G_\star}}{f_G}. 
\end{equation*}
\end{proof}

\end{document}